\documentclass[twoside,a4paper,reqno,11pt]{amsart}
\usepackage{amsfonts, amsbsy, amsmath, amssymb, latexsym,hyperref, mathtools, bold-extra, url}
\usepackage{mathrsfs,array}
\usepackage[top=24mm,right=28mm,bottom=24mm,left=28mm]{geometry}
\usepackage{stmaryrd}
\usepackage{bm}
\usepackage[pdftex]{color,graphicx}
\usepackage{longtable}
\usepackage[figuresright]{rotating}

\usepackage{chngcntr}
\counterwithin{table}{section}

\renewcommand{\a}{\alpha}
\renewcommand{\b}{\beta}
\newcommand{\normeq}{\trianglelefteqslant}

\newcommand{\br}{\genfrac{[}{]}{0pt}{}}

\newcommand{\Q}{\mathcal{Q}}
\newcommand{\F}{\mathbb{F}_{q}}
\newcommand{\Z}{\mathbb{Z}}
\newcommand{\e}{\epsilon}
 
 \renewcommand{\O}{\Omega}

 \renewcommand{\to}{\rightarrow}
 \newcommand{\s}{\sigma}
 
 \newcommand{\C}{\mathcal{C}}

\newcommand{\la}{\langle}
\newcommand{\ra}{\rangle}

\newcommand{\leqs}{\leqslant}
\newcommand{\geqs}{\geqslant}

 \newcommand{\vs}{\vspace{3mm}}

\newcommand{\cL}{\mathcal{L}}
\newcommand{\cP}{\mathcal{P}}
\newcommand{\cQ}{\mathcal{Q}}

\makeatletter
\newcommand{\imod}[1]{\allowbreak\mkern4mu({\operator@font mod}\,\,#1)}
\makeatother

\newtheorem{theorem}{Theorem}
\newtheorem*{theorem*}{Theorem}
\newtheorem*{conj*}{Conjecture}

\newtheorem{thm}{Theorem}[section]
\newtheorem{prop}[thm]{Proposition}
\newtheorem{lem}[thm]{Lemma}
\newtheorem{cor}[thm]{Corollary}

\newtheorem{hypothesis}[thm]{Hypothesis}

\theoremstyle{definition}
\newtheorem{rem}[thm]{Remark}
\newtheorem{remk}{Remark}

\newtheorem{defn}[thm]{Definition}
\newtheorem{ex}[thm]{Example}

\begin{document}

\author{Timothy C. Burness}
\address{T.C. Burness, School of Mathematics, University of Bristol, Bristol BS8 1UG, UK}
\email{t.burness@bristol.ac.uk}

\author{Tao Feng}
\address{School of Mathematical Sciences, Zhejiang University, Hangzhou 310027, Zhejiang, People's Republic of China}
\email{tfeng@zju.edu.cn}

\author{Jianbing Lu}
\address{College of Science, National University of Defense Technology, Changsha 410073, Hunan, People's Republic of China}
\email{jianbinglu@nudt.edu.cn}

\title[Primitive flag-transitive generalized quadrangles]{Primitive flag-transitive generalized quadrangles}

\dedicatory{\rm Dedicated to Professor Hendrik Van Maldeghem on the  occasion of his retirement}

\begin{abstract}
Let $\mathcal{Q}$ be a finite thick generalized quadrangle of order
$(s,t)$ with a flag-transitive automorphism group $G$. A
well-known conjecture of Kantor from the early 1990s
asserts that up to duality, $\mathcal{Q}$ is either of classical type,
or it is one of two sporadic examples of order $(3,5)$ or
$(15,17)$. Although this fundamental problem in finite
geometry remains open, it has been resolved in some special cases. For example, a theorem of Buekenhout and Van Maldeghem from 1994 proves Kantor's conjecture under the stronger assumption that $G$ acts distance-transitively on the set of points of $\mathcal{Q}$.

In this paper, we study Kantor's conjecture in the setting where $G$
acts primitively on both the points and lines of $\mathcal{Q}$. Here a theorem of Bamberg et al.\ implies that $G$ is almost simple
with socle $X$ a finite simple group of Lie type, and several special cases where the Lie rank of $X$ is at most $2$ have recently been settled in a sequence of papers. Our main theorem
resolves the problem for all almost simple groups of Lie type, thereby completing the proof of Kantor's conjecture
for generalized quadrangles with a flag-transitive automorphism group acting primitively on points and lines.
\end{abstract}

\date{\today}

\maketitle

\setcounter{tocdepth}{1}
\tableofcontents

\section{Introduction}\label{s:intro}

Let $n \geqs 3$ be an integer. A finite \emph{generalized $n$-gon} $\mathcal{Q}$ is a finite point-line incidence geometry whose bipartite incidence graph has diameter $n$ and girth $2n$. For example, if $n=3$ then $\mathcal{Q}$ is a projective plane. And for $n=4$ we refer to $\mathcal{Q}$ as a \emph{generalized quadrangle}, or GQ for short.

A finite generalized $n$-gon $\mathcal{Q}$ is \emph{thick} if each line contains at least three points, and each point lies on at least three lines. In this situation, there exist constants $s,t \geqs 2$ such that every line is incident with $s+1$ points, and every point is incident with $t+1$ lines (see \cite[Corollary 1.5.3]{GenPoly}); we say that $\mathcal{Q}$ has \emph{order} $(s,t)$. The point-line dual of a generalized $n$-gon of order $(s,t)$ is a generalized $n$-gon of order $(t,s)$. By a fundamental theorem of Feit and Higman \cite{Feit}, a finite thick generalized $n$-gon exists only for $n=3,4,6$ or $8$.

Generalized $n$-gons were first introduced by Jacques Tits
\cite{Tits} in the 1950s in order to provide a geometric framework
for studying the finite simple groups of Lie type, revealing a new
bridge between geometry and group theory. In the language of buildings, the finite thick generalized $n$-gons are precisely the finite thick irreducible spherical buildings of rank
two of Coxeter type $I_2(n)$. If a group $G$ has a spherical
$BN$-pair of rank two, then the points and lines of the associated
generalized polygon may be identified with the cosets of the two
standard maximal parabolic subgroups of $G$, with two cosets incident if
and only if they have nonempty intersection; see
\cite{TitsBuildings}. We refer the reader to the monographs \cite{Payne,GenPoly} for further background on generalized polygons and generalized quadrangles.

Let $\cQ = (\cP,\cL)$ be a generalized $n$-gon with point set $\cP$ and line set $\cL$. An \emph{automorphism} (or \emph{collineation}) of $\mathcal{Q}$ is a permutation of $\cP$ that also induces a permutation of $\cL$ in such a way that the incidence relation is preserved. An \emph{automorphism group} of $\mathcal{Q}$ is a group of automorphisms, and its full automorphism group is denoted by ${\rm Aut}(\mathcal{Q})$. The automorphism groups of the so-called \emph{classical} generalized polygons are almost simple groups of Lie type (we will say more about the classical generalized quadrangles in Section \ref{ClassGQ}). In addition, the latter automorphism groups act primitively on both points and lines, and they also act transitively on the set of \emph{flags} of the geometry (i.e. on the set of incident point-line pairs). Furthermore, these automorphism groups act distance-transitively on points with respect to the associated incidence graph.
	
The above observations lead us naturally to the following question: to what extent do these symmetry properties characterize the classical generalized polygons? This problem has been studied by many authors for more than $50$ years, with a particular focus on generalized quadrangles.

In his classic text \cite{GenPoly} on generalized polygons from the late 1990s, Van Maldeghem records the classification of finite flag-transitive generalized polygons as Problem 7 in his list of \textit{famous open problems} (see \cite[Appendix E]{GenPoly}). For finite generalized $3$-gons, i.e., projective planes, flag-transitivity implies primitivity on both points and lines by \cite[Proposition 3]{HigMc}. And by a theorem of Kantor \cite[Theorem B]{PtPrimPlane}, a finite point-primitive projective plane is either classical, or belongs to a hypothetical family which is conjectured not to exist. However, for finite generalized $n$-gons with $n\geqs 4$, flag-transitivity no longer implies primitivity on points or lines, and we will see such examples in the statement of Kantor's conjecture given below.

Returning to the general problem posed above, one of the first major results in this direction is a theorem of Fong and Seitz \cite{FS} from the early 1970s, which classifies the finite generalized polygons satisfying the so-called \emph{Moufang condition}. Later, in the 1990s, Buekenhout and Van Maldeghem \cite{Buekenhout1994} determined the finite thick generalized polygons admitting an automorphism group that acts distance-transitively on the set of points. For generalized quadrangles, their main result states that such a GQ is either classical, dual classical, or the unique generalized quadrangle $\mathrm{GQ}(3,5)$ of order $(3,5)$. More recently, Bamberg, Li and Swartz have classified the antiflag-transitive generalized quadrangles in \cite{BLS2018} (an \emph{antiflag} is a non-incident point-line pair) and the locally $2$-transitive generalized quadrangles in \cite{BLS2021} (in the latter situation, an automorphism group of $\cQ$ acts transitively on ordered pairs of collinear points and ordered pairs of concurrent lines).

It is worth noting that both antiflag-transitive and locally $2$-transitive generalized quadrangles are locally primitive (this means that the stabilizer of a vertex in the associated incidence graph is primitive on its neighbours), and it is clear that a locally primitive generalized quadrangle is flag-transitive. Locally primitive graphs form a rich branch of algebraic graph theory (see \cite{LocPrimGraph,Weiss}, for example), so a complete classification of locally primitive generalized quadrangles is a promising direction. In their recent paper \cite{FTh}, Feng and Thas outline a Thompson-Wielandt type approach to the classification of such quadrangles.

One of the main motivations for this earlier work stems from the following well-known and widely studied conjecture of Kantor, which first appeared in print in 1991 \cite{Kantor}.

\begin{conj*}[Kantor \cite{Kantor}]
Any finite thick generalized quadrangle with a flag-transitive automorphism group is either classical, or $\mathrm{GQ}(3,5)$, or it is the generalized quadrangle of order $(15,17)$ arising from the Lunelli-Sce hyperoval, up to duality.
\end{conj*}

Here it is worth noting that the two non-classical GQs appearing in Kantor's conjecture admit flag-transitive automorphism groups that act primitively on points and imprimitively on lines, up to duality.

Despite the fact that Kantor's conjecture has attracted significant attention for more than $30$ years, it still remains out of reach. However, progress has been possible under the assumption of some additional symmetry conditions.
In this paper, through the introduction of new algebraic, geometric, arithmetic and computational techniques, we prove Kantor's conjecture under a natural primitivity hypothesis.

\begin{theorem}\label{c:main}
Let $\mathcal{Q}$ be a finite thick generalized quadrangle with a flag-transitive, point-primitive and line-primitive automorphism group. Then $\mathcal{Q}$ is classical.
\end{theorem}

So in addition to flag-transitivity, we are assuming that $\mathcal{Q}$ admits an automorphism group $G$ that acts primitively on both points and lines. Here one of the foundational results in this direction is the main theorem of  \cite{Bamberg2012} by Bamberg et al., which shows that $G$ must be \emph{almost simple}. This means that $G$ has a unique minimal normal subgroup $X$, which is a nonabelian simple group, so we have
\[
X \normeq G \leqs {\rm Aut}(X)
\]
and we refer to $X$ as the \emph{socle} of $G$. In the same paper, the authors show that $X$ is not a sporadic simple group. In addition, under the assumption that $G$ is flag-transitive and point-primitive, they prove that if $X$ is an alternating group, then $G = A_6$ or $S_6$ and $\cQ$ is the unique generalized quadrangle of order $(2,2)$, which is the classical GQ denoted by $\mathrm{W}(3,2)$ in Table \ref{Classical} (see \cite[Theorems 1.1 and 1.2]{Bamberg2012}).

Therefore, to prove Theorem \ref{c:main} it just remains to
consider the GQs admitting an almost simple automorphism group whose socle $X$ is a finite simple group of Lie type. Here several recent papers have resolved the problem for certain low-rank groups. Indeed, the cases where $X=\mathrm{L}_2(q)$, $\mathrm{U}_3(q)$ and $\mathrm{PSp}_4(q)'$ are treated in \cite{Feng,Lu,Ou}, respectively; the groups with socle $X={}^2B_2(q)$ or ${}^2G_2(q)$ are handled in \cite{Arumugam}, and the case $X={}^2F_4(q)'$ is dealt with in \cite{Arumugam_thesis}.

In this paper, we classify all the finite thick generalized quadrangles that admit a flag-transitive, point-primitive and line-primitive automorphism group whose socle is a finite simple group of Lie type, without imposing any conditions on the rank of the group, thereby completing the proof of Theorem \ref{c:main}.

\begin{theorem}\label{t:main}
Let $\mathcal{Q}$ be a finite thick generalized quadrangle with an automorphism group $G$ whose socle is a finite simple group of Lie type. Assume $G$ is flag-transitive and primitive on both points and lines. Then $\mathcal{Q}$ is a classical generalized quadrangle.
\end{theorem}

\begin{remk}\label{r:main}
Let $\cQ$ be a finite thick generalized quadrangle and let $G$ be an almost simple group of automorphisms of $\cQ$. Assume the socle $X$ of $G$ is a classical group with natural module $V$ and set
$n = \dim V$.
\begin{itemize}\addtolength{\itemsep}{0.2\baselineskip}
\item[{\rm (a)}] Under the weaker assumption that $G$ is flag-transitive and point-primitive (but not necessarily line-primitive), our proof of Theorem \ref{t:main} shows that $n \leqs 80$.
\item[{\rm (b)}] It is also worth noting that certain cases arising in the proof of Theorem \ref{t:main} can be handled without assuming flag-transitivity. For example, as a consequence of Theorem \ref{t:subspace}, we deduce that if $G$ is a point-primitive classical group and a point stabilizer $G_{\a}$ is a maximal parabolic subgroup, then $\cQ$ is a classical GQ.
\end{itemize}
\end{remk}

\begin{remk}\label{r:earlier}
In view of the earlier work highlighted above, in the proof of Theorem \ref{t:main} we exclude the groups with socle one of the following:
\[
{\rm L}_2(q),\; {\rm U}_3(q), \; {\rm PSp}_4(q)', \; {}^2B_2(q), \; {}^2G_2(q)', \; {}^2F_4(q)'.
\]
However, it is worth noting that our proof is independent of the main theorem of \cite{Buekenhout1994} on distance-transitive generalized quadrangles.
\end{remk}

\begin{remk}
Our assumption in Theorem \ref{c:main} that $\mathcal{Q}$ admits an automorphism group that is both point-primitive and line-primitive means that we are imposing
global transitivity conditions on $\cQ$, in contrast to the local transitivity assumptions in \cite{BLS2018,BLS2021}. Of course, if we drop line primitivity, then the two non-classical GQs in Kantor's conjecture will arise. This indicates that additional techniques and insight will be needed to classify flag-transitive and point-primitive generalized quadrangles, which is a direction we intend to pursue in future work.
\end{remk}

For the remainder of the introduction, we briefly outline the main steps in the proof of Theorem \ref{t:main}. The proof combines several new techniques that we anticipate will also be useful in future work. So let $\cQ = (\cP,\cL)$ be a finite thick GQ of order $(s,t)$ and let $G$ be an almost simple group of automorphisms of $\cQ$ with socle $X$, a finite simple group of Lie type over $\mathbb{F}_q$. Write $q = p^f$, where $p$ is a prime.

The main bulk of the work concerns the case where $X$ is a classical group, so let us make this assumption (we will comment briefly below on the case where $X$ is an exceptional group of Lie type). Let $V$ be the natural module for $X$ and set $n = \dim V$ and
\[
\theta=|{\rm Out}(X)|=\frac{|{\rm Aut}(X)|}{|X|}.
\]
By combining the earlier work highlighted above with the existence of certain isomorphisms between some of the low-dimensional classical groups (see \cite[Proposition 2.9.1]{KL}, for example), we may assume that $X$ is one of the following, where we adopt the standard notation from \cite{KL} for classical groups:
\[
{\rm L}_n(q) \, (n \geqs 3), \; {\rm U}_n(q) \, (n \geqs 4), \; {\rm PSp}_n(q) \, (n \geqs 6), \; {\rm P\O}_n^{\e}(q) \, (n \geqs 7).
\]

We assume throughout that $G$ is point-primitive, which means that each point stabilizer $G_{\a}$ is a core-free maximal subgroup of $G$ (here core-free simply means that $G_{\a}$ does not contain $X$, which is the unique minimal normal subgroup of $G$). The possibilities for $G_{\a}$ are described by a theorem of Aschbacher \cite{Aschbacher}, which partitions the core-free maximal subgroups of $G$ into nine separate collections, denoted $\C_1, \ldots, \C_8$ and $\mathcal{S}$. Here the subgroups in each collection $\C_i$ are defined in terms of the geometry of $V$. For example, $\C_1$ includes the stabilizers of appropriate subspaces of $V$, while $\C_2$ comprises the stabilizers of certain direct sum decompositions of $V$. For this reason, we refer to each subgroup in $\C_1 \cup \cdots \cup \C_8$ as a \emph{geometric subgroup}. The remaining maximal subgroups $H \in \mathcal{S}$ are almost simple and embedded in $G$ via an absolutely irreducible representation of a suitable covering group of $H \cap X$. We will say more about Aschbacher's theorem in Section \ref{sss:sub}.

Our proof of Theorem \ref{t:main} for classical groups proceeds via a sequence of reductions based on Aschbacher's framework. In Section \ref{s:subspace}, we begin by handling the cases where $G_{\a}$ is a so-called \emph{subspace subgroup}. Here we can typically identify the set of points $\cP$ with a certain collection of subspaces (or pairs of subspaces) of $V$, which allows us to apply a combination of different techniques in order to obtain the desired reduction. This includes working with fixed points, as described below, together with specialized arithmetic, computational and geometric techniques, some of which rely on detailed information concerning the subdegrees for the action of $G$ on $\cP$.

For the remaining possibilities, we work initially with a useful inequality from Lemma \ref{large}, which tells us that
\begin{equation}\label{e:11}
|X_{\a}| > \theta^{-5/6}|X|^{1/6}.
\end{equation}
This observation immediately allows us to eliminate many cases where $G_{\a}$ is a small order maximal subgroup of $G$ and this phase of the reduction is presented in Section \ref{s:irred1}. For example, if $X = {\rm L}_n(q)$ and $G_{\a} \in \C_3$ is a field extension subgroup of type ${\rm GL}_{n/k}(q^k)$, where $k$ is a prime divisor of $n$, then we can use \eqref{e:11} to show that $k \leqs 5$. However, this approach does not allow us to obtain any restrictions on $n$ or $q$ in these remaining cases.

Next, in Sections \ref{s:irred2} and \ref{s:irred3}, we seek further restrictions on $G$ and $G_{\a}$ by considering fixed points. Given $g \in G$, we write $\cP_g = \{ \a \in \cP \,:\, \a^g = \a\}$ for the set of fixed points of $g$ on $\cP$. By appealing to Lemma \ref{fixednumber}, if we can identify an element $1 \ne g \in G$ with
\begin{equation}\label{e:12}
|\cP_g| \geqs |\cP|^{4/5},
\end{equation}
then the candidate pair $(G,G_{\a})$ is eliminated.

Given an element $g \in X$, we will compute (or estimate) $|\cP_g|$ by using the fact that
\[
|\cP_g| = |\cP| \cdot \frac{|g^G \cap G_{\a}|}{|g^G|} = [X:X_{\a}] \cdot  \frac{|g^X \cap X_{\a}|}{|g^X|},
\]
where $g^L$ denotes the conjugacy class of $g$ in $L$ (see Lemma \ref{l:fpr}). For example, if $X = {\rm L}_n(q)$ and $G_{\a} \in \C_3$ is of type ${\rm GL}_{n/2}(q^2)$, then by taking $g \in X$ to be a unipotent element with Jordan form $(J_2^2,J_1^{n-4})$ on $V$ (where $J_i$ denotes a standard unipotent Jordan block of size $i$) we can show that \eqref{e:12} holds if $n \geqs 20$, or if $n = 18$ and $q \geqs 4$. So this allows us to assume that $n \leqs 16$ or $(n,q) = (18,2)$, $(18,3)$ in this case, but this method does not yield any additional restrictions on $q$.

\begin{remk}\label{r:method}
Several cases arise where we can use the inequalities in \eqref{e:11} or \eqref{e:12} to reduce the possibilities for a candidate pair $(G,G_{\a})$ to finitely many values of $q$ (with $n$ fixed). Each remaining value of $q$ can then be eliminated by showing that there are no feasible values for the order $(s,t)$ of $\cQ$.

For example, suppose $X = {\rm PSp}_6(q)$ and $G_{\a} \in \mathcal{S}$ has socle ${\rm U}_3(3)$. By inspecting \cite[Table
8.29]{BHR} we see that $q=p$ and either $X_{\a} = {\rm U}_3(3)$ and $q \equiv \pm 7, \pm 17, \pm 19, \pm 29 \imod{60}$, or $X_{\a} = {\rm U}_3(3).2$ and either $q=2$ or $q \equiv \pm 1 \imod{12}$. It is straightforward to check that \eqref{e:11} fails unless $q \in \{2,7,11,13\}$. Suppose $q=2$ and $\cQ$ has order $(s,t)$. Then Lemma \ref{HD} gives
\[
|\cP| = [X:X_{\a}] = (s+1)(st+1) = 120
\]
and thus $s\in\{2,3,4,5,7\}$ since $s+1$ divides $120$ and $2s^2 \leqs ts^2 < 120$. It is now easy to check that $(s,t) = (7,2)$ is the only solution, but this can be ruled out because Higman's inequality $s \leqs t^2$ is violated (see Lemma \ref{HD}(ii)).

We refer the reader to Remarks \ref{r:c1} and \ref{rem_subdegBound} for a more in-depth discussion of some of the arithmetic techniques we will use in this paper.
\end{remk}

By point-line duality, our initial reduction on the possibilities for $G_{\a}$ applies equally well to the stabilizer $G_{\ell}$ of a line $\ell \in \cL$. So in this way, the proof of Theorem \ref{t:main} for classical groups is reduced to a collection of candidate triples $(G,G_{\a},G_{\ell})$ and our goal in Section \ref{s:main1} is to eliminate all of these remaining possibilities. In each case, $G_{\a}$ and $G_{\ell}$ are core-free maximal subgroups of $G$, which are non-conjugate in $G$ by Lemma \ref{isomorphism}(i). In addition, we have $n \leqs 80$, as noted above in Remark \ref{r:main}. And for every candidate triple we can compute
\begin{equation}\label{e:params}
|X|, \; |X_{\a}|, \; |X_{\ell}|, \; v = [X:X_{\a}], \, b = [X:X_{\ell}], \; \theta = |{\rm Out}(X)|.
\end{equation}

The remaining candidate triples $(G,G_{\a},G_{\ell})$ will be eliminated in two phases. First, in Section \ref{ss:comp}, we present a sequence of computational methods, which rely on some of the arithmetic conditions that must be satisfied by the defining parameters in \eqref{e:params} (for example, see Lemmas \ref{HD} and \ref{inequalities}). These techniques, which can be implemented in {\sc Magma} \cite{magma}, allow us to eliminate many, but not all, of the remaining possibilities for $(G,G_{\a},G_{\ell})$. We refer the reader to Remark \ref{rem:computational-strategy} for a description of our main computational techniques, and to Tables \ref{tab:E1} and \ref{tab:E2} for the list of surviving triples, which comprise the collection $\mathcal{E}$.

So to complete the proof of Theorem \ref{t:main} for classical groups, it just remains for us to eliminate every triple in $\mathcal{E}$. To do this, we develop further algebraic and arithmetic techniques in Section \ref{ss:further}, some of which rely on leveraging the number of fixed points of certain elements in $X$ on both points and lines. We then apply these methods in Section \ref{ss:comp_class} to eliminate all of the triples in $\mathcal{E}$, thereby completing the proof for classical groups.

Finally, in Section \ref{s:excep}, we complete the proof of Theorem \ref{t:main} by handling the case where $G$ is an almost simple exceptional group of Lie type. Our starting point is the fact that through the work of many authors, spanning several decades, we now have an almost complete description of the maximal subgroups of $G$, up to conjugacy. And in view of the lower bound in \eqref{e:11}, together with the analogous bound for the line stabilizer, we can quickly reduce the problem to a more restrictive list of candidate triples, where in each case we can compute the associated data in \eqref{e:params}. At this point, we can then apply the computational methods presented in Remark \ref{rem:computational-strategy} to further reduce the list of possibilities to the cases comprising the collection $\mathcal{F}$, which is presented in Table \ref{tab:F}. This reduction process is presented in Section \ref{ss:ex_red} and the proof is then completed in Section \ref{ss:ex_complete}, where we use a combination of the techniques in Section \ref{ss:further} to eliminate all of the triples in $\mathcal{F}$.

\vs

\noindent \textbf{Acknowledgements.} Tao Feng is supported by the National Key Research and Development Program of China (Grant No. 2025YFA1017700) and the National Natural Science Foundation of China (Grant No. 12225110). Jianbing Lu is supported by the National Natural Science Foundation of China (Grant No. 12501467) and the Innovation Research Foundation of the College of Science at the National University of Defense Technology (Grant No. 202501-YJRC-LXY-02). Tim Burness thanks the School of Mathematical Sciences at Zhejiang University for their generous hospitality during a research visit in June 2026.

\vs

\noindent \textbf{AI statement.} None of the results or proofs in this paper were obtained through the use of AI tools. In addition, the manuscript has been written without any AI assistance, except for proofreading and looking up references with the aid of ChatGPT 6.1. In particular, reference \cite{Whittaker} was suggested by ChatGPT.

\section{Preliminaries}\label{s:prel}

In this section we present several preliminary results that we will need in the proof of Theorem \ref{t:main}. We begin by recalling some general properties of generalized quadrangles in Section \ref{ss:GQ}. This includes the inequalities presented in Lemmas \ref{fixednumber} and \ref{large}(ii), which play a key role in our reduction strategy for classical groups in Sections \ref{s:subspace}-\ref{s:irred3}. Next, in Section \ref{ClassGQ}, we introduce the classical generalized quadrangles and we briefly discuss Kantor's conjecture on the classification of flag-transitive GQs. The finite classical groups are the main focus of Section \ref{ss:class}. Here we fix our notation and we recall some standard order estimates, which we will use repeatedly throughout the paper. We also provide an overview of Aschbacher's theorem \cite{Aschbacher} on the maximal subgroups of classical groups, and we briefly discuss conjugacy classes, with a view towards computing fixed point estimates via Lemma \ref{l:fpr}. And in Section \ref{ss:ex} we comment briefly on the finite exceptional groups of Lie type and their maximal subgroups.

\subsection{Generalized quadrangles}\label{ss:GQ}
	
Let $\mathcal{Q} = (\mathcal{P},\mathcal{L})$ be a finite thick generalized quadrangle (GQ for short) of order $(s,t)$ with point set $\mathcal{P}$ and line set $\mathcal{L}$. Recall that $s,t \geqs 2$ and every line contains exactly $s+1$ points and every point lies on exactly $t+1$ lines. Let us also recall that every generalized quadrangle satisfies the following additional condition:
\[
\begin{array}{ll}
\mbox{\emph{GQ Axiom.}} & \mbox{For every non-incident point-line pair $(\a,\ell)$ there exists a} \\
& \mbox{unique point $\beta$ on $\ell$ that is collinear with $\a$.}
\end{array}
\]

Throughout this paper, we set
\[
v = |\mathcal{P}|,\;\; b = |\mathcal{L}|.
\]
We begin by recording the following lemma, which brings together several well-known results.

\begin{lem}\label{HD}
Let $\mathcal{Q} = (\mathcal{P},\mathcal{L})$ be a finite thick generalized quadrangle of order $(s,t)$. Then the following hold:
\begin{itemize}\addtolength{\itemsep}{0.2\baselineskip}
\item[{\rm (i)}] Any two distinct points are incident with at most one common line, and any two distinct lines are incident with at most one common point.
\item[{\rm (ii)}] Higman's inequalities: $s\leqs t^{2}$ and $t\leqs s^{2}$.
\item[{\rm (iii)}] Divisibility condition: $s+t$ divides $st(s+1)(t+1)$.
\item[{\rm (iv)}] We have $v=(s+1)(st+1)$ and $b=(t+1)(st+1)$.
\item[{\rm (v)}] In addition, $v^4<b^{5}$, $b^4<v^{5}$ and
\[
v^{2}/b<(s+1)^{3}<2v^{2}/b,\;\; b^{2}/v<(t+1)^{3}<2b^{2}/v,\;\; vb/2<(st+1)^{3}<vb.
\]
\end{itemize}
\end{lem}

\begin{proof}
Part (i) is \cite[Lemma 1.4.1]{GenPoly}. For parts (ii)-(iv), we refer the reader to \cite[1.2.1-1.2.3]{Payne}.

Finally, let us consider the inequalities in (v). First, observe that
\[
(x^{1/2}+1)^5(x^{3/2}+1)>(x+1)^4
\]
for every real number $x>1$, so Higman's inequality $t \geqs s^{1/2}$ implies that
\[
(t+1)^5(st+1)\geqs (s^{1/2}+1)^5(s^{3/2}+1)>(s+1)^4.
\]
Then by appealing to (iv) we get
\[
b=(t+1)(st+1)>\left(\frac{s+1}{t+1}\right)^4=
\left(\frac{v}{b}\right)^4
\]
and thus $v^4<b^{5}$. Dually, we deduce that $b^4<v^{5}$. In addition,
\[
v=(s+1)(st+1)<(s+1)^{2}(t+1)<2(s+1)(st+1)=2v
\]
and
\[
(s+1)^{2}(t+1)=(s+1)^{3}\left(\frac{t+1}{s+1}\right)=(s+1)^{3}\left(\frac{b}{v}\right).
\]
This yields $v^{2}/b<(s+1)^{3}<2v^{2}/b$ and the dual inequalities $b^{2}/v<(t+1)^{3}<2b^{2}/v$ are proved in the same way. The final statement in (v) now follows since $st+1=v/(s+1)$.
\end{proof}

\begin{defn}\label{d:feas}
Let $(s,t)$ be a pair of integers with $s,t \geqs 2$. We say that $(s,t)$ is \emph{feasible} if all of the arithmetic conditions in Lemma
\ref{HD} are satisfied. In other words, it is feasible that $(s,t)$ is the order of a finite thick generalized quadrangle $\cQ = (\cP,\cL)$ with $v = |\cP|$ and $b = |\cL|$.
\end{defn}

In order to state our next result, we need to introduce some additional notation and terminology. Given $\mathcal{Q} = (\mathcal{P},\mathcal{L})$ as above and an automorphism $g \in {\rm Aut}(\mathcal{Q})$, we write $\mathcal{P}_{g}$ (respectively, $\mathcal{L}_g$) for the set of points (respectively, lines) fixed by $g$. In addition, we write $\mathcal{Q}_{g}=(\cP_g,\cL_g)$ for the induced incidence substructure on $\mathcal{P}_{g} \times \mathcal{L}_g$.

To describe the possibilities for $\mathcal{Q}_{g}$, let us recall that a \emph{grid} with parameters $(s_{1},s_{2})$ is a point-line incidence structure $(\mathcal{P}', \mathcal{L}')$ with
\[
\mathcal{P}'= \{\a_{i, j} \, : \, 0 \leqslant i \leqslant s_{1}, 0 \leqslant j \leqslant s_{2} \},\;\;  \mathcal{L}'= \{\ell_{i},\ell_j^{\prime} \,: \, 0 \leqs i \leqs s_1, \, 0 \leqs j \leqs s_2 \}
\]
such that $\a_{i, j}, \ell_{k}$ are incident if and only if $i=k$, and $\a_{i, j}, \ell_{k}^{\prime}$ are incident if and only if $j=k$. Similarly, a \emph{dual grid} with parameters $(s_1,s_2)$ is the point-line dual of a grid with parameters $(s_2,s_1)$. A grid with parameters $(s,s)$ is a generalized quadrangle of order $(s,1)$, and a dual grid with parameters $(t,t)$ is a generalized quadrangle of order $(1,t)$.

\begin{lem}\label{subquadrangle}
Let $\mathcal{Q} = (\mathcal{P},\mathcal{L})$ be a finite thick generalized quadrangle of order $(s,t)$ and let $1 \ne g \in {\rm Aut}(\mathcal{Q})$. Then one of the following holds:
\begin{itemize}\addtolength{\itemsep}{0.2\baselineskip}
\item[$(0)$] $|\mathcal{P}_{g}|=|\mathcal{L}_{g}|=0$.

\item[$(1)$] $|\mathcal{L}_{g}|=0$ and $\mathcal{P}_{g}$ is a set of pairwise non-collinear points with $1 \leqs |\mathcal{P}_{g}|\leqs st+1$.

\item[$(1')$] $|\mathcal{P}_{g}|=0$ and $\mathcal{L}_{g}$ is a set of pairwise non-concurrent lines with $1 \leqs |\mathcal{L}_{g}|\leqs st+1$.

\item[$(2)$] $1 \leqs |\mathcal{P}_{g}|\leqs s(t+1)+1$, $1 \leqs |\mathcal{L}_{g}|\leqs t+1$ and $\mathcal{P}_{g}$ contains a point that is collinear with each point in $\mathcal{P}_{g}$ and is on each line in $\mathcal{L}_{g}$.

\item[$(2')$] $1 \leqs |\mathcal{P}_{g}|\leqs s+1$, $1 \leqs |\mathcal{L}_{g}|\leqs t(s+1)+1$ and $\mathcal{L}_{g}$ contains a line that is concurrent with each line in $\mathcal{L}_{g}$ and contains each point in $\mathcal{P}_{g}$.

\item[$(3)$] $\mathcal{Q}_{g}$ is a grid with parameters $(s_1,s_2)$, $s_{1}<s_{2}$, $|\mathcal{P}_{g}|=(s_{1}+1)(s_{2}+1)<(s+1)^{2}$, $|\mathcal{L}_{g}|=(s_{1}+1)+(s_{2}+1)<2(s+1)$.
	
\item[$(3')$] $\mathcal{Q}_{g}$ is a dual grid with parameters $(s_1,s_2)$, $s_{1}<s_{2}$, $|\mathcal{P}_{g}|=(s_{1}+1)+(s_{2}+1)<2(t+1)$, $|\mathcal{L}_{g}|=(s_{1}+1)(s_{2}+1)<(t+1)^2$.

\item[$(4)$] $\mathcal{Q}_{g}$ is a generalized quadrangle of order $\left(s^{\prime}, t^{\prime}\right)$, $|\mathcal{P}_{g}|=(s^{\prime}+1)(s^{\prime}t^{\prime}+1)\leqs (s+1)(t+1)$, $|\mathcal{L}_{g}|=(t^{\prime}+1)(s^{\prime}t^{\prime}+1)\leqs (s+1)(t+1)$.
\end{itemize}
\end{lem}

\begin{proof}
The possible structural types for $\mathcal{Q}_g$ are described in \cite[Theorem 2.4.1]{Payne}. The upper bounds on $|\mathcal{P}_{g}|$ and $|\mathcal{L}_{g}|$ in each case are derived in the proof of  \cite[Lemma 2.2]{Bamberg Popiel}.
\end{proof}

As an application of Lemma \ref{subquadrangle}, we obtain the following key result (see \cite[Theorem 1.3]{Bamberg Popiel}). Note that the so-called \emph{classical} GQs  will be discussed in Section \ref{ClassGQ}.

\begin{lem}\label{fixednumber}
Let $\mathcal{Q} = (\mathcal{P},\mathcal{L})$ be a finite thick generalized quadrangle of order $(s,t)$ and let $1 \ne g \in {\rm Aut}(\mathcal{Q})$. Then either
\begin{itemize}\addtolength{\itemsep}{0.2\baselineskip}
\item[{\rm (i)}] $|\mathcal{P}_g| < |\mathcal{P}|^{4/5}$; or
\item[{\rm (ii)}] $\mathcal{Q}$ is the classical generalized quadrangle $\mathrm{Q}^{-}(5,2)$ of order $(2,4)$, with $|\mathcal{P}| = 27$ and $|\mathcal{P}_g| = 15$.
\end{itemize}
\end{lem}

\begin{rem}
In the special case arising in part (ii) of Lemma \ref{fixednumber}, we have
\[
G = {\rm Aut}(\mathcal{Q}) = {\rm O}_6^{-}(2) = \O_6^{-}(2).2 \cong {\rm U}_4(2).2
\]
and $G_{\a}$ is a $P_1$ parabolic subgroup (in other words, $G_{\a}$ is the stabilizer of a $1$-dimensional totally singular subspace of the natural $6$-dimensional module $V$ for $G$). In addition, $g \in G$ is an involution with Jordan form $(J_2,J_1^4)$ on $V$.
\end{rem}

To conclude our preliminary discussion on generalized quadrangles, we present one or two more specialized results that apply when $\mathcal{Q}$ admits a flag-transitive or almost simple group of automorphisms.

Recall that if $G$ is an almost simple group with socle $X$, then $X$ is a nonabelian simple group and it is the unique minimal normal subgroup of $G$. By identifying $X$ with its group of inner automorphisms, we have $X \normeq G \leqs {\rm Aut}(X)$ and we set
\[
\theta = |{\rm Out}(X)| = \frac{|{\rm Aut}(X)|}{|X|}
\]
for the order of the outer automorphism group of $X$.

\begin{lem}\label{isomorphism}
Let $\mathcal{Q}=(\mathcal{P},\mathcal{L})$ be a finite thick generalized quadrangle of order $(s,t)$ admitting a flag-transitive automorphism group $G$.
\begin{itemize}\addtolength{\itemsep}{0.2\baselineskip}
\item[{\rm (i)}] For each pair $(\a,\ell) \in \mathcal{P} \times \mathcal{L}$, the stabilizers $G_{\a}$ and $G_{\ell}$ are not $G$-conjugate.
\item[{\rm (ii)}] Let $\alpha \in \mathcal{P}$ and let $\mathcal{C}$ be the set of points $\beta \ne \alpha$ that are collinear with $\alpha$. Then $|\mathcal{C}| = s(t+1)$ and the length of each $G_{\a}$-orbit on $\mathcal{C}$ is divisible by $t+1$.
\end{itemize}
\end{lem}

\begin{proof}
First consider (i). Seeking a contradiction, let us assume that $G_{\alpha} = (G_{\ell})^g = G_{\ell^{g}}$ for some $g \in G$. If $\alpha$ is incident with $\ell^{g}$, then $G_{\alpha}=G_{\alpha}\cap G_{\ell^{g}}=:G_{\alpha,\ell^{g}}$ and the flag-transitivity of $G$ implies that $t+1=[G_{\alpha} : G_{\alpha,\ell^{g}}]$. But this forces $t +1= 1$, which is a contradiction. On the other hand, if $\alpha$ is not incident with $\ell^{g}$, then the GQ Axiom implies that there is a line $\ell_{1}$ through $\alpha$   concurrent with $\ell^{g}$. This means that $G_{\alpha}$ fixes $\ell_{1}$, so $G_{\alpha}=G_{\alpha,\ell_{1}}$ and once again this is incompatible with flag-transitivity.
	
Next let us turn to (ii). The point $\alpha$ lies on exactly $t+1$ lines, each of which contains $s$ points distinct from $\alpha$. In addition, distinct lines through $\alpha$ meet only at $\alpha$ (see Lemma \ref{HD}(i)), so we have $|\mathcal{C}|=s(t+1)$ as claimed (note that this does not require flag-transitivity). Now let $\Gamma$ be a $G_{\a}$-orbit on $\mathcal{C}$. For any line $\ell$ through $\alpha$, let $r$ be the number of points in $\Gamma$ that are incident with $\ell$. Note that the flag-transitivity of $G$ implies that $r$ is independent of the choice of $\ell$. By enumerating the flags of the form $(\gamma,\ell)$ in two different ways, where $\gamma\in\Gamma$ and $\ell$ is a line through $\alpha$, we deduce that $|\Gamma| = r(t+1)$ and the result follows.
\end{proof}

\begin{lem}\label{large}
Let $\mathcal{Q}=(\mathcal{P},\mathcal{L})$ be a finite thick generalized quadrangle of order $(s,t)$ with a flag-transitive automorphism group $G$, which is almost simple with socle $X$. Then for each $(\a,\ell) \in \mathcal{P} \times \mathcal{L}$, the following hold:
\begin{itemize}\addtolength{\itemsep}{0.2\baselineskip}
\item[{\rm (i)}] $t+1$ divides $\theta|X_{\alpha}|$ and $s+1$ divides $\theta|X_{\ell}|$.
\item[{\rm (ii)}] If $G$ is point-primitive, then $|X_{\alpha}|> \theta^{-5/6}|X|^{1/6}$. Similarly, $|X_{\ell}|> \theta^{-5/6}|X|^{1/6}$ if $G$ is line-primitive.
\end{itemize}
\end{lem}

\begin{proof}
Since $G$ is flag-transitive, $G_{\alpha}$ acts transitively on the set of $t+1$ lines through $\alpha$, so $t+1$ divides $|G_{\alpha}|$.  And since $[G_{\alpha} : X_{\alpha}]$ divides $\theta=|\mathrm{Out}(X)|$, it follows that $t+1$ divides $\theta|X_{\alpha}|$. Similarly, we deduce that  $s+1$ divides $\theta|X_{\ell}|$. This proves part (i).

Now assume $G$ is point-primitive, which implies that the normal subgroup $X$ of $G$ acts transitively on $\cP$. Using parts (ii) and (iv) in Lemma \ref{HD}, it follows that
\[
|\cP| = [X: X_{\alpha}]=(s+1)(st+1) \leqs (t^2+1)(t^3+1)<(t+1)^{5}
\]
and thus part (i) yields $[X:X_{\a}] < (\theta|X_{\alpha}|)^5$, which in turn implies that $|X_{\alpha}|^{6}>\theta^{-5}|X|$. Similarly, we get $|X_{\ell}|^{6}>\theta^{-5}|X|$ if $G$ is line-primitive.
\end{proof}

By combining the inequalities in Lemma \ref{HD}(v) with the divisibility conditions from Lemma \ref{large}(i), we obtain the following corollary.

\begin{cor}\label{gcd_bound}
Let $\mathcal{Q}=(\mathcal{P},\mathcal{L})$ be a finite thick generalized quadrangle of order $(s,t)$ with a flag-transitive automorphism group $G$, which is almost simple with socle $X$.
Then
\begin{itemize}\addtolength{\itemsep}{0.2\baselineskip}
\item[{\rm (i)}] $\left(v^{2}/b\right)^{1/3}<s+1 \leqs \gcd(v,\theta|X_{\ell}|)$; and
\item[{\rm (ii)}] $\left(vb/2\right)^{1/3}<st+1\leqs \gcd(v,b)$.
\end{itemize}
\end{cor}

\begin{rem}\label{r:c1}
As above, let $\mathcal{Q}=(\mathcal{P},\mathcal{L})$ be a finite thick generalized quadrangle of order $(s,t)$ with a point-primitive automorphism group $G$. Assume that $G$ is almost simple with socle $X$. We can use some of the properties recorded above to develop  arithmetic techniques for ruling out possible candidates for a point stabilizer $G_{\a}$. In the proof of Theorem \ref{t:main}, when we use such an approach to eliminate a given pair $(G,G_{\a})$, we will write that there are no \emph{feasible parameters} for $\cQ$ because we have ruled out all possibilities for the order $(s,t)$.

\begin{itemize}\addtolength{\itemsep}{0.2\baselineskip}
\item[{\rm (a)}] First recall that
\[
v = |\cP| = [X:X_{\alpha}] = (s+1)(st+1),
\]
so there exists a divisor $r \geqs 3$ of $v$ such that
\[
s = r-1 \; \mbox{ and } \; t = \frac{v/r-1}{r-1}.
\]
In addition, the integers $s$ and $t$ must satisfy Higman's inequalities $s \leqs t^2$ and $t \leqs s^2$, and $st(s+1)(t+1)$ must be divisible by $s+t$ (see Lemma \ref{HD}).

So if we can determine the set of divisors of $v=[X:X_{\a}]$ (with the aid of {\sc Magma} \cite{magma}, for example) and we can show that there is no divisor $r$ satisfying all of the above conditions, then we have reached a contradiction. In particular, the candidate pair $(G,G_{\a})$ is eliminated.

\item[(b)] The previous method relies on determining the set of divisors of $v = [X:X_{\a}]$, which in some cases may be a prohibitively large integer (for instance, see Example \ref{ex:c1} below). Here we present an alternative approach.

Assume $G$ is flag-transitive and point-primitive. Higman's inequality $s \leqs t^2$ implies that
\[
v = [X:X_{\a}] = (s+1)(st+1) \leqs (t^2+1)(t^3+1) < T^5,
\]
where $T=t+1$, so we have
\[
v=(s+1)(s(T-1)+1) =(T-1)s^2+Ts+1
\]
and thus $s$ is a root of the quadratic equation
\[
(T-1)x^2+Tx+1-v=0.
\]
More precisely,
\[
s=\frac{-T+\sqrt{T^2+4(T-1)(v-1)}}{2(T-1)},
\]
which means that the discriminant
$\Delta = T^2+4(T-1)(v-1)$
must be a perfect square, and the given expression for $s$ must be an integer at least $2$. And as before, we also have the inequalities $s \leqs t^2$ and $t \leqs s^2$, and we know that $s+t$ divides $st(s+1)(t+1)$.

Now Lemma \ref{large}(i) implies that $T$ divides $w=|{\rm Out}(X)||X_{\a}|$. So if we can show that there is no divisor $T$ of $w$ satisfying all of the above conditions, then we can eliminate the candidate pair $(G,G_{\a})$. In practice, we do not need to construct the entire set of divisors of $w$ to do this. Indeed, since we are only interested in the divisors $T$ of $w$ satisfying the bound $T^5>v$, we can work with the prime factorization of $w$ to identify the set $\mathcal{D}$ of divisors $d$ of $w$ such that $d \leqs \lfloor wv^{-1/5} \rfloor$. And then it just remains to inspect the integers in the set $\{w/d \,:\, d \in \mathcal{D}\}$.
\end{itemize}
\end{rem}

\begin{ex}\label{ex:c1}
To illustrate the above methods, let us consider the case where $G$ is point-primitive and flag-transitive, with socle $X = {\rm P\O}_{38}^{-}(3)$. Set $\theta = |{\rm Out}(X)| = 8$ and assume $G_{\a}$ is a $\C_2$-subgroup of type ${\rm O}_1(3) \wr S_{38}$ (we will discuss $\C_i$-subgroups and Aschbacher's theorem for finite classical groups in Section \ref{sss:sub}). Here
\[
|X_{\alpha}| = 2^{36}|A_{38}| > \theta^{-5/6}|X|^{1/6},
\]
so we cannot rule out $(G,G_{\a})$ by appealing to Lemma \ref{large}(ii). In addition, we observe that $v = [X:X_{\alpha}] > 10^{279}$ is too large for {\sc Magma} to be able to determine the set of divisors of $v$, which means that the method in Remark \ref{r:c1}(a) is ineffective.

So let us turn to the approach presented in Remark \ref{r:c1}(b). Here $w = \theta|X_{\a}| = 2^{38}(38!)$ and $\lfloor wv^{-1/5}\rfloor = 1$, so $T = w$ is the only divisor of $w$ with $T^5>v$.  But then we check that
\[
T^2+4(T-1)(v-1) \equiv 7 \imod{17},
\]
which implies that $T^2+4(T-1)(v-1)$ is a nonsquare. This is a contradiction and so we can conclude that there is no finite thick generalized quadrangle admitting a flag-transitive
and point-primitive almost simple automorphism group $G$ with socle $X = {\rm P\O}_{38}^{-}(3)$ and point stabilizer $G_{\a} \in \C_2$ of type ${\rm O}_1(3) \wr S_{38}$.
\end{ex}

\subsection{Classical GQs}\label{ClassGQ}

Let $\mathcal{Q}=(\mathcal{P},\mathcal{L})$ be a finite thick generalized quadrangle of order $(s,t)$. We say that $\cQ$ is \emph{classical} if it is isomorphic to one of the generalized quadrangles listed in Table \ref{Classical} (also see \cite[Table 1]{BLS2018}).

\begin{table}
\[
\begin{array}{lll}\hline
\Q  & \mbox{Order} & {\rm Aut}(\Q) \\ \hline
{\rm W}(3,q)  & (q,q) & {\rm P\Gamma Sp}_4(q)  \\
{\rm Q}(4,q) & (q,q) & {\rm P\Gamma O}_5(q)  \\
{\rm Q}^{-}(5,q) & (q,q^2) & {\rm P\Gamma O}_6^{-}(q) \\
{\rm H}(3,q^2) & (q^2,q) & {\rm P\Gamma U}_4(q)  \\
{\rm H}(4,q^2) & (q^2,q^3) & {\rm P\Gamma U}_5(q) \\
{\rm H}(4,q^2)^D & (q^3,q^2) & {\rm P\Gamma U}_5(q) \\ \hline
\end{array}
\]
\caption{The classical generalized quadrangles}
\label{Classical}
\end{table}

In each case, the socle $X$ of ${\rm Aut}(\cQ)$ is a low-dimensional classical group and the subgroups $X_{\a}$ and $X_{\ell}$ for $(\a,\ell) \in \cP \times \cL$ represent the two conjugacy classes of maximal parabolic subgroups of $X$, which we can identify with the stabilizers of totally singular $1$-dimensional and $2$-dimensional subspaces of the natural module $V$ for $X$. Moreover, the point-line incidence relation on $\mathcal{Q}$ corresponds to the usual inclusion relation on the set of  nonzero totally singular subspaces of $V$. Notice that the  classical GQs arise in dual pairs: $\mathrm{W}(3,q)$ is isomorphic to the dual of $\mathrm{Q}(4,q)$, $\mathrm{Q}^{-}(5,q)$ is isomorphic to the dual of $\mathrm{H}(3,q^{2})$, and $\mathrm{H}(4,q^{2})^{D}$ denotes the dual of $\mathrm{H}(4,q^{2})$.
	
Each classical generalized quadrangle $\cQ$ has the property that ${\rm Aut}(\cQ)$ acts transitively on the set of flags of the geometry. In addition, ${\rm Aut}(\cQ)$ is primitive on both points and lines. The following well-known conjecture of Kantor \cite{Kantor} asserts that flag-transitivity alone is almost enough to characterize the classical GQs (in the statement, ${\rm GQ}(3,5)$ is the unique generalized quadrangle of order $(3,5)$).

\vspace{2mm}

\noindent \textbf{Kantor's conjecture.} \emph{
Any finite thick generalized quadrangle with a flag-transitive automorphism group is either classical, or $\mathrm{GQ}(3,5)$, or it is the generalized quadrangle of order $(15,17)$ arising from the Lunelli-Sce hyperoval, up to duality.}

\vspace{2mm}

Kantor's conjecture remains open despite the fact that it has attracted significant attention since it first appeared in print in 1991. However, several special cases have been resolved and below we briefly summarize some of the main advances. In each case, we have a complete classification of the finite thick generalized quadrangles admitting a group $G$ of automorphisms with the given properties.

\begin{itemize}\addtolength{\itemsep}{0.2\baselineskip}
\item[{\rm (a)}] $G$ satisfies the Moufang condition: Fong and Seitz \cite{FS}.
\item[{\rm (b)}] $G$ acts distance-transitively on points: Buekenhout and Van Maldeghem \cite{Buekenhout1994}.
\item[{\rm (c)}] $G$ acts transitively on the set of antiflags: Bamberg, Li and Swartz \cite{BLS2018}.
\item[{\rm (d)}] $G$ acts transitively on ordered pairs of collinear points and ordered pairs of concurrent lines: Bamberg, Li and Swartz \cite{BLS2021}.
\item[{\rm (e)}] $G$ acts primitively on points and lines, and transitively on flags, and $G$ is not an almost simple group of Lie type: Bamberg et al. \cite{Bamberg2012}.
\end{itemize}

The latter result highlighted in (e) reduces the proof of Kantor's conjecture under the additional point-primitive and line-primitive hypothesis to the case where $\cQ$ admits an almost simple group of automorphisms $G$ with socle $X$, a simple group of Lie type. In a series of recent papers, several special cases in the latter situation have been resolved:
\begin{itemize}\addtolength{\itemsep}{0.2\baselineskip}
\item[{\rm (i)}] If $X = {\rm L}_2(q)$, then $q=9$ and $\mathcal{Q} = \mathrm{W}(3,2)$: Feng and Lu \cite{Feng}.
\item[{\rm (ii)}] If $X = {\rm PSp}_4(q)$ with $q \geqs 3$, then either
$\mathcal{Q} = \mathrm{W}(3,q)$, or $q=3$ and $\cQ = \mathrm{H}(3,4)$, up to duality: Ou and Lu \cite{Ou}.
\item[{\rm (iii)}] $X \ne {\rm U}_3(q)$: Lu, Zhang and Zou \cite{Lu}.
\item[{\rm (iv)}] $X \ne {}^2B_2(q), {}^2G_2(q)'$: Arumugam, Bamberg and Giudici \cite{Arumugam}.
\item[{\rm (v)}] $X \ne {}^2F_4(q)'$: Arumugam \cite{Arumugam_thesis}.
\end{itemize}

As explained in Section \ref{s:intro}, our goal in this paper is to fully resolve Kantor's conjecture under the additional primitivity assumption on points and lines. To do this, we will need to extend the earlier work on low rank groups to the general setting, where we impose no restrictions at all on the group of Lie type $X$. This will require an in-depth analysis of the almost simple groups of Lie type and we will need to recall some key properties of these groups. This leads us naturally to the next section, where we focus on the finite classical groups.

\subsection{Classical groups}\label{ss:class}

Let $G$ be a finite almost simple classical group over $\mathbb{F}_q$ with socle $X$. Write $q = p^f$ and $n = \dim V$, where $p$ is a prime and $n$ is the dimension of the natural module $V$ for $X$ (in particular, note that $V$ is an $n$-dimensional vector space over $\mathbb{F}_{q^2}$ if $X$ is a unitary group). Our notation for classical groups follows the book \cite{KL} by Kleidman and Liebeck. For example, we write ${\rm L}_n(q) = {\rm L}_n^{+}(q)$ and ${\rm U}_n(q) = {\rm L}_n^{-}(q)$ for the projective linear and unitary groups ${\rm PSL}_n(q)$ and ${\rm PSU}_n(q)$, respectively. We refer the reader to \cite[Chapter 1]{BHR}, \cite[Chapter 2]{BG} and \cite[Chapter 2]{KL} as general references on the finite classical groups and their associated forms and geometries.

Due to the existence of isomorphisms between certain low-dimensional classical groups (as recorded in \cite[Proposition 2.9.1]{KL}), and in view of the recent work towards Theorem \ref{t:main} in some low rank cases highlighted above in Section \ref{ClassGQ}, we may assume that the socle $X$ of $G$ is one of the groups recorded in Table \ref{tab:G}, which we will denote by writing $X \in \mathcal{G}$. In the table, we also record the orders of $X$ and ${\rm Out}(X)$, which are taken from \cite[Table 5.1.A]{KL}. In the table, $\delta_{8,n}$ is the familiar Kronecker delta, so
$\delta_{8,n}=1$ if $n=8$, otherwise $\delta_{8,n}=0$.

{\small
\begin{table}
\[
\begin{array}{llcl}\hline
X & d &  |\mathrm{Out}(X)| & |X| \\ \hline
\mathrm{L}_n(q), \, n \geqs 3 &
\gcd(n,q-1) & 2df & \frac{1}{d} q^{\frac{1}{2}n(n-1)}\prod\limits_{i=2}^n\left(q^i-1\right) \\

 {\rm U}_n(q),\, n \geqs 4 & \gcd(n, q+1) &  2df & \frac{1}{d}q^{\frac{1}{2}n(n-1)}\prod\limits_{i=2}^n\left(q^i-(-1)^i\right) \\

\mathrm{PSp}_{n}(q),\, n \geqs 6 & \gcd(2, q-1) & df & \frac{1}{d}q^{\frac{1}{4}n^2} \prod\limits_{i=1}^{n/2}\left(q^{2i}-1\right) \\

\Omega_{n}(q), \, n \geqs 7, \, p \geqs 3 & 2 & 2f & \frac{1}{2}q^{\frac{1}{4}(n-1)^2}\prod\limits_{i=1}^{(n-1)/2}(q^{2i}-1)  \\

\mathrm{P}\Omega^{+}_{n}(q),\, n \geqs 8 & \gcd(4, q^{\frac{1}{2}n}-1) & (2+4\delta_{8,n})df & \frac{1}{d} q^{\frac{1}{4}n(n-2)}(q^{\frac{1}{2}n}-1) \prod\limits_{i=1}^{(n-2)/2}(q^{2 i}-1) \\

\mathrm{P}\Omega^{-}_{n}(q),\, n \geqs 8 & \gcd(4, q^{\frac{1}{2}n}+1) & 2df & \frac{1}{d} q^{\frac{1}{4}n(n-2)}(q^{\frac{1}{2}n}+1) \prod\limits_{i=1}^{(n-2)/2}(q^{2 i}-1) \\ \hline
\end{array}
\]
\caption{The finite simple classical groups in $\mathcal{G}$}
\label{tab:G}
\end{table}}

\subsubsection{Order estimates}

In the proof of Theorem \ref{t:main}, we will often need to work with upper and lower bounds on the order of a given finite classical group. The following lemma records the main estimates that we will use throughout the paper (note that we include odd $n$ in part {\rm (iii)}, in which case $q$ is odd and $\e$ is the empty symbol).

\begin{lem}\label{l:bounds}
The following bounds are satisfied:
\begin{itemize}\addtolength{\itemsep}{0.2\baselineskip}
\item[{\rm (i)}] $\frac{1}{2}q^{n^2-1} < |{\rm SL}_n(q)| \leqs |{\rm SU}_{n}(q)| < q^{n^2-1}$ for all $n,q \geqs 2$.
\item[{\rm (ii)}] $\frac{1}{2}q^{\frac{1}{2}n(n+1)} < |{\rm Sp}_n(q)| < q^{\frac{1}{2}n(n+1)}$ for all $n,q \geqs 2$.
\item[{\rm (iii)}] $\frac{1}{e}q^{\frac{1}{2}n(n-1)} < |{\rm SO}_n^{\e}(q)| < (3-e)q^{\frac{1}{2}n(n-1)}$ for all $n \geqs 3$, $q \geqs 2$, where $e=\gcd(2,q-1)$.
\item[{\rm (iv)}] $\frac{1}{2}q^{n^2} < |{\rm GL}_n(q)| < q^{n^2}$ for all $n \geqs 2$, $q \geqs 3$.
\item[{\rm (v)}] $q^{n^2} < |{\rm GU}_n(q)| < (q+1)q^{n^2-1}$ for all $n,q \geqs 2$.
\end{itemize}
\end{lem}

\begin{proof}
These estimates are standard; see  \cite[Proposition 3.9]{FPR2}, for example.
\end{proof}

\subsubsection{Subgroup structure}\label{sss:sub}

In order to study the faithful primitive actions of $G$, which is essential for the proof of Theorem \ref{t:main}, we need detailed information on the conjugacy classes of core-free maximal subgroups of $G$. Here the main theorem is due to Aschbacher. In \cite{Aschbacher}, eight collections of subgroups of $G$ are defined, labelled $\mathcal{C}_i$ for $1 \leqs i \leqs 8$, and it is shown that every core-free maximal subgroup $H$ of $G$ is either  contained in one of these subgroup collections, or $H$ belongs to a family of almost simple subgroups which act irreducibly on $V$ (we will write $\mathcal{S}$ to denote the latter collection). Strictly speaking, the orthogonal groups with socle $X = {\rm P\O}_8^{+}(q)$ require special attention and a complete description of their maximal subgroups was given in later work of Kleidman \cite{Kleidman} (see Remark \ref{r:N}).

In Table \ref{tab0} we provide a rough description of the subgroups in each $\mathcal{C}_i$ collection, which we refer to as the \emph{geometric subgroups} of $G$. A detailed analysis of the order, maximality and conjugacy of these subgroups is given by Kleidman and Liebeck \cite{KL}, and throughout this paper we adopt the notation therein. We also follow \cite{KL} when defining the membership of each $\C_i$ collection, which differs slightly from Aschbacher's original formulation in \cite{Aschbacher} (see \cite[p.58]{KL} and \cite[Table 3.5.J]{KL}).

We will often refer to the \emph{type} of a $\mathcal{C}_i$-subgroup  (see \cite[p.58]{KL}), which provides an approximate description of its structure. For example, if $X = {\rm L}_n(q)$ and $H \in \mathcal{C}_2$ is the stabilizer in $G$ of a direct sum decomposition $V = V_1 \oplus V_2$ with $\dim V_i = m$, then we say that $H$ is a subgroup of type ${\rm GL}_m(q) \wr S_2$. We will also use standard notation for parabolic subgroups, writing $P_i$ for the stabilizer in $X$ of an
$i$-dimensional totally singular subspace of $V$, and referring to a
subgroup $H$ of $G$ as being of type $P_i$ if $H\cap X=P_i$ (for $X = {\rm L}_n(q)$  we adopt the convention that every subspace of $V$ is said to be totally singular). In addition, we will refer to the subgroups comprising $\C_1$, together with the $\C_8$-subgroups of type ${\rm O}_n^{\pm}(q)$ when $X = {\rm Sp}_n(q)$ and $q$ is even, as the collection of \emph{subspace subgroups} of $G$.

The maximal subgroups of the low-dimensional classical groups with $n \leqs 12$ are determined up to conjugacy in \cite{BHR} and we will make extensive use of these results. For $n \geqs 13$, the maximality and conjugacy of the $\C_i$-subgroups is completely resolved in \cite{KL}. The analogous problem for the subgroups in $\mathcal{S}$ is intractable in general, but the results in \cite{BHR} for $n \leqs 12$ have been extended to $n \leqs 17$ via the PhD theses of Rogers \cite{Rogers} and Schr\"{o}der \cite{Schroder}.

\begin{table}
\[
\begin{array}{cl} \hline
\mbox{Collection} & \mbox{Description} \\ \hline
\C_1 & \mbox{Stabilizers of subspaces, or pairs of subspaces, of $V$} \\
\C_2 & \mbox{Stabilizers of direct sum decompositions $V=\bigoplus_{i=1}^{t}V_i$, where $\dim V_i  = m$} \\
\C_3 & \mbox{Stabilizers of prime degree extension fields of $\mathbb{F}_q$} \\
\C_4 & \mbox{Stabilizers of tensor product decompositions $V=V_1 \otimes V_2$} \\
\C_5 & \mbox{Stabilizers of prime index subfields of $\mathbb{F}_q$} \\
\C_6 & \mbox{Normalizers of symplectic-type $r$-groups, $r \ne p$} \\
\C_7 & \mbox{Stabilizers of tensor product decompositions $V=\bigotimes_{i=1}^{t}V_i$, where $\dim V_i  = m$} \\
\C_8 & \mbox{Stabilizers of nondegenerate forms on $V$} \\ \hline
\end{array}
\]
\caption{The geometric subgroup collections of a finite classical group}
\label{tab0}
\end{table}

The non-geometric maximal subgroups $H \in \mathcal{S}$ will require special attention. Here we will work with the definition of $\mathcal{S}$ given by Kleidman and Liebeck \cite[p.3]{KL}. In particular, $H$ is almost simple with socle $Y$, and the embedding of $H$ in $G$ corresponds to an absolutely irreducible representation $\rho:L \to {\rm GL}(V)$, where $L$ is a covering group of $Y$. In addition, the image $\rho(L)$ must satisfy several additional conditions that are designed to ensure that $H$ is not contained in one of the geometric subgroup collections.

In the proof of Theorem \ref{t:main}, it will be convenient to partition the subgroups in $\mathcal{S}$ into two subcollections, denoted by $\mathcal{S}_1$ and $\mathcal{S}_2$. Here we write $H \in \mathcal{S}_1$ if $Y = A_k$ is an alternating group and $V$ is the \emph{fully deleted permutation module} for $Y$ over $\mathbb{F}_q$, where $q=p$ is a prime. In order to define this module, first consider the natural action of $S_k$ on the permutation module $\mathbb{F}_{p}^k$ and define the following $S_k$-invariant subspaces
\[
U=\{(a_1, \ldots, a_k) \,:\, \sum_{i=1}^{k}a_i=0\},\;\; W=\{(a,\ldots, a) \,:\, a \in \mathbb{F}_{p}\}.
\]
Then $V=U/(U \cap W)$ is the fully deleted permutation module for $Y$ over $\mathbb{F}_{p}$. This module is absolutely irreducible and $\dim V = k-\iota$, where $\iota=2$ if $p$ divides $k$, otherwise $\iota=1$. In addition, $Y$ preserves the symmetric bilinear form $\b:U \times U \to \mathbb{F}_{p}$ defined by
\[
\b((a_1, \ldots, a_k),(b_1, \ldots, b_k)) = \sum_{i=1}^{k}a_ib_i,
\]
which induces a symmetric bilinear form on $V$. In this way, the action of $Y$ on $V$ yields an embedding of $Y$ in an appropriate symplectic or orthogonal group defined over $\mathbb{F}_p$, and the possibilities that arise are recorded in Table \ref{atab}.

\begin{table}
$$\begin{array}{lcl} \hline
X & k & \mbox{Conditions} \\ \hline
{\rm Sp}_n(2) & n+2 & n \geqs 8,\, n \equiv 0 \imod{4} \\
\O_{n}^{+}(2) & n+2 & n \geqs 14,\, n \equiv 6 \imod{8} \\
& n+1 & n \geqs 8,\, n \equiv 0 \imod{8} \\
\O_{n}^{-}(2) & n+2 & n \geqs 10,\, n \equiv 2 \imod{8} \\
 & n+1 & n \geqs 12,\, n \equiv 4 \imod{8} \\
{\rm P\Omega}_n^{\e}(p) & n+2 & n \geqs 7, \, p \geqs 3, \, \mbox{$p$ divides $n+2$} \\
& n+1 & n \geqs 7, \, p \geqs 3, \, \gcd(p,n+2) = \gcd(p,n+1) = 1 \\ \hline
\end{array}$$
\caption{The collection $\mathcal{S}_1$, $Y=A_{k}$}
\label{atab}
\end{table}

As highlighted above, it remains an open problem to determine the subgroups in $\mathcal{S}_2$ when $n \geqs 18$. However, there are numerous results in the literature concerning general properties of these subgroups, which we can appeal to. For our purposes, in view of Lemma \ref{large}(ii) for example, it will be important to know that most subgroups arising in $\mathcal{S}_2$ have small order relative to $|X|$. More precisely, we will work with the main theorem of \cite{BW}, which determines all the subgroups $H \in \mathcal{S}_2$ with $|H| \geqs q^n$. The latter extends earlier work in \cite{BGS2014} and \cite{Liebeck}, where similar results are established, but with weaker bounds.

\begin{rem}\label{r:N}
As noted above, the orthogonal groups with socle $X = {\rm P\O}_8^{+}(q)$ require special attention due to the existence of triality automorphisms. If $G \leqs {\rm P\Gamma O}_8^{+}(q)$, then Aschbacher's main theorem in \cite{Aschbacher} applies and the maximal subgroups of $G$ can be described in terms of the collections $\C_i$ and $\mathcal{S}$ as before. However, the subgroup structure of the groups $G \not\leqs {\rm P\Gamma O}_8^{+}(q)$ containing triality graph or graph-field automorphisms is more complicated. In particular, a small additional collection of so-called \emph{novelty} maximal subgroups arises in this situation; these are certain core-free maximal subgroups $H$ of $G$ with the property that $H \cap X$ is non-maximal in $X$. We write $\mathcal{N}$ for this additional collection, which comprises the subgroups recorded in Table \ref{tab:N} (see \cite[Table 5.9.1]{BG}).  For example, if $H$ is a parabolic subgroup of type $P_{1,3,4}$, then $H \cap X < P_i < X$ for $i = 1,3,4$. And for the subgroup $H$ in the third row, we note that $H \cap X =  [2^9].{\rm SL}_3(2)$ is contained in a $\C_2$-subgroup of $X$ of type ${\rm O}_1(q) \wr S_8$.
\end{rem}

\begin{table}
\[
\begin{array}{ll} \hline
\mbox{Type of $H$} & \mbox{Conditions} \\ \hline
{\rm GL}_3^{\e}(q) \times {\rm GL}_1^{\e}(q) & (\e,q) \ne (+,2) \\
{\rm O}_2^{-}(q^2) \times {\rm O}_2^{-}(q^2) & \\
\mbox{$[2^9].{\rm SL}_3(2)$} & q = p \geqs 3  \\
P_{1,3,4} & \\
G_2(q) & \\ \hline
\end{array}
\]
\caption{The collection $\mathcal{N}$ with $X = {\rm P\O}_8^{+}(q)$}
\label{tab:N}
\end{table}

\subsubsection{Conjugacy classes}

We will also need some information concerning conjugacy classes in the finite classical groups. For example, this will be required for computing (or estimating) the number of fixed points in appropriate group actions via Lemma \ref{l:fpr} below. There is an extensive literature on this topic and our main references are \cite[Chapter 3]{BG} and \cite{LS_book}. For the purposes of this preliminary section, we will briefly discuss our notation for involutions and unipotent elements, which we will use repeatedly in the proof of Theorem \ref{t:main}.

Set $G = {\rm GL}_n(q)$, where $q = p^f$ and $p$ is a prime. Suppose $q$ is odd and let $g \in {\rm GL}_n(q)$ be a non-central involution. Then $g$ is diagonalizable over $\mathbb{F}_q$, with eigenvalues $1$ and $-1$ on $V$, and the conjugacy class $g^G$ is uniquely determined by the dimension of the $1$-eigenspace $C_V(g)$. In particular, if $\dim C_V(g) = k$ then we will write $g = (-I_{n-k},I_k)$ to encode the action of $g$ on $V$ with respect to a decomposition $V = U \oplus W$ with $W = C_V(g)$.

Similarly, if $g \in {\rm GL}_n(q)$ is unipotent (in other words, if the order of $g$ is a $p$-power) then $g^G$ is uniquely determined by the Jordan form of $g$ on $V$. So if $J_i$ denotes a standard unipotent Jordan block of size $i$, then we write $g = (J_n^{a_n}, \ldots, J_1^{a_1})$ to indicate that $J_i$ has multiplicity $a_i$ in the Jordan form of $g$ on $V$. In the special case $g = (J_2,J_1^{n-2})$, we will often refer
to $g$ as a \emph{long root element} (or a \emph{transvection}).

This notation extends naturally to the other classical groups. However, it may no longer be the case that there is a unique conjugacy class of elements of a given type. For example, if $G = {\rm SO}_{n}^{+}(q)$ with $n \geqs 8$ even and $q$ odd, then there are two classes of involutions of type $(-I_{n-2},I_{2})$ represented by the elements $g_{+}$ and $g_{-}$, where $W = C_{V}(g_{\e})$ is a nondegenerate $2$-space of type $\e$ (that is to say, if $\e=+$ then $W$ is a hyperbolic $2$-space, whereas $W$ is an elliptic $2$-space when $\e=-$).

In the special case where $G = {\rm Sp}_n(q)$ or ${\rm O}_n^{\pm}(q)$ with $q$ even, we will adopt the standard notation for involutions introduced by Aschbacher and Seitz in \cite{AS} (also see \cite[Sections 3.4.4, 3.5.4]{BG}). For example, suppose $G = {\rm Sp}_n(q)$ and $g = (J_2^s,J_{1}^{n-2s})$ is an involution. If $s$ is odd, then there is a unique conjugacy class of such elements and we write $g = b_s$. On the other hand, if $s$ is even then there are two distinct conjugacy classes of elements with Jordan form $(J_2^s,J_1^{n-2s})$, with class representatives denoted by $a_s$ and $c_s$. In the latter case, the two representatives can be distinguished by their action on $V$. Indeed, if $\b:V \times V \to \mathbb{F}_q$ is the defining symplectic form, then $g$ is of type $a_s$ if and only if
\[
\b(v,v^g) = 0 \mbox{ for all $v \in V$.}
\]
There is a very similar description of the involution classes in ${\rm O}_{n}^{\pm}(q)$ when $q$ is even and we refer the reader to \cite[Section 3.5.4]{BG} for more details.

\subsection{Exceptional groups}\label{ss:ex}

Let $G$ be a finite almost simple exceptional group of Lie type over $\mathbb{F}_q$ with socle $X$, where $q = p^f$ and $p$ is a prime. In view of the main theorems of \cite{Arumugam_thesis,Arumugam}, we may assume that $X \ne {}^2B_2(q)$, ${}^2G_2(q)'$, ${}^2F_4(q)'$, so it remains for us to deal with the following possibilities for $X$:
\[
{}^3D_4(q), \, G_2(q), \, F_4(q), \, E_6^{\e}(q), \, E_7(q), \, E_8(q),
\]
where we write $E_6^{+}(q) = E_6(q)$ and $E_6^{-}(q) = {}^2E_6(q)$. And in view of \cite{Lu} and the isomorphism $G_2(2)' \cong {\rm U}_3(3)$, we may assume $q \geqs 3$ when $X = G_2(q)$. We use $\mathcal{H}$ to denote this collection of simple groups and for each $X \in \mathcal{H}$ we record the orders of $X$ and ${\rm Out}(X)$ in Table \ref{tab:H} (see \cite[Table 5.1.B]{KL}).

{\small
\begin{table}
\[
\begin{array}{lcl}\hline
X &   |\mathrm{Out}(X)| & |X| \\ \hline
{}^3D_4(q) &  3f & q^{12}(q^2-1)(q^6-1)(q^8+q^4+1) \\
G_2(q) & (1+\delta_{3,p})f & q^6(q^2-1)(q^6-1) \\
F_4(q)  & (1+\delta_{2,p})f & q^{24}(q^2-1)(q^6-1)(q^8-1)(q^{12}-1) \\
E_6^{\e}(q) & 2df & \frac{1}{d}q^{36}\prod_{i \in \Lambda}(q^i-\e^i), \; \Lambda = \{2,5,6,8,9,12\},\, d = \gcd(3,q-\e) \\
E_7(q) & df & \frac{1}{d}q^{63}\prod_{i \in \Lambda}(q^i-1),\; \Lambda = \{2,6,8,10,12,14,18\},\; d = \gcd(2,q-1) \\
E_8(q) & f & q^{120}\prod_{i \in \Lambda}(q^i-1),\; \Lambda = \{2,8,12,14,18,20,24,30\} \\ \hline
\end{array}
\]
\caption{The finite simple exceptional groups in $\mathcal{H}$}
\label{tab:H}
\end{table}}

There is a substantial literature on the subgroup structure of exceptional groups of Lie type, stretching back several decades. Through the work of many authors, the maximal subgroups of $G$ have been completely determined up to conjugacy, except for the groups with socle $E_7(q)$ or $E_8(q)$. For example, recent work of Craven \cite{Craven1} provides a complete treatment for the groups with socle $F_4(q)$, $E_6(q)$ or ${}^2E_6(q)$. In addition, an almost complete result for $E_7(q)$ is presented in \cite{Craven2}, and detailed subgroup information for groups with socle $E_8(q)$ is also available in the literature. Indeed, the only uncertainty in the latter two cases concerns the existence, maximality and conjugacy of certain low order almost simple subgroups. However, by appealing to Lemma \ref{large}(ii), we see that the existence (or otherwise) of such subgroups does not cause any difficulties in the proof of Theorem \ref{t:main} for exceptional groups (see Lemma  \ref{l:e8_red}, for example), which will be presented in Section \ref{s:excep}.

\subsection{Fixed points}\label{ss:fix}

To conclude this preliminary section, let us return to Lemma \ref{fixednumber}. In order to effectively apply this result, we will need to recall one or two basic observations concerning group actions and fixed points. For a group $G$ acting on a set $\Omega$, we write
\[
\Omega_g = \{\a \in \Omega \, : \, \a^{g}=\a\}
\]
for the set of fixed points of an element $g \in G$.

For transitive actions, the following lemma is a key tool for calculating the number of fixed points of an element.

\begin{lem}\label{l:fpr}
Let $G \leqs {\rm Sym}(\O)$ be a finite transitive permutation group with point stabilizer $H = G_{\a}$. Then
\[
|\O_g|  = |\O| \cdot \frac{|g^G \cap H|}{|g^G|}
\]
for all $g \in G$.
\end{lem}

\begin{proof}
This is \cite[Lemma 2.5]{Liebeck1991}.
\end{proof}

\begin{rem}\label{r:fpr}
Note that if $N$ is a transitive normal subgroup of $G$, then
\[
|\O_g|  = [N:N_{\a}] \cdot \frac{|g^N \cap N_{\a}|}{|g^N|}
\]
for all $g \in N$.
\end{rem}

\section{Subspace subgroups}\label{s:subspace}

Let $\mathcal{Q} = (\mathcal{P},\mathcal{L})$ be a finite thick generalized quadrangle of order $(s,t)$ and assume $\cQ$ admits an almost simple automorphism group $G$ with socle $X \in \mathcal{G}$, so $X$ is one of the classical groups over $\mathbb{F}_q$ recorded in Table \ref{tab:G}. Let $V$ be the natural module for $X$ and write $n = \dim V$ and $q = p^f$, where $p$ is a prime. Throughout this section we will assume $G$ acts primitively on $\mathcal{P}$, with point stabilizer $G_{\a}$.

Here we begin the proof of Theorem \ref{t:main} by considering the special cases where $G_{\a}$ is a so-called \emph{subspace subgroup} of $G$. The possibilities for $G_{\a}$ can be divided into the following six cases (in the descriptions given below, we refer to the subgroup collections arising in the statement of Aschbacher's theorem in \cite{Aschbacher}, as briefly described in Section \ref{sss:sub}):

\begin{itemize}\addtolength{\itemsep}{0.2\baselineskip}
\item[(\texttt{S1})] $G_{\a} \in \C_1$ is of type $P_i$. Here $G_{\a}$ is the stabilizer in $G$ of an $i$-dimensional totally singular subspace of $V$, which is a maximal parabolic subgroup of $G$. Note that for $X = {\rm L}_n(q)$, we adopt the convention that every subspace of $V$ is totally singular. Also note that for $X = {\rm U}_n(q)$ or ${\rm PSp}_n(q)$, the terms \emph{totally singular} and \emph{totally isotropic} are synonymous.
\item[(\texttt{S2})] $X = {\rm L}_n(q)$ and $G_{\a} \in \C_1$ is of type $P_{i,n-i}$ with $1 \leqs i < n/2$. Here $G_{\a}$ is the stabilizer of a flag $0 < U < W < V$, where $\dim U = i$ and $\dim W = n-i$.
\item[(\texttt{S3})] $X = {\rm L}_n(q)$ and $G_{\a} \in \C_1$ is of type ${\rm GL}_i(q) \oplus {\rm GL}_{n-i}(q)$, where $1\leqs i<n/2$. In this case, $G_{\a}$ is the stabilizer of a decomposition $V= U \oplus W$ with $\dim U = i$ and $\dim W = n-i$.
\item[(\texttt{S4})] $X$ is a unitary, symplectic or orthogonal group, and $G_{\a} \in \C_1$ is the stabilizer of an $i$-dimensional nondegenerate subspace $U$ of $V$, where $1\leqs i<n$ is odd if $X$ is an odd-dimensional orthogonal group, and $i\leqs n/2$ otherwise. In this case, $G_{\a}$ preserves an orthogonal decomposition $V = U \perp U^{\perp}$.
\item[(\texttt{S5})] $X = \O_{n}^{\pm}(q)$, $q$ is even and $G_{\a} \in \C_1$ is of type ${\rm Sp}_{n-2}(q)$. Here $G_{\a}$ is the stabilizer of a $1$-dimensional nonsingular subspace of $V$.
\item[(\texttt{S6})] $X = {\rm Sp}_n(q)$, $q$ is even and $G_{\a} \in \C_8$ is of type ${\rm O}_n^{\e}(q)$ with $\e = \pm$.
\end{itemize}

\begin{rem}
Notice that the subgroups in (\texttt{S2}) and (\texttt{S3}) are maximal only if $G \not\leqs {\rm P\Gamma L}_n(q)$, so our point-primitivity hypothesis implies that $G$ must contain a nontrivial graph or graph-field automorphism of $X$. And in case (\texttt{S4}), let us observe that $i = n/2$ only if $X = {\rm P\O}_n^{-}(q)$ and $n \equiv 0 \imod{4}$ (see \cite[Tables 3.5.C--3.5.F]{KL}).
\end{rem}

Our main result for subspace subgroups is the following, which provides a reduction to the restricted list of cases recorded in Table \ref{tab:A}. We use $\mathcal{A}$ to denote the latter collection, which we will return to later, when we complete the proof of Theorem \ref{t:main} for classical groups in Section \ref{s:main1}. Note that we do not assume $G$ acts flag-transitively in the statement of Theorem \ref{t:subspace}.

\begin{thm}\label{t:subspace}
Let $\mathcal{Q} = (\mathcal{P},\mathcal{L})$ be a finite thick generalized quadrangle admitting an almost simple point-primitive automorphism group $G$ with socle $X \in \mathcal{G}$. If $G_{\a}$ is a subspace subgroup of $G$, then either
\begin{itemize}\addtolength{\itemsep}{0.2\baselineskip}
\item[{\rm (i)}] $\mathcal{Q}$ is a classical GQ; or
\item[{\rm (ii)}] $(G,G_{\a}) \in \mathcal{A}$ is one of the cases in Table \ref{tab:A}.
\end{itemize}
\end{thm}

{\small
\begin{table}
\[
\begin{array}{lll} \hline
X & \mbox{Type of $G_{\a}$} & \mbox{Conditions} \\ \hline
{\rm L}_n(q) &
{\rm GL}_i(q)\oplus{\rm GL}_{n-i}(q) &
(n,i)\in\{(5,2),(6,2),(7,3)\} \\
{\rm U}_n(q) & {\rm GU}_i(q) \perp {\rm GU}_{n-i}(q) & (n,i) \in \{ (5,2), (6,2), (7,2), (7,3), (8,3),(9,4) \} \\
\O_n(q) & {\rm O}_i(q) \perp {\rm O}_{n-i}^{\pm}(q) & (n,i) \in \{(7,3), (7,5), (9,3), (9,5), (9,7)\} \\
{\rm P\O}_n^{\pm}(q) & {\rm O}_i^{\e_1}(q) \perp {\rm O}_{n-i}^{\e_2}(q) & \mbox{$8 \leqs n \leqs 20$, $2 \leqs i \leqs n/2$} \\ \hline
\end{array}
\]
\caption{The collection $\mathcal{A}$}
\label{tab:A}
\end{table}
}

\begin{rem}\label{r:classical}
Recall that the classical generalized quadrangles $\cQ = (\cP,\cL)$ are recorded in Table \ref{Classical} (see Section \ref{ClassGQ}). For each pair $(\a,\ell) \in \cP \times \cL$, the stabilizers $X_{\a}$ and $X_{\ell}$ represent the two conjugacy classes of maximal parabolic subgroups of $X$ associated with the rank-two polar space, and point-line duality interchanges the two types of parabolic subgroups. For example, if
$\cQ={\rm W}(3,q)$ then $X={\rm PSp}_4(q)$ and we have $X_{\a} = P_1$ and $X_{\ell} = P_2$. In each case, the precise structure of $X_{\a}$ and $X_{\ell}$ can be read off from \cite[Section 4.1]{KL} (or by consulting the relevant tables in \cite[Chapter 8]{BHR}).
\end{rem}

Our proof of Theorem \ref{t:subspace} involves eliminating candidate pairs $(G,G_{\a})$ through a combination of different techniques. In most cases, we start by identifying an element $1 \ne g \in G$ such that $|\cP_g| \geqs |\cP|^{4/5}$ if $n$ is sufficiently large. Then by  appealing to Lemma \ref{fixednumber}, this allows us to reduce the problem to certain low-dimensional groups. At this point, we switch our focus to the  orbits of $G_{\a}$ (or $X_{\a}$) on the set of points in $\cP \setminus \{\a\}$ that are collinear to $\a$. As we will explain in more detail below, we are particularly interested in the lengths of the
$G_{\a}$-orbits on $\cP \setminus \{\a\}$, which in most cases are available in the literature. Notice that a subspace subgroup typically has large order, which explains why we do not work with the bound in Lemma \ref{large}(ii) in this section.

Fix a point $\a \in \cP$ and recall that an \emph{$X$-orbital} on $\cP$ is simply an orbit of $X$ on $\cP \times \cP$. Since the primitivity of $G$ on $\cP$ implies that $X$ acts transitively on $\cP$, there is a natural bijection between the orbitals of $X$ and the orbits of $X_{\a}$ on $\mathcal{P}$; the latter are called the \emph{suborbits} of $X$ and the size of each suborbit is called a \emph{subdegree} (or an \emph{$X$-subdegree} if we wish to emphasize that we are working with $X_{\a}$-orbits, rather than $G_{\a}$-orbits). Since the set of collinear point pairs in $\cP \times \cP$  is $G$-invariant, it can be expressed as a union of $G$-orbitals (and hence $X$-orbitals). Throughout this section, we do not assume that the set of ordered pairs of distinct collinear points forms a single $G$-orbital.

Let $\mathcal{O}_1, \ldots, \mathcal{O}_k$ be the complete set of orbits of $X_{\a}$ on $\mathcal{P} \setminus \{\a\}$. Let $\mathcal{C}$ be the set of points in $\mathcal{P} \setminus \{\a\}$ that are collinear to $\a$ and note that $|\mathcal{C}| = s(t+1)$, where $(s,t)$ is the order of $\cQ$ (see Lemma \ref{isomorphism}(ii)). Since $\mathcal{C}$ is $X_{\a}$-invariant, it follows that
\begin{equation}\label{e:orb}
|\mathcal{C}| = s(t+1) = \sum_{j=1}^k c_jd_j,
\end{equation}
where $c_j \in \{0,1\}$ and $d_j = |\mathcal{O}_j|$ for all $j$.

Given a subspace subgroup $G_{\a}$, the corresponding subdegrees $d_1, \ldots, d_k$ are known and available in the literature in most cases. By combining these subdegrees with various arithmetic conditions imposed on the parameters $s$ and $t$ (see Lemmas \ref{HD} and \ref{lem_subdegBound}, for example), our goal is to severely limit the possibilities for the $k$-tuple $(c_1, \ldots, c_k) \in \{0,1\}^k$ in \eqref{e:orb}. In particular, it may be possible to show that a given candidate decomposition of $\mathcal{C}$ into $X_{\a}$-orbits is only valid for a finite number of possibilities for $q$, which can then be ruled out by showing that there are no feasible parameters for $\cQ$ in these cases (see Definition \ref{d:feas}), using the approach described in Remark \ref{r:c1}(a).

In view of the strategy outlined above, the arithmetic conditions obtained in the following lemma will be very useful. In the statement, recall that we write $v = |\cP|$ and $b = |\cL|$.

\begin{lem}\label{lem_subdegBound}
Let $\cQ = (\cP,\cL)$ be a finite thick generalized quadrangle of order $(s,t)$ and set $d=s(t+1)$, $e=\gcd(d,v-1)$ and $e'=\gcd(d^2,v-1-d)$. Then $s \leqs e$, $s^2 = e'$ and we have
\[
d<\frac{3}{2}v^{3/4},\;\; v < (e+1)d+1.
\]
\end{lem}

\begin{proof}
First recall that $v = (s+1)(st+1)$ by Lemma \ref{HD}(iv). In particular, $s$ divides $d$ and $v-1$, so $s \leqs e$. Since $d^2=s^2(t+1)^2$ and $v-1-d=s^2t$, we see that $s^2=e'$ and
\begin{equation}\label{e:v1}
(v-1-d)^2=s^4t^2<s^2d^2
\end{equation}
since $st<d$. In particular, $v< (s+1)d+1$, and so $v<(e+1)d+1$. Finally, since $v>s^2t$ and $2 \leqs t \leqs s^2$, we get
\[
  \frac{d^4}{v^3}<\frac{s^4(t+1)^4}{s^6t^3}\leqs \left(\frac{3}{2}\right)^4\frac{t}{s^2}\leqs \left(\frac{3}{2}\right)^4
\]
and thus $d<\frac{3}{2}v^{3/4}$.
\end{proof}

We will also need the following elementary observation.

\begin{lem}\label{polynomial-bound}
Let $f(x) \in \mathbb{R}[x]$ be a monic polynomial of degree $n \geqs 1$. Then there are at most finitely many positive integers $m$ such that  $f(m)<0$.
\end{lem}

\begin{proof}
Write $f(x)=x^n+\sum_{i=0}^{n-1}a_ix^i \in \mathbb{R}[x]$ and set $b_i = \max\{0,-a_i\}$ and
\[
B = \max\{ b_i^{1/(n-i)} \,:\, 0 \leqs i < n\}.
\]
Then
\[
x^{-n}f(x) \geqs 1 - \sum_{i=0}^{n-1}b_ix^{i-n} \geqs 1 - \sum_{j=1}^n2^{-j} = 2^{-n}>0
\]
for all $x \geqs 2B$ with $x>0$ and the result follows.
\end{proof}

\begin{rem}\label{rem_subdegBound}
Let us explain how we will use Lemma \ref{lem_subdegBound} in the proof of Theorem \ref{t:subspace}.

Let $\mathcal{F} = \{X(q) \, : \, q \in \Lambda \}$ be a family of low-dimensional finite simple classical groups of the same type, where $\Lambda$ is an infinite set of prime powers. For example, we could take $\mathcal{F} = \{ {\rm L}_4(q) \,:\, q \in \Lambda\}$, where $\Lambda$ is the set of all prime powers. Set $X=X(q)$ for some fixed $q \in \Lambda$ and let $d_1,\ldots,d_k$ be the lengths of the orbits of $X_{\a}$ on $\cP \setminus \{\a\}$, as above. Let us assume that $k$ is independent of the choice of $q$, and let us also assume that $v = |\cP|$ and each $d_i$ can be expressed as polynomials in $q$ with rational coefficients, and assume that these polynomials are also independent of the choice of $q$. To be specific, write $v=U(q)$ and $d_i=D_i(q)$, where $U(x)$ and $D_i(x)$ are polynomials in $\mathbb{Q}[x]$. For example, if $\mathcal{F} = \{ {\rm L}_4(q) \,:\, q \in \Lambda\}$ as above and we assume $X_{\a}$ is a $P_2$ parabolic, then
\[
k = 2, \; U(x) = x^4+x^3+2x^2+x+1,\; D_1(x) = x^3+2x^2+x,\; D_2(x) = x^4.
\]
We say that $d_i$ is \emph{large} if $d_i \geqs \frac{3}{2}v^{3/4}$ for all but at most finitely many $q \in\Lambda$, otherwise $d_i$ is \emph{small}. By Lemma \ref{polynomial-bound}, $d_i$ is large if  $4\deg(D_i)>3\deg(U)$.

Set $d=s(t+1)$ and express $d$ as in \eqref{e:orb} for some $k$-tuple $\tau = (c_1, \ldots, c_k) \in \{0,1\}^k$. Suppose $d_i$ is a large subdegree and let $\Gamma$ be the finite set of prime powers $q \in \Lambda$ such that $d_i<\frac{3}{2}v^{3/4}$. For each $q \in \Gamma$, we use the approach in Remark \ref{r:c1}(a) to check whether or not there are feasible parameters for $\cQ$ (see Definition \ref{d:feas}). If there are no feasible parameters, then we conclude that $c_i = 0$ in \eqref{e:orb} and this reduces the number of possibilities for $\tau$.

Next we consider each remaining $k$-tuple $\tau = (c_1, \ldots, c_k) \in \{0,1\}^k$ in turn. For each possibility, we can express $d = \sum_i c_id_i = D(q)$ for some rational polynomial $D(x) \in \mathbb{Q}[x]$ and we define
\[
\Delta(x)=\gcd(D(x)^2,U(x)-1-D(x)),\;\; E(x)=\gcd(D(x),U(x)-1)
\]
as polynomials in $\mathbb{Q}[x]$. Then with the aid of the {\sc Magma} \cite{magma} command \texttt{XGCD}, we can find polynomials $g_0(x),h_0(x)\in\mathbb{Q}[x]$ such that
\[
g_0(q)d^2+h_0(q)(v-1-d)=\Delta(q).
\]

Let $c$ be the least common multiple of all the denominators of the nonzero coefficients in $g_0(x)$ and $h_0(x)$. Then $g(x)=cg_0(x)$ and $h(x)=ch_0(x)$ are integer polynomials and we define
\[
w(q)=g(q)d^2+h(q)(v-1-d)=c\Delta(q)\in \mathbb{N}.
\]
It follows that $\gcd(d^2,v-1-d)\leqs w(q)$, and so \eqref{e:v1} in the proof of Lemma \ref{lem_subdegBound} implies that
\begin{equation}\label{e:v2}
(v-1-d)^2\leqs w(q)d^2.
\end{equation}
If $2\deg(U)> \deg(\Delta)+2\deg(D)$, then Lemma \ref{polynomial-bound} implies that the inequality in \eqref{e:v2} can only hold for at most finitely many prime powers $q \in \Lambda$.

Similarly, there exist $u(x),v(x)\in\mathbb{Z}[x]$ and $c'\in\mathbb{N}$ such that
\[
  w'(q)=u(q)d+v(q)(v-1)=c' E(q).
\]
We have $e=\gcd(d,v-1)\leqs c'E(q)$, and so
\begin{equation}\label{e:v3}
  v<(w'(q)+1)d+1
\end{equation}
by Lemma \ref{lem_subdegBound}. If $\deg(U)>\deg(E)+\deg(D)$, then \eqref{e:v3} holds for at most finitely many $q \in \Lambda$.

Let us assume that the inequalities in \eqref{e:v2} and \eqref{e:v3} are only both simultaneously satisfied for a finite set $\Sigma$ of prime powers $q \in \Lambda$. For each $q \in \Sigma$, if we can show that there are no feasible parameters for $\cQ$ (for example, by applying the approach in Remark \ref{r:c1}(a)), then we have ruled out the candidate $k$-tuple $\tau = (c_1, \ldots, c_k) \in \{0,1\}^k$. And if we can exclude all candidate $k$-tuples in this way, then we have eliminated the pair $(G,G_{\a})$, where $G$ has socle $X = X(q)$ with $q \in \Lambda$.
\end{rem}

\subsection{Parabolic subgroups}\label{ss:parab}

We begin the proof of Theorem \ref{t:subspace} by handling the parabolic subgroups arising in cases (\texttt{S1}) and (\texttt{S2}).

\begin{lem}\label{l:R1linear}
The conclusion of Theorem \ref{t:subspace} holds in case {\rm (\texttt{S1})} when $X = {\rm L}_n(q)$.
\end{lem}

\begin{proof}
Here $X={\rm L}_n(q)$ with $n\geqs 3$ and we may identify $\cP$ with the set of all $i$-dimensional subspaces of the natural module $V$, where $1\leqs i\leqs n/2$. It is easy to see that there are $i+1$ $X$-orbitals on $\cP$ in total, which take the following form:
\[
\Omega_j=\{(\pi,\pi') \, : \, \dim(\pi\cap\pi')=j\}	
\]
for $0\leqs j \leqs i$, where $\Omega_i = \{(\pi,\pi) \,:\, \pi \in \cP\}$ is the diagonal orbital.
Let $j$ be maximal such that $\Omega_j$ consists of distinct collinear point pairs. Clearly, $j \leqs i-1$ and we claim that $j=i-1$.

To see this, set $d=i-j \geqs 1$ and consider a decomposition
\begin{equation}\label{e:dec1}
V = V_1 \oplus V_2 \oplus V_3 \oplus V_4 \oplus W,
\end{equation}
where $\dim V_1 = \dim V_4 = j$, $\dim V_2 = \dim V_3 = d$ and $\dim W = n-2i$. Set $\pi_1 = V_1 \oplus V_2$ and $\pi_2 = V_1 \oplus V_3$, so $\pi_1,\pi_2 \in \cP$ and $(\pi_1,\pi_2)\in\Omega_j$. Let $\ell \in \cL$ be the line containing $\pi_1$ and $\pi_2$.

Given a matrix $A \in {\rm GL}_d(q)$, which we can view as representing an invertible linear map $\varphi:V_2 \to V_3$, we define the following $i$-dimensional subspace
\[
\pi(A)=\{(x,y,\varphi(y),0,0)\, :\, x\in V_1, \, y \in V_2 \} \subseteq V
\]
with respect to the decomposition in \eqref{e:dec1}. Since $\dim(\pi_k \cap\pi(A))=j$ for $k = 1,2$, it follows that $(\pi_1,\pi(A))$ and $(\pi_2,\pi(A))$ belong to $\Omega_j$. In particular, $\pi(A)$ is on the line $\ell$. Note that if $I_d$ is the identity matrix in ${\rm GL}_d(q)$, then
\[
\dim(\pi(B)\cap\pi(I_d))=i-\mathrm{rank}(B-I_d)
\]
for all $B \in {\rm GL}_d(q)$.

Suppose $d\geqs 2$. Then $\mathrm{rank}(B-I_d)=1$ for some $B \in {\rm GL}_d(q)$, which means that $\pi(B)$ and $\pi(I_d)$ are both on $\ell$. It follows that $\Omega_{i-1}$ consists of collinear point pairs, so $j=i-1$ and thus $d = i-j = 1$, a contradiction. Therefore $d=1$ and $j=i-1$ as claimed.

Fix a basis $\{e_1,\ldots,e_n\}$ for $V$. We will use  induction to show that $\Omega_k$ consists of collinear points for all $0 \leqs k \leqs i-1$. This implies that every pair of points in $\cP$ is a collinear pair, which is clearly a contradiction.

By the above claim, we know that $\Omega_{i-1}$ consists of collinear point pairs, which establishes the base case for the induction argument. Fix $0 \leqs k \leqs i-2$ and assume that $\Omega_j$ consists of collinear point pairs for each $j \geqs k+1$. We need to show that $\Omega_k$ also comprises collinear pairs. To do this, set
\[
\pi_1'=\la e_1,\ldots,e_i\ra, \;\; \pi_2'=\la e_1,\ldots,e_{k+1},e_{i+1},\ldots,e_{2i-k-1}\ra,
\]
so that $\dim(\pi_1'\cap\pi_2')=k+1$ and $(\pi_1',\pi_2') \in \O_{k+1}$. By assumption, $\pi_1'$ and $\pi_2'$ are collinear and we take $\ell'$ to be the line containing these two points. Now define
\[
 \pi_3'=\la e_1,\ldots,e_k,e_{k+2},\ldots,e_{i+1}\ra,\;\;
  \pi_4'=\la e_1,\ldots,e_{k+1},e_{i+2},\ldots,e_{2i-k}\ra
\]
and observe that
\[
\dim(\pi_1'\cap\pi_3')=\dim(\pi_2'\cap\pi_4')=i-1, \;\; \dim(\pi_1'\cap\pi_4')=\dim(\pi_2'\cap\pi_3')=k+1.
\]
It follows that $\pi_3'$ and $\pi_4'$ are also on the line $\ell'$. And since
$\dim(\pi_3'\cap\pi_4')=k$, we deduce that $\Omega_k$ consists of collinear point pairs. This completes the proof of the lemma.
\end{proof}

The uniform approach we adopted in the proof of Lemma \ref{l:R1linear} for $X = {\rm L}_n(q)$ is difficult to apply when $X$ is a unitary, symplectic or orthogonal group because the description of the $X$-orbitals on $\cP$ is more complicated. Therefore, we adopt a different approach to  handle these cases. First we work with fixed points (via Lemma \ref{fixednumber}) to reduce the problem to a short list of possibilities for $(n,i)$, as recorded in Table \ref{tab:A1}. We  then apply a combination of geometric, arithmetic and computational techniques to deal with each of the remaining cases (see the proof of Proposition \ref{p:pi}, which implements some of the general techniques we discussed in Remark \ref{rem_subdegBound}).

{\small
\begin{table}
\[
\begin{array}{lll} \hline
j & X & \mathcal{A}_j \\ \hline
1 & {\rm U}_n(q) & (7,2),(7,3),(8,3),(8,4),(9,3),(9,4),(10,4), (10,5), (11,4),(11,5), \\
\vspace{2mm}
& &  (12,5), (12,6),(13,6),(14,7),(15,7),(16,8),(17,8),(18,9) \\
\vspace{2mm}
2 & {\rm PSp}_n(q) & (6,2),(6,3),(8,3),(8,4),(10,4),(10,5),(12,6),(14,7),(16,8) \\ \vspace{2mm}
3 & \O_n(q) & (7,2),(7,3),(9,3),(9,4),(11,4),(11,5),(13,6),(15,7) \\
4 & {\rm P\O}_n^{\e}(q)& (12,2),(12,3),(12,4),(12,5),(12,6),(14,3),(14,4),\\
& & (14,5),(14,6),(14,7),(16,4),(16,5),(16,6),(16,8) \\
\hline
\end{array}
\]
\caption{The sets $\mathcal{A}_j$ for $G_{\a}$ of type $P_i$}
\label{tab:A1}
\end{table}
}

\vs

Consider case (\texttt{S1}) for unitary groups. We begin by  introducing some additional notation. Let $m$ and $k$ be integers such that $1 \leqs k \leqs m/2$ and let $q$ be a prime power. Let
\[
U_q(m,k) = \frac{|{\rm GU}_m(q)|}{q^{k(2m-3k)}|{\rm GL}_k(q^2)||{\rm GU}_{m-2k}(q)|}
\]
be the number of totally singular $k$-spaces in an $m$-dimensional unitary space over $\mathbb{F}_{q^2}$. In addition, set $U_q(m,0) = 1$ and $U_q(m,k) = 0$ if $k<0$ or $k>m/2$.

\begin{lem}\label{l:uni_pi}
Let $X = {\rm U}_n(q)$ and let $G_{\a}$ be a maximal parabolic subgroup of type $P_{i}$, where $1 \leqs i \leqs n/2$. If $g \in X$ has Jordan form $(J_2,J_1^{n-2})$, then $|\mathcal{P}| = U_q(n,i)$ and
\[
|\mathcal{P}_g| = U_q(n-2,i-1) + q^{2i}U_q(n-2,i).
\]
\end{lem}

\begin{proof}
Working in ${\rm SU}(V)$, write $g = I+N$ where $I$ is the identity map and $N$ is a nilpotent matrix with $N^2 = 0$ and ${\rm rank}(N) = 1$. Set $K = \ker(N)$, $L = {\rm im}(N)$ and note that $L \subseteq K = L^{\perp}$ with respect to the defining unitary form on $V$. In addition, observe that $K/L$ is an $(n-2)$-dimensional unitary space.

Let $U$ be an $i$-dimensional totally singular subspace of $V$. Then $U$ is fixed by $g$ if and only if $U \subseteq K$ or $L \subseteq U$. Of course, if $L \subseteq U$ then $U \subseteq L^{\perp} = K$ since $U$ is totally singular, so $\cP_g$ is simply the set of totally singular $i$-spaces in $K$.

If $L \subseteq U$, then $U/L$ is a totally singular $(i-1)$-space in $K/L$ and so there are $U_q(n-2,i-1)$ such spaces. Note that if $i=n/2$ then we must have $L \subseteq U$ (otherwise $L \oplus U$ is totally singular, which is not possible since $U$ is a maximal totally singular subspace of $V$) and thus $|\cP_g| = U_q(n-2,i-1)$ in this case.

Now assume $i < n/2$ and $L$ is not contained in $U$. Here $(U \oplus L)/L$ is an $i$-dimensional totally singular subspace of $K/L$ and there are $U_q(n-2,i)$ such spaces in total. Given such a space $J/L$, there are exactly $q^{2i}$ distinct subspaces $U$ of $J$ such that $J = U \oplus L$, all of which are totally singular. This yields  $q^{2i}U_q(n-2,i)$ choices for $U$ in total and the result follows.
\end{proof}

In the next lemma, we refer to the set $\mathcal{A}_1$ defined in Table \ref{tab:A1}. Note that the special cases with $n \in \{4,5,6\}$ or $(n,i) \in \mathcal{A}_1$ that are excluded in the statement will be handled later in Proposition \ref{p:pi}.

\begin{lem}\label{l:uni_pi21}
The conclusion of Theorem \ref{t:subspace} holds in case {\rm (\texttt{S1})} when $X = {\rm U}_n(q)$ with $n \geqs 7$ and $(n,i) \not\in \mathcal{A}_1$.
\end{lem}

\begin{proof}
Let $g \in X$ be a unipotent element with Jordan form $(J_2,J_1^{n-2})$. Then $|\cP_g|$ is given by Lemma \ref{l:uni_pi} and it is a routine exercise to check that $|\cP_g| \geqs |\cP|^{4/5}$ when $n \geqs 7$ and $(n,i) \not\in \mathcal{A}_1$ (see Table \ref{tab:A1}). Therefore, the result follows via Lemma \ref{fixednumber}.
\end{proof}

Next let us turn to case {\rm (\texttt{S1})} for symplectic groups. Let $m$ and $k$ be integers with $m$ even and $1 \leqs k \leqs m/2$. For a prime power $q$, let
\[
S_q(m,k) = \frac{|{\rm Sp}_m(q)|}{q^{k(2m-3k+1)/2}|{\rm GL}_k(q)||{\rm Sp}_{m-2k}(q)|}
\]
be the number of totally singular $k$-spaces in an $m$-dimensional symplectic space over $\mathbb{F}_q$. In addition, set $S_q(m,0) = 1$ and $S_q(m,k) = 0$ if $k<0$ or $k>m/2$.

\begin{lem}\label{l:symp_pi}
Let $X = {\rm PSp}_n(q)$ and let $G_{\a}$ be a maximal parabolic subgroup of type $P_{i}$, where $1 \leqs i \leqs n/2$. If $g \in X$ has Jordan form $(J_2,J_1^{n-2})$, then $|\mathcal{P}| = S_q(n,i)$ and
\[
|\mathcal{P}_g| = S_q(n-2,i-1) + q^iS_q(n-2,i).
\]
\end{lem}

\begin{proof}
This is essentially identical to the proof of Lemma \ref{l:uni_pi}. Working in ${\rm Sp}(V)$ we write $g = I+N$ with $N^2=0$ and ${\rm rank}(N)=1$. Setting $K = \ker(N)$ and $L = {\rm im}(N)$, we observe that $L \subseteq K = L^{\perp}$ and $K/L$ is an $(n-2)$-dimensional symplectic space. Then $\cP_g$ coincides with the set of totally singular $i$-spaces $U \subseteq K$, which we can count by separately considering the cases $L \subseteq U$ and $L \not\subseteq U$. We omit the details.
\end{proof}

\begin{lem}\label{l:symp_pi20}
The conclusion of Theorem \ref{t:subspace} holds in case {\rm (\texttt{S1})} when $X = {\rm PSp}_n(q)$ with $n \geqs 6$ and $(n,i) \not\in \mathcal{A}_2$.
\end{lem}

\begin{proof}
We proceed exactly as in the proof of Lemma \ref{l:uni_pi21}, working with
an element $g \in X$ with Jordan form $(J_2,J_1^{n-2})$. We obtained expressions for $|\cP|$ and $|\cP_g|$ in Lemma \ref{l:symp_pi} and we now conclude by appealing to Lemma \ref{fixednumber}.
\end{proof}

Now assume $X$ is an orthogonal group. Let $m$ and $k$ be integers with $m$ odd and $1 \leqs k \leqs (m-1)/2$. For an odd prime power $q$, let
\[
O_q(m,k) = \frac{|{\rm O}_m(q)|}{q^{k(2m-3k-1)/2}|{\rm GL}_k(q)||{\rm O}_{m-2k}(q)|}
\]
be the number of totally singular $k$-spaces in an $m$-dimensional orthogonal space over $\mathbb{F}_q$. In addition, we set $O_q(m,0) = 1$ and $O_q(m,k) = 0$ if $k<0$ or $k>(m-1)/2$.

Similarly, if $m$ is even and $1 \leqs k < m/2$, then for a prime power $q = p^f$ we define
\[
O_q^{\e}(m,k) = \frac{|{\rm O}_m^{\e}(q)|}{q^{k(2m-3k-1)/2}|{\rm GL}_k(q)||{\rm O}_{m-2k}^{\e}(q)|},
\]
where $\e=\pm$. This coincides with the number of totally singular $k$-spaces in an $m$-dimensional orthogonal space over $\mathbb{F}_q$ of type $\e$. Similarly, if $k = m/2$ and $\e=+$, then
\[
O_q^{+}(m,m/2) = \frac{|{\rm O}_m^{+}(q)|}{q^{m(m-2)/8}|{\rm GL}_{m/2}(q)|}
\]
is the total number of maximal totally singular subspaces in an $m$-dimensional plus-type orthogonal space over $\mathbb{F}_q$. Note that if $m \geqs 4$ then ${\rm P\O}_{m}^{+}(q)$ has two equal-sized orbits on the set of maximal totally singular subspaces of the natural module, whereas ${\rm PO}_m^{+}(q)$ is transitive on the latter set. We also set $O_q^{\e}(m,0)=1$, and
$O_q^{\e}(m,k)=0$ if $k<0$, $k>m/2$ or $(\e,k)=(-,m/2)$.

First assume $X = \O_n(q)$, where $n \geqs 7$ is odd and $q$ is odd. Here $X_{\a} = P_i$ with $1 \leqs i \leqs (n-1)/2$ and we define an element $g \in X$ as follows:
\begin{equation}\label{e:gdef}
g = \left\{ \begin{array}{ll}
g_+ = (-I_{n-1},I_1) & \mbox{if $n \equiv 1 \imod{4}$ or $q \equiv 1 \imod{4}$} \\
g_{-} = (-I_{n-1},I_1) & \mbox{if $n \equiv 3 \imod{4}$, $q \equiv 3 \imod{4}$ and $i < (n-1)/2$} \\
g_{u} = (J_2^2,J_1^{n-4}) & \mbox{otherwise,}
\end{array}\right.
\end{equation}
where the $(-1)$-eigenspace of $g_{\e}$ (for $\e=\pm$) is a nondegenerate orthogonal space of type $\e$.

\renewcommand{\arraystretch}{1.3}

\begin{lem}\label{l:odd_orth_pi}
Let $X = \O_n(q)$ with $n \geqs 7$ odd and let $G_{\a}$ be a maximal parabolic subgroup of type $P_{i}$, where $1 \leqs i \leqs (n-1)/2$. Define the element $g \in X$ as in \eqref{e:gdef}. Then $|\mathcal{P}| = O_q(n,i)$ and
\[
|\mathcal{P}_g| = \left\{\begin{array}{ll}
\hspace{1.65mm} O_q^{\e}(n-1,i) & \mbox{if $g = g_{\e}$ with $\e=\pm$} \\
\begin{array}{l}
q^{2i}O_q(n-4,i) + (q+1)q^{i-1}O_q(n-4,i-1) \\
+ ((q+1)q^{n-i-2}+1)O_q(n-4,i-2)
\end{array}
& \mbox{if $g = g_{u}$} \\
\end{array}\right.
\]
\end{lem}

\renewcommand{\arraystretch}{1}

\begin{proof}
First assume $g = g_{\e}$ with $\e=\pm$. Since the $1$-eigenspace of $g$ is $1$-dimensional and nondegenerate, we immediately deduce that $\cP_g$ coincides with the set of totally singular $i$-dimensional subspaces of the $(-1)$-eigenspace of $g$.  So we get $|\cP_g| = O_q^{\e}(n-1,i)$.

Now assume $g = g_{u}$ is a unipotent element with Jordan form $(J_2^2,J_1^{n-4})$ on $V$. As usual, write $g = I+N$ and set $K = \ker(N)$ and $L = {\rm im}(N)$. Note that $\dim K = n-2$, $\dim L = 2$ and $L \subseteq K = L^{\perp}$. Moreover, $L$ is the radical of $K$ and $K/L$ is a nondegenerate orthogonal space of dimension $n-4$. Let $U$ be a totally singular $i$-dimensional subspace of $V$ fixed by $g$ and note that $N(U) \subseteq U \cap L$. Set $d = \dim(U \cap L)$. More concretely, if we fix a standard orthogonal basis for $V$ of the form
\[
\{e_1, \ldots, e_{\ell}, f_1, \ldots, f_{\ell},x\}
\]
as in \cite[Proposition 2.5.3]{KL}, then we may choose $g$ so that
$e_2^g = e_1+e_2$, $f_1^g = f_1-f_2$ and $g$ fixes all other basis elements, which means that
\[
K = \la e_1,f_2, e_i, f_i,x \,:\, 3 \leqs i \leqs \ell\ra,\;\; L = \la e_1, f_2 \ra.
\]

First assume $N(U) = 0$, so $U \subseteq K$ and we can consider the image $W/L = (U+L)/L$ of $U$ in $K/L$. If $d=0$ then $W/L = (U \oplus L)/L$ is a totally singular $i$-space in $K/L$ and we deduce that there are $q^{2i}O_q(n-4,i)$ possibilities for $U$. Similarly, if $d=2$ then $L \subseteq U$ and $W/L = U/L$ is a totally singular subspace of $K/L$ with dimension $i-2$, so there are $O_q(n-4,i-2)$ possibilities for $U$ of this form. Finally, suppose $d=1$ and note that there are $q+1$ choices for $U \cap L$ as a subspace of $L$. Here $W/L$ is a totally singular subspace in $K/L$ of dimension $i-1$ and we deduce that there are $(q+1)q^{i-1}O_q(n-4,i-1)$ choices for $U$. Putting all of this together, it follows that $g$ fixes exactly
\[
q^{2i}O_q(n-4,i) + (q+1)q^{i-1}O_q(n-4,i-1) + O_q(n-4,i-2)
\]
totally singular $i$-spaces $U$ with $N(U) = 0$.

Now assume $N(U) \ne 0$. If $d=2$, then
$L\subseteq U$, and hence $U\subseteq L^\perp=K$, a
contradiction. Thus $d=1$ and ${\rm rank}(N|_U)=1$, so there are $q+1$ possibilities for $U \cap L$ as a subspace of $L$. Set $J = U \cap K$, so $J \cap L = U \cap L$, $\dim J = i-1$  and $(J+L)/L$ is a totally singular subspace of $K/L$ of dimension $i-2$. Given a totally singular subspace $A/L \subseteq K/L$ of dimension $i-2$ and a $1$-space $W$ of $L$, there are $q^{i-2}$ possibilities for $J$ such that $J+L = A$ and $J \cap L = W$. And then one checks that we can extend such a subspace $J$ to a totally singular $i$-space fixed by $g$ in exactly $q^{n-2i}$ ways. So it follows that $g$ fixes precisely
\[
(q+1)q^{i-2}q^{n-2i}O_q(n-4,i-2) = (q+1)q^{n-i-2} O_q(n-4,i-2)
\]
totally singular $i$-spaces $U$ with $N(U) \ne 0$. The result follows.
\end{proof}

\begin{lem}\label{l:odd_orth_pi20}
The conclusion of Theorem \ref{t:subspace} holds in case {\rm (\texttt{S1})} when $X = \O_n(q)$ with $n \geqs 7$ odd and $(n,i) \not\in \mathcal{A}_3$.
\end{lem}

\begin{proof}
Define the elements $g_{+}$, $g_{-}$ and $g_{u}$ in $X$ as in \eqref{e:gdef}. Set $\e = -$ if $n \equiv 3 \imod{4}$ and $q \equiv 3 \imod{4}$, otherwise $\e=+$. We now apply the previous lemma in order to calculate $\max\{|\cP_{g_{\e}}|,|\cP_{g_{u}}|\}$; the reader can then check that each pair $(n,i) \not\in \mathcal{A}_3$ is eliminated via
Lemma \ref{fixednumber}.
\end{proof}

Finally, we consider the case where $X$ is an even-dimensional orthogonal group.

\begin{lem}\label{l:even_orth_pi}
Let $X = {\rm P\O}_n^{\e}(q)$ with $n \geqs 8$ even and let $G_{\a}$ be a maximal parabolic subgroup of type $P_{i}$, where $1 \leqs i < n/2$ or $(\e,i) = (+,n/2)$. Suppose $g \in X$ has Jordan form $(J_2^2,J_1^{n-4})$, which is an $a_2$-type involution if $q$ is even.
\begin{itemize}\addtolength{\itemsep}{0.2\baselineskip}
\item[{\rm (i)}] We have $|\mathcal{P}| = a^{-1}O^{\e}_q(n,i)$, where $a = 2$ if $(\e,i) = (+,n/2)$, otherwise $a=1$.

\item[{\rm (ii)}] If $1 \leqs i < n/2$, then
\begin{align*}
|\cP_g| = & \; q^{2i}O_q^{\e}(n-4,i) + (q+1)q^{i-1}O_q^{\e}(n-4,i-1) \\
& \; + ((q+1)q^{n-i-2}+1)O_q^{\e}(n-4,i-2).
\end{align*}

\item[{\rm (iii)}] If $(\e,i) = (+,n/2)$, then $|\cP_g| = \frac{1}{2}((q+1)q^{i-2}+1)O_q^{+}(n-4,i-2)$.
\end{itemize}
\end{lem}

\begin{proof}
Part (i) follows immediately from the discussion preceding the statement of Lemma \ref{l:odd_orth_pi}. Indeed, if $(\e,i) \ne (+,n/2)$ then we may identify $\cP$ with the set of $i$-dimensional totally singular subspaces of $V$. And for $(\e,i) = (+,n/2)$ we have previously noted that $X$ has two equal-sized orbits on this set and we may identify $\cP$ with either one of them. In particular, in the latter case we get $|\cP| = \frac{1}{2}O^{+}_q(n,n/2)$.

In order to derive the expression for $|\cP_g|$ in part (ii) we can essentially repeat the argument in the proof of Lemma \ref{l:odd_orth_pi} for the unipotent element $g_{u}$ and we omit the details. In particular, no special difficulties arise when $p=2$.

Finally, let us turn to part (iii). Here we first observe that the expression in (ii) gives the exact number of fixed points of $g$ on the complete set of totally singular $i$-spaces, noting that $O_q^{+}(n-4,i) = O_q^{+}(n-4,i-1) = 0$. The result now follows by identifying $\cP$ with one of the two $X$-orbits on this set, and observing that $g$ has the same number of fixed points on both orbits.
\end{proof}

\begin{lem}\label{l:even_orth_pi20}
The conclusion of Theorem \ref{t:subspace} holds in case {\rm (\texttt{S1})} when $X = {\rm P\O}_n^{\e}(q)$ with $n \geqs 12$ even and $(n,i) \not\in \mathcal{A}_4$.
\end{lem}

\begin{proof}
This follows by combining the previous result with Lemma \ref{fixednumber}.
\end{proof}

By combining Lemmas \ref{l:R1linear}, \ref{l:uni_pi21}, \ref{l:symp_pi20}, \ref{l:odd_orth_pi20} and \ref{l:even_orth_pi20}, we see that it remains to deal with the following cases in order to complete the proof of Theorem \ref{t:subspace} when $G_{\a}$ is a maximal parabolic subgroup of type $P_i$:

\begin{itemize}\addtolength{\itemsep}{0.2\baselineskip}
\item[{\rm (a)}] $X = {\rm U}_n(q)$ and either $n \in \{4,5,6\}$ or $(n,i) \in \mathcal{A}_1$.
\item[{\rm (b)}] $X = {\rm PSp}_n(q)$ and $(n,i) \in \mathcal{A}_2$.
\item[{\rm (c)}] $X = {\rm \O}_n(q)$ and $(n,i) \in \mathcal{A}_3$.
\item[{\rm (d)}] $X = {\rm P\O}_n^{\e}(q)$ and either $n \in \{8,10\}$ or $(n,i) \in \mathcal{A}_4$.
\end{itemize}

\begin{prop}\label{p:pi}
The conclusion of Theorem \ref{t:subspace} holds in case {\rm (\texttt{S1})}.
\end{prop}

\begin{proof}
We need to handle cases (a)-(d) as described above. To do this, it will be convenient to define $m$ to be the dimension of a maximal totally singular subspace of $V$, so
\[
m = \left\{ \begin{array}{ll}
(n-2)/2 & \mbox{if $X = {\rm P\O}_n^{-}(q)$} \\
\lfloor n/2 \rfloor & \mbox{otherwise.}
\end{array}\right.
\]
In addition, set $\mathbb{F} = \mathbb{F}_{q^u}$, where $u = 2$ if $X = {\rm U}_n(q)$ and $u = 1$ in all other cases. We also define a parameter $e$ as follows:
\[
e = \left\{ \begin{array}{ll}
1 & \mbox{if $X = {\rm PSp}_n(q)$ or $\O_n(q)$} \\
1-\e & \mbox{if $X = {\rm P\O}_n^{\e}(q)$, $n$ even} \\
3/2 & \mbox{if $X = {\rm U}_n(q)$, $n$ odd} \\
1/2 & \mbox{if $X = {\rm U}_n(q)$, $n$ even.}
\end{array}\right.
\]

For $1 \leqs i \leqs m$, let $\mathcal{P}_i$ be the set of $i$-dimensional totally singular subspaces of $V$. Recall that we may identify $\mathcal{P}$ with $\mathcal{P}_i$, unless $X = {\rm P\O}_{n}^{+}(q)$ and $i=m$. In the latter case, $\mathcal{P}_{m}$ can be divided into two equivalence classes, where the subspaces $\pi_1,\pi_2 \in \cP_m$ are in the same class if and only if $m - \dim(\pi_1\cap\pi_2)$ is even (see  \cite[Theorem 6.11]{cam} or \cite[p.30]{KL}, for example). We denote the two equivalence classes by $\cP_m^+$ and $\cP_m^-$, respectively, and we note that $X$ acts transitively on both $\cP_m^+$ and $\cP_m^-$. Moreover, since the actions of $X$ on $\cP_m^+$ and $\cP_m^-$ are permutation isomorphic, we may identify $\mathcal{P}$ with either one of these classes.

By \cite[Lemma 9.4.1]{Brouwer}, we have
\begin{equation}\label{e:par_ord}
|\cP_i| = \br{m}{i}_{q^u} \cdot \prod_{j=0}^{i-1}\left( q^{u(m+e-j-1)}+1\right),
\end{equation}
where the Gaussian binomial coefficient
\[
\br{m}{i}_{q^u} = \prod_{j=0}^{i-1}\frac{q^{u(m-j)}-1}{q^{u(i-j)}-1}
\]
is the number of $i$-dimensional subspaces in an $m$-dimensional vector space over $\mathbb{F}$.

For a pair $\pi,\pi'\in\cP_i$, we define $a=\dim(\pi\cap\pi')$ and
\[
  r(\pi,\pi')=i-\dim(\pi' \cap \pi^\perp),
\]
where $\pi^\perp$ denotes the orthogonal complement of $\pi$.
Then the subspace $\la \pi,\pi'\cap\pi^{\perp}\ra$ is totally singular and we set
\[
b = \dim \la \pi,\pi'\cap\pi^{\perp}\ra = 2i-a-r(\pi,\pi').
\]
In particular, note that $b=m$ if $i=m$.

For $1\leqs i\leqs m$ and non-negative integers $a$ and $b$, we define
\begin{equation}\label{eqn_Riab}
\mathcal{R}_{i}^{a,b} =\{(\pi,\pi')\in\cP_i\times\cP_i \,:\, \dim(\pi \cap\pi')=a, \; \dim(\langle\pi,\pi'\cap\pi^{\perp}\rangle)=b\},
\end{equation}
which is nonempty if and only if $0\leqs a\leqs i\leqs b\leqs \min\{m,2i-a\}$ (see \cite[Proposition 4.9]{Stanton}). If we fix a subspace $\pi\in\cP_i$, then \cite[Lemma 9.4.2]{Brouwer} implies that there are precisely
\begin{equation}\label{flag}
N = q^{uc}  \br{i}{a}_{q^u}\br{i-a}{r}_{q^u}\br{m-i}{z}_{q^u}  \prod_{j=0}^{z-1} (q^{u(m+e-i-j-1)}+1)
\end{equation}
distinct $i$-spaces $\pi'\in\cP_i$ such that $(\pi,\pi')\in\mathcal{R}_{i}^{a,b}$,
where $z=b-i$, $r=2i-a-b$ and
\[
c=z(i-a)+r(2m-i-a-2z+e)-\frac{1}{2}r(r+1).
\]

If $X={\rm P\O}_{n}^+(q)$ and $\cP=\cP_i$ with $i<m$, then we define
\begin{equation}\label{eqn_Riab2}
\mathcal{R}_{i}^{a,m,\epsilon}=\{(\pi,\pi')\in\cP\times\cP \,:\, \dim(\pi \cap\pi')=a,\, \langle\pi,\pi'\cap\pi^{\perp}\rangle\in\cP_m^\epsilon\}
\end{equation}
for each $\epsilon\in\{+,-\}$ and non-negative integer $a$. If the set $\mathcal{R}_{i}^{a,m}$ in \eqref{eqn_Riab} is nonempty, then
\[
\mathcal{R}_{i}^{a,m} = \mathcal{R}_{i}^{a,m,+} \cup \mathcal{R}_{i}^{a,m,-}
\]
defines a partition into two equal-sized subsets. In particular, for each $\pi\in\cP_i$ and $\e \in \{+,-\}$, there are precisely $N/2$ distinct $\pi'\in\cP_i$ such that $(\pi,\pi')\in\mathcal{R}_{i}^{a,m,\epsilon}$, where $N$ is the expression in \eqref{flag} with $b=m$.

Similarly, if $X={\rm P\O}_{n}^+(q)$ and $\cP=\cP_m^\epsilon$ with $\epsilon\in\{+,-\}$, then we define
\[
  \mathcal{R}_{m}^{a,m,\epsilon}=\{(\pi,\pi')\in\cP\times\cP \,:\, \dim(\pi \cap\pi')=a\}.
\]
It is clear that $\mathcal{R}_{m}^{a,m,\epsilon}\subseteq \mathcal{R}_{m}^{a,m}$, where $\mathcal{R}_{m}^{a,m}$ is defined as in \eqref{eqn_Riab}. In addition, the definition of $\mathcal{P}$ implies that $\mathcal{R}_{m}^{a,m,\epsilon}$ is empty if $m-a$ is odd. On the other hand, if $m-a$ is even and $\pi$ is an element of $\cP$, then
\[
|\{\pi'\in\cP \,:\, (\pi,\pi')\in\mathcal{R}_{m}^{a,m,\epsilon}\}|
=|\{\pi'\in\cP_m\, :\, (\pi,\pi')\in\mathcal{R}_{m}^{a,m}\}|
\]
is equal to the expression $N$ in \eqref{flag}, where we set $i = b = m$.

\vs

We claim that the $X$-orbitals on $\cP$ are as follows:

\begin{itemize}\addtolength{\itemsep}{0.2\baselineskip}
\item[(A)] If $X \ne {\rm P\O}_{n}^{+}(q)$, then $\cP=\cP_i$ and the $X$-orbitals are precisely the nonempty sets of the form $\mathcal{R}_{i}^{a,b}$ in \eqref{eqn_Riab}. The corresponding $X$-subdegrees are given by the expression in \eqref{flag}.

\item[(B)] If $X = {\rm P\O}_{n}^{+}(q)$ and $i<m$, then $\cP=\cP_i$ and the $X$-orbitals comprise the nonempty sets $\mathcal{R}_{i}^{a,b}$ in \eqref{eqn_Riab} with $b<m$ (the corresponding subdegrees are given by \eqref{flag}), together with the nonempty sets of the form $\mathcal{R}_{i}^{a,m,\pm}$ as in \eqref{eqn_Riab2}, whose subdegrees are given by $N/2$, where $N$ is the expression in \eqref{flag} with $b=m$.

\item[(C)] If $X= {\rm P\O}_{n}^{+}(q)$ and $\cP=\cP_m^+$, then the $X$-orbitals comprise the nonempty sets $\mathcal{R}_{m}^{a,m,+}$ whose subdegrees are given by \eqref{flag}, setting $b=m$.
\end{itemize}

In order to justify the claim, let $I$ be the full isometry group of $V$ and let $S$ be the subgroup of isometries with determinant $1$. So for example, if $X = {\rm P\O}_n^{+}(q)$, then $I = {\rm O}_n^{+}(q)$ and $S = {\rm SO}_{n}^{+}(q)$, so in this case $I = S$ when $q$ is even. The isometry group $I$ is transitive on $\cP_i$ by Witt's Lemma. If $q$ is odd, then  \cite[Proposition 4.9]{Stanton} implies that the $I$-orbitals on $\mathcal{P}_i$ comprise the nonempty sets of the form  $\mathcal{R}_{i}^{a,b}$. In fact, it is easy to see that the proof of this proposition is also valid when $q$ is even, so the same conclusion holds for all $q$.

As a consequence, we immediately deduce that (A) holds when $X = {\rm PSp}_n(q)$ is a symplectic group. And in view of the discussion preceding the claim, we also see that (C) holds. Therefore, it just remains to justify the claims in (A) and (B) when $X$ is a unitary or orthogonal group.

Let $X$ be a unitary or orthogonal group. Fix a standard basis $\mathcal{B}$ for $V$ as in \cite[Propositions 2.3.2, 2.5.3]{KL} so that
\[
\la e_1, \ldots, e_m,f_1, \ldots, f_m\ra
\]
is a $2m$-dimensional nondegenerate subspace (of plus-type if $X$ is an orthogonal group).

Suppose $\mathcal{R}_{i}^{a,b}$ is nonempty and set $z=b-i$, $r=2i-a-b$. Then $(\pi,\pi')\in \mathcal{R}_{i}^{a,b}$, where
\[
\pi=\la e_1,\ldots,e_i\ra,\;\;
\pi'=\la e_1,\ldots,e_a,f_{a+1},\ldots,f_{a+r},e_{i+1},\ldots,e_{i+z}\ra.
\]
Fix a primitive element $\gamma \in \mathbb{F}^{\times}$.

Suppose $X = {\rm U}_n(q)$ is a unitary group and let $h \in I \setminus S$ be the isometry mapping $e_1 \mapsto \gamma e_1$, $f_1 \mapsto
\gamma^{-q}f_1$ and fixing all other basis elements in $\mathcal{B}$.
Then $I=S\la h\ra$ and $h$ stabilizes both $\pi$ and $\pi'$. In particular, if we write $I_{\pi,\pi'}$ (respectively, $S_{\pi,\pi'}$) for the stabilizer in $I$ (respectively, $S$) of the subspaces $\pi$ and $\pi'$, then $I_{\pi,\pi'}=S_{\pi,\pi'}\la h\ra$ and we deduce that $\mathcal{R}_{i}^{a,b}$ is an $X$-orbital.

Finally, let us assume $X = {\rm P\O}_n^{\e}(q)$ is an orthogonal group and $i < n/2$. We make the following claims:
\begin{itemize}\addtolength{\itemsep}{0.2\baselineskip}
\item[(i)] If $n \ne 2m$, then $\mathcal{R}_{i}^{a,b}$ is an $X$-orbital.
\item[(ii)] Suppose $n=2m$ and $i<m$. If $b<m$, then $\mathcal{R}_{i}^{a,b}$ is an $X$-orbital. And if $b=m$, then $\mathcal{R}_{i}^{a,b}$ splits into two equal-sized $X$-orbitals, denoted by $\mathcal{R}_{i}^{a,m,+}$ and $\mathcal{R}_{i}^{a,m,-}$ in terms of our earlier notation.
\end{itemize}

Let $Q$ be the defining quadratic form on $V$ and let $\b(x,y)=Q(x+y)-Q(x)-Q(y)$ be the associated symmetric bilinear form. Given a nonsingular vector $v \in V$, we define the reflection $r_v : V \to V$ as follows:
\[
r_v(u)=u-\frac{\b(u,v)}{Q(v)}v.
\]

First assume $q$ is even, so $X = \O_{n}^{\e}(q)$ and $\e=\pm$.
If $(\pi+\pi')^\perp$ contains a nonsingular vector $v$, then $I = X\la r_v \ra$ and we note that $r_v$ stabilizes $\pi$ and $\pi'$. So in this case,  the $I$-orbital $\mathcal{R}_{i}^{a,b}$ is also an $X$-orbital. Now assume $(\pi+\pi')^\perp$ is totally singular. If $2m<n$ then the subspace $\la e_1, \ldots, e_m,f_1, \ldots, f_m\ra^\perp$ of $(\pi+\pi')^\perp$ would contain a nonsingular vector, whence $n = 2m$, $\e = +$ and we are in case (ii). In addition, note that $e_m+f_m$ is a nonsingular vector in $(\pi+\pi')^\perp$ if $i+z < m$, so we must have  $i+z=m$ and thus  $b=m$. Then since $i<m$, it follows that  the $I$-orbital $\mathcal{R}_{i}^{a,m}$ is a union of two $X$-orbitals, corresponding to the sets $\mathcal{R}_{i}^{a,m,+}$ and $\mathcal{R}_{i}^{a,m,-}$ defined above.

Now assume $q$ is odd and recall that $\gamma \in \mathbb{F}_q^{\times}$ is a primitive element. For each $\lambda\in\F^{\times}$, the reflection $r_{e_1+\lambda f_1} \in I$ has determinant $-1$ and spinor norm $\lambda(\F^{\times})^2 \in \F^{\times}/(\F^{\times})^2$, which means that $h=r_{e_1+f_1}r_{e_1+\gamma f_1} \in S \setminus \O_n^{\e}(q)$. Since $h$ stabilizes both $\pi$ and $\pi'$, it follows that the
$S$-orbital containing $(\pi,\pi')$ is equal to the corresponding $X$-orbital and we deduce that the $X$-orbitals on $\mathcal{P}$ coincide with the $S$-orbitals. To complete the argument, we need to consider the action of $S$
on each $I$-orbital; this is very similar to the previous argument for $q$ even and we omit the details.

This justifies the above claims in (i) and (ii), which in turn establishes (A) and (B).  So we now have a complete description of the $X$-orbitals on $\mathcal{P}$ and we can compute all of the $X$-subdegrees via \eqref{flag}. We are now ready to complete the proof of the proposition.

\vs

By \cite[Proposition 1.1]{GaussBin}, the Gaussian binomial coefficient $\br{a}{b}_{t}$ is a polynomial in $t$ of degree $b(a-b)$ with non-negative integer coefficients. As a consequence, it follows that $v = |\cP|$ and each $X$-subdegree on $\cP$ can be expressed as polynomials in $q$ with non-negative rational coefficients.

For example, suppose $X = {\rm U}_4(q)$ and $i=1$, so \eqref{e:par_ord} yields
\[
v = |\cP| = \br{2}{1}_{q^2}\left(q^{2\left(2+\frac{1}{2}-1\right)}+1\right) = (q^2+1)(q^3+1).
\]
In terms of a standard basis $\{e_1,e_2,f_1,f_2\}$ for $V$ (see \cite[Proposition 2.3.2]{KL}), we see that the two non-diagonal $X$-orbitals on $\cP$, namely $\mathcal{R}_1^{0,1}$ and $\mathcal{R}_1^{0,2}$, are represented by the pairs $(\la e_1 \ra,\la f_1\ra)$ and $(\la e_1 \ra,\la e_2\ra)$, respectively. And by appealing to \eqref{flag}, we calculate that the corresponding subdegrees are $q^5$ and $q^2(q+1)$.

To complete the proof, we first implement the computational method presented in Remark \ref{rem_subdegBound}, examining each case $(X,n,i)$ in Table \ref{tab:A1} in turn. In order to illustrate the general approach, let us consider the special case where $X={\rm P\O}_8^+(q)$ and $i=3$.

First observe that \eqref{e:par_ord} gives $v=(q+1)^2(q^2+1)^2(q^3+1)$. From our earlier work, we know that there are precisely $9$ distinct non-diagonal $X$-orbitals, namely
\[
\mathcal{R}_3^{j,3},\;\; \mathcal{R}_3^{j,4,+},\;\; \mathcal{R}_3^{j,4,-}
\]
with $j \in \{0,1,2\}$, and the corresponding subdegrees are as follows: \[
d_{0,3}=q^9, \; d_{1,3}=q^5(q^2+q+1), \; d_{2,3}=q^2(q^2+q+1),
\]
\[
d_{0,4}^{\pm}=q^6(q^2+q+1), \; d_{1,4}^{\pm}=q^3(q+1)(q^2+q+1),\; d_{2,4}^{\pm}=q(q^2+q+1).
\]
As integer polynomials in $q$, we observe that $d_{0,3}$, $d_{1,3}$ and $d_{0,4}^{\pm}$ all have degree larger than $\frac{3}{4}\deg(v)=\frac{27}{4}$, and we note that $d_{0,3}\geqs \frac{3}{2}v^{3/4}$ for all $q$.
So by setting $d = s(t+1)$, it follows that
\[
d=c_1d_{1,3}+c_2d_{2,3}+c_3d_{0,4}^++c_4d_{1,4}^++c_5d_{2,4}^+
\]
for some $5$-tuple $(c_1,\ldots,c_5)$ such that $c_1,c_2\in\{0,1\}$ and $c_3,c_4,c_5\in\{0,1,2\}$.

Fix such a $5$-tuple and evaluate the above expression for $d$. If $\deg(d)>\frac{3}{4}\deg(v)$ then the inequality $d<\frac{3}{2}v^{3/4}$ only holds for a small finite set of prime powers $q$ (see Lemma \ref{polynomial-bound}); in each case, we can check that there are no feasible parameters $(s,t)$ for $\cQ$ such that $d = s(t+1)$ (see Remark \ref{r:c1}(a)). On the other hand, if $\deg(d)\leqs \frac{3}{4}\deg(v)$ then we can work with the {\sc Magma} command \texttt{XGCD} (as explained in Remark \ref{rem_subdegBound}) to show that the inequality in \eqref{e:v2} is only satisfied for a small finite collection of prime powers $q$. And once again, the remaining possibilities for $q$ can be ruled out by showing that there are no feasible parameters for $\cQ$.

As expected, the classical generalized quadrangles with $X$ isomorphic to ${\rm U}_4(q)$ or ${\rm U}_5(q)$ emerge naturally in this process. For example, let us return to the case $X = {\rm U}_4(q)$ with $i=1$. As noted above, we have $v = (q^2+1)(q^3+1)$ and there are two nontrivial $X$-subdegrees, namely $d_1 = q^5$ and $d_2 = q^2(q+1)$. Setting $d = s(t+1)= c_1d_1+c_2d_2$ with $c_i \in \{0,1\}$, the bound $d < \frac{3}{2}v^{3/4}$ from Lemma \ref{lem_subdegBound} implies that $d = d_2$ is the only possibility. Therefore,
\[
v = (s+1)(st+1) = (q^2+1)(q^3+1), \;\; d = s(t+1) = q^2(q+1)
\]
and Lemma \ref{lem_subdegBound} gives $s^2=\gcd(d^2,v-1-d)=q^4$, so $s=q^2$ and $t=q$. Since $d=d_2$, we see that $\cQ$ has the same collinearity relation as the classical generalized quadrangle labelled ${\rm H}(3,q^2)$ in Table \ref{Classical}. And since the lines of a generalized quadrangle are uniquely determined by its collinearity relation, we conclude that $\cQ$ is ${\rm H}(3,q^2)$. In particular, $X_{\ell}$ is a $P_2$ parabolic subgroup of $X$.

Similarly, if $X = {\rm U}_4(q)$ and $i=2$, then $v = (q+1)(q^3+1)$ and the nontrivial subdegrees are $d_1 = q^4$ and $d_2 = q(q^2+1)$. Here the bound $d < \frac{3}{2}v^{3/4}$ implies that either $d = d_2$, or $d = d_1$ and $q=2$. In the latter case, $v = 27$ and we check that $(s,t) = (2,4)$ is the only possibility, but this is incompatible with the fact that $d = s(t+1) = 16$. For $d = d_2 = q(q^2+1)$ we deduce that $(s,t) = (q,q^2)$ and $X_{\ell}$ is a $P_1$ parabolic subgroup, so $\cQ$ is the dual of ${\rm H}(3,q^2)$, which in turn is isomorphic to ${\rm Q}^{-}(5,q)$ (recall that $X \cong {\rm P\O}_6^{-}(q)$).

In the same way, the reader can check that if $X = {\rm U}_5(q)$ and $i = 1$, then $(s,t) = (q^2,q^3)$ is the only possibility and $\cQ$ is isomorphic to the classical generalized quadrangle ${\rm H}(4,q^2)$ in Table \ref{Classical}. Similarly, if $i=2$ then $(s,t) = (q^3,q^2)$ and $\cQ$ is isomorphic to the dual of ${\rm H}(4,q^2)$.

By systematically implementing the above procedure, handling each possibility for $(X,n,i)$ from Table \ref{tab:A1} in turn, the proof of the proposition is reduced to the following cases:

\begin{itemize}\addtolength{\itemsep}{0.2\baselineskip}
\item[{\rm (i)}] $X = {\rm P\O}_{10}^{+}(q)$ with $i = 5$
and $d=q(q^2+1)(q^4+q^3+q^2+q+1)$.
\item[{\rm (ii)}] $X = \O_9(q)$ or ${\rm PSp}_{8}(q)$ with $i = 4$
and $d=q(q^2+1)(q^4+q^3+q^2+q+1)$.
\end{itemize}

Consider case (i). Here one can check that there is a unique small $X$-subdegree, in the sense of Remark \ref{rem_subdegBound}, namely $q(q^2+1)(q^4+q^3+q^2+q+1)$, which corresponds to the $X$-orbital labelled $\mathcal{R}_5^{3,5,+}$ above. To handle this case, it is helpful to observe that the orbital graph of $\mathcal{R}_5^{3,5,+}$ coincides with the so-called \emph{collinearity graph} of $\mathcal{Q}$, denoted by $\Gamma(\mathcal{Q})$. The  vertices of $\Gamma(\mathcal{Q})$ are labelled by the points in $\cP$, which we can identify with one of the two equivalence classes of $5$-dimensional totally singular subspaces of $V$, and two vertices are adjacent if and only if they are collinear in $\mathcal{Q}$. Furthermore, we can also view $\Gamma(\mathcal{Q})$ as the  \emph{halved dual polar graph} of the orthogonal geometry: the vertices are labelled by $\cP$, and two vertices $\pi$ and $\pi'$ are adjacent if and only if $\dim(\pi\cap\pi')=3$.

By \cite[Theorem 2.2.20]{SRG}, the eigenvalues of the adjacency matrix of the halved dual polar graph are as follows:
\[
q(q^2+1)(q^4+q^3+q^2+q+1), \;\; q^5+q^4+q^3-1,\;\; -q^2-1.
\]
On the other hand, \cite[Proposition 2.2.18]{SRG} states that the collinearity graph of $\cQ$ has eigenvalues
\[
s(t+1), \;\; s-1, \;\; -t-1.
\]
By comparing these two sets of eigenvalues, we deduce that $s=q^5+q^4+q^3$ and $t=q^2$. However, this violates Higman's inequality $s\leqs t^2$ in Lemma \ref{HD}, so this case does not occur.

Finally, let us turn to case (ii). Here $d=s(t+1)$ is the sum of the subdegrees corresponding to the orbitals $\mathcal{R}_4^{2,4}$ and $\mathcal{R}_4^{3,4}$, which are $q^3(q^2+1)(q^2+q+1)$ and $q(q+1)(q^2+1)$, respectively. Therefore,
\[
v=\prod_{j=1}^4(q^j+1), \;\; d=q(q^2+1)r \, \mbox{ with } \, r=\frac{q^5-1}{q-1}
\]
and thus Lemma \ref{lem_subdegBound} implies that
$s^2=\gcd(d^2,v-1-d)=q^2r$,
which in turn yields
\[
s=q\sqrt{r}, \;\; t=\frac{v-1-d}{s^2}=q^4, \;\; t=\frac{d}{s}-1=(q^2+1)\sqrt{r}-1.
\]
By comparing the two given expressions for $t$, we obtain
$(q^4+1)^2=(q^2+1)^2r$.
However, it is easy to check that this is false for every prime power $q$, so this case is also eliminated.

This completes the proof of the proposition.
\end{proof}

Now let us turn to {\rm (\texttt{S2})}. Here $X = {\rm L}_n(q)$ with $n \geqs 3$ and $G_{\a}$ is a parabolic subgroup of type $P_{i,n-i}$ with $1 \leqs i < n/2$, which means that $G_{\a}$ is the stabilizer of a flag $0 < U < W < V$ with $\dim U = i$ and $\dim W = n-i$. Note that
\[
|\cP| = \br{n}{n-i}_q\br{n-i}{i}_q
\]
and we extend the notation by setting $\br{m}{k}_q = 0$ if $k<0$ or $k>m$.

\begin{lem}\label{l:pini}
Let $X = {\rm L}_n(q)$ and let $G_{\a}$ be a maximal subgroup of $G$ of type $P_{i,n-i}$, where $1 \leqs i < n/2$. If $g \in X$ has Jordan form $(J_2,J_1^{n-2})$, then
\[
|\mathcal{P}_g| = \br{n-1}{n-i}_q\br{n-i}{i}_q + q^i\br{n-2}{n-2-i}_q \left(\br{n-1-i}{i}_q + q^{n-2i}\br{n-2-i}{i-2}_q\right).
\]
\end{lem}

\begin{proof}
Working in ${\rm SL}(V)$, write $g = I+N$ with $N^2 = 0$ and ${\rm rank}(N) = 1$. Set $K = \ker(N)$ and $L = {\rm im}(N)$ so that $L \subseteq K$ with $\dim L = 1$ and $\dim K = n-1$. Observe that an arbitrary subspace $S$ of $V$ is $g$-invariant if and only if $S\subseteq K$ or $L \subseteq S$. Let $0<U<W<V$ be a flag in $\cP$.

If $W \subseteq K$, then both $U$ and $W$ are fixed by $g$. And  since $\dim K = n-1$, it is clear that there are exactly
\[
\br{n-1}{n-i}_q\br{n-i}{i}_q
\]
flags of this form.

Now assume $W \not\subseteq K$, so $W$ is $g$-invariant if and only if $L \subseteq W$. Here there are $a-b$ possibilities for $W$, where $a$ is the number of $(n-i)$-spaces in $V$ containing $L$, and $b$ is the number of these spaces contained in $K$. We compute
\[
a-b = \br{n-1}{n-1-i}_q - \br{n-2}{n-1-i}_q = q^i\br{n-2}{n-2-i}_q.
\]
For a fixed subspace $W$ of this form, we have $\dim(W \cap K) = n-1-i$ and we now need to count the number of $g$-invariant $i$-spaces $U \subseteq W$. Clearly, every $i$-space in $W \cap K$ is fixed by $g$ and there are
\[
\br{n-1-i}{i}_q
\]
such spaces. Now assume $U \not\subseteq K$, which means that $U$ is fixed by $g$ if and only if $L \subseteq U$ (in particular, this situation can only arise if $i \geqs 2$). There are $c-d$ such spaces, where $c$ is the number of $i$-spaces in $W$ containing $L$, and $d$ is the number of these spaces contained in $K$. We get
\[
c-d = \br{n-1-i}{i-1}_q - \br{n-2-i}{i-1}_q = q^{n-2i}\br{n-2-i}{i-2}_q
\]
and the result follows.
\end{proof}

\begin{lem}\label{l:s21}
The conclusion of Theorem \ref{t:subspace} holds in case {\rm (\texttt{S2})} if $n \geqs 7$ and $(n,i) \not\in \mathcal{B}_1$, where $\mathcal{B}_1 = \{ (7,2),(7,3),(8,3),(9,3),(9,4)\}$.
\end{lem}

\begin{proof}
This follows by combining the previous result with Lemma \ref{fixednumber}, working with a unipotent element $g \in X$ with Jordan form $(J_2,J_1^{n-2})$ on $V$. Indeed, if $n \geqs 7$ and $(n,i) \not\in \mathcal{B}_1$ then it is straightforward to check that $|\cP_g| \geqs |\cP|^{4/5}$ and the result follows.
\end{proof}

In order to complete the proof of Theorem \ref{t:subspace} in case {\rm (\texttt{S2})} we need to eliminate the remaining cases where $3 \leqs n \leqs 6$ or $(n,i) \in \mathcal{B}_1$.

Recall that we identify $\cP$ with the set of flags of $V$ of the form $0<U<W<V$, where $\dim U = i$ and $\dim W = n-i$. It will be convenient to write $(U,W)$ for such a flag. In order to eliminate the remaining possibilities for $(n,i)$ we will work with the explicit description of the ${\rm GL}_n(q)$-orbitals on $\cP$ given by Karp and Thomas in \cite[Section 4]{Whittaker}.

\renewcommand{\arraystretch}{1.2}

Let $M$ be a $3 \times 3$ matrix with non-negative integer entries such that the row and column sums are both equal to $(i,n-2i,i)$ (so for example, the sum of entries in the second column is $n-2i$, and the sum of entries in the third row is $i$). Associated to $M$, we fix a direct sum decomposition
\begin{equation}\label{e:Vdec}
V = V_{11} \oplus V_{12} \oplus V_{13} \oplus \cdots \oplus V_{33}
\end{equation}
such that $\dim V_{jk} = M_{jk}$ for each pair $(j,k)$. Set
\begin{equation}\label{eqn_UWorbM}
\begin{array}{ll}
U = V_{11} \oplus V_{12} \oplus V_{13}, & W = U \oplus V_{21} \oplus V_{22} \oplus V_{23} \\
U' = V_{11} \oplus V_{21} \oplus V_{31}, & W' = U' \oplus V_{12} \oplus V_{22} \oplus V_{32}.
\end{array}
\end{equation}

\renewcommand{\arraystretch}{1}

By appealing to \cite[Propositions 4.25 and 4.32]{Whittaker}, if we now associate $M$ with the ${\rm GL}_n(q)$-orbital containing the pair $((U,W),(U',W'))$, then in this way we obtain a complete description of all the orbitals of ${\rm GL}_n(q)$ on $\mathcal{P}$. Moreover, \cite[Corollary 4.29]{Whittaker} implies that
\begin{equation}\label{e:dM}
d_M=q^{e_M}ab^{-1}
\end{equation}
is the ${\rm GL}_n(q)$-subdegree corresponding to $M$, where
\[
a = \prod_{\ell=1}^i(q^\ell-1)^2\prod_{\ell=1}^{n-2i}(q^\ell-1),\;\; b = \prod_{(j,k) \in \mathcal{A}}\prod_{\ell=1}^{M_{jk}}(q^\ell-1)
\]
and
\[
e_M=\sum_{1\leqs j'<j\leqs 3}\,\sum_{1\leqs k<k'\leqs 3} M_{jk}M_{j'k'},\;\; \mathcal{A} = \{ (j,k) \,: \, 1 \leqs j,k \leqs 3,\, M_{jk}>0\}.
\]

For example, suppose $(n,i) = (3,1)$. Here the $X$-orbitals are parameterised by the six $3 \times 3$ monomial matrices with entries in $\{0,1\}$. If we take such a matrix, say
\begin{equation}\label{e:M}
M = \begin{pmatrix}1&0&0\\0&0&1\\0&1&0\end{pmatrix},
\end{equation}
then we compute $e_M = 1$ and $a = b = (q-1)^3$, so the corresponding subdegree is $q$.

From the above description, it is clear that $X = {\rm L}_n(q)$ acts transitively on each ${\rm GL}_n(q)$-orbital, so the $X$-orbitals on $\cP$ are described in exactly the same way. Furthermore, we see that each $X$-orbital is invariant under all diagonal and field automorphisms of $X$. This means that if $L$ is any almost simple group with $X \normeq L \leqs {\rm P\Gamma L}_n(q)$, then each $L$-subdegree on $\mathcal{P}$ is of the form $d_M$ as in \eqref{e:dM}, and the complete set of subdegrees for $L$ is obtained by running over the set of all compatible matrices $M$ satisfying the required conditions on row and column sums.

So in order to compute the $G$-subdegrees on $\cP$, we need to determine the action of a graph automorphism of $X$ on the set of $X$-orbitals. This is resolved by the following lemma.

\begin{lem}\label{lem_PiniFuse}
Let $\gamma$ be a graph automorphism of $X$ and let $\mathcal{O}$ be an $X$-orbital on $\cP$ with associated matrix $M$. Then $\mathcal{O}^{\gamma}$ is the $X$-orbital with matrix $M'$, where $M'_{jk}=M_{4-j,4-k}$ for all $1\leqs j,k\leqs 3$.
\end{lem}

\begin{proof}
Take the decomposition of $V$ in \eqref{e:Vdec} and consider the pair $((U,W),(U',W'))$ in the $X$-orbital corresponding to $M$, where the component subspaces are defined as in \eqref{eqn_UWorbM}. Referring to the summands in \eqref{e:Vdec}, set $V_{jk}'=V_{4-j,4-k}$ for each pair $(j,k)$.

Let $\b$ be a nondegenerate bilinear form on $V$ such that the summands in \eqref{e:Vdec} are pairwise orthogonal. Then as usual, for any subspace $Y$ of $V$, we define
\[
Y^\perp=\{v\in V \,:\, \mbox{$\b(v,y) =0$ for all $y\in Y$} \}.
\]
Since we have already noted that each $X$-orbital is preserved by every diagonal and field automorphism of $X$, we may assume without loss of generality that $\gamma$ induces the map $Y \mapsto Y^\perp$ on the set of subspaces of $V$. It is routine to check that
\[
W^\perp= V_{11}' \oplus V_{12}' \oplus V_{13}',\;\; U^\perp=W^\perp \oplus V_{21}' \oplus V_{22}' \oplus V_{23}',
\]
with similar expressions for $W'^\perp$ and $U'^\perp$. The result now follows since the $X$-orbital $\mathcal{O}^\gamma$ contains the pair
$((W^\perp,U^\perp), (W'^\perp,U'^\perp))$ and \cite[Proposition 4.32]{Whittaker} implies that the corresponding matrix is $M'$.
\end{proof}

For example, suppose $(n,i) = (3,1)$ and consider the $X$-orbital $\mathcal{O}$ corresponding to the matrix $M$ in \eqref{e:M}. By the lemma, it follows that the $X$-orbital $\mathcal{O}^{\gamma}$  has matrix
\[
M' = \begin{pmatrix}0&1&0\\1&0&0\\0&0&1\end{pmatrix},
\]
so the corresponding $\la X, \gamma\ra$-orbital is $\mathcal{O} \cup \mathcal{O}^{\gamma}$ and the associated $\la X, \gamma\ra$-subdegree is $2q$.

\begin{prop}\label{p:R2_2}
The conclusion of Theorem \ref{t:subspace} holds in case {\rm (\texttt{S2})}.
\end{prop}

\begin{proof}
In view of Lemma \ref{l:s21}, we may assume that $3 \leqs n \leqs 6$ or
\[
(n,i) \in \{ (7,2),(7,3),(8,3),(9,3),(9,4)\}.
\]
Recall that the primitivity of $G$ on $\cP$ implies that $G \not\leqs {\rm P\Gamma L}_n(q)$, so $G$ contains a nontrivial graph or graph-field automorphism of $X$.

Notice that for each pair $(n,i)$, we can express $v = |\cP|$ and each $X$-subdegree on $\cP$ as rational polynomials in $q$. Moreover, these polynomials are independent of the choice of $q$ and so we are in a position to apply the approach presented in Remark \ref{rem_subdegBound}.

\vs

\noindent \emph{Case 1.} $i=1$

\vs

First assume $i=1$, in which case $3\leqs n\leqs 6$ and
\[
v=\frac{(q^n-1)(q^{n-1}-1)}{(q-1)^2}.
\]
If $q \leqs 3$ then it is easy to check that there are no feasible parameters for $\cQ$, using the approach in Remark \ref{r:c1}(a). So we may assume $q \geqs 4$.

Next let us discuss the $X$-subdegrees. For $n\geqs 4$, there are six (non-diagonal) $X$-orbitals, with two pairs each fusing to give one $G$-orbital by Lemma \ref{lem_PiniFuse}, so there are four nontrivial $G$-subdegrees in total, namely
\[
d_1=\frac{2q(q^{n-2}-1)}{q-1}, \; d_2=\frac{q^2(q^{n-3}-1)(q^{n-2}-1)}{(q-1)^2}, \; d_3=\frac{2q^{n-1}(q^{n-2}-1)}{q-1}, \; d_4=q^{2n-3}.
\]
Similarly, if $n=3$ then $G_{\a}$ has three orbits on $\cP \setminus \{\a\}$ and the corresponding subdegrees are $d_1$, $d_3$ and $d_4$ (note that $d_2 = 0$ when $n=3$). We set $d = s(t+1)$.

First assume $n \geqs 4$. For $j=3,4$, it is easy to check that $d_j \geqs \frac{3}{2}v^{3/4}$, so $d\in\{d_1,d_2,d_1+d_2\}$ by Lemma \ref{lem_subdegBound} and we consider each possibility for $d$ in turn. Proceeding as in Remark \ref{rem_subdegBound}, we find that the inequality in \eqref{e:v3} only holds for a small finite collection of prime powers $q$. And for every such $q$, we can argue as in Remark \ref{r:c1}(a) to show that there are no feasible parameters for $\cQ$.

Now assume $n=3$, so $v=(q+1)(q^2+q+1)$ and the nontrivial $G$-subdegrees are $2q$, $2q^2$ and $q^3$. Since $q \geqs 4$, we have $q^3\geqs \frac{3}{2}v^{3/4}$ and thus $d\in\{2q,2q^2,2q(q+1)\}$ by Lemma \ref{lem_subdegBound}. We now consider each possibility in turn. Set $e' = \gcd(d^2,v-1-d)$ and recall that $e' = s^2$ by Lemma \ref{lem_subdegBound}.

Suppose $d=2q$. Here $s^2=e' = q^2\gcd(4,q+2)$, so $s\in\{q,2q\}$ and thus $d = s(t+1) \geqs 3q$, which is a contradiction. Next assume $d=2q^2$, so $s^2=e' = \gcd(2,q)q$ and we get
\[
(s,t) = \left\{ \begin{array}{ll}
(q^{1/2},2q^{3/2}-1) & \mbox{if $q$ is odd} \\
(\sqrt{2}q^{1/2},\sqrt{2}q^{3/2}-1) & \mbox{if $q$ is even.}
\end{array}\right.
\]
Since $q \geqs 4$, we see that $t>s^2$ in both cases, which is incompatible with Lemma \ref{HD}(ii). Finally, suppose $d=2q(q+1)$. Here we have $s^2=\gcd(4,q)q^2$ and thus $s=\gcd(2,q)q$. Since $d=s(t+1)$ it follows that
\[
t=\frac{2(q+1)}{\gcd(2,q)}-1
\]
and it is now straightforward to check that
\[
v = (q+1)(q^2+q+1) \ne (s+1)(st+1)
\]
for all $q$. This final contradiction completes the proof in Case 1.

\vs

\noindent \emph{Case 2. $i \in \{2,3,4\}$, $(n,i) \ne (9,3)$}

\vs

First assume $i=2$, so $n \in \{5,6,7\}$. For $(n,q)\in\{(5,2),(7,2),(7,3)\}$, it is easy to show that there are no feasible parameters for $\cQ$ via  Remark \ref{r:c1}(a), so we can exclude these three cases for the remainder of Case 2.

The reader can check that there are precisely $k$ nontrivial $G$-subdegrees less than $\frac{3}{2}v^{3/4}$, where $k = 3,8,9$ for $n = 5,6,7$ (counting multiplicities). So if $d_1, \ldots, d_k$ are the relevant subdegrees, then
\begin{equation}\label{e:orb2}
d = s(t+1) = \sum_{j=1}^kc_jd_j
\end{equation}
for some $k$-tuple $(c_1, \ldots, c_k) \in \{0,1\}^k$. By proceeding as in Remark \ref{rem_subdegBound}, we can use {\sc Magma} to rule out every possible $k$-tuple, which allows us to eliminate every pair $(n,i)$ with $i=2$.

Now assume $(n,i)=(7,3)$, $(8,3)$ or $(9,4)$. For $q \leqs 5$ it is easy to check that there are no feasible parameters for $\cQ$, so we may assume $q \geqs 7$. Let $d_1, \ldots, d_k$ be the nontrivial $G$-subdegrees with $d_j < \frac{3}{2}v^{3/4}$. Then $k = 5,12,6$ in the three respective cases we are considering, and we can eliminate each one by arguing as in the previous paragraph. No special difficulties arise and we omit the details.

\vs

\noindent \emph{Case 3. $(n,i)=(9,3)$}

\vs

To complete the proof, we may assume $(n,i)=(9,3)$. First observe that
\[
v = \br{9}{3}_{q}\br{6}{3}_{q} = (q^2+1)^2(q^3+1)(q^4+1)(q^6+q^3+1)\left(\frac{q^5-1}{q-1}\right)\left(\frac{q^7-1}{q-1}\right)
\]
is a monic polynomial in $q$ of degree $27$ with non-negative integer coefficients. By working with \eqref{e:dM} and Lemma \ref{lem_PiniFuse}, we calculate that there are precisely $22$ nontrivial $G$-subdegrees less than $\frac{3}{2}v^{3/4}$, counting multiplicities. So if we label them $d_1, \ldots, d_{22}$, then we are left with $2^{22} = 4194304$ possibilities for $(c_1, \ldots, c_{22})$ to consider in \eqref{e:orb2}. To make this computationally feasible, we will use some geometric arguments in order to reduce the size of the search space.

Fix a basis $\{e_1,\ldots,e_9\}$ for $V$ and set $U=\la e_1,e_2,e_3\ra$ and $W_1=\la e_1,\ldots,e_6\ra$. Let $W_2 \ne W_1$ be another $6$-dimensional subspace that contains $U$ and set $j=\dim(W_1\cap W_2)-3$, so we have $0 \leqs j\leqs 2$. Let $\mathcal{O}_j$ be the $G$-orbital containing $((U,W_1),(U,W_2))$ and note that
\[
N_j=\begin{pmatrix}3&0&0\\0&j&3-j\\0&3-j&j\end{pmatrix}
\]
is the matrix associated to the $X$-orbital of $((U,W_1),(U,W_2))$.

Let $\overline{V}=V/U$, $\overline{W}_1=W_1/U$ and $\overline{W}_2=W_2/U$ so that $\dim(\overline{W}_1 \cap \overline{W}_2)=j$. The stabilizer of $U$ in ${\rm SL}_9(q)$ induces ${\rm GL}_6(q)$ on $\overline{V}$. By working with the set of all $3$-dimensional subspaces of $\overline{V}$, as in the proof of Lemma \ref{l:R1linear}, we deduce that if some $\mathcal{O}_j$ with $0\leqs j\leqs 2$ consists of collinear point pairs, then all three $G$-orbitals $\mathcal{O}_0$, $\mathcal{O}_1$ and $\mathcal{O}_2$ have this property. In other words, if we label the corresponding $G$-subdegrees $d_1$, $d_2$ and $d_3$, then this means that we only need to consider \eqref{e:orb2} for $k$-tuples of the form
$(a,a,a,c_4, \ldots, c_{k})$,
where $k = 22$ and $a, c_j \in \{0,1\}$ for all $4 \leqs j \leqs k$. So this observation allows us to reduce the search space by a factor of $4$.

To further reduce the size of the search space, let us consider the $X$-orbital on $\cP$ corresponding to the matrix
\[
M_1=\begin{pmatrix}0&0&3\\3&0&0\\0&3&0\end{pmatrix}.
\]
By Lemma \ref{lem_PiniFuse}, this fuses with the $X$-orbital aligned with
\[
M_1'=\begin{pmatrix}0&3&0\\0&0&3\\3&0&0\end{pmatrix}
\]
to form one $G$-orbital $\mathcal{O}$. Here the $G$-subdegree corresponding to $\mathcal{O}$ is $2q^{18}$, which is less than $\frac{3}{2}v^{3/4}$, and so we may assume that $d_4 = 2q^{18}$.

Now, if we take a direct sum decomposition
\begin{equation}\label{e:dec3}
V = V_1 \oplus V_2 \oplus V_3,
\end{equation}
where each $V_i$ is $3$-dimensional, then we see that the $X$-orbital corresponding to $M_1'$ contains the pair $((U,W),(U',W'))$, where
\[
U=V_1, \; W = V_1 \oplus V_3, \; U'=V_2,\;  W'=V_1 \oplus V_2.
\]

We claim that $\mathcal{O}$ consists of non-collinear point pairs. In particular, we only need to consider $k$-tuples of the form $(a,a,a,0,c_5, \ldots, c_k)$ in \eqref{e:orb2}, which means that we have reduced our initial search space of size $2^{22}$ to one of size $2^{19} = 524288$.

Seeking a contradiction, suppose the points $(U,W)$ and $(U',W')$ are collinear and let $\ell \in \cL$ be the line they determine. Let $A \in {\rm GL}_3(q)$ be an invertible matrix, which we can view as representing an invertible linear map $\varphi:V_1 \to V_3$. Using $A$, we define the following subspaces
\[
U_A=\{(u,0,\varphi(u)) \, : \, u \in V_1\}, \;\; W_A=\{(u,v,\varphi(u)) \, : \, u \in V_1, v \in V_2 \}
\]
with respect to the decomposition of $V$ in \eqref{e:dec3}.
Then both point-pairs $((U,W),(U_A,W_A))$ and $((U',W'),(U_A,W_A))$ are in the same orbital, so the point $(U_A,W_A)$ is also on the line $\ell$.

Let $I_3$ be the identity matrix in ${\rm GL}_3(q)$ and let $A \in {\rm GL}_3(q)$ be an invertible matrix such that ${\rm rank}(A-I_3)=3$. By the previous argument, the points $(U_A,W_A)$ and $(U_{I_3},W_{I_3})$ are collinear and we note that the $X$-orbital containing the pair $((U_A,W_A),(U_{I_3},W_{I_3}))$ has associated matrix
\[
\begin{pmatrix}0&0&3\\0&3&0\\3&0&0\end{pmatrix}.
\]
But by considering \eqref{e:dM} and Lemma \ref{lem_PiniFuse}, we calculate that the $G$-subdegree corresponding to this orbital is $q^{27}$, which is
incompatible with the bound $d< \frac{3}{2}v^{3/4}$ from Lemma \ref{lem_subdegBound}. So this contradicts our assumption that the point pairs in the $G$-orbital $\mathcal{O}$ are collinear.

So as explained above, this leaves us with $2^{19}$ possibilities for the $k$-tuple $(c_1, \ldots, c_k)$ in \eqref{e:orb2}. And the analysis of these tuples can be further streamlined by observing that there are $5$ pairs of identical subdegrees in the list $d_1, \ldots, d_{22}$. To eliminate all of the remaining $k$-tuples we now proceed as in Remark \ref{rem_subdegBound}, working with {\sc Magma} and the inequality in  \eqref{e:v3}. No special difficulties arise and this allows us to complete the proof of the proposition.
\end{proof}

\begin{rem}
Let $X = {\rm L}_{9}(q)$ and assume $G_{\a}$ is of type $P_{3,6}$, as in Case 3 in the proof of Proposition \ref{p:R2_2}. Here one can check that there are $38$ nontrivial $X$-subdegrees smaller than $\frac{3}{2}v^{3/4}$. So if we work with $X$-orbitals, rather than $G$-orbitals, then our initial search space comprises $2^{38} = 274877906944$ possibilities for the $k$-tuple $(c_1, \ldots, c_k)$ in \eqref{e:orb2}, which is too large for an exhaustive analysis. This explains why we need to work with $G$-orbitals, rather than $X$-orbitals, and it explains why we need Lemma \ref{lem_PiniFuse}. Indeed, as we noted in the proof of Proposition \ref{p:R2_2}, the initial search space when working with $G$-orbitals has size $2^{22}$, which can be reduced to a more manageable $2^{19} = 524288$ possibilities by bringing some additional geometric arguments into play.
\end{rem}

\subsection{Non-parabolic subgroups}\label{ss:nparab}

For the remainder of Section \ref{s:subspace}, we may assume $G_{\a}$ is a non-parabolic subspace subgroup of $G$ and so it remains to consider the cases we labelled {\rm (\texttt{S3})}-{\rm (\texttt{S6})} above.

We begin by considering case {\rm (\texttt{S3})}. Here $X = {\rm L}_n(q)$ and $G_{\a}$ is the stabilizer in $G$ of a decomposition $V = U \oplus W$, where $\dim U = i < n/2$. As in case {\rm (\texttt{S2})}, we first use fixed points to reduce the problem to a handful of possibilities for $(n,i)$.

\begin{lem}\label{l:R23_1}
The conclusion of Theorem \ref{t:subspace} holds in case {\rm (\texttt{S3})} if $n \geqs 6$ and $(n,i) \ne (6,2),(7,3)$.
\end{lem}

\begin{proof}
Here $G_{\a}$ is of type ${\rm GL}_i(q) \oplus {\rm GL}_{n-i}(q)$ with $1 \leqs i < n/2$ and we have
\[
|\mathcal{P}| = \frac{|{\rm GL}_n(q)|}{|{\rm GL}_{i}(q)||{\rm GL}_{n-i}(q)|} = \frac{q^{i(n-i)}\prod_{j=1}^i (q^{n-i+j}-1)}{\prod_{j=1}^{i}(q^j-1)}.
\]
Let $g \in X$ be unipotent with Jordan form $(J_2,J_1^{n-2})$ and note that
\[
|g^X| = \frac{|{\rm SL}_n(q)|}{q^{2n-3}|{\rm GL}_{n-2}(q)|} = \frac{(q^{n-1}-1)(q^n-1)}{q-1}
\]
(see \cite[Table B.3]{BG}, for example).

First assume $i=1$. Here $|g^X \cap X_{\a}|$ coincides with the total number of unipotent transvections in ${\rm SL}_{n-1}(q)$, so
\[
|g^X \cap X_{\a}| = \frac{(q^{n-2}-1)(q^{n-1}-1)}{q-1}
\]
and by applying Lemma \ref{l:fpr} we deduce that
\[
|\mathcal{P}_g| = \frac{q^{n-1}(q^{n-2}-1)}{q-1} \geqs |\mathcal{P}|^{4/5}
\]
for all $n \geqs 6$.

Now assume $2 \leqs i < n/2$. Here $|g^X \cap X_{\a}| = a_i + a_{n-i}$, where $a_k$ is the number of transvections in ${\rm SL}_k(q)$. Therefore,
\[
|g^X \cap X_{\a}| = \frac{(q^{i-1}-1)(q^i-1)}{q-1} + \frac{(q^{n-i-1}-1)(q^{n-i}-1)}{q-1}
\]
and it is straightforward to check that $|\mathcal{P}_g| \geqs |\mathcal{P}|^{4/5}$ unless $n \leqs 5$ or $(n,i) = (6,2)$, $(7,3)$.
\end{proof}

\begin{prop}\label{p:R3}
The conclusion of Theorem \ref{t:subspace} holds in case {\rm (\texttt{S3})}.
\end{prop}

\begin{proof}
In view of the previous lemma, we may assume that $n \in \{3,4,5\}$ and $i=1$. We may identify $\cP$ with the set of pairs of subspaces $(U,W)$ of $V$, where $V = U \oplus W$ and $\dim U = 1$. We have
\[
v = |\cP| = \frac{q^{n-1}(q^{n}-1)}{q-1}.
\]

Set $\mathcal{B}_2 = \{(4,2),(4,3),(4,4),(5,2)\}$. If $(n,q) \in \mathcal{B}_2$ then we can use the method in Remark \ref{r:c1}(a) to show that there are no feasible parameters for $\cQ$, so for the remainder we may assume $(n,q) \not\in \mathcal{B}_2$.

Fix a basis $\{e_1, \ldots, e_n\}$ for $V$ and a primitive element $\gamma \in \mathbb{F}_q^{\times}$. Let $\b(u,v) = \sum_{i}u_iv_i$ be the standard inner product on $V$ with respect to this basis. For each $1$-space $\la u \ra$ in $V$, we define the hyperplane
\[
[u] = \la u \ra^{\perp} = \{v\in V \,: \, \b(u,v) = 0\}.
\]
By \cite{BannaiSong}, there are $q+2$ nontrivial ${\rm GL}_n(q)$-orbitals on $\mathcal{P}$, namely:
\begin{align*}
\mathcal{O}_1 & = \{((\la v\ra,[u]),(\la v'\ra,[u'])) \in \cP \times \cP \,:\,  \b(v,u') = \b(v',u) = 0\}\\
\mathcal{O}_{q+1} & = \{((\la v\ra,[u]),(\la v'\ra,[u'])) \in\cP \times \cP \,:\,  \mbox{exactly one of $\b(v,u')$, $\b(v',u)$ is $0$} \}\\
\mathcal{O}_{q+2} & = \{((\la v\ra,[u]),(\la v'\ra,[u'])) \in\cP \times \cP \,:\, \mbox{exactly one of $\la v\ra=\la v'\ra$, $[u]=[u']$ holds}\}
\end{align*}
and
\[
\mathcal{O}_j = \left\{((\la v\ra,[u]),(\la v'\ra,[u'])) \in \cP \times \cP \,:\, \begin{array}{l}
\mbox{$\b(v,u')\b(v',u) \ne 0$, $\frac{\b(u,v)\b(u',v')}{\b(v,u')\b(v',u)}=\gamma^{j-1}$,} \\
\la v \ra \ne \la v' \ra, [u] \ne [u']
\end{array} \right\},
\]
where $2\leqs j\leqs q$. The corresponding ${\rm GL}_n(q)$-subdegrees are as follows:
\[
d_1=\frac{q^{n-2}(q^{n-1}-1)}{q-1}, \; d_{q+1}=2q^{n-2}(q^{n-1}-1), \; d_{q+2}=2(q^{n-1}-1)
\]
and
\[
d_j=q^{n-2}(q^{n-1}-1),\; d_q=(q^{n-2}-1)(q^{n-1}-1)
\]
for all $2\leqs j\leqs q-1$.

We need to determine the $X$-orbitals on $\cP$. If $\gcd(n,q-1)=1$, then $X={\rm PGL}_n(q)$ and so the $X$-orbitals coincide precisely with the ${\rm GL}_n(q)$-orbitals described above.

Define the following collection of points in $\cP$:
\[
\alpha=(\la e_1\ra,[e_1]), \; \beta_1=(\la e_2\ra,[e_2]), \; \beta_{q+1}=(\la e_2\ra,[e_1+e_2]), \; \beta_{q+2}=(\la e_1\ra,[e_1+e_2])
\]
and
\[
\beta_j=(\la e_1+e_2\ra,[e_1+(\gamma^{j-1}-1)e_2]), \; \beta_q=(\la e_1+e_2\ra,[e_1+e_3])
\]
for all $2\leqs j\leqs q-1$. It is routine to check that $(\alpha,\beta_j)\in \mathcal{O}_j$ for each $1\leqs j \leqs q+2$. In addition, we note that the diagonal matrix ${\rm diag}(1,\ldots,1,\gamma) \in {\rm GL}_n(q)$ stabilizes each of these point pairs, with the exception of $(\a,\b_q)$ when $n=3$. This allows us to conclude that each $\mathcal{O}_j$ is also an $X$-orbital, with the possible exception of $\mathcal{O}_q$ when $n=3$ and $q \equiv 1 \imod{3}$.

So to obtain a complete description of the $X$-orbitals on $\cP$, we just need to consider the ${\rm GL}_n(q)$-orbital $\mathcal{O}_q$ when $n=3$ and $q\equiv 1\imod{3}$.

Let $H$ be the stabilizer of $\alpha$ in ${\rm SL}_3(q)$, which has order $q(q-1)(q^2-1)$. Take $g\in H$ so that it stabilizes $\beta_q$. Then $g$ stabilizes $\la e_1\ra$, $\la e_1+e_2\ra$ and $[e_1]\cap [e_1+e_3]=\la e_2\ra$, so there is a scalar $\lambda\in\F^{\times}$ such that $e_i^g=\lambda e_i$ for $i=1,2$. And since $g$ stabilizes $[e_1]$ and $\det(g)=1$, we also have $e_3^g=\lambda^{-2}e_3+\mu e_2$ for some $\mu \in\F$. Now $[e_1+e_3]=\la e_2,e_1-e_3\ra$ and $(e_1-e_3)^g\in [e_1+e_3]$, which implies that $\lambda^3=1$, so the stabilizer of $\beta_q$ in $H$ has order $3q$. So if $q \equiv 1 \imod{3}$, then the $X$-orbital containing $(\alpha,\beta_q)$ has corresponding subdegree $\frac{1}{3}d_q$.

Set $\beta_{q,j}=(\la e_1+e_2\ra,[e_1+\gamma^je_3])$ for $0\leqs j\leqs 2$ and let $\mathcal{O}_{q,j}$ be the $X$-orbital that contains $(\alpha,\beta_{q,j})$ for each $j$. If $q \equiv 1 \imod{3}$ then it is easy to show that the $\mathcal{O}_{q,j}$ form a partition of $\mathcal{O}_q$ (with ${\rm PGL}_3(q)$ transitively permuting the component parts) and each of the corresponding subdegrees is $\frac{1}{3}d_q$. So this explains how the ${\rm PGL}_3(q)$-orbital $\mathcal{O}_q$ splits into $X$-orbitals.

We are now ready to prove the proposition. First assume $n \in \{4,5\}$ and recall that we may assume $(n,q) \not\in \mathcal{B}_2$. It is routine to check that $d_1$ and $d_{q+2}$ are the only nontrivial $X$-subdegrees less than $\frac{3}{2}v^{3/4}$ and we can now proceed as we have done in several previous cases, working with the approach described in Remark \ref{rem_subdegBound} to eliminate each possibility $d \in \{ d_1, d_{q+2}, d_1+d_{q+2}\}$, where $d = s(t+1)$. More precisely, with the aid of {\sc Magma}, we show that the inequality in \eqref{e:v2} only holds for a small finite set of prime powers $q$, and in each of these cases we can argue as in Remark \ref{r:c1}(a) to show that there are no feasible parameters for $\cQ$.

Finally, suppose $n=3$, in which case $v=q^2(q^2+q+1)$. For $q \leqs 11$ we can use the method in Remark \ref{r:c1}(a) to show that there are no feasible parameters for $\cQ$, so we are free to assume that $q\geqs 13$. Recall that the nontrivial ${\rm PGL}_3(q)$-subdegrees on $\mathcal{P}$ are as follows:
\[
d_1=q(q+1), \; d_{q+1}=2q(q^2-1), \;d_{q+2}=2(q^2-1), \;d_q=(q-1)(q^2-1)
\]
and $d_{j}=q(q^2-1)$ for $2\leqs j\leqs q-1$. We have shown that each $d_j$ with $j \ne q$ is an $X$-subdegree, and the ${\rm PGL}_3(q)$-orbital corresponding to $d_q$ is an $X$-orbital unless $q \equiv 1 \imod{3}$, in which case it splits into three distinct $X$-orbitals and each of the corresponding $X$-subdegrees is $\frac{1}{3}d_q$. Since $q \geqs 13$, it is easy to check that $d_{q+1}$, $2d_{2}$ and $\frac{2}{3}d_q+d_2$ are all larger than $\frac{3}{2}v^{3/4}$. So by appealing to Lemma \ref{lem_subdegBound}, we deduce that
\[
d = s(t+1) = c_1d_1 + c_2d_{q+2}+c_3q(q^2-1) +\frac{1}{3}c_4(q-1)(q^2-1),
\]
where $c_1,c_2,c_3\in \{0,1\}$ and $c_4 \in \{0,1,2,3\}$, with the extra condition that $c_4 \leqs 1$ if $c_3=1$.

We now seek to eliminate each candidate $4$-tuple $(c_1,c_2,c_3,c_4)$ by arguing as in Remark \ref{rem_subdegBound}. In this way, we reduce the problem to the case where $c_j = 1$ for $j = 1,2,3$ and $c_4 = 0$, so we have
\[
d = q(q+1) + 2(q^2-1) + q(q^2-1) = q^3+3q^2-2.
\]
Now Lemma \ref{lem_subdegBound} gives $s^2=\gcd(d^2,v-1-d)=(q+1)^2$, so $s=q+1$ and thus $t=q^2+2q-3$ since $d=s(t+1)$. However, for these values of $s$ and $t$ it is easy to check that
\[
v = q^2(q^2+q+1) \ne (s+1)(st+1)
\]
for all $q$, which is a contradiction. This completes the proof of the proposition.
\end{proof}

\begin{prop}\label{p:R4}
The conclusion of Theorem \ref{t:subspace} holds in case {\rm (\texttt{S4})}.
\end{prop}

\begin{proof}
Here $X$ is a unitary, symplectic or orthogonal group, and $G_{\a}$ is the stabilizer of an $i$-dimensional nondegenerate subspace of $V$. There are several cases to consider.

\vs

\noindent \emph{Case 1.} $X = {\rm U}_n(q)$

\vs

Here $G_{\a}$ is of type ${\rm GU}_i(q) \perp {\rm GU}_{n-i}(q)$ with $1 \leqs i < n/2$ and we have
\[
|\mathcal{P}| = \frac{|{\rm GU}_n(q)|}{|{\rm GU}_{i}(q)||{\rm GU}_{n-i}(q)|} = \frac{q^{i(n-i)}\prod_{j=1}^i (q^{n-i+j}-(-1)^{n-i+j})}{\prod_{j=1}^{i}(q^j-(-1)^j)}.
\]
Recall that $X \in \mathcal{G}$ in the statement of Theorem \ref{t:subspace}, which means that $n \geqs 4$ (see Table \ref{tab:G}).

Let $g \in X$ be unipotent with Jordan form $(J_2,J_1^{n-2})$ on $V$, so
\[
|g^X| = \frac{|{\rm SU}_n(q)|}{q^{2n-3}|{\rm GU}_{n-2}(q)|} = \frac{(q^{n-1}-(-1)^{n-1})(q^n-(-1)^n)}{q+1}.
\]
Then for $i=1$ we compute
\[
|\cP_g| = \frac{q^{n-1}(q^{n-2}-(-1)^{n-2})}{q+1} \geqs |\cP|^{4/5}
\]
for all $n \geqs 7$. Similarly, if $i \geqs 2$ then
\[
|g^X \cap X_{\a}| = \frac{(q^{i-1}-(-1)^{i-1})(q^i-(-1)^i)}{q+1} + \frac{(q^{n-i-1}-(-1)^{n-i-1})(q^{n-i}-(-1)^{n-i})}{q+1}
\]
and we get $|\cP_g| \geqs |\cP|^{4/5}$ unless $n \leqs 7$ or $(n,i) = (8,3)$, $(9,4)$.

So to complete the proof of the proposition in Case 1, we need to eliminate the groups with $n \in \{4,5,6,7\}$ and $i=1$. Here
\[
v = |\cP| = \frac{q^{n-1}(q^n-(-1)^n)}{q+1}
\]
and we identify $\cP$ with the set of $1$-dimensional nondegenerate subspaces of $V$. By \cite[Section 1.1]{BannaiShen2}, the nontrivial ${\rm GU}_n(q)$-subdegrees on $\cP$ are as follows:
\[
d_1=\frac{q^{n-2}(q^{n-1}-(-1)^{n-1})}{q+1},\;\; d_2 = q^{2n-3}+(-1)^{n-1}(q-1)q^{n-2}-1
\]
and $d_j = q^{n-2}(q^{n-1}-(-1)^{n-1})$ for all $3 \leqs j \leqs q$. As noted in the first paragraph of \cite[Section 1]{BannaiShen2}, the $X$-orbitals are the same as the ${\rm GU}_n(q)$-orbitals.

We now proceed as in Remark \ref{rem_subdegBound}. First, we note that there are no feasible parameters for $\cQ$ if $(n,q)=(6,2)$, so we may assume $q \geqs 3$ if $n=6$. For $n \in \{6,7\}$, we check that each nontrivial $X$-subdegree is greater than $\frac{3}{2}v^{3/4}$, so these cases are immediately ruled out by Lemma \ref{lem_subdegBound}.

Now assume $n\in\{4,5\}$. Here we find that $d_j \geqs \frac{3}{2}v^{3/4}$ for all $j \geqs 2$, so Lemma \ref{lem_subdegBound} implies that $d = s(t+1) = d_1$.

Suppose $n=5$, so $d= q^3(q-1)(q^2+1)$. With the aid of the \texttt{XGCD} command in {\sc Magma}, we can find integer polynomials  $u(x),v(x)\in\mathbb{Z}[x]$ such that
\[
 w(q)=u(q)d^2+v(q)(v-1-d)=20(q-1)(q^2+1),
\]
and we check that the inequality in \eqref{e:v2} holds only if $q \leqs 19$. For each $q \leqs 19$, it is straightforward to show that there are no feasible parameters for $\cQ$ via Remark \ref{r:c1}(a).

Now assume $n=4$ and note that $d=q^2(q^2-q+1)$. As above, we find polynomials $u(x),v(x)\in\mathbb{Z}[x]$ such that
\[
w(q)=u(q)d^2+v(q)(v-1-d)=3(q^2-q+1),
\]
and we see that \eqref{e:v2} holds only if $q=2$. For $q=2$, the only  feasible parameter pair is $(s,t)=(3,3)$. And we cannot eliminate this possibility because under the isomorphism ${\rm U}_4(2) \cong {\rm PSp}_4(3)$, we see that it corresponds to the classical generalized quadrangle labelled ${\rm W}(3,3)$ in Table \ref{Classical}. Indeed, under this isomorphism, a subgroup in $X$ of type ${\rm GU}_1(2) \perp {\rm GU}_3(2)$ is mapped to a $P_1$ parabolic subgroup of ${\rm PSp}_4(3)$.

\vs

\noindent \emph{Case 2.} $X = {\rm PSp}_n(q)$

\vs

Here $G_{\a}$ is of type ${\rm Sp}_i(q) \perp {\rm Sp}_{n-i}(q)$, where $2 \leqs i < n/2$ is even, and we have
\[
|\cP| = \frac{|{\rm Sp}_n(q)|}{|{\rm Sp}_i(q)||{\rm Sp}_{n-i}(q)|} = \frac{q^{\frac{1}{2}i(n-i)}\prod_{j=1}^{i/2}(q^{n-i+2j}-1)}{\prod_{j=1}^{i/2}(q^{2j}-1)}.
\]
Let $g \in X$ be unipotent with Jordan form $(J_2,J_1^{n-2})$ and set $e = \gcd(2,q-1)$. Then
\[
|g^X| = \frac{1}{e}(q^n-1),\;\; |g^X \cap X_{\a}| \geqs \frac{1}{e}(q^{n-i}+q^i-2)
\]
and this allows us to conclude that $|\cP_g| \geqs |\cP|^{4/5}$ unless $(n,i) = (6,2)$.

So let us assume $(n,i) = (6,2)$ and note that $v = |\cP| = q^4(q^4+q^2+1)$. For $q=2$ we check that there are no feasible parameters for $\cQ$, so we may assume $q \geqs 3$. By \cite[Lemma 2.2.1]{BannaiShen2}, the nontrivial $X$-subdegrees on $\mathcal{P}$ are
\[
d_1 = (q+1)(q^4-1), \; d_2=q^2(q^2+1), \; d_3=q(q^2-1)(q^4-1), \;  d_4=q^2(q+1)(q^4-1)
\]
and $d_j = q^3(q^4-1)$ for all $5 \leqs j \leqs q+2$. Since $d_j \geqs \frac{3}{2}v^{3/4}$ for all $j\geqs 3$, we deduce that $d\in\{d_1,d_2,d_1+d_2\}$ by Lemma \ref{lem_subdegBound}. And then for each possibility for $d$, it is straightforward to check that the inequality in \eqref{e:v2} fails to hold for all $q \geqs 3$, which allows us to eliminate the case $(n,i) = (6,2)$.

\vs

\noindent \emph{Case 3.} $X = \O_n(q)$, $n \geqs 7$ odd

\vs

Write $n=2m+1$ and note that we may assume $1 \leqs i < n$ is odd. Suppose $i=1$, so
\[
|\cP| = \frac{|{\rm SO}_n(q)|}{2|{\rm SO}_{n-1}^{\e}(q)|} = \frac{1}{2}q^m(q^m+\e).
\]
Let $g \in X$ be an involution of type $(-I_{n-1},I_1)$ and note that
\[
|g^X| = \frac{|{\rm SO}_n(q)|}{2|{\rm SO}_{n-1}^{\e'}(q)|} \leqs \frac{1}{2}q^m(q^m+1),
\]
where $\e'=+$ if and only if $q^m \equiv 1 \imod{4}$ (see \cite[Table B.8]{BG}). Then
\[
|g^X \cap X_{\a}| \geqs \frac{|{\rm SO}_{n-1}^{\e}(q)|}{2|{\rm SO}_{n-2}(q)|} = \frac{1}{2}q^{m-1}(q^m-\e)
\]
and it is easy to check that $|\cP_g| \geqs |\cP|^{4/5}$ for all $n \geqs 7$.

Now assume $i \geqs 3$. Taking $g \in X$ to be an involution of type $(-I_{n-1},I_1)$ as above, we calculate the bounds
\[
|\cP| > \frac{1}{4}\left(\frac{q}{q+1}\right)q^{i(n-i)},\;\; |g^X| < q^{n-1},\;\; |g^X \cap X_{\a}| > \frac{1}{4}q^{a(n,i)},
\]
where $a(n,i) = i-1$ if $i > n/2$, otherwise $a(n,i) = n-i-1$. Then $|\cP_g| \geqs |\cP|^{4/5}$ if the following inequality
\[
q^{i(n-i)} \geqs 2^{12}\left(\frac{q+1}{q}\right)q^{5(n-1-a(n,i))}
\]
is satisfied. We check that this holds unless $n \leqs 9$, or if $(n,i,q) = (11,5,3)$, $(11,5,5)$. And the latter two possibilities can be eliminated by applying the method presented in Remark \ref{r:c1}(a), working with the set of divisors of $v=[X:X_{\a}]$ to show that there are no feasible parameters for $\cQ$.

\vs

\noindent \emph{Case 4.} $X = {\rm P\O}_n^{\e}(q)$, $n \geqs 8$ even

\vs

To complete the proof of the proposition, we may assume $X$ is an even-dimensional orthogonal group. Here \cite[Proposition 4.1.6]{KL} gives
\[
|\cP| = \frac{|\O_n^{\e}(q)|}{2^a|\O_i^{\e_1}(q)||\O_{n-i}^{\e_2}(q)|} > \frac{1}{4}\left(\frac{q}{q+1}\right)q^{i(n-i)},
\]
where $a=1$ if $p=2$ or $i=1$, and $a=2$ otherwise. Note that if $\e=+$ then $\e_1 = \e_2$ and we may assume $1 \leqs i < n/2$. On the other hand, if $\e=-$ then either $i$ is odd and we may assume $i<n/2$, or $i$ is even, $\e_1 = -\e_2$ and we may assume $i \leqs n/2$. Also note that $q$ is odd if $i$ is odd.

Suppose $q$ is odd and let $g \in X$ be an involution of type $(-I_{n-2},I_{2})$. Then
\[
|g^X| < \frac{|{\rm SO}_{n}^{-}(q)|}{2|{\rm SO}_{2}^{+}(q)||{\rm SO}_{n-2}^{+}(q)|} < 2q^{2n-4}
\]
and
\[
|g^X \cap X_{\a}| \geqs \frac{|{\rm SO}_{n-i}^{\e_2}(q)|}{2|{\rm SO}_2^{-}(q)||{\rm SO}_{n-i-2}^{\e_3}(q)|} > \frac{1}{4}\left(\frac{q}{q+1}\right)q^{2(n-i)-4},
\]
where we set $\e_3=-\e_2$ if $i$ is even. It follows that $|\cP_g| \geqs |\cP|^{4/5}$ if
\[
q^{i(n-i)} \geqs 2^{17}\left(\frac{q+1}{q}\right)^6q^{10i}.
\]
One can check that this inequality holds if $n \geqs 24$, or if $n = 22$ and $(i,q) \ne (1,3)$. The case $(n,i,q) = (22,1,3)$ can be ruled out because there are no feasible parameters for $\cQ$ (via the method in Remark \ref{r:c1}(a)), giving us the reduction $n \leqs 20$ recorded in Table \ref{tab:A}.

Similarly, if $q$ is even and we take $g \in X$ to be an $a_2$-type involution, then $|g^X| < 2q^{2n-6}$ and $|g^X \cap X_{\a}| > \frac{1}{2}q^{2(n-i)-6}$, noting that the latter expression is a lower bound on the number of $a_2$-type involutions in $\O_{n-i}^{\e_2}(q)$. Then $|\cP_g| \geqs |\cP|^{4/5}$ if
\[
q^{i(n-i)} \geqs 2^{12}\left(\frac{q+1}{q}\right)q^{10i}
\]
and we check that this holds for all $n \geqs 22$.

So to complete the proof in Case 4, we may assume $8 \leqs n \leqs 20$, $q$ is odd and $i=1$. Write $n= 2m$ and observe that there are two isometry classes of nondegenerate $1$-dimensional subspaces of $V$ (distinguished by the discriminant of the defining quadratic form on a representative $1$-space, which is either a square or nonsquare in $\mathbb{F}_q$). Both classes have the same size and we are free to identify $\cP$ with either one of them. Note that
\[
v = |\cP| =\frac{1}{2}q^{m-1}(q^m-\epsilon).
\]
By inspecting \cite[Lemmas 4.1,5.1]{BannaiShenSong} we see that every nontrivial $X$-subdegree on $\cP$ is at least $a = \frac{1}{2}q^{m-1}(q^{m-1}-1)$. And since $a \geqs \frac{3}{2}v^{3/4}$ for all odd $q$, we can now conclude by appealing to Lemma \ref{lem_subdegBound}.
\end{proof}

\begin{prop}\label{p:R5}
The conclusion of Theorem \ref{t:subspace} holds in case {\rm (\texttt{S5})}.
\end{prop}

\begin{proof}
Here $X = \O_{n}^{\e}(q)$, $q$ is even and
\[
|\cP| = [\O_{n}^{\e}(q) : {\rm Sp}_{n-2}(q)] = q^{m-1}(q^m-\e),
\]
where $n=2m$. Let $g \in X$ be an involution of type $a_2$, so from \cite[Table 3.5.1]{BG} we read off
\[
|g^X| = \frac{|{\rm O}_n^{\e}(q)|}{q^{2n-7}|{\rm O}_{n-4}^{\e}(q)||{\rm Sp}_2(q)|} = \frac{(q^{n-2}-1)(q^{m-2}+\e)(q^m-\e)}{q^2-1}.
\]
By considering the embedding of $X_{\a}$ in $X$ (as described in the proof of \cite[Proposition 4.1.7]{KL}, for example), we see that $g^X \cap X_{\a}$ coincides with the class of $a_2$-type involutions in ${\rm Sp}_{n-2}(q)$, so by appealing to \cite[Table 3.4.1]{BG} we get
\[
|g^X \cap X_{\a}| = \frac{|{\rm Sp}_{n-2}(q)|}{q^{2n-9}|{\rm Sp}_{n-6}(q)||{\rm Sp}_2(q)|} = \frac{(q^{n-4}-1)(q^{n-2}-1)}{q^2-1}.
\]
It is now straightforward to check that $|\cP_g| \geqs |\cP|^{4/5}$ if $n \geqs 12$.

Finally, let us assume $n \in \{8,10\}$. By inspecting \cite[Lemmas 2.1, 3.1]{BannaiShenSong} we see that every nontrivial $X$-subdegree
is at least $a = q^{m-1}(q^{m-1}-1)$. We then check that the inequality $a < \frac{3}{2}v^{3/4}$ holds if and only if $(\e,n,q) = (-,8,2)$, and the latter possibility can be ruled out by showing there are no feasible parameters for $\cQ$, using the approach in Remark \ref{r:c1}(a).
\end{proof}

\begin{prop}\label{p:R6}
The conclusion of Theorem \ref{t:subspace} holds in case {\rm (\texttt{S6})}.
\end{prop}

\begin{proof}
Here $X = {\rm Sp}_n(q)$ and $X_{\a} = {\rm O}_{n}^{\e}(q)$ is a $\C_8$-subgroup of $G$, where $n = 2m \geqs 6$ and $q$ is even, so we have
\[
|\mathcal{P}| = [{\rm Sp}_n(q) : {\rm O}_n^{\e}(q)] = \frac{1}{2}q^m(q^m+\e).
\]
Let $g \in X$ be an involution with Jordan form $(J_2,J_1^{n-2})$ on $V$ (so in the notation of \cite{AS}, $g$ is an involution of type $b_1$). There is a unique conjugacy class of such elements in both $X_{\a}$ and $X$, and by appealing to the centralizer information in \cite[Tables 3.4.1, 3.5.1]{BG} we deduce that
\[
|g^X \cap X_{\a}| = \frac{|{\rm O}_n^{\e}(q)|}{2|{\rm Sp}_{n-2}(q)|} = q^{m-1}(q^m-\e),\;\; |g^X| = \frac{|{\rm Sp}_n(q)|}{q^{n-1}|{\rm Sp}_{n-2}(q)|} = q^n-1.
\]
Therefore, in view of Lemma \ref{l:fpr} we deduce that
$|\mathcal{P}_g| = \frac{1}{2}q^{n-1}$.

For $(n,q,\e) \ne (6,2,+)$ it is straightforward to show that $|\cP_g| \geqs |\cP|^{4/5}$ and the result follows via Lemma \ref{fixednumber}. Finally, if
$(n,q,\e) = (6,2,+)$ then it is easy to check that there are no feasible parameters for $\cQ$ via Remark \ref{r:c1}(a), so this case is also eliminated.
\end{proof}

\vs

By combining Propositions \ref{p:pi}, \ref{p:R2_2}, \ref{p:R3}, \ref{p:R4}, \ref{p:R5} and \ref{p:R6}, we conclude that the proof of Theorem \ref{t:subspace} is complete.

\section{Geometric subgroups: Small order reduction}\label{s:irred1}

Let $G$ be a finite almost simple classical group over $\mathbb{F}_q$ with socle $X$. Write $q = p^f$ with $p$ a prime and let $V$ be the natural module for $X$. Set $n=\dim V$ and recall that we may assume $X \in \mathcal{G}$, where $\mathcal{G}$ comprises the groups recorded in Table \ref{tab:G}. Fix a point $\a \in \cP$.

With reference to the subgroup collections arising in the statement of Aschbacher's theorem (see Section \ref{sss:sub}), in this section we assume that
\[
G_{\a} \in \C_2 \cup \C_3 \cup \C_4 \cup \C_5 \cup \C_6 \cup \C_7.
\]
Our aim is to use Lemma \ref{large}(ii) in order to reduce the possibilities for $(G,G_{\a})$ to the restricted list of cases recorded in Table \ref{tab:B0}, which we refer to as collection $\mathcal{B}_0$.

Our main result is the following. Note that in this section, we assume  that $G$ is both point-primitive and flag-transitive.

\begin{thm}\label{t:red1}
Let $\mathcal{Q} = (\mathcal{P},\mathcal{L})$ be a finite thick generalized quadrangle admitting an almost simple point-primitive and flag-transitive automorphism group $G$ with socle $X \in \mathcal{G}$. If
\[
G_{\a} \in \C_2 \cup \C_3 \cup \C_4 \cup \C_5 \cup \C_6 \cup \C_7,
\]
then $(G,G_{\a}) \in \mathcal{B}_0$ is one of the cases in Table \ref{tab:B0}.
\end{thm}

{\small
\begin{table}
\[
\begin{array}{lcll} \hline
X & \mbox{Class} & \mbox{Type of $G_{\a}$} & \mbox{Conditions} \\ \hline
{\rm L}_n(q) & \C_2 & {\rm GL}_m(q) \wr S_k & \mbox{$k \leqs 6$, or $(m,k) = (1,7)$} \\
& \C_3 & {\rm GL}_{n/k}(q^k) & k \in \{2,3,5\} \\
& \C_4 & {\rm GL}_m(q) \otimes {\rm GL}_{n/m}(q) & \mbox{$m = 2$, $n \geqs 6$}  \\
& \C_5 & {\rm GL}_n(q_0) & \mbox{$q=q_0^k$, $k \in \{2,3,5\}$} \\
& \C_7 & {\rm GL}_m(q) \wr S_k & (n,m,k) = (9,3,2) \\
& & & \\
{\rm U}_n(q) & \C_2 & {\rm GU}_m(q) \wr S_k & \mbox{$k \leqs 6$, or  $(m,k) = (1,7)$} \\
& & {\rm GL}_{n/2}(q^2) & n \geqs 4 \\
& \C_3 & {\rm GU}_{n/k}(q^k) & k \in \{3,5\} \\
& \C_4 & {\rm GU}_m(q) \otimes {\rm GU}_{n/m}(q) & m = 2 \\
& \C_5 & {\rm GU}_n(q_0) & \mbox{$q = q_0^k$, $k \in \{3,5\}$} \\
& & {\rm Sp}_n(q),\, {\rm O}_n^{\e}(q) & n \geqs 4 \\
& \C_7 & {\rm GU}_m(q) \wr S_k & (n,m,k) = (9,3,2) \\
& & & \\
{\rm PSp}_n(q) & \C_2 & {\rm Sp}_m(q) \wr S_k & \mbox{$k \leqs 6$, or $m=2$ and $k \leqs 9$, or $(m,k) = (4,7)$} \\
& & {\rm GL}_{n/2}(q).2 & n \geqs 6 \\
& \C_3 & {\rm Sp}_{n/k}(q^k) & \mbox{$k \in \{2,3,5\}$, or $k =7$ and $n \in \{14,28\}$} \\
& & {\rm GU}_{n/2}(q) & n \geqs 6 \\
& \C_4 & {\rm Sp}_m(q) \otimes {\rm O}_{n/m}^{\e}(q) & \mbox{$m=2$, or $(n,m) = (12,4)$} \\
& \C_5 & {\rm Sp}_n(q_0) & \mbox{$q = q_0^k$, $k \in \{2,3,5\}$} \\
& \C_7 & {\rm Sp}_m(q) \wr S_k & (n,m,k) = (8,2,3) \\
& & & \\
\O_n(q) & \C_2 & {\rm O}_m(q) \wr S_k & k \in \{3,5\} \\
& \C_3 & {\rm O}_{n/k}(q^k) & k \in \{3,5\} \\
& \C_5 & {\rm O}_n(q_0) & \mbox{$q = q_0^k$, $k \in \{2,3,5\}$} \\
& \C_7 & {\rm O}_m(q) \wr S_k & (n,m,k) = (9,3,2) \\
& & & \\
{\rm P\O}_n^{\e}(q), \, \e = \pm & \C_2 & {\rm O}_m^{\e'}(q) \wr S_k & k \leqs 5 \\
& & {\rm O}_{n/2}(q)^2 & n \equiv 2 \imod{4} \\
& & {\rm GL}_{n/2}(q).2 & \e = + \\
& \C_3 & {\rm O}_{n/k}^{\e'}(q^k) & k \in \{2,3,5\} \\
& & {\rm GU}_{n/2}(q) & n \geqs 8 \\
& \C_4 & {\rm Sp}_m(q) \otimes {\rm Sp}_{n/m}(q) & (\e,m)=(+,2) \\
& \C_5 & {\rm O}_n^{\e'}(q_0) & \mbox{$q = q_0^k$, $k \in \{2,3,5\}$} \\
& \C_7 & {\rm Sp}_m(q) \wr S_k & (\e,n,m,k) = (+,16,4,2) \\ \hline
\end{array}
\]
\caption{The collection $\mathcal{B}_0$}
\label{tab:B0}
\end{table}
}

Our proof of Theorem \ref{t:red1} relies on the following lemma, which in turn is based on Lemma \ref{large}(ii). In the statement, the function $h(n)$ is defined in Table \ref{tab:fn}.

\begin{table}
\[
\begin{array}{lll} \hline
X & h(n) & \mbox{Conditions} \\ \hline
{\rm L}_n(q) & (n^2-12)/6 & n \geqs 3 \\
{\rm U}_n(q) & (n^2-17)/6 & n \geqs 4 \\
{\rm PSp}_n(q) & (n^2+n-10)/12 & n \geqs 6 \\
\O_n(q) & (n^2-n-10)/12 & \mbox{$n \geqs 7$ odd} \\
{\rm P\O}_n^{\e}(q) & (n^2-n-24)/12 & \mbox{$n \geqs 8$ even, $(n,\e) \ne (8,+)$} \\
{\rm P\O}_8^{+}(q) & 11/6 & \\ \hline
\end{array}
\]
\caption{The function $h(n)$ in Lemma \ref{l:small}}
\label{tab:fn}
\end{table}

\begin{lem}\label{l:small}
Let $\mathcal{Q} = (\mathcal{P},\mathcal{L})$ be a finite thick generalized quadrangle admitting an almost simple point-primitive and flag-transitive automorphism group $G$ with socle $X \in \mathcal{G}$. Then $|X_{\alpha}| > q^{h(n)}$.
\end{lem}

\begin{proof}
This is a straightforward application of Lemma \ref{large}(ii). Set $\theta = |{\rm Out}(X)|$ and first assume $X = {\rm L}_n(q)$ with $n \geqs 3$. Then $\theta = 2df$, where $d = \gcd(n,q-1)$, and we have $|X|>\frac{1}{2d}q^{n^2-1}$ by Lemma \ref{l:bounds}. It follows that
\[
\theta^{-5}|X| > \frac{q^{n^2-1}}{(2d)^6(\log_pq)^5}
\]
and it is easy to check that $(2d)^6(\log_pq)^5 \leqs q^{11}$. Therefore, $\theta^{-5}|X|>q^{n^2-12}$ and thus Lemma \ref{large}(ii) implies that
\[
|X_{\alpha}| > q^{(n^2-12)/6} = q^{h(n)}
\]
as required. A very similar argument applies when $X = {\rm U}_n(q)$, based on the observation that $(2d)^6(\log_pq)^5 \leqs q^{16}$ for $d = \gcd(n,q+1)$.

Next assume $X = {\rm PSp}_n(q)$ and $n \geqs 6$, so $\theta = df$, $|X|>\frac{1}{2d}q^{n(n+1)/2}$ and thus
\[
\theta^{-5}|X| > \frac{q^{n(n+1)/2}}{2d^6(\log_pq)^5},
\]
where $d = \gcd(2,q-1)$. The result now follows via Lemma \ref{large}(ii) since $2d^6(\log_pq)^5 \leqs q^5$. The case $X = \O_n(q)$ with $n \geqs 7$ odd is entirely similar since $|X|>\frac{1}{4}q^{n(n-1)/2}$, $\theta = 2f$ and $2^7(\log_pq)^5 \leqs q^5$.

Finally, suppose $X = {\rm P\O}_n^{\e}(q)$ and $n \geqs 8$ is even. For now, let us assume $(n,\e) \ne (8,+)$ and set $d = \gcd(4,q^{n/2}-\e)$. Then $|X|>\frac{1}{2d}q^{n(n-1)/2}$ and $\theta = 2df$, so
\[
\theta^{-5}|X| > \frac{q^{n(n-1)/2}}{(2d)^6(\log_pq)^5}
\]
and the result follows since $(2d)^6(\log_pq)^5 \leqs q^{12}$. Now assume $(n,\e) = (8,+)$. Here we observe that $|X|>\frac{1}{2d}q^{28}$, $\theta = 6df$ and we check that $(2d)(6d\log_pq)^5 \leqs q^{17}$ for all $q \geqs 2$. This completes the proof of the lemma.
\end{proof}

\begin{prop}\label{p:b_psl}
The conclusion of Theorem \ref{t:red1} holds when $X = {\rm L}_n(q)$.
\end{prop}

\begin{proof}
Set $d = \gcd(n,q-1)$. We inspect the list of cases presented in \cite[Table 3.5.A]{KL}.

First assume $G_{\a}$ is a $\C_2$-subgroup of type ${\rm GL}_m(q) \wr S_k$, where $n = mk$ and $m \geqs 1$. We may assume $k \geqs 7$. Here \cite[Proposition 4.2.9]{KL} gives
\[
|X_{\a}|  = d^{-1}(q-1)^{k-1}|{\rm SL}_m(q)|^kk!
\]
and thus $|X_{\a}|< q^{km^2+k\log_2k-1}$ since $|{\rm SL}_m(q)|\leqs q^{m^2-1}$ and $k! < k^k$ (see Lemma \ref{l:bounds}). It is easy to check that this upper bound on $|X_{\a}|$ is less than $q^{h(n)}$ unless $m = 1$ and $k \leqs 37$, or if $(m,k)$ is one of the following:
\begin{equation}\label{e:list1}
( 2, 7 ), ( 2, 8 ), ( 2, 9 ), ( 2, 10 ), ( 2, 11 ), ( 3, 7 ), ( 3, 8 ), ( 4, 7 ).
\end{equation}

So in view of Lemma \ref{l:small}, we may assume $(m,k)$ is one of these special cases. For each possibility in \eqref{e:list1}, we can use precise values to check that
\begin{equation}\label{e:xa}
|X_{\a}| \leqs \theta^{-5/6}|X|^{1/6},
\end{equation}
where $\theta = |{\rm Out}(X)|$. And similarly if $m = 1$ and $11 \leqs k \leqs 37$, noting that the maximality of $G_{\a}$ implies that $q \geqs 5$ (see \cite[Proposition 2.3.6]{BHR}). Finally, if $m = 1$ and $k \in \{8,9,10\}$ then we check that \eqref{e:xa} holds with the exception of a short list of possibilities for $q$; in each remaining case, we can use the arithmetic method presented in Remark \ref{r:c1}(b) to show that there are no feasible parameters for $\cQ$. So this just leaves us with the case $(m,k) = (1,7)$, which is included in Table \ref{tab:B0}.

Next assume $G_{\a} \in \C_3$ is of type ${\rm GL}_{n/k}(q^k)$, where $k \geqs 7$ is a prime divisor of $n$. Here \cite[Proposition 4.3.6]{KL} gives
\[
|X_{\a}| = d^{-1}\left(\frac{q^k-1}{q-1}\right)|{\rm SL}_{n/k}(q^k)|k < 2kq^{n^2/k-1}
\]
and one checks that this upper bound is less than $q^{h(n)}$ if $k \geqs 11$, or if $k = 7$ and $n \geqs 21$. If $(n,k) = (14,7)$ then \eqref{e:xa} holds, and the same is true if $n=k=7$ and $q \geqs 5$.
To eliminate the groups with $n=k=7$ and $q \in \{2,3,4\}$, we can apply the arithmetic method presented in Remark \ref{r:c1}(a), working with the set of divisors of $[X:X_{\a}]$ to show that there are no feasible parameters for $\cQ$.

Now suppose $G_{\a}$ is a $\C_4$-subgroup of type ${\rm GL}_m(q) \otimes {\rm GL}_{n/m}(q)$, where $3 \leqs m < \sqrt{n}$. Here \cite[Proposition 4.4.10]{KL} gives
\[
|X_{\a}| = d^{-1}\gcd(m,n/m,q-1)|{\rm SL}_{m}(q)||{\rm SL}_{n/m}(q)|<
q^{m^2+(n/m)^2-2}
\]
and this upper bound is less than $q^{h(n)}$ unless $(n,m) = (12,3)$. In the latter case, using precise values, it is easy to check that \eqref{e:xa} holds.

Next assume $G_{\a} \in \C_5$ is a subfield subgroup of type ${\rm GL}_n(q_0)$, where $q = q_0^k$ and $k \geqs 7$ is a prime. By  \cite[Proposition 4.5.3]{KL} and \cite[Lemma A.6]{BG} we see that $|X_{\a}| = e|{\rm L}_n(q_0)|$, where
\[
e = (q_0-1)^{-1}\gcd(n,q_0-1)\gcd\left(q_0-1,\frac{q-1}{d}\right) = \left\{ \begin{array}{ll} k & \mbox{if $(q-1)_k > n_k > 1$} \\ 1 & \mbox{otherwise,} \end{array}\right.
\]
and thus
\[
|X_{\a}| \leqs |{\rm SL}_n(q_0)| < q^{(n^2-1)/k}
\]
since $e \leqs \gcd(n,q_0-1)$ (here we write $r_k$ for the $k$-part of $r$, which is the highest $k$-power dividing $r$). This implies that $|X_{\a}| < q^{h(n)}$ if $n \geqs 9$. For the remaining cases with $3 \leqs n \leqs 8$, it is straightforward to verify the inequality in \eqref{e:xa} using exact values, with the exception of the groups with $q_0 \in \{2,3,4,5,8\}$ and $(n,k) = (3,7)$. But in each of the latter cases, it is a routine exercise to show that there are no feasible parameters for $\cQ$ (see Remark \ref{r:c1}(a)).

Now consider the case where $G_{\a}$ is a $\C_6$-subgroup of type $k^{2m}.{\rm Sp}_{2m}(k)$. Here $n = k^m$ and $k$ is a prime divisor of $q-1$ (in addition, if $k = 2$ then $q \equiv 1 \imod{4}$). By inspecting \cite[Propositions 4.6.5, 4.6.6]{KL}, we deduce that  $|X_{\a}| < q^{2m^2+3m}$ and we check that this upper bound is sufficient unless
\[
(k,m) \in \{ ( 2, 2 ), ( 2, 3 ), ( 2, 4 ), ( 3, 1 ), ( 3, 2 ), ( 5, 1 )\}.
\]

For $(k,m)  = (2,4)$ or $(3,2)$ we can check that \eqref{e:xa} is satisfied for all possible $q$. Similarly, if $(k,m) = (2,3)$ then \eqref{e:xa} holds unless $q=5$, whereas for $(k,m) = (5,1)$ we are left with $q = 11,16$. As above, we rule out these remaining possibilities by showing that there are no feasible parameters for $\mathcal{Q}$. The two remaining cases $(k,m) \in \{(2,2),(3,1)\}$ can be handled in the same way, noting that $q$ is a prime when $n = k^m = 3$ (see \cite[Table 8.3]{BHR}, for example).

Finally, let us assume $G_{\a}$ is a $\C_7$-subgroup of type ${\rm GL}_{m}(q) \wr S_k$, where $n = m^k$, $m \geqs 3$ and $k \geqs 2$. We may assume $(m,k) \ne (3,2)$. Here \cite[Proposition 4.7.3]{KL} implies that
\[
|X_{\a}| < q^{k(m^2-1)+k\log_2k}
\]
and we immediately deduce that $|X_{\a}| < q^{h(n)}$ unless
\[
(m,k) \in \{ ( 3, 3 ), ( 4, 2 ), ( 5, 2 ), ( 6, 2 )\}.
\]
In each of these remaining cases we can check directly that the inequality in \eqref{e:xa} is satisfied, and this completes the proof of the proposition.
\end{proof}

We now complete the proof of Theorem \ref{t:red1}.

\begin{proof}[Proof of Theorem \ref{t:red1}]
In view of Proposition \ref{p:b_psl}, we may assume $X$ is a unitary, symplectic or orthogonal group. In each case we proceed as before, working through the list of possibilities for $G_{\a}$ recorded in \cite[Tables 3.5.B-3.5.F]{KL}. Since no special difficulties arise, we only provide details in a couple of cases.
\begin{itemize}\addtolength{\itemsep}{0.3\baselineskip}
\item[{\rm (a)}] Suppose $X = {\rm PSp}_n(q)$ and $G_{\a}$ is a $\C_4$-subgroup of type ${\rm Sp}_m(q) \otimes {\rm O}_{n/m}^{\e}(q)$. Here $m \geqs 2$, $n/m \geqs 3$ and $q$ is odd. Since the case $m=2$ is recorded in Table \ref{tab:B0}, we may assume $m \geqs 4$. By inspecting \cite[Proposition 4.4.11]{KL} we see that
\[
|X_{\a}| = \frac{1}{2}\gcd(n/m,2)|{\rm Sp}_m(q)||{\rm SO}^{\e}_{n/m}(q)| < q^{\frac{1}{2}m(m+1)+\frac{1}{2}n(n/m-1)/m}.
\]
It is routine to check that this upper bound on $|X_{\a}|$ is less than $q^{h(n)}$ unless $(n,m) = (12,4)$. And in the latter case, we find that $|X_{\a}| > \theta^{-5/6}|X|^{1/6}$ when $q \ne p$, which explains why it is included in Table \ref{tab:B0}.

\item[{\rm (b)}] Suppose $X = {\rm P\O}_n^{\e}(q)$ and $G_{\a}$ is a $\C_2$-subgroup of type ${\rm O}_m^{\e'}(q) \wr S_k$, where $n = mk$. We may assume $k \geqs 6$ and there are several cases to consider. By \cite[Proposition 2.3.6]{BHR}, the maximality of $G_{\a}$ implies that $(m,q) \ne (3,3)$ and $(\e',m,q) \ne (+,2,2)$, $(+,2,3)$, $(+,2,4)$ or $(+,4,2)$.

First assume $m = 1$, so $q=p \geqs 3$ and $|X_{\a}| \leqs 2^{n-\b}n!$, where $\b = 1$ if $n$ is odd, otherwise $\beta = 2$ (see \cite[Proposition 4.2.15]{KL}). Since $n! < 2^{n\log_2 n}$, it is easy to check that $|X_{\a}| < q^{h(n)}$ if $n \geqs 39$, so we may assume $7 \leqs n \leqs 38$. In these remaining cases, by working with the exact values for $|X|$, $|X_{\a}|$ and $\theta = |{\rm Out}(X)|$, we can show that $|X_{\a}| \leqs \theta^{-5/6}|X|^{1/6}$, with the exception of a finite list of possibilities for $X$. For example, if $n = 7$, then we may assume $q \leqs 41$. In each of these remaining cases, we can use the method in Remark \ref{r:c1}(b) to show that there are no feasible parameters for $\cQ$ (for instance, see Example \ref{ex:c1} for the case $(\e,n,q) = (-,38,3)$).

Next assume $m \geqs 3$. By carefully inspecting \cite[Propositions 4.2.11, 4.2.14]{KL}, we deduce that
\[
|X_{\a}| \leqs \frac{1}{2}|{\rm O}_{m}^{\e'}(q)|^kk! < 2^{k-1}q^{\frac{1}{2}n(m-1)}k! < q^{\frac{1}{2}n(m-1)+k-1+k\log_2k}.
\]
It is then straightforward to check that this upper bound implies that  $|X_{\a}|<q^{h(n)}$ unless $(m,k)$ belongs to a short list of exceptions (recall that we are free to assume that $k \geqs 6$). For example, if $q$ is even (in which case $m$ is even), then the above estimate is sufficient unless
\[
(m,k) \in \{ (4,6), (4,7), (6,6), (8,6) \}.
\]
Working with precise expressions for $|X|$, $|X_{\a}|$ and $|{\rm Out}(X)|$, we can reduce to the special case $(m,k,q) = (4,6,2)$. Here $X = \O_{24}^{+}(2)$ and $X_{\a} = \O_{4}^{-}(2)^6.2^5.S_6$, recalling that $(\e',m,q) \ne (+,4,2)$ by \cite[Proposition 2.3.6]{BHR}. As above, we can eliminate this possibility by proceeding as in Remark \ref{r:c1}(b).

Finally, suppose $m=2$, in which case $k = n/2$. Here we have
\[
|X_{\a}| \leqs 2^{k-1}(q+1)^kk! < 2^{k-1}(q+1)^kq^{k\log_qk}
\]
and we check that this upper bound yields $|X_{\a}|< q^{h(n)}$ for all $n \geqs 44$. For the remaining cases with $12\leqs n \leqs 42$ we can use precise values for $|X|$, $|X_{\a}|$ and $|{\rm Out}(X)|$ to reduce to a finite list of possibilities for $(n,q)$, with $n \leqs 32$ in each case. And then each of these possibilities can be eliminated by applying the method in Remark \ref{r:c1}(b).
\end{itemize}

This completes the proof of Theorem \ref{t:red1}.
\end{proof}

\section{Geometric subgroups: Fixed point reduction}\label{s:irred2}

In this section, we seek to extend and strengthen the reduction obtained in Section \ref{s:irred1} by considering the classical groups with point stabilizer a non-subspace geometric maximal subgroup.
To do this, we will study the fixed points of carefully chosen automorphisms of $\cQ = (\cP,\cL)$, recalling that a candidate pair $(G,G_{\a})$ is eliminated via Lemma \ref{fixednumber} if there exists an element $1\ne g \in G$ with
\begin{equation}\label{e:fix}
|\mathcal{P}_g| \geqs |\cP|^{4/5}.
\end{equation}

Recall that a maximal subgroup $H$ of $G$ is \emph{non-subspace} and \emph{geometric} if it is contained in one of the $\C_i$ collections for $2 \leqs i \leqs 8$, with the additional condition that $H$ is not a $\C_8$-subgroup of type ${\rm O}_n^{\e}(q)$ when $X = {\rm Sp}_n(q)$ and $q$ is even.

The main theorem of this section is the following.

\begin{thm}\label{t:red2}
Let $\mathcal{Q} = (\mathcal{P},\mathcal{L})$ be a finite thick generalized quadrangle admitting an almost simple point-primitive and flag-transitive automorphism group $G$ with socle $X \in \mathcal{G}$. If $G_{\a}$ is a non-subspace geometric subgroup, then $(G,G_{\a}) \in \mathcal{B}$ is one of the cases in Table \ref{tab:B}.
\end{thm}

{\small
\begin{table}
\[
\begin{array}{lcll} \hline
X & \mbox{Class} & \mbox{Type of $G_{\a}$} & \mbox{Conditions} \\ \hline
{\rm L}_n(q) & \C_2 & {\rm GL}_m(q) \wr S_k & \mbox{$(m,k) \in \{ (2,2),(2,3),(2,4),(3,2),(3,3),(4,2)\}$,} \\
& & & \mbox{or $m=1$ and $3 \leqs k \leqs 7$} \\
& \C_3 & {\rm GL}_{n/k}(q^k) & \mbox{$k \in \{2,3,5\}$ and $n \leqs 8k$} \\
& \C_4 & {\rm GL}_m(q) \otimes {\rm GL}_{n/m}(q) & \mbox{$m=2$ and $6 \leqs n \leqs 18$} \\
& \C_5 & {\rm GL}_{n}(q_0) & \mbox{$q=q_0^k$, $k \in \{2,3,5\}$, $n \leqs 9$} \\
& \C_7 & {\rm GL}_m(q) \wr S_k & (n,m,k) = (9,3,2) \\
& \C_8 & {\rm Sp}_n(q) & n \in \{4,6,8\} \\
& & {\rm O}_n(q) & \mbox{$n \in \{3,5,7\}$, $q$ odd}  \\
& & {\rm O}_n^{\pm}(q) & \mbox{$4 \leqs n \leqs 16$ even, $q$ odd}  \\
& & {\rm U}_{n}(q_0) & \mbox{$q = q_0^2$, $n \leqs 9$} \\
& & & \\
{\rm U}_n(q) & \C_2 & {\rm GU}_m(q) \wr S_k & \mbox{$(m,k) \in \{ (2,2),(2,3),(2,4),(3,2),(3,3),(4,2)\}$,} \\
& & & \mbox{or $m=1$ and $4 \leqs k \leqs 7$} \\
& & {\rm GL}_{n/2}(q^2) & 4 \leqs n \leqs 16 \\
& \C_3 & {\rm GU}_{n/k}(q^k) & \mbox{$k \in \{3,5\}$ and $n \leqs 8k$} \\
& \C_4 & {\rm GU}_m(q) \otimes {\rm GU}_{n/m}(q) & \mbox{$m = 2$ and $6 \leqs n \leqs 18$} \\
& \C_5 & {\rm GU}_n(q_0) & \mbox{$q = q_0^k$, $k \in \{3,5\}$ and $n \leqs 9$} \\
& & {\rm Sp}_n(q) & n \in \{4,6,8\} \\
& & {\rm O}_n(q) & \mbox{$n \in \{5,7\}$, $q$ odd}  \\
& & {\rm O}_n^{\pm}(q) & \mbox{$4 \leqs n \leqs 16$ even, $q$ odd}  \\
& \C_7 & {\rm GU}_m(q) \wr S_k & (n,m,k) = (9,3,2) \\
& & & \\
{\rm PSp}_n(q) & \C_2 & {\rm Sp}_m(q) \wr S_k & (m,k) \in \{ (2,3),(2,4),(2,5),(4,2)\} \\
& & {\rm GL}_{n/2}(q).2 & \mbox{$6 \leqs n \leqs 18$, $q$ odd}  \\
& \C_3 & {\rm Sp}_{n/k}(q^k) & \mbox{$k \in \{2,3,5\}$ and $n \leqs 8k$, or $k=7$ and $n \in \{14,28\}$} \\
& & {\rm GU}_{n/2}(q) & \mbox{$6 \leqs n \leqs 18$, $q$ odd} \\
& \C_4 & {\rm Sp}_m(q) \otimes {\rm O}_{n/m}^{\e}(q) & \mbox{$m=2$ and $n \leqs 16$, or $(n,m) = (12,4)$} \\
& \C_5 & {\rm Sp}_n(q_0) & \mbox{$q = q_0^k$, $k \in \{2,3,5\}$ and $n \leqs 14$} \\
& \C_7 & {\rm Sp}_m(q) \wr S_k & (n,m,k) = (8,2,3) \\
& & & \\
\O_n(q) & \C_2 & {\rm O}_m(q) \wr S_k & (m,k) = (3,3) \\
& \C_3 & {\rm O}_{n/k}(q^k) & \mbox{$k \in \{3,5\}$ and $n \leqs 7k$}  \\
& \C_5 & {\rm O}_n(q_0) & \mbox{$q = q_0^k$, $k \in \{2,3,5\}$ and $n \in \{7,9\}$} \\
& \C_7 & {\rm O}_m(q) \wr S_k & (n,m,k) = (9,3,2) \\
& & & \\
{\rm P\O}_n^{\e}(q) & \C_2 & {\rm O}_m^{\e'}(q) \wr S_k & \mbox{$k \leqs 5$ and $n \leqs 18$} \\
\e = \pm
& & {\rm O}_{n/2}(q)^2 & \mbox{$n \equiv 2 \imod{4}$, $q$ odd and $n \leqs 18$} \\
& & {\rm GL}_{n/2}(q).2 & \mbox{$\e = +$ and $n \leqs 36$} \\
& \C_3 & {\rm O}_{n/k}^{\e'}(q^k) & \mbox{$k=2$ and $n \leqs 20$, or $k \in \{3,5\}$ and $n \leqs 16k$} \\
& & {\rm GU}_{n/2}(q) & n \leqs 36 \\
& \C_4 & {\rm Sp}_m(q) \otimes {\rm Sp}_{n/m}(q) & \mbox{$(\e,m)=(+,2)$ and $n \leqs 36$} \\
& \C_5 & {\rm O}_n^{\e'}(q_0) & \mbox{$q = q_0^k$, $k \in \{2,3,5\}$ and $n \leqs 18$} \\
& \C_7 & {\rm Sp}_m(q) \wr S_k & (\e,n,m,k) = (+,16,4,2) \\
\hline
\end{array}
\]
\caption{The collection $\mathcal{B}$}
\label{tab:B}
\end{table}
}

Suppose we can find an element $1 \ne g \in X$ such that
\[
|g^X \cap X_{\a}| \geqs a,\;\; |g^X| \leqs b.
\]
Then by appealing to Lemma \ref{l:fpr} (and Remark \ref{r:fpr}) we see that $|\cP_g| \geqs ab^{-1}|\cP|$ and thus \eqref{e:fix} holds if
\[
|\cP| = [X:X_{\a}] \geqs (ba^{-1})^5.
\]
In particular, \eqref{e:fix} holds if we have a lower bound $|\cP| \geqs c$ such that $(ba^{-1})^5 \leqs c$. In this situation, we can then eliminate $(G,G_{\a})$ via Lemma \ref{fixednumber}, and this is how we proceed in the proof of Theorem \ref{t:red2}.

\begin{rem}
As indicated in the preceding discussion, throughout this section we will write $b$ for an upper bound on $|g^X|$ for some suitably chosen element $g \in X$. This should not be confused with $|\cL|$, which is sometimes denoted by $b$ elsewhere in this paper.
\end{rem}

In view of Theorem \ref{t:red1}, we may assume throughout that either
\begin{itemize}\addtolength{\itemsep}{0.2\baselineskip}
\item[{\rm (i)}] $(G,G_{\a}) \in \mathcal{B}_0$ is one of the cases in  Table \ref{tab:B0}; or
\item[{\rm (ii)}] $X = {\rm L}_n(q)$ and $G_{\a}$ is a $\C_8$-subgroup of type ${\rm Sp}_n(q)$, ${\rm O}_n^{\e}(q)$ or ${\rm U}_n(q_0)$, with $q =q_0^2$ in the latter case.
\end{itemize}

\begin{prop}\label{p:lin1}
The conclusion of Theorem \ref{t:red2} holds when $X = {\rm L}_n(q)$.
\end{prop}

\begin{proof}
We may assume throughout the proof that $(G,G_{\a})$ is not one of the cases recorded in Table \ref{tab:B}. Set $d = \gcd(n,q-1)$ and first assume $G_{\a} \in \C_2$ is of type ${\rm GL}_m(q) \wr S_k$ with $n = mk$. By inspecting Table \ref{tab:B0}, we may assume that $k \leqs 6$ and $m \geqs 2$. Let $g \in X$ be a unipotent element with Jordan form $(J_2,J_1^{n-2})$ on $V$, in which case
\[
|g^X| = \frac{(q^{n-1}-1)(q^n-1)}{q-1} = b
\]
(see \cite[Table B.3]{BG}). By considering the embedding of $G_{\a}$ in $G$, we observe that $|g^X \cap X_{\a}|$ coincides with the number of elements in the direct product ${\rm GL}_m(q)^k$ of the form $(x_1, \ldots, x_k)$ such that one component has Jordan form $(J_2,J_1^{m-2})$ and the remainder are trivial. This means that
\[
|g^X \cap X_{\a}| = \binom{k}{1} \cdot \frac{(q^{m-1}-1)(q^m-1)}{q-1} = a
\]
and by applying Lemma \ref{l:bounds} we see that
\[
|\cP|=\frac{|{\rm SL}_n(q)|}{(q-1)^{k-1}|{\rm SL}_m(q)|^kk!}>\frac{1}{2}(k!)^{-1}q^{n(n-m)}=c.
\]
It is now straightforward to check that $(ba^{-1})^5 \leqs c$ unless
\[
(m,k) \in \{(2,2),(2,3),(2,4),(3,2),(3,3),(4,2) \},
\]
or if $(m,k) = (2,5)$ and $q \in \{2,3\}$. In the latter two cases, it is easy to verify that there are no feasible parameters for $\cQ$ by applying the method in Remark \ref{r:c1}(a). The remaining possibilities for $(m,k)$ are included in Table \ref{tab:B}.

Next suppose $G_{\a} \in \C_3$ is of type ${\rm GL}_{n/k}(q^k)$ with $k \in \{2,3,5\}$. Then
\[
|\cP| = \frac{|{\rm GL}_n(q)|}{k|{\rm GL}_{n/k}(q^k)|} > \frac{1}{4k}q^{n^2\left(1-\frac{1}{k}\right)} = c
\]
and we take $g \in X$ unipotent with Jordan form $(J_2^k,J_1^{n-2k})$. Then
\[
|g^X| = \frac{|{\rm GL}_n(q)|}{q^{k(2n-3k)}|{\rm GL}_k(q)||{\rm GL}_{n-2k}(q)|} < 2^{1+\delta_{2,q}}q^{2k(n-k)} = b
\]
and we observe that $|g^X \cap X_{\a}|$ coincides with the number of transvections in ${\rm GL}_{n/k}(q^k)$. Therefore,
\[
|g^X \cap X_{\a}|=\frac{(q^{n-k}-1)(q^n-1)}{q^k-1}=a
\]
and we deduce that $(ba^{-1})^5 \leqs c$ unless one of the following holds:
\begin{itemize}\addtolength{\itemsep}{0.2\baselineskip}
\item[{\rm (a)}] $k = 5$ and $n \leqs 40$.
\item[{\rm (b)}] $k=3$ and either $n \leqs 24$, or $(n,q) = (27,2)$.
\item[{\rm (c)}] $k=2$ and either $n \leqs 16$, or $n=18$ and $q \leqs 13$.
\end{itemize}
Here the cases with $n \leqs 8k$ are included in Table \ref{tab:B}. In case (c) with $n=18$, we calculate that $|\cP_g| \geqs |\cP|^{4/5}$ for $4 \leqs q \leqs 13$, and we can show that there are no feasible parameters when $q\in \{2,3\}$.

So this just leaves us with the case $(n,q) = (27,2)$ in (b), where we have $X = {\rm L}_{27}(2)$ and $X_{\a} = {\rm GL}_9(8).3$. Here we observe that $|\cP_g| < |\cP|^{4/5}$ and $v = |\cP|$ is too large to rule out the existence of feasible parameters using the method in Remark \ref{r:c1}(a). In addition, the size of $w = |{\rm Out}(X)||X_{\a}|$ and $\lfloor wv^{-1/5}\rfloor$ makes it difficult to apply the technique presented in Remark \ref{r:c1}(b). However, we can work with a modified version of the latter approach in order to eliminate this case.

To do this, recall that it suffices to show that there is no divisor $r$ of $w$ such that
\begin{equation}\label{e:cond}
\mbox{$r^5>v$ and $\delta = r^2+4(r-1)(v-1)$ is a perfect square.}
\end{equation}
Now, if $e$ is a prime divisor of $r$, then $\delta \equiv -4(v-1) \imod{e}$ and so we need $-4(v-1)$ to be a quadratic residue modulo $e$. By considering the prime factorization of $w$, we deduce that $r$ is not divisible by $7$, $13$ or $241$, which means that $r$ is a divisor of $w' = w/f$, where $f = 7^{10}13^2241$. We can now use {\sc Magma} to check that no divisor $r$ of $w'$ satisfies the two conditions in \eqref{e:cond} and so this allows us to conclude that there are no feasible parameters in this case.

Now assume $G_{\a} \in \C_4$ is a tensor product subgroup of type ${\rm GL}_2(q) \otimes {\rm GL}_{n/2}(q)$, where $q \geqs 3$ and $n \geqs 6$. Set $e = \gcd(2,q-1,n/2)$ and note that
\[
|\cP| = \frac{|{\rm SL}_n(q)|}{|{\rm SL}_2(q)||{\rm SL}_{n/2}(q)|e} > \frac{1}{4}q^{\frac{3}{4}(n^2-4)} = c.
\]
Take $g \in X$ unipotent with Jordan form $(J_2^2,J_1^{n-4})$ on $V$. Then
\[
|g^X| = \frac{|{\rm GL}_n(q)|}{q^{4n-12}|{\rm GL}_{n-4}(q)||{\rm GL}_2(q)|} < 2q^{4n-8}= b
\]
and we observe that an element $x = (x_1,x_2) \in {\rm PGL}_2(q) \times {\rm PGL}_{n/2}(q)$ is conjugate to $g$ if and only if $x_1 = 1$ and $x_2$ is a transvection. Therefore,
\[
|g^X \cap X_{\a}| = \frac{(q^{n/2-1}-1)(q^{n/2}-1)}{q-1} = a
\]
and we check that $(ba^{-1})^5 \leqs c$ for all $n \geqs 20$.

Next suppose $G_{\a} \in \C_5$ is a subfield subgroup of type ${\rm GL}_n(q_0)$ with $q = q_0^k$ and $k \in \{2,3,5\}$. Here
\[
|\cP| \geqs \frac{|{\rm SL}_n(q)|}{d|{\rm SL}_n(q_0)|} > \frac{1}{2d}q^{(n^2-1)\left(1-\frac{1}{k}\right)} = c
\]
and we take $g \in X$ unipotent with Jordan form $(J_2,J_1^{n-2})$. Then
\[
|g^X \cap X_{\a}| = \frac{(q_0^{n-1}-1)(q_0^n-1)}{q_0-1} = a,\;\; |g^X| = \frac{(q^{n-1}-1)(q^n-1)}{q-1} = b
\]
and we check that $(ba^{-1})^5 \leqs c$ for all $n \geqs 10$.

Finally, let us assume $G_{\a} \in \C_8$. First suppose $G_{\a}$ is of type ${\rm Sp}_n(q)$, so $n \geqs 4$ is even. Setting $e = \gcd(n/2,q-1)$, we have
\[
|\cP| = \frac{|{\rm SL}_n(q)|}{e|{\rm Sp}_n(q)|} >\frac{1}{2e}q^{\frac{1}{2}(n^2-n-2)} = c.
\]
The groups with $n \in \{4,6,8\}$ are recorded in Table \ref{tab:B}, so we may assume $n \geqs 10$. Let $g \in X$ be unipotent with Jordan form $(J_2,J_1^{n-2})$ on $V$. Then $|g^X \cap X_{\a}|$ coincides with the number of transvections in ${\rm PSp}_n(q)$, so
\[
|g^X \cap X_{\a}| = q^n-1 = a,\;\; |g^X| = \frac{(q^{n-1}-1)(q^n-1)}{q-1} = b
\]
and we check that $(ba^{-1})^5 \leqs c$ unless $(n,q) = (10,2)$. In the latter case, $X_{\a} = {\rm Sp}_{10}(2)$ and $|\cP_g|<|\cP|^{4/5}$. However, it is easy to check that there are no feasible parameters for $\cQ$, so this possibility is also eliminated.

Suppose $G_{\a}$ is of type ${\rm O}_n^{\e}(q)$, so $q$ is odd and
\[
|\cP| = \frac{|{\rm SL}_n(q)|}{d|{\rm SO}_n^{\e}(q)|} > \frac{1}{2d}q^{\frac{1}{2}(n^2+n-2)} = c.
\]
First assume $n \geqs 3$ is odd and let $g \in X$ be an involution of type $(-I_{n-1},I_1)$. We compute
\[
|g^X\cap X_{\a}| = \frac{|{\rm SO}_n(q)|}{2|{\rm SO}_{n-1}^{+}(q)|} +
\frac{|{\rm SO}_n(q)|}{2|{\rm SO}_{n-1}^{-}(q)|} = q^{n-1} = a
\]
and
\[
|g^X| = \frac{|{\rm GL}_{n}(q)|}{|{\rm GL}_{n-1}(q)||{\rm GL}_1(q)|} = \frac{q^{n-1}(q^n-1)}{q-1} = b,
\]
and we check that $(ba^{-1})^5 \leqs c$ for all $n \geqs 9$.

Now assume $n \geqs 4$ is even. Here we take $g \in X$ to be an involution of type $(-I_{n-2},I_2)$, so
\[
|g^X| = \frac{|{\rm GL}_n(q)|}{|{\rm GL}_{n-2}(q)||{\rm GL}_2(q)|} = \frac{q^{2n-4}(q^{n-1}-1)(q^n-1)}{(q-1)(q^2-1)} = b.
\]
For $\e = +$ we compute
\begin{align*}
|g^X \cap X_{\a}| & = \frac{|{\rm SO}_n^{+}(q)|}{2|{\rm SO}_{n-2}^{+}(q)||{\rm SO}_2^{+}(q)|} + \frac{|{\rm SO}_n^{+}(q)|}{2|{\rm SO}_{n-2}^{-}(q)||{\rm SO}_2^{-}(q)|} \\
 & = \frac{1}{2}q^{n-2}(q^{n/2}-1)\left(\frac{q^{(n-2)/2}+1}{q-1}+\frac{q^{(n-2)/2}-1}{q+1}\right) = a
 \end{align*}
 and we check that $(ba^{-1})^5 \leqs c$ if $n \geqs 18$. Similarly, if $\e=-$ then
 \[
 |g^X \cap X_{\a}| = \frac{1}{2}q^{n-2}(q^{n/2}+1)\left(\frac{q^{(n-2)/2}+1}{q+1}+\frac{q^{(n-2)/2}-1}{q-1}\right) = a
 \]
 and the same conclusion holds.

Finally, suppose $G_{\a} \in \C_8$ is of type ${\rm U}_n(q_0)$ with $q = q_0^2$. We have
\[
|\cP| \geqs \frac{|{\rm SL}_n(q)|}{d|{\rm SU}_n(q_0)|} > \frac{1}{2d}q^{\frac{1}{2}(n^2-1)} = c
\]
and we take $g \in X$ unipotent with Jordan form $(J_2,J_1^{n-2})$. Then
\[
|g^X \cap X_{\a}| = \frac{(q_0^{n-1}-(-1)^{n-1})(q_0^n-(-1)^n)}{q_0+1} = a,\;\; |g^X| = \frac{(q^{n-1}-1)(q^n-1)}{q-1} = b
\]
and we check that $(ba^{-1})^5 \leqs c$ for all $n \geqs 10$.
\end{proof}

\begin{prop}\label{p:uni1}
The conclusion of Theorem \ref{t:red2} holds when $X = {\rm U}_n(q)$.
\end{prop}

\begin{proof}
As in the proof of the previous proposition, we may assume that $(G,G_{\a})$ is not one of the cases recorded in Table \ref{tab:B}. Here $n \geqs 4$ and we begin by assuming $G_{\a} \in \C_2$ is of type ${\rm GU}_m(q) \wr S_k$. In view of Theorem \ref{t:red1}, we may assume $k \leqs 6$ and $m \geqs 2$, in which case we can proceed as in the proof of the previous proposition, working with a unipotent element $g \in X$ of type $(J_2,J_1^{n-2})$. We omit the details.

Now suppose $G_{\a} \in \C_2$ is of type ${\rm GL}_{n/2}(q^2)$. Here
\[
|\cP| = \frac{|{\rm GU}_n(q)|}{2|{\rm GL}_{n/2}(q^2)|} > \frac{1}{4}q^{n^2/2} = c
\]
and we take $g \in X$ unipotent with Jordan form $(J_2^2,J_1^{n-4})$. Then
\begin{equation}\label{e:j22}
|g^X| = \frac{|{\rm GU}_n(q)|}{q^{4n-12}|{\rm GU}_{2}(q)||{\rm GU}_{n-4}(q)|} = \frac{q(q^{n-3}+1)(q^{n-2}-1)(q^{n-1}+1)(q^n-1)}{(q+1)(q^2-1)} = b
\end{equation}
and we note that $|g^X \cap X_{\a}|$ coincides with the number of transvections in ${\rm SL}_{n/2}(q^2)$, so
\[
|g^X \cap X_{\a}| = \frac{(q^{n-2}-1)(q^n-1)}{q^2-1} = a.
\]
It is easy to check that $(ba^{-1})^5 \leqs c$ for all $n \geqs 18$.

Next suppose $G_{\a} \in \C_3$ is of type ${\rm GU}_{n/k}(q^k)$ with $k \in \{3,5\}$. Here we proceed as in the proof of the previous proposition, working with a unipotent element $g \in X$ with Jordan form $(J_2^k,J_1^{n-2k})$. We have
\[
|\cP| = \frac{|{\rm GU}_n(q)|}{|{\rm GU}_{n/k}(q^k)|k} >\frac{1}{2k}\left(\frac{q^k}{q^k+1}\right)q^{n^2\left(1-\frac{1}{k}\right)} = c
\]
and
\[
|g^X \cap X_{\a}| = \frac{(q^{n-k}-(-1)^{n/k-1})(q^n-(-1)^{n/k})}{q^k+1} = a,\;\; |g^X| <2q^{2k(n-k)} = b.
\]
For $k =5$ we check that $(ba^{-1})^5 \leqs c$ if $n \geqs 45$. And for $k=3$ the same conclusion holds if $n \geqs 27$, unless $(n,q) = (27,2)$. In the latter case, we can check that $|\cP_g| \geqs |\cP|^{4/5}$ by working with the exact values for $|\cP|$ and $|\cP_g|$.

Next assume $G_{\a} \in \C_4$ is of type ${\rm GU}_{2}(q) \otimes {\rm GU}_{n/2}(q)$, so
\[
|\cP| = \frac{|{\rm SU}_n(q)|}{|{\rm SU}_2(q)||{\rm SU}_{n/2}(q)|e}>\frac{1}{2e}q^{\frac{3}{4}(n^2-4)} = c
\]
with $e = \gcd(2,q-1,n/2)$. Setting $m=n/2$ and taking $g \in X$ of type
$(J_2^2,J_1^{n-4})$ as in the proof of the previous proposition, we compute
\[
|g^X \cap X_{\a}| = \frac{(q^{m-1}-(-1)^{m-1})(q^m-(-1)^m)}{q+1} = a
\]
and $|g^X| = b$ as in \eqref{e:j22}. It is then easy to check that $(ba^{-1})^5 \leqs c$ for all $n \geqs 20$.

Finally, let $G_{\a}$ be a $\C_5$-subgroup. If $G_{\a}$ is of type ${\rm GU}_n(q_0)$, where $q = q_0^k$ and $k \in \{3,5\}$, then we can proceed as in the previous proposition, working with a unipotent element $g \in X$ of type $(J_2,J_1^{n-2})$. Similarly, if $G_{\a} \in \C_5$ is of type ${\rm Sp}_n(q)$ or ${\rm O}_n^{\e}(q)$, then we can argue as before (see the analysis of the corresponding $\C_8$-subgroups for $X = {\rm L}_n(q)$ in the proof of Proposition \ref{p:lin1}). We omit the details in order to avoid unnecessary repetition.
\end{proof}

\begin{prop}\label{p:pspfix}
The conclusion of Theorem \ref{t:red2} holds when $X = {\rm PSp}_n(q)$.
\end{prop}

\begin{proof}
Here $n \geqs 6$ and we set $d = \gcd(2,q-1)$. First assume $G_{\a} \in \C_2$ is of type ${\rm Sp}_m(q) \wr S_k$ and let $g \in X$ be a unipotent element with Jordan form $(J_2,J_1^{n-2})$. Then
\[
|\cP| = \frac{|{\rm Sp}_n(q)|}{|{\rm Sp}_{m}(q)|^kk!} > \frac{1}{2}(k!)^{-1}q^{\frac{1}{2}n(n-m)} = c
\]
and we compute $|g^X \cap X_{\a}| \geqs \frac{1}{d}k(q^m-1)=a$ and $|g^X| \leqs q^n-1=b$. It is now easy to check that $(ba^{-1})^5 \leqs c$ unless $(m,k) = (2,3)$, $(2,4)$, $(2,5)$ or $(4,2)$.

Next assume $G_{\a} \in \C_2$ is of type ${\rm GL}_{n/2}(q).2$. Here $q$ is odd,
\[
|\cP| = \frac{|{\rm Sp}_n(q)|}{2|{\rm GL}_{n/2}(q)|} > \frac{1}{2}q^{\frac{1}{4}n(n+2)} = c
\]
and we take $g \in X$ with Jordan form $(J_2^2,J_1^{n-4})$. Then
\[
|g^X \cap X_{\a}| = \frac{(q^{(n-2)/2}-1)(q^{n/2}-1)}{q-1} = a
\]
and
\[
|g^X| \leqs \frac{|{\rm Sp}_n(q)|}{q^{2n-5}|{\rm Sp}_{n-4}(q)||{\rm O}_2^{+}(q)|} < q^{2n-2} = b.
\]
We now check that $(ba^{-1})^5 \leqs c$ for all $n \geqs 20$.

Now assume $G_{\a} \in \C_3$ is of type ${\rm Sp}_{n/k}(q^k)$, where $k \in \{2,3,5\}$. First observe that
\[
|\cP| = \frac{|{\rm Sp}_n(q)|}{|{\rm Sp}_{n/k}(q^k)|k} > \frac{1}{2k}q^{\frac{1}{2}n^2\left(1-\frac{1}{k}\right)} = c
\]
and let $g \in X$ be a unipotent element with Jordan form $(J_2^k,J_1^{n-2k})$ on $V$ (if $q$ is even and $k=2$, then we take $g$ to be a $c_2$-type involution in the notation of \cite{AS}). Then
\[
|g^X \cap X_{\a}| \geqs \frac{1}{d}(q^n-1) = a
\]
and by inspecting \cite[Table B.7]{BG} we see that
\[
|g^X| \leqs \frac{|{\rm Sp}_n(q)|}{q^{k(2n-3k+1)/2}|{\rm Sp}_{n-2k}(q)||{\rm SO}_{k}^{\e}(q)|} < 2q^{k(n-k+1)} = b
\]
if $q$ is odd (with $\e=+$ when $k=2$). And it is easy to check that the same upper bound $|g^X|<b$ is still valid when $q$ is even. Then $(ba^{-1})^5 \leqs c$ for all $n > 8k$.

Next suppose $G_{\a} \in \C_3$ is of type ${\rm GU}_{n/2}(q)$ with $q$ odd. This is very similar to the case handled above, where  $G_{\a}$ is a $\C_2$-subgroup of type ${\rm GL}_{n/2}(q)$. First observe that
\[
|\cP| = \frac{|{\rm Sp}_{n}(q)|}{2|{\rm GU}_{n/2}(q)|} > \frac{1}{4}\left(\frac{q}{q+1}\right)q^{\frac{1}{4}n(n+2)} = c
\]
and take $g \in X$ with Jordan form $(J_2^2,J_1^{n-4})$. Then $|g^X| < q^{2n-2}=b$,
\[
|g^X \cap X_{\a}|  = \frac{(q^{(n-2)/2}-(-1)^{(n-2)/2})(q^{n/2}-(-1)^{n/2})}{q+1} = a
\]
and we check that $(ba^{-1})^5 \leqs c$ for all $n \geqs 20$.

Now assume $G_{\a} \in \C_4$ is of type ${\rm Sp}_{2}(q) \otimes {\rm O}_{n/2}^{\e}(q)$ with $q$ odd. Then
\[
|\cP| \geqs \frac{|{\rm Sp}_{n}(q)|}{|{\rm Sp}_{2}(q)||{\rm O}_{n/2}^{\e}(q)|} > \frac{1}{4}q^{\frac{3}{8}(n^2+2n-8)} = c
\]
and by taking $g \in X$ to be an involution of type $(-I_2,I_{n-2})$ we compute
\[
|g^X|  = \frac{|{\rm Sp}_n(q)|}{|{\rm Sp}_2(q)||{\rm Sp}_{n-2}(q)|} = \frac{q^{n-2}(q^n-1)}{q^2-1} = b
\]
and
\[
|g^X \cap X_{\a}|  \geqs \frac{|{\rm SO}_{n/2}^{\e}(q)|}{|{\rm SO}_{n/2-1}^{\e'}(q)|} > \frac{1}{2}q^{n/2-1} = a,
\]
where $\e'=-$ if $n \equiv 2 \imod{4}$. We then check that $(ba^{-1})^5 \leqs c$ for all $n \geqs 18$.

Finally, let us assume $G_{\a} \in \C_5$ is a subfield subgroup of type ${\rm Sp}_n(q_0)$, where $q = q_0^k$ with $k \in \{2,3,5\}$. If $q$ is even, or if $k$ is odd, then
\[
|\cP| = \frac{|{\rm Sp}_n(q)|}{|{\rm Sp}_n(q_0)|}> \frac{1}{2}q^{\frac{1}{4}n(n+1)} = c
\]
and it is straightforward to check that the same lower bound also holds when $q$ is odd and $k=2$. Take $g \in X$ unipotent of type $(J_2,J_1^{n-2})$. Then $|g^X \cap X_{\a}| \geqs \frac{1}{d}(q_0^n-1) = a$, $|g^X| \leqs q^n-1=b$ and it is easy to check that $(ba^{-1})^5 \leqs c$ if $n \geqs 16$. (In fact, if $k=3$ then we only need $n \geqs 14$, while for $k=2$ we get $(ba^{-1})^5 \leqs c$ if $n \geqs 10$.)
\end{proof}

\begin{prop}\label{p:orthfix1}
The conclusion of Theorem \ref{t:red2} holds when $X = \O_n(q)$.
\end{prop}

\begin{proof}
Here $n \geqs 7$ is odd and $q$ is also odd. First assume $G_{\a} \in \C_2$ is of type ${\rm O}_m(q) \wr S_k$, where $k \in \{3,5\}$. Then $m \geqs 3$ and we have
\[
|\cP| = \frac{|{\rm SO}_n(q)|}{2^{k-1}|{\rm SO}_m(q)|^kk!} > \frac{1}{2^{k}}(k!)^{-1}q^{\frac{1}{2}n(n-m)} = c.
\]
For an involution $g \in X$ of type $(-I_{n-1},I_1)$ we compute the bounds
\[
|g^X \cap X_{\a}| \geqs k\left(\frac{|{\rm SO}_m(q)|}{2|{\rm SO}_{m-1}^{-}(q)|}\right) > \frac{1}{4}kq^{m-1} = a
\]
and
\[
|g^X| \leqs \frac{|{\rm SO}_n(q)|}{2|{\rm SO}_{n-1}^{+}(q)|} < q^{n-1} = b.
\]
We then check that $(ba^{-1})^5 \leqs c$ unless $(m,k) = (3,3)$.

Next suppose $G_{\a} \in \C_3$ is of type ${\rm O}_{n/k}(q^k)$ with $k \in \{3,5\}$. Here
\[
|\cP| = \frac{|{\rm SO}_n(q)|}{|{\rm SO}_{n/k}(q^k)|k} > \frac{1}{2k}q^{\frac{1}{2}n^2\left(1-\frac{1}{k}\right)} = c
\]
and taking $g \in X$ to be an involution of type $(-I_{n-k},I_k)$ we see that
\[
|g^X \cap X_{\a}| \geqs \frac{|{\rm SO}_{n/k}(q^k)|}{2|{\rm SO}_{n/k-1}^{-}(q^k)|}>\frac{1}{4}q^{n-k} =a
\]
and
\[
|g^X| \leqs \frac{|{\rm SO}_{n}(q)|}{2|{\rm SO}_k(q)||{\rm SO}_{n-k}^{+}(q)|} < q^{k(n-k)} = b.
\]
When $k=5$ we get $(ba^{-1})^5 \leqs c$ for all $n \geqs 45$. And the same conclusion holds if $k=3$ and $n \geqs 33$, or if $n=27$ and $q \geqs 19$. Finally, if $n=27$ and $q \leqs 17$ then it is an easy exercise to check that $(ba^{-1})^5 \leqs c$ when we set
\[
a = \frac{|{\rm SO}_{9}(q^3)|}{2|{\rm SO}_{8}^{-}(q^3)|},\;\; b = \frac{|{\rm SO}_{27}(q)|}{2|{\rm SO}_3(q)||{\rm SO}_{24}^{+}(q)|},\;\; c = \frac{|{\rm SO}_{27}(q)|}{|{\rm SO}_{9}(q^3)|3}.
\]

Finally, suppose $G_{\a} \in \C_5$ is of type ${\rm O}_n(q_0)$, where $q = q_0^k$ and $k \in \{2,3,5\}$. We have
\[
|\cP| = \frac{|{\rm SO}_n(q)|}{e|{\rm SO}_n(q_0)|} > \frac{1}{2}q^{\frac{1}{2}n(n-1)\left(1-\frac{1}{k}\right)} = c,
\]
where $e = 2$ if $k=2$, otherwise $e=1$. Taking $g \in X$ to be an involution of type $(-I_{n-1},I_1)$ gives
\[
|g^X \cap X_{\a}| \geqs \frac{|{\rm SO}_n(q_0)|}{2|{\rm SO}_{n-1}^{-}(q_0)|}>\frac{1}{4}q_0^{n-1} = a
\]
and we have $|g^X|<q^{n-1} = b$ as above. We now check that $(ba^{-1})^5 \leqs c$ unless $n\leqs 9$, or $(k,n,q) = (2,11,9)$. In the latter case, we can take
\[
a = \frac{|{\rm SO}_{11}(3)|}{2|{\rm SO}_{10}^{-}(3)|},\;\; b = \frac{|{\rm SO}_{11}(9)|}{2|{\rm SO}_{10}^{+}(9)|},\;\; c = \frac{|{\rm SO}_{11}(9)|}{2|{\rm SO}_{11}(3)|}
\]
and it is easy to check that $(ba^{-1})^5 \leqs c$.
\end{proof}

\begin{prop}\label{p:orthfix2}
The conclusion of Theorem \ref{t:red2} holds when $X = {\rm P\O}_n^{\e}(q)$ with $n$ even.
\end{prop}

\begin{proof}
Set $d = \gcd(4,q^{n/2}-\e)$ and let us first assume $G_{\a}$ is a $\C_2$-subgroup of type ${\rm O}_m^{\e'}(q) \wr S_k$. By inspecting Table \ref{tab:B0}, we see that we may assume $k \leqs 5$. In addition, since the cases with $n\leqs18$ are recorded
in Table \ref{tab:B}, we may assume $n\geqs 20$ and thus $m\geqs 4$. By reading off the order of $X_{\a}$ from \cite[Propositions 4.2.11, 4.2.14]{KL} we get
\[
|\cP| > \frac{1}{2^kd}(k!)^{-1}q^{\frac{1}{2}n(n-m)} = c.
\]

Suppose $q$ is odd and let $g \in X$ be an involution of type $(-I_{2},I_{n-2})$ so that
\[
|g^X| \leqs \frac{|{\rm SO}_n^{\e}(q)|}{2|{\rm SO}_2^{+}(q)||{\rm SO}_{n-2}^{+}(q)|} < 2q^{2n-4} = b.
\]
Then
\[
|g^X \cap X_{\a}| \geqs \binom{k}{2}\cdot \left(\frac{|{\rm SO}_{m}^{\e'}(q)|}{2|{\rm SO}_{m-1}^{\delta}(q)|}\right)^2
>\frac{1}{32}k(k-1)q^{2(m-1)} = a
\]
where $\delta=-$ if $m$ is odd, and we deduce that $(ba^{-1})^5 \leqs c$ unless $n=20$ and $(k,q) = (2,3)$, $(2,5)$, $(2,7)$, $(4,3)$ or $(5,3)$.  The latter cases can be eliminated by working with more accurate values for $a$ and $b$. For example, if $(n,k) = (20,2)$ then we can take
\[
a = \left(\frac{|{\rm SO}_{10}^{+}(q)|}{2|{\rm SO}_{9}(q)|}\right)^2,\;\; b = \frac{|{\rm SO}_{20}^{-}(q)|}{2|{\rm SO}_2^{+}(q)||{\rm SO}_{18}^{+}(q)|}
\]
and it is easy to check that $(ba^{-1})^5 \leqs c$, where $c$ is defined as above.

Now assume $q$ is even, so $m$ is also even. Here we take $g \in X$ to be a $c_2$-type involution, so
\[
|g^X|  = \frac{|{\rm O}_n^{\e}(q)|}{2|{\rm Sp}_{n-4}(q)|q^{n-2}}< 2q^{2n-4} = b
\]
and
\[
|g^X \cap X_{\a}| \geqs \binom{k}{2} \cdot \left(\frac{|{\rm O}_{m}^{\e'}(q)|}{2|{\rm Sp}_{m-2}(q)|}\right)^2 > \frac{1}{8}k(k-1)q^{2m-2} = a.
\]
We now check that $(ba^{-1})^5 \leqs c$ unless $n \leqs 18$, or $(n,q) = (20,2)$ and $k \in \{2,5\}$. As above, the latter two cases can be ruled out by working with more accurate estimates for $a$ and $b$.

Next suppose $n \equiv 2 \imod{4}$, $q$ is odd and $G_{\a} \in \C_2$ is of type ${\rm O}_{n/2}(q)^2$. Here $|X_{\a}| = |{\rm SO}_{n/2}(q)|^2$, so $|\cP| > \frac{1}{2d}q^{n^2/4} = c$ and we take $g \in X$ to be an involution of type $(-I_2,I_{n-2})$. Then $|g^X| < 2q^{2n-4} = b$ and
\[
|g^X \cap X_{\a}| \geqs \left(\frac{|{\rm SO}_{n/2}(q)|}{2|{\rm SO}_{n/2-1}^{-}(q)|}\right)^2 > \frac{1}{16}q^{n-2} = a,
\]
which gives $(ba^{-1})^5 \leqs c$ for all $n \geqs 22$.

Next assume $G_{\a} \in \C_2$ is of type ${\rm GL}_{n/2}(q).2$, in which case $\e=+$ and $G_{\a}$ is the stabilizer of a decomposition $V = U \oplus W$, where $U$ and $W$ are totally singular subspaces of maximal dimension. If $q$ is even and we take $g \in X$ to be an involution of type $a_2$, then
\[
|\cP| \geqs \frac{|\O_n^{+}(q)|}{2|{\rm GL}_{n/2}(q)|} > \frac{1}{4}q^{\frac{1}{4}n(n-2)} = c
\]
and we note that $|g^X| < 2q^{2n-6} = b$ and $|g^X \cap X_{\a}|$ coincides with the number of transvections in ${\rm SL}_{n/2}(q)$, so
\[
|g^X \cap X_{\a}| = \frac{(q^{n/2-1}-1)(q^{n/2}-1)}{q-1} = a.
\]
It is straightforward to check that $(ba^{-1})^5 \leqs c$ for all $n \geqs 20$.

Now suppose $q$ is odd, so we have
\[
|\cP| \geqs \frac{2|X|}{|{\rm GL}_{n/2}(q)|} > \frac{1}{d}q^{\frac{1}{4}n(n-2)} = c.
\]
If $n \equiv 2 \imod{4}$ or $q \equiv 1 \imod{4}$, then we can take  $g \in X$ to be an involution of type $(-I_{2},I_{n-2})$, where both eigenspaces are hyperbolic (i.e. of plus-type). Then
\[
|g^X \cap X_{\a}| = \frac{|{\rm GL}_{n/2}(q)|}{|{\rm GL}_{n/2-1}(q)||{\rm GL}_1(q)|} = \frac{q^{n/2-1}(q^{n/2}-1)}{q-1} = a, \;\;  |g^X| < 2q^{2n-4} = b
\]
and we check that $(ba^{-1})^5 \leqs c$ for all $n \geqs 22$. The remaining case where $n \equiv 0 \imod{4}$ and $q \equiv 3 \imod{4}$ requires special attention because $X$ does not contain involutions of type $(-I_2,I_{n-2})$ with hyperbolic eigenspaces (see \cite[Table B.10]{BG}). Let us assume $n \geqs 12$ and take $g \in X$ to be an involution of type $(-I_4,I_{n-4})$ with hyperbolic eigenspaces. Then
\[
|g^X \cap X_{\a}| = \frac{|{\rm GL}_{n/2}(q)|}{|{\rm GL}_{n/2-2}(q)||{\rm GL}_2(q)|} > \frac{1}{2}q^{2n-8} = a, \;\; |g^X| < 2q^{4n-16} = b
\]
and we get  $(ba^{-1})^5 \leqs c$ for all $n \geqs 40$.

Next assume $G_{\a} \in \C_3$ is of type ${\rm GU}_{n/2}(q)$, in which case $\e=+$ if $n \equiv 0 \imod{4}$, otherwise $\e=-$. For $q$ even, we can proceed as in the previous case, taking $g \in X$ to be an $a_2$-type involution and noting that
\[
|\cP|>\frac{1}{4}\left(\frac{q}{q+1}\right)q^{\frac{1}{4}n(n-2)} = c,\;\; |g^X| < 2q^{2n-6} = b
\]
and
\[
|g^X \cap X_{\a}| = \frac{(q^{m-1}-(-1)^{m-1})(q^{m}-(-1)^{m})}{q+1} = a,
\]
where $n = 2m$ (as explained in the proof of \cite[Proposition 5.3.13]{BG}, each transvection in $X_{\a}$ embeds in $X$ as an $a_2$-type involution, rather than a $c_2$-type involution). We now check that $(ba^{-1})^5 \leqs c$ for all $n \geqs 22$.

Now assume $q$ is odd, so
\[
|\cP|>\frac{1}{d}\left(\frac{q}{q+1}\right)q^{\frac{1}{4}n(n-2)} = c.
\]

If $\e=-$, or if $\e=+$ and $q \equiv 3 \imod{4}$, then we can take $g \in X$ to be an involution of type $(-I_{n-2},I_{2})$ with an elliptic $1$-eigenspace (for example, if $\e=-$ then $g$ is an involution of type $t_1'$ in the notation of \cite[Table B.11]{BG}). Then
\[
|g^X \cap X_{\a}| = \frac{|{\rm GU}_{m}(q)|}{|{\rm GU}_{m-1}(q)||{\rm GU}_1(q)|} \geqs \frac{q^{m-1}(q^{m}-1)}{q+1} = a,\;\; |g^X| < 2q^{2n-4} = b
\]
and once again we check that $(ba^{-1})^5 \leqs c$ for all $n \geqs 22$. Finally, suppose $\e=+$ and $q \equiv 1 \imod{4}$. Here we choose $g \in X$ to be an involution of type $(-I_4,I_{n-4})$ with hyperbolic eigenspaces. If we assume $n \geqs 12$ then
\[
|g^X \cap X_{\a}| = \frac{|{\rm GU}_{m}(q)|}{|{\rm GU}_{m-2}(q)||{\rm GU}_2(q)|} > \frac{1}{2}\left(\frac{q}{q+1}\right)q^{2n-8} = a,
\;\; |g^X| < 2q^{4n-16} = b
\]
and we have $(ba^{-1})^5 \leqs c$ for all $n \geqs 40$.

Next suppose $G_{\a} \in \C_3$ is a field extension subgroup of type ${\rm O}_{n/k}^{\e'}(q^k)$ with $k \in \{2,3,5\}$, so either $\e=\e'$, or $k=2$, $q$ is odd and $n \equiv 2 \imod{4}$.

First assume $k=2$. The precise structure of $X_{\a}$ is given in \cite[Lemma 5.3.4]{BG} and we deduce that $|\cP| > \frac{1}{8}q^{n^2/4} = c$. For $q$ odd we take $g \in X$ to be an involution of type $(-I_{n-2},I_2)$, while for $q$ even we let $g \in X$ be a $c_2$-type involution. Then $|g^X|< 2q^{2n-4}=b$ and we calculate that
$|g^X \cap X_{\a}|>\frac{1}{4}q^{n-2}=a$. For example, if $q$ is even, then $|g^X \cap X_{\a}|$ is the number of $b_1$-type involutions in ${\rm O}_{n/2}^{\e}(q^2)$ and thus
\[
|g^X \cap X_{\a}| = \frac{|{\rm O}_{n/2}^{\e}(q^2)|}{2|{\rm Sp}_{n/2-2}(q^2)|} >\frac{1}{2}q^{n-2}.
\]
We now check that $(ba^{-1})^5 \leqs c$ for all $n \geqs 20$, with the exception of the cases $(n,q) = (20,2)$, $(20,3)$. The latter two possibilities can be eliminated by working with a more accurate lower bound for $|g^X \cap X_{\a}|$. For example, if $(n,q) = (20,3)$ then
\[
|g^X \cap X_{\a}| \geqs \frac{|{\rm SO}_{10}^{+}(9)|}{2|{\rm SO}_9(9)|} = a
\]
and this yields $(ba^{-1})^5 \leqs c$ with $b$ and $c$ defined as above.

Now assume $k \in \{3,5\}$. Here $n \geqs 4k$ (see \cite[Tables 3.5.E, 3.5.F]{KL}) and we have
\[
|\cP| > \frac{1}{2dk}q^{\frac{1}{2}n^2\left(1-\frac{1}{k}\right)} = c.
\]
For $q$ odd, we take $g \in X$ to be an involution of type $(-I_{n-2k},I_{2k})$, and for $q$ even we choose a $c_{2k}$-type involution $g \in X$, so we have $|g^X| < 2q^{2k(n-2k)} = b$. For $q$ odd we get
\[
|g^X \cap X_{\a}| \geqs \frac{|{\rm SO}_{n/k}^{+}(q^k)|}{2|{\rm SO}_{2}^{-}(q^k)||{\rm SO}_{n/k-2}^{-}(q^k)|} > \frac{1}{4}\left(\frac{q^k}{q^k+1}\right)q^{2(n-2k)} = a.
\]
Similarly, if $q$ is even, then $g^X \cap X_{\a}$ is the conjugacy class of $c_2$-type involutions in $\O_{n/k}^{\e}(q^k)$ and the previous lower bound $|g^X \cap X_{\a}| > a$ still applies.

We now check that $(ba^{-1})^5 \leqs c$ for all $n>16k$, unless $(n,k) = (54,3)$ and $q \in \{2,3\}$. For $q=2$ we can take
\[
a = \frac{|{\rm O}_{18}^{+}(8)|}{2|{\rm Sp}_{14}(8)|8^{16}},
\]
which allows us to conclude that $(ba^{-1})^5 \leqs c$ with $b,c$ defined as above. And the same conclusion holds for $q=3$ if we take
\[
a = \frac{|{\rm SO}_{18}^{+}(27)|}{2|{\rm SO}_{2}^{-}(27)||{\rm SO}_{16}^{-}(27)|}.
\]

Now assume $\e = +$ and $G_{\a}$ is a $\C_4$-subgroup of type ${\rm Sp}_{2}(q) \otimes {\rm Sp}_{n/2}(q)$, in which case $n \equiv 0 \imod{4}$ and
\[
|\cP| > \frac{1}{4}q^{\frac{3}{8}(n^2-2n-8)} = c.
\]
Let us assume $n \geqs 12$. For $q$ even, let $g \in X$ be an involution of type $a_2$, so $|g^X|<2q^{2(n-2)} = b$ and $|g^X \cap X_{\a}| = q^{n/2}-1 = a$ coincides with the number of $b_1$-type involutions in ${\rm Sp}_{n/2}(q)$. Similarly, if $q$ is odd and $g \in X$ is an involution of type $(-I_{n-4},I_4)$ then $|g^X|< 2q^{4(n-4)} = b$ and we compute
\[
|g^X \cap X_{\a}| \geqs \frac{|{\rm Sp}_{n/2}(q)|}{|{\rm Sp}_{2}(q)||{\rm Sp}_{n/2-2}(q)|} > \frac{1}{2}q^{n-4} = a.
\]
It is now an easy exercise to check that $(ba^{-1})^5 \leqs c$ for all $n \geqs 40$.

Finally, suppose $G_{\a} \in \C_5$ is a subfield subgroup of type ${\rm O}_n^{\e'}(q_0)$, where $q = q_0^k$ and $k \in \{2,3,5\}$. By inspecting \cite[Proposition 4.5.10]{KL} we see that
\[
|\cP| > \frac{1}{2d}q^{\frac{1}{2}n(n-1)\left(1-\frac{1}{k}\right)} = c.
\]
For $q$ even, we take $g \in X$ to be an $a_2$-type involution, in which case $|g^X| < 2q^{2(n-3)} = b$ and $|g^X \cap X_{\a}| > \frac{1}{2}q^{2(n-3)/k} = a$. And for $q$ odd, if we work with an involution $g \in X$ of type $(-I_{2},I_{n-2})$, then $|g^X| < 2q^{2(n-2)} = b$ and
\[
|g^X \cap X_{\a}| \geqs \frac{|{\rm SO}_{n}^{+}(q_0)|}{2|{\rm SO}_{2}^{-}(q_0)||{\rm SO}_{n-2}^{-}(q_0)|} > \frac{1}{4}\left(\frac{q_0}{q_0+1}\right)q^{2(n-2)/k} = a.
\]
It is now straightforward to check that $(ba^{-1})^5 \leqs c$ for all $n \geqs 20$, with the exception of the case $(n,k,q) = (20,2,9)$. And the latter can be eliminated in the same way, using a more accurate lower bound on $|g^X \cap X_{\a}|$.
\end{proof}

\vs

This completes the proof of Theorem \ref{t:red2}.

\section{Irreducible subgroups}\label{s:irred3}

In this section we prove our main reduction theorem for classical groups when a point stabilizer in $G$ is contained in Aschbacher's collection $\mathcal{S}$. Recall that each subgroup $H$ in
$\mathcal{S}$ is almost simple and the embedding of $H$ in $G$ corresponds to an absolutely irreducible representation $\rho: L \to {\rm GL}(V)$, where $L$ is a covering group of the socle $Y$ of $H$. The proof of our main result divides naturally into two parts according to the partition
\[
\mathcal{S} = \mathcal{S}_1 \cup \mathcal{S}_2
\]
we defined in Section \ref{sss:sub}. Recall that $\mathcal{S}_1$ comprises the subgroups with socle an alternating group so that the representation $\rho$ is afforded by the fully deleted permutation module over $\mathbb{F}_p$ (see Table \ref{atab}). And then $\mathcal{S}_2$ is simply the remaining collection of subgroups in $\mathcal{S}$. Recall that we have a complete description of the maximal subgroups in $\mathcal{S}$ when $n \leqs 17$ via the work in \cite{BHR,Rogers,Schroder}.

For the groups with $G_{\a} \in \mathcal{S}_1$, our main tool is Lemma \ref{fixednumber}. Indeed, with the exception of an explicit list of cases that arises for certain low-dimensional groups, we can find an involution $g \in X$ such that $|\mathcal{P}_g| \geqs |\cP|^{4/5}$, which is incompatible with Lemma \ref{fixednumber}. And for the remaining cases with $n$ small, we will apply Lemma \ref{large}(ii) to derive an upper bound on the field size $p$, which leaves us with a finite number of possibilities for the pair $(G,G_{\a})$. And then these remaining candidates can be ruled out by showing that there are no feasible parameters for $\cQ$, using the approach described in Remark \ref{r:c1}. In this way, we can eliminate every candidate pair $(G,G_{\a})$ with $G_{\a} \in \mathcal{S}_1$.

Now assume $G_{\a} \in \mathcal{S}_2$. Here an analysis of fixed points is less effective and so we adopt a different approach. Setting $\theta = |{\rm Out}(X)|$, our first aim is to show that
\[
|X_{\a}| \leqs \theta^{-5/6}|X|^{1/6}
\]
unless $n = \dim V$ is small, which allows us to rule out the pair $(G,G_{\a})$ via Lemma \ref{large}(ii). For $n \geqs 18$, we will work closely with the main theorem of \cite{BW}, which determines all the maximal subgroups $H \in \mathcal{S}_2$ with $|H| \geqs q^n$. And for each possibility for $X$ with $n \leqs 17$ we can appeal to \cite{BHR,Rogers,Schroder}, where the maximal subgroups of $G$ are determined up to conjugacy. As before, we can typically rule out a given pair $(G,G_{\a})$ for specific values of $n$ and $q$ by showing that there are no feasible parameters for $\cQ$, using the arithmetic techniques discussed in Remark \ref{r:c1}.

The main result of this section is the following.

\begin{thm}\label{t:red3}
Let $\mathcal{Q} = (\mathcal{P},\mathcal{L})$ be a finite thick generalized quadrangle admitting an almost simple point-primitive and flag-transitive automorphism group $G$ with socle $X \in \mathcal{G}$. If $G_{\a} \in \mathcal{S}$, then $(G,G_{\a}) \in \mathcal{C}$ is one of the cases in Table \ref{tab:C}.
\end{thm}

{\small
\begin{table}
\[
\begin{array}{lll} \hline
X & Y & \mbox{Conditions} \\ \hline
\O_{25}(q) & F_4(q)  & p = 3 \\
{\rm PSp}_{20}(q) & {\rm L}_{6}(q) & p \geqs 3 \\
\O_{20}^{+}(q) & {\rm L}_6^{\e'}(q) & p = 2 \\
{\rm L}^{\e}_{16}(q) & {\rm P\O}_{10}^{\e}(q) & \\
{\rm P\O}_{16}^{+}(q) & \O_9(q) & p \geqs 3\\
& {\rm Sp}_8(q) & p = 2 \\
& {\rm PSp}_4(q^2) & p \geqs 3 \\
{\rm PSp}_{14}(q) & {\rm PSp}_6(q) & p \geqs 3 \\
{\rm P\O}_{14}^{\e}(q) & {\rm PSp}_6(q) & \\
& {\rm L}_4^{\e'}(q) & \\
\O_{13}(q) & {\rm PSp}_6(q) & p = 3 \\
{\rm L}_{10}^{\e}(q) & {\rm L}_5^{\e}(q) & \\
{\rm P\O}_{10}^{\e}(q) & {\rm PSp}_4(q) & q \equiv \e \imod{4} \\
{\rm L}_9^{\e}(q) & {\rm L}_3(q^2) & \\
\O_9(q) & {\rm L}_2(q^2) & \\
{\rm PSp}_8(q) & {\rm L}_2(q^3) & p \geqs 3 \\
{\rm P\O}_8^{+}(q) & \O_7(q) & p \geqs 3 \\
& {\rm Sp}_6(q) & p = 2 \\
& {\rm L}_3^{\e}(q) & q \equiv \e \imod{3} \\
& {\rm P\O}_8^{-}(q_0) & q = q_0^2 \\
& {}^3D_4(q_0) & q = q_0^3 \\
{\rm P\O}_8^{-}(q) & {\rm L}_3^{\e}(q) & q \equiv -\e \imod{3} \\
\O_7(q) & G_2(q) & \\
{\rm L}_6^{\e}(q) & {\rm L}_3^{\e}(q) & p \geqs 3 \\
 {\rm Sp}_6(q) & G_2(q) & p = 2 \\ \hline
\end{array}
\]
\caption{The collection $\mathcal{C}$}
\label{tab:C}
\end{table}
}

\begin{rem}
In the second column of Table \ref{tab:C} we write $Y$ for the socle of $X_{\a}$, and in the third column we record conditions on $q$ that are necessary for the existence or maximality of $G_{\a}$ as a subgroup of $G$. Here it is worth noting that there may be some additional restrictions on $q$ that we have not listed, which can be read off by inspecting the relevant tables in \cite{BHR,Rogers,Schroder}. For example, if $X = {\rm P\O}_{14}^{+}(q)$ and $Y = {\rm PSp}_6(q)$, then either $q \equiv 1,7 \imod{12}$ or $q=2^f$ with $f$ even (see \cite[Table 11.0.14]{Schroder}).
\end{rem}

\begin{rem}\label{r:precise}
Let $(X,Y)$ be one of the cases arising in Table \ref{tab:C} and let $n$ be the dimension of the natural module for $X$. For the groups with $n \leqs 16$, we can read off the precise structure of $X_{\a}$ by consulting the relevant tables in \cite{BHR,Rogers,Schroder}. So this just leaves us with the cases in the first three rows of the table.
\begin{itemize}\addtolength{\itemsep}{0.2\baselineskip}
\item[{\rm (a)}] If $X = \O_{25}(q)$ and $Y = F_4(q)$ with $p=3$, then $X_{\a} = Y$ is the only possibility since the highest weight of $V$ (as an irreducible module for $Y$) is not preserved by nontrivial field automorphisms.
\item[{\rm (b)}] Similarly, in the other two cases it is easy to see that
$[X_{\a}:Y] \leqs 2d$, where $d = \gcd(6,q-\e')$ and $Y = {\rm L}_6^{\e'}(q)$. More precisely, by considering the action of diagonal and graph automorphisms of $Y$ on $V$, we can show that $[X_{\a} :Y] = 2e$, where $e = \gcd(3,q-\e')$.
\end{itemize}
\end{rem}

\subsection{The case $G_{\a} \in \mathcal{S}_1$}

We begin by considering the groups with $G_{\a} \in \mathcal{S}_1$. Recall that the possibilities for $X$ and $Y$ are recorded in Table \ref{atab}, where $Y = A_k$ is the socle of $X_{\a}$.

We will need the following lemma. Recall that if $a$ is a positive integer and $p$ is an odd prime, then the \emph{Legendre symbol} is defined as follows:
\[
\left(\frac{a}{p}\right) = \left\{\begin{array}{rl}
1 & \mbox{if $a$ is a quadratic residue modulo $p$, with $a \not\equiv 0 \imod{p}$} \\
-1 & \mbox{if $a$ is a quadratic nonresidue modulo $p$} \\
0 & \mbox{if $a \equiv 0 \imod{p}$}
\end{array}\right.
\]

\begin{lem}\label{l:fdpm}
Let $G$ be an almost simple classical group over $\mathbb{F}_p$ with socle $X$ and let $H \in \mathcal{S}_1$ be a maximal subgroup with socle $A_k$. Let $n$ be the dimension of the natural module for $X$.
\begin{itemize}\addtolength{\itemsep}{0.2\baselineskip}
\item[{\rm (i)}] We have $H \cap X = S_k$ if and only if $X = {\rm Sp}_n(2)$, or $np$ is odd and $\left(\frac{(n+1)/2}{p}\right) = 1$.
\item[{\rm (ii)}] If $X = {\rm P\O}_n^{\e}(p)$, where $n$ is even and $p$ is odd, then $\e=+$ if and only if
\[
\left(\frac{n+1}{p}\right) = (-1)^{\frac{1}{4}n(p-1)}.
\]
\end{itemize}
\end{lem}

\begin{proof}
Part (i) is \cite[Lemma 3.1]{BGiu18} and part (ii) is explained on \cite[p.364]{BGiu18}.
\end{proof}

The following result establishes Theorem \ref{t:red3} for the groups with $G_{\a} \in \mathcal{S}_1$.

\begin{prop}\label{p:s1}
Let $\mathcal{Q} = (\mathcal{P},\mathcal{L})$ be a finite thick generalized quadrangle admitting an almost simple point-primitive and flag-transitive automorphism group $G$ with socle $X \in \mathcal{G}$. Then $G_{\a} \not\in \mathcal{S}_1$.
\end{prop}

\begin{proof}
Seeking a contradiction, assume $G_{\a} \in \mathcal{S}_1$. We consider each case in Table \ref{atab} in turn.

\vs

\noindent \emph{Case 1.} $X = {\rm Sp}_n(2)$ and $Y = A_{n+2}$

\vs

Here $n \geqs 8$, $n \equiv 0 \imod{4}$ and we have $G = X$ and $G_{\a} = S_{n+2}$. Let $g \in G$ be an involution with Jordan form $(J_2,J_1^{n-2})$ on $V$. Then by considering the  embedding of $G_{\a}$ in $G$, it follows that $g^G \cap G_{\a}$ coincides with the set of transpositions in $G_{\a}$, so
\[
|g^G \cap G_{\a}| = \frac{1}{2}(n+2)(n+1), \;\; |g^G| = 2^n-1
\]
and by appealing to Lemma \ref{l:fpr} we deduce that $|\mathcal{P}_g| \geqs |\mathcal{P}|^{4/5}$ if
\[
|\mathcal{P}| \geqs \left(\frac{2(2^n-1)}{(n+2)(n+1)}\right)^5 =:g(n).
\]
By Lemma \ref{l:bounds}(ii) we have
\[
|\mathcal{P}| = [{\rm Sp}_n(2): S_{n+2}] > \frac{2^{n(n+1)/2-1}}{(n+2)!} =: h(n)
\]
and it is easy to check that $h(n)>g(n)$ for all $n \geqs 8$. This is incompatible with Lemma \ref{fixednumber} and we have reached a contradiction.

\vs

\noindent \emph{Case 2.} $X = \O_n^{\e}(2)$, $Y = A_{n+2}$ and $n \equiv 2 \imod{4}$

\vs

Write $n = 2m$ and note that $G = X$ or $X.2 = {\rm O}_n^{\e}(2)$, so there are two cases to consider.

First assume $G = {\rm O}_n^{\e}(2)$. Here $G_{\a} = S_{n+2}$ and we can proceed as in the previous case, working with a transvection $g \in G$. Indeed, we compute
\[
|g^G \cap G_{\a}| = \frac{1}{2}(n+2)(n+1),\;\; |g^G| = 2^{m-1}(2^m-\e)\]
and it follows that $|\cP_g| \geqs |\cP|^{4/5}$ for all $n \geqs 10$.

Now suppose $G = X$, in which case $G_{\a} = A_{n+2}$ (see Lemma \ref{l:fdpm}(i)). Here we take $g \in G$ to be a $c_2$-type involution in the notation of \cite{AS} and we read off
\[
|g^G| = \frac{|{\rm O}_n^{\e}(2)|}{2^{n-1}|{\rm Sp}_{n-4}(2)|} = 2^{m-2}(2^{n-2}-1)(2^m-\e)
\]
from \cite[Table 3.5.1]{BG}. Now $g^G \cap G_{\a}$ is the set of double transpositions in $A_{n+2}$ (as noted in the proof of \cite[Proposition 3.4]{BGiu18}, the double transpositions in $A_{n+2}$ embed in $G$ as $c_2$-type involutions, rather than $a_2$-type) and thus
\[
|g^G \cap G_{\a}| = \frac{(n+2)!}{2^2(n-2)!2!} = \frac{1}{8}(n+2)(n+1)n(n-1).
\]
It is now a routine exercise to check that $|\cP_g| \geqs |\cP|^{4/5}$ for all $n \geqs 18$. Finally, for $n \in \{10,14\}$ we show that there are no feasible parameters for $\cQ$ by inspecting the set of divisors of $[X:X_{\a}]$, appealing to the method presented in Remark \ref{r:c1}(a). Therefore, we have reached a contradiction in every case.

\vs

\noindent \emph{Case 3.} $X = \O_n^{\e}(2)$, $Y = A_{n+1}$ and $n \equiv 0 \imod{4}$

\vs

This is very similar to Case 2. If $G = {\rm O}_n^{\e}(2)$ and $g \in G$ is a transvection, then $|\cP_g| \geqs |\cP|^{4/5}$ and the result follows. The same conclusion holds if $G = X$, $n \geqs 20$ and $g \in G$ is a $c_2$-type involution. In the latter case, if $n \in \{12,16\}$ then we can show that there are no feasible parameters for $\cQ$. Finally, suppose $G = \O_8^{+}(2)$ and $G_{\a} = A_9$. Here $|\cP| = [G:G_{\a}] = 960$ and $(s,t) = (7,17)$ is the only feasible order for $\cQ$, which means that
\[
|\cL| = (t+1)(st+1) = 2160.
\]
However, we can use {\sc Magma} to show that no subgroup of $G$ has index $2160$, so this final possibility is also eliminated.

\vs

\noindent \emph{Case 4.} $X = {\rm P\O}_n^{\e}(p)$, $Y = A_{n+2}$ and $p$ odd

\vs

Here $p$ divides $n+2$ and to begin with we will assume $G_{\a} = S_{n+2}$. The latter assumption means that we can choose $g \in G$ to be an involution of type $(-I_1,I_{n-1})$.

If $n = 2m$ is even, then
\[
|g^G| = \frac{|{\rm SO}_n^{\e}(p)|}{2|{\rm SO}_{n-1}(p)|} = \frac{1}{2}p^{m-1}(p^m-\e)
\]
and we note that $g^G \cap G_{\a}$ coincides with the set of transpositions in $G_{\a}$. Therefore,
\[
|g^G \cap G_{\a}| = \frac{1}{2}(n+2)(n+1)
\]
and we deduce that $|\cP_g| \geqs |\cP|^{4/5}$ if $n \geqs 10$. Finally, if $n=8$ then $(\e,p) = (+,5)$ and $X_{\a} = A_{10}$ (see Lemma \ref{l:fdpm}), and one can check that there are no feasible parameters for $\cQ$.

Now suppose $n = 2m+1$ is odd. Here
\[
|g^G| \leqs \frac{|{\rm SO}_n(p)|}{2|{\rm SO}_{n-1}^{+}(p)|} = \frac{1}{2}p^m(p^m+1)
\]
and $|g^G \cap G_{\a}| = (n+2)(n+1)/2$ as above, which allows us to deduce that $|\cP_g| \geqs |\cP|^{4/5}$ if $n \geqs 11$. For $n=9$ we have $G = \O_9(11)$, $G_{\a} = S_{11}$ (see Lemma \ref{l:fdpm}(i)) and it is easy to check that there are no feasible parameters for $\cQ$. Similarly, if $n=7$, then $G=\O_7(3)$
and $G_{\a}=S_9$. Here $v=[G:G_{\a}]=12636$, and $(s,t)=(35,10)$ is the only feasible order. However, the flag-transitivity of $G$ implies that  $t+1$ divides $|G_{\a}|$, which is false since $9!$ is indivisible by $11$.

Now let us assume $p$ divides $n+2$ and $G_{\a} = A_{n+2}$. Here we choose $g \in X$ to be an involution of type $(-I_2,I_{n-2})$ so that $g^G \cap G_{\a}$ is the set of double transpositions in $A_{n+2}$ and thus $|g^G \cap G_{\a}| = (n+2)(n+1)n(n-1)/8$. If $n = 2m$ is even, then
\[
|g^G| \leqs \frac{|{\rm SO}_{n}^{\e}(p)|}{2|{\rm SO}_{2}^{+}(p)||{\rm SO}_{n-2}^{+}(p)|} = \frac{1}{2}(p-1)^{-1}p^{n-2}(p^{m-1}+1)(p^m-\e)
\]
and we deduce that $|\cP_g| \geqs |\cP|^{4/5}$ if $n \geqs 20$. Similarly, if $n =2m+1$ is odd then
\[
|g^G| \leqs \frac{|{\rm SO}_{n}(p)|}{2|{\rm SO}_{2}^{+}(p)||{\rm SO}_{n-2}(p)|} = \frac{p^{n-2}(p^{n-1}-1)}{2(p-1)}
\]
and we get $|\cP_g| \geqs |\cP|^{4/5}$ if $n \geqs 19$.

By appealing to Lemma \ref{l:fdpm}, we see that it remains to deal with the following possibilities for $(X,X_{\a})$:
\[
({\rm P\O}_{18}^{+}(5),A_{20}), \, ({\rm P\O}_{16}^{-}(3),A_{18}), \, (\O_{13}(5),A_{15}),\, ({\rm P\O}_{12}^{-}(7),A_{14}),
\]
\[
(\O_{11}(13),A_{13}),\, ({\rm P\O}_{10}^{+}(3),A_{12}), \, ({\rm P\O}_{8}^{+}(5),A_{10}).
\]
In each case, we check that there are no feasible parameters for $\cQ$.

\vs

\noindent \emph{Case 5.} $X = {\rm P\O}_n^{\e}(p)$, $Y = A_{n+1}$ and $p$ odd

\vs

Here $n+2$ and $n+1$ are both indivisible by $p$. By inspecting \cite[Tables 8.40, 8.50, 8.53]{BHR}, we may assume $n \geqs 9$.

First assume $G_{\a} = S_{n+1}$. Here we can take $g \in G$ to be an involution of type $(-I_1,I_{n-1})$ so that $g^G \cap G_{\a}$ is the set of transpositions in $S_{n+1}$. Then by arguing as in Case 4, we deduce that $|\cP_g| \geqs |\cP|^{4/5}$ if $n \geqs 11$. Similarly, if $G_{\a} = A_{n+1}$ then we choose $g \in G$ to be an involution of type $(-I_2,I_{n-2})$ so that $g^G \cap G_{\a}$ coincides with the set of double transpositions in $A_{n+1}$. Then $|\cP_g| \geqs |\cP|^{4/5}$ if $n \geqs 19$.

So to complete the argument, we need to handle the remaining groups with $9 \leqs n \leqs 18$. Set $\theta = |{\rm Out}(X)|$. For each possibility for $n$, we first show that
\begin{equation}\label{e:xa2}
|X_{\a}| \leqs \theta^{-5/6}|X|^{1/6}
\end{equation}
if $p \geqs c$ for some explicit constant $c$ (in every case, $c \leqs 17$). Then Lemma \ref{large}(ii) allows us to assume that $p<c$ and we can consider each remaining possibility for $p$ in turn, using Lemma \ref{l:fdpm} to determine the precise structure of $X_{\a}$. And then by considering the divisors of $[X:X_{\a}]$, we can show in every case that there are no feasible parameters for $\cQ$. For brevity, we give details in two cases (all of the other cases are entirely similar).

\begin{itemize}\addtolength{\itemsep}{0.3\baselineskip}
\item[{\rm (a)}] Suppose $X = {\rm P\O}_{18}^{\e}(p)$ and $Y = A_{19}$, so $\theta = 2\gcd(4,p^9-\e)$, $p \ne 5,19$ and Lemma \ref{l:fdpm} implies that $X_{\a} = A_{19}$. It is straightforward to check that \eqref{e:xa2} holds if $p \geqs 5$, so by appealing to Lemma \ref{large}(ii), we may assume $p=3$. Here $\e=-$ by Lemma \ref{l:fdpm}(ii) and it is easy to check that there are no feasible parameters for $\cQ$.

\item[{\rm (b)}] Now assume $X = \O_9(p)$, so $Y = A_{10}$ and $p \ne 5,11$. Here $\theta=2$, $|X_{\a}| \leqs 10!$ and we check that $10! \leqs \theta^{-5/6}|X|^{1/6}$ for all $p \geqs 17$. So we may assume $p \in \{3,7,13\}$. In each case we have $X_{\a} = A_{10}$ and it is routine to check that there are no feasible parameters.
\end{itemize}

This completes the proof of the proposition.
\end{proof}

\subsection{The case $G_{\a} \in \mathcal{S}_2$}

To complete the proof of Theorem \ref{t:red3}, we may assume $G_{\a} \in \mathcal{S}_2$ and our goal is to reduce to the specific list of pairs $(G,G_{\a})$ recorded in Table \ref{tab:C}. We begin by handling the groups with $n \geqs 18$. As before, we set $\theta = |{\rm Out}(X)|$.

\begin{prop}\label{p:s2_1}
The conclusion of Theorem \ref{t:red3} holds if $n \geqs 18$ and $G_{\a} \in \mathcal{S}_2$.
\end{prop}

\begin{proof}
In view of Lemma \ref{large}(ii), a given candidate $(G,G_{\a})$ is eliminated if the inequality in \eqref{e:xa2} is satisfied. And by Lemma \ref{l:small}, this holds if $|X_{\a}| \leqs q^{h(n)}$, where the function $h(n)$ is defined in Table \ref{tab:fn}. Let $Y$ be the socle of $X_{\a}$.

First assume $X = {\rm L}^{\e}_n(q)$, so $h(n) = (n^2-c)/6$ with $c =12$ if $\e=+$ and $c = 17$ if $\e=-$. If $|X_{\a}| < q^{2n+4}$ then we immediately deduce that $|X_{\a}| < q^{h(n)}$, so we may assume $|X_{\a}| \geqs q^{2n+4}$. Then the possibilities for $(G,G_{\a})$ are identified in \cite[Theorem 2.10]{BGS2014} and we deduce that one of the following holds:
\begin{itemize}\addtolength{\itemsep}{0.2\baselineskip}
\item[{\rm (a)}] $X = {\rm L}_n^{\e}(q)$ and $Y = {\rm L}^{\e}_a(q)$, where $n = a(a-1)/2$ and $a \geqs 7$.
\item[{\rm (b)}] $X = {\rm L}_{27}^{\e}(q)$ and $Y = E_6^{\e}(q)$.
\end{itemize}
In case (a), we have $|X_{\a}| < q^{a^2-1}$ and we check that $a^2-1 \leqs h(n)$ for all $a \geqs 7$. Similarly, $|X_{\a}|<q^{78} < q^{h(n)}$ in case (b).

Next suppose $X = {\rm PSp}_n(q)$, so $h(n) = (n^2+n-10)/12$ and thus $|X_{\a}|< q^{h(n)}$ if $|X_{\a}| < q^n$. Therefore, we may assume $|X_{\a}| \geqs q^n$ and this means that the possibilities for $(G,G_{\a})$ are determined by the main theorem of \cite{BW}. Since we are assuming $n \geqs 18$, we read off the following two cases with $|X_{\a}| > q^{h(n)}$:
\begin{itemize}\addtolength{\itemsep}{0.2\baselineskip}
\item[{\rm (a)}] $X = {\rm PSp}_{20}(q)$ and $Y = {\rm L}_6(q)$ with  $q$ odd.
\item[{\rm (b)}] $X = {\rm Sp}_{22}(2)$ and $Y = {\rm Co}_3$.
\end{itemize}
Here case (a) is recorded in Table \ref{tab:C}, whereas it is easy to check that \eqref{e:xa2} holds in (b).

Next assume $X = \O_n(q)$ with $n \geqs 19$ odd. Here $n < h(n)$ and so we can use \cite{BW} to reduce the problem to the following cases:
\begin{itemize}\addtolength{\itemsep}{0.2\baselineskip}
\item[{\rm (a)}] $X = \O_{21}(3)$ and $Y = {\rm U}_6(2)$.
\item[{\rm (b)}] $X = \O_{21}(q)$ and $Y = {\rm PSp}_6(q)$ or $\O_7(q)$.
\item[{\rm (c)}] $X = \O_{23}(3)$ and $Y = {\rm Co}_2$.
\item[{\rm (d)}] $X = \O_{23}(q)$, $p=5$ and $Y = {\rm L}_5^{\e}(q)$.
\item[{\rm (e)}] $X = \O_{25}(q)$, $p=3$ and $Y = F_4(q)$.
\item[{\rm (f)}] $X = \O_{27}(q)$ and $Y = {\rm PSp}_8(q)$.
\end{itemize}
Here case (e) is recorded in Table \ref{tab:C}. And in each of the remaining cases it is easy to verify the inequality in \eqref{e:xa2}.

Finally, suppose $X = {\rm P\O}_n^{\e}(q)$ where $n \geqs 18$ is even. As before, if $|X_{\a}|<q^n$ then \eqref{e:xa2} holds, so we may assume $|X_{\a}| \geqs q^{n}$. Then by the main theorem of \cite{BW}, the possibilities for $(X,Y)$ are as follows (for some unspecified choices of signs $\e$ and $\e'$):
\begin{itemize}\addtolength{\itemsep}{0.2\baselineskip}
\item[{\rm (a)}] $X = \O_{20}^{+}(2)$ and $Y = {\rm U}_4(3)$.
\item[{\rm (b)}] $X = \O_{20}^{+}(q)$, $q$ even and $Y = {\rm L}_6^{\e'}(q)$.
\item[{\rm (c)}] $X = \O_{22}^{+}(2)$ and $Y = {\rm Co}_2$ or ${\rm McL}$.
\item[{\rm (d)}] $X = {\rm P\O}_{22}^{+}(3)$ and $Y = {\rm Co}_3$.
\item[{\rm (e)}] $X = {\rm P\O}_{24}^{+}(q)$, $q \in \{2,3,5\}$ and $Y = {\rm Co}_1$.
\item[{\rm (f)}] $X = {\rm P\O}_{24}^{\e}(q)$, $p \ne 5$ and $Y = {\rm L}_5^{\e'}(q)$.
\item[{\rm (g)}] $X = \O_{26}^{\e}(q)$, $q$ even and $Y = {\rm Sp}_8(q)$, $\O_8^{\e'}(q)$ or ${}^3D_4(q)$.
\item[{\rm (h)}] $X = {\rm P\O}_{26}^{\e}(q)$, $p \ne 3$ and $Y = F_4(q)$.
\item[{\rm (i)}] $X = \O_{28}^{-}(2)$ and $Y = {\rm Ru}$.
\item[{\rm (j)}] $X = {\rm P\O}_{28}^{\e}(q)$, $q$ odd and $Y = {\rm P\O}_8^{\e'}(q)$ or ${}^3D_4(q)$.
\item[{\rm (k)}] $X = \O_{56}^{+}(q)$, $q$ even and $Y =  E_7(q)$.
\end{itemize}
Case (b) is recorded in Table \ref{tab:C}. In case (c) with $Y = {\rm Co}_2$, and also case (e) with $q=2$, we check that there are no feasible parameters for $\cQ$. For all of the remaining possibilities,  it is straightforward to verify the bound in \eqref{e:xa2}. For example,  we have $|X_{\a}| < q^{52} < q^{h(n)}$ in case (h) since $h(n) = \frac{313}{6}>52$.
\end{proof}

For the remainder, we may assume $n \leqs 17$ and we note that the possibilities for $X_{\a}$ are recorded in \cite{BHR} (for $n \leqs 12$), \cite{Schroder} (for $n = 13,14,15$) and \cite{Rogers} (for $n = 16,17$).

\begin{prop}\label{p:s2_2}
The conclusion of Theorem \ref{t:red3} holds if $n \leqs 17$ and $G_{\a} \in \mathcal{S}_2$.
\end{prop}

\begin{proof}
Here we need to consider each possibility for $X$ in turn, recalling that the assumption $X \in \mathcal{G}$ implies that
\[
X \ne {\rm L}_2(q), \; {\rm U}_3(q), \; {\rm PSp}_4(q).
\]
We read off the possibilities for $X_{\a}$ from the relevant tables in \cite{BHR,Rogers,Schroder} and in every case we do one of the following:
\begin{itemize}\addtolength{\itemsep}{0.2\baselineskip}
\item[{\rm (i)}] We verify the inequality in \eqref{e:xa2};
\item[{\rm (ii)}] We show that there are no feasible parameters for $\cQ$ via Remark \ref{r:c1}(a); or
\item[{\rm (iii)}] We include the pair $(X,Y)$ in Table \ref{tab:C}.
\end{itemize}
No special difficulties arise in any of these cases.

For example, suppose $X = {\rm PSp}_6(q)$. Here the possibilities for $X_{\a}$ are recorded in \cite[Table 8.29]{BHR}. Clearly, if $q \geqs 4$ is even and $X_{\a} = G_2(q)$, then
\begin{equation}\label{e:xa3}
|X_{\a}| > \theta^{-5/6}|X|^{1/6}
\end{equation}
and so this case is included in Table \ref{tab:C}. By carefully working through the remaining cases in \cite[Table 8.29]{BHR}, we deduce that \eqref{e:xa3} holds if and only if $(X,X_{\a})$ is one of the following:
\begin{itemize}\addtolength{\itemsep}{0.2\baselineskip}
\item[{\rm (a)}] $X = {\rm PSp}_6(3)$ and $X_{\a} = A_5$ or ${\rm L}_2(13)$.
\item[{\rm (b)}] $X = {\rm PSp}_6(7)$ and $X_{\a} = {\rm U}_3(3)$.
\item[{\rm (c)}] $X = {\rm PSp}_6(9)$ and $X_{\a} = A_7$.
\item[{\rm (d)}] $X = {\rm PSp}_6(q)$, $q \in \{2,11,13\}$ and $X_{\a} = {\rm U}_3(3).2$.
\item[{\rm (e)}] $X = {\rm PSp}_6(q)$, $q \in \{5,9,11,19,29,31,41,49\}$ and $X_{\a} = {\rm J}_2$.
\end{itemize}
In every case, it is a routine exercise to apply the arithmetic method in Remark \ref{r:c1}(a) in order to conclude that there are no feasible parameters for $\cQ$. Therefore, we can eliminate all of these possibilities.
\end{proof}

\vs

This completes the proof of Theorem \ref{t:red3}.

\section{Proof of Theorem \ref{t:main}: Classical groups}\label{s:main1}

In this section we will complete the proof of Theorem \ref{t:main} for classical groups. Let $\mathcal{Q}=(\mathcal{P},\mathcal{L})$ be a finite thick generalized quadrangle of order $(s,t)$ with a flag-transitive, point-primitive and line-primitive automorphism group $G$, which is almost simple with socle $X \in \mathcal{G}$, where we recall that the classical groups in $\mathcal{G}$ are recorded in Table \ref{tab:G}. Fix a point $\a \in \cP$ and a line $\ell \in \cL$.

In view of our earlier work in Sections \ref{s:subspace}-\ref{s:irred3}, it suffices to
eliminate all of the candidate triples $(G,G_{\a},G_{\ell})$, where each pair $(G,G_{\a})$ and $(G,G_{\ell})$ is contained in one of the collections labelled $\mathcal{A}$, $\mathcal{B}$ or $\mathcal{C}$ (see Tables \ref{tab:A}, \ref{tab:B} and \ref{tab:C}, respectively), together with the collection $\mathcal{N}$ (see Table \ref{tab:N}) in the special case $X = {\rm P\O}_8^{+}(q)$. We will write $\mathcal{D}$ for this collection of candidate triples $(G,G_{\a},G_{\ell})$, recalling that we are free to assume that $G_{\alpha}$ and $G_{\ell}$ are not $G$-conjugate (see Lemma \ref{isomorphism}(i)).

We will proceed in two steps. First we apply an arithmetic-computational approach to severely restrict the possibilities for $(G,G_{\a},G_{\ell})$, which is based on some of the arithmetic conditions in Lemma \ref{HD} and Corollary \ref{gcd_bound}. This general method, which can be implemented using {\sc Magma} \cite{magma}, will be explained in Section \ref{ss:comp} (see Remark \ref{rem:computational-strategy}) and we will illustrate our main techniques with the aid of several concrete examples. In this way, we reduce the proof of Theorem \ref{t:main} for classical groups to a much smaller collection $\mathcal{E}$ of candidate triples, which is recorded in Tables \ref{tab:E1} and \ref{tab:E2}. Finally, we complete the proof for classical groups in Section \ref{ss:comp_class}, where we use a combination of algebraic and arithmetic techniques to eliminate all of the triples in $\mathcal{E}$. These methods, which are presented in Section \ref{ss:further}, will also be useful in our analysis of the exceptional groups in Section \ref{s:excep}, where we will complete the proof of Theorem \ref{t:main}.

For the remainder of this section, unless stated otherwise, we will adopt the following hypothesis:

\begin{hypothesis}\label{hypothesis}
Let $\mathcal{Q}=(\mathcal{P},\mathcal{L})$ be a finite thick generalized quadrangle of order $(s,t)$ with a flag-transitive, point-primitive and line-primitive automorphism group $G$, which is almost simple with socle $X \in \mathcal{G}$. Given $\a \in \cP$ and $\ell \in \cL$, we assume that $(G,G_{\a},G_{\ell}) \in \mathcal{D}$.
\end{hypothesis}

In addition, we write $X_\alpha=G_\alpha\cap X$,  $X_\ell=G_\ell\cap X$ and we set
\[
v=|\mathcal P|=[X:X_\alpha],\;\; b=|\mathcal L|=[X:X_\ell],\;\; \theta=|{\rm Out}(X)|.
\]

\begin{rem}\label{r:params}
As before, let $n$ be the dimension of the natural module for $X$. For each triple $(G,G_{\a},G_{\ell})$ in $\mathcal{D}$, it is important to note that we can compute $v$, $b$, $|X_{\a}|$ and $|X_{\ell}|$ precisely:
\begin{itemize}\addtolength{\itemsep}{0.2\baselineskip}
\item[{\rm (a)}] For $n \leqs 12$, we can consult the relevant tables in \cite[Chapter 8]{BHR}.
\item[{\rm (b)}] If $n \geqs 13$ and $G_{\a}$ is a geometric maximal subgroup in one of the $\C_i$ collections, then we can read off $|X_{\a}|$ (and hence $v$) by inspecting \cite[Chapter 4]{KL}.
\item[{\rm (c)}] If $13 \leqs n \leqs 17$ and $G_{\a} \in \mathcal{S}$, then we can consult \cite{Schroder} (for $n \leqs 15$) and \cite{Rogers} (for $n = 16,17$).
\item[{\rm (d)}] Finally, if $n \geqs 18$ and $G_{\a} \in \mathcal{S}$, then we refer the reader to Remark \ref{r:precise}, where the relevant cases in Table \ref{tab:C} are discussed.
\end{itemize}
In addition, we recall that $\theta = |{\rm Out}(X)|$ is listed in Table \ref{tab:G}.
\end{rem}

\subsection{Computational methods}\label{ss:comp}

In this section we present the main computational methods we will use to reduce our initial set of candidate triples in $\mathcal{D}$ to the more restricted list of possibilities in $\mathcal{E}$. These methods will also be essential in Section \ref{ss:ex_red}, where we establish reduction results in the case where $X$ is a finite simple exceptional group of Lie type. Therefore, throughout Section \ref{ss:comp} we will simply assume that $X$ is a simple group of Lie type over $\mathbb{F}_q$, where $q = p^f$ with $p$ a prime, without specifying whether $X$ is classical or exceptional.

We begin with the following preliminary result. Recall that if $m$ is a positive integer and $r$ is a prime, then $m_{r}$ and $m_{r'}$ denote the $r$-part and $r'$-part of $m$, respectively (so for example, $12_2 = 4$ and $12_{2'} = 3$).

\begin{lem}\label{inequalities}
Let $r$ be a prime.
\begin{itemize}\addtolength{\itemsep}{0.2\baselineskip}
\item[{\rm (i)}] If $v_r\leqs b_r$, then either $(st+1)_r=2_r$ or $(s+1)_r\leqs2_r$, and the latter holds if $(v_r/2_r)^5\geqs 2v^{2}$.
\item[{\rm (ii)}] If $b_r\leqs v_r$, then either $(st+1)_r=2_r$ or $(t+1)_r\leqs2_r$, and the latter holds if $(b_r/2_r)^5\geqs 2b^{2}$.
\item[{\rm (iii)}] If $(v_r/2_r)^5\geqs2v^2$ or $(b_r/2_r)^5\geqs2b^2$, then $(s+1)_r\leqs2_r\max\{1,v_r/b_r\}$. The same conclusion holds if $(v_r/2_r)^3\geqs2v^2/b$ or $(b_r/2_r)^3\geqs2b^2/v$.
\end{itemize}
\end{lem}

\begin{proof}
In (i), we have $v_r\leqs b_r$, which means that  $(s+1)_r \leqs (t+1)_r$. We may as well assume that $(s+1)_r > 2_r$, so  $(s+1)_r$ is either an odd prime power, or a multiple of $4$. From the inequality $(s+1)_r \leqs (t+1)_r$, we deduce that the same prime power also divides $t+1$, so either $r$ is odd and $r$ divides $s+1$ and $t+1$, or $r=2$ and both $s+1$ and $t+1$ are divisible by $4$. Since
\[
\gcd(st+1,s+1,t+1)=\gcd(s+1,t+1,t-1) \leqs 2
\]
it follows that $(st+1)_r=2_r$ as required.

Moreover, Higman's inequality $t \geqs s^{1/2}$ (see Lemma \ref{HD}(ii)) implies that
\begin{equation}\label{e:vineq}
v \geqs (s+1)(s^{3/2}+1) > 2^{-1/2}(s+1)^{5/2}
\end{equation}
for all $s \geqs 2$, so we have $s+1<2^{1/5}v^{2/5}$. Suppose that  $(v_r/2_r)^5\geqs 2v^{2}$. If $(st+1)_r=2_r$, then
\[
(s+1)_r=v_r/2_r\geqs 2^{1/5}v^{2/5}>s+1,
\]
which is absurd. Therefore, $(st+1)_r \ne 2_r$ and we conclude that $(s+1)_r \leqs 2_r$. This completes the proof of part (i), and we note that (ii) follows by point-line duality.

Finally, let us consider part (iii). It will be helpful to observe that the inequality in \eqref{e:vineq} gives $(s+1)^5 < 2v^2$, and similarly $(t+1)^5 < 2b^2$.

Seeking a contradiction, suppose that $(st+1)_r = 2_r$, in which case $(s+1)_r = v_r/2_r$ and $(t+1)_r = b_r/2_r$. If $(v_r/2_r)^5\geqs2v^2$ then
\[
(s+1)^5 \geqs (s+1)_r^5 = \left(v_r/2_r\right)^5 \geqs 2v^2
\]
and we reach a contradiction. Similarly, we get $(t+1)_r^5 \geqs 2b^2$ if $(b_r/2_r)^5\geqs2b^2$, which once again is a contradiction. And if we assume $(v_r/2_r)^3\geqs2v^2/b$, then Lemma \ref{HD}(v) gives
\[
(s+1)^3<2v^2/b\leqs(v_r/2_r)^3=(s+1)_r^3\leqs(s+1)^3,
\]
again a contradiction. By point-line duality, we also reach a contradiction if $(b_r/2_r)^3\geqs2b^2/v$.

Therefore, $(st+1)_r \ne 2_r$ and we can now conclude by applying parts (i) and (ii). Indeed, if $v_r \leqs b_r$ then (i) gives $(s+1)_r \leqs 2_r$ and the result follows. Similarly, if $b_r \leqs v_r$ then (ii) implies that $(t+1)_r \leqs 2_r$ and thus
\[
(s+1)_r = (t+1)_r\left(v_r/b_r\right) \leqs 2_r\left(v_r/b_r\right) = 2_r\max\{1,v_r/b_r\}
\]
as required.
\end{proof}

\begin{rem}\label{rem:computational-strategy}
In this extended remark we explain how to combine Lemma \ref{inequalities} with some of the earlier results in Section \ref{s:prel} in order to develop a computational approach for eliminating candidate triples, which can be implemented using {\sc Magma} \cite{magma}.

Let $\mathcal{J} = \{X(q) \, : \, q \in \Lambda \}$ be a family of low-rank  finite simple groups of Lie type, where $X(q)$ is defined over $\mathbb{F}_q$ and $\Lambda$ is an infinite set of prime powers. Here we assume that the groups in $\mathcal{J}$ are all of the same type; for instance, we could take $\mathcal{J} = \{ {\rm PSp}_6(q) \,:\, q \in \Lambda\}$, where $\Lambda$ is the set of all prime powers. Set $X=X(q)$ for some $q \in \Lambda$ and fix the types of the point and line stabilizers $X_{\a}$ and $X_{\ell}$. Let us assume that $v=|\cP|$, $b = |\cL|$, $|X_{\alpha}|$ and $|X_{\ell}|$ can all be expressed as polynomials in $z$ with rational coefficients, where $q = z^r$ for some fixed positive integer $r$ (here $r$ is independent of the choice of $q$). Moreover, let us assume that these polynomials are independent of the choice of $q$, up to multiplication by a finite set of rational scalars. Then for each fixed choice of scalars, write
 \begin{equation}\label{e:fi}
 v=f_1(z),\;\; b = f_2(z), \;\; |X_{\a}| = f_3(z), \;\; |X_{\ell}| = f_4(z),
 \end{equation}
where each $f_i(x)$ is a polynomial in $\mathbb{Q}[x]$.

For example, if $X={\rm PSp}_6(q)$ and $X_{\a}$ and $X_{\ell}$ are of type ${\rm Sp}_2(q) \wr S_3$ and ${\rm Sp}_6(q_0)$, respectively, where $q = q_0^3$, then
\begin{align*}
f_1(z) & = \frac{1}{6}z^{18}(z^6+1)(z^{12}+z^6+1) \\
f_2(z) & = z^{18}(z^4+z^2+1)(z^8+z^4+1)(z^{12}+z^6+1) \\
f_3(z) & = \frac{6}{y}z^{9}(z^6-1)^3 \\
f_4(z) & = \frac{1}{y}z^9(z^2-1)(z^4-1)(z^6-1),
\end{align*}
where $z = q_0$ and $y = \gcd(2,q-1)$.

\vs

\noindent \emph{Method 1.} Bounds on $v$ and $b$.

\vs

Suppose that at least one of the polynomials $f_1(x)^4-f_2(x)^5$ and $f_2(x)^4-f_1(x)^5$ has positive leading coefficient. Then Lemma \ref{polynomial-bound} implies that the required inequalities $v^4<b^5$ and $b^4<v^5$ in Lemma \ref{HD}(v) are only simultaneously satisfied for at most finitely many prime powers $q \in \Lambda$, so the triple $(G,G_{\a},G_{\ell})$ is eliminated for all sufficiently large $q \in \Lambda$. And for any remaining possibilities for $q$, we can seek to eliminate $(G,G_{\a},G_{\ell})$ by showing there are no feasible parameters for $\cQ$, using the computational techniques presented in Remark \ref{r:c1}.

\vs

\noindent \emph{Method 2.} Bounds on $s+1$.

\vs

Write $z = p^j$ for some prime $p$ and integer $j \geqs 1$ and consider the decomposition
\[
f_i(z) = p^{a_i}z^{b_i}g_i(z)
\]
for $1 \leqs i \leqs 4$,  where $a_i,b_i \in \mathbb{Z}$ with $b_i \geqs 0$ and $g_i(x)\in\mathbb{Q}[x]$ has a nonzero constant term. We may assume that $p^{a_i}z^{b_i}$ and $g_i(z)$ are the $p$-part and $p'$-part of $f_i(z)$, respectively.

Let us assume we can show that at least one of the following inequalities
\[
(v_p/2_p)^5 \geqs 2v^2, \; (b_p/2_p)^5 \geqs 2b^2, \; (v_p/2_p)^3 \geqs2v^2/b,\; (b_p/2_p)^3 \geqs2b^2/v
\]
is valid for all but finitely many $q\in\Lambda$, which we can test with the aid of {\sc Magma}. Then Lemma \ref{inequalities}(iii) yields
\begin{equation}\label{e:splus1}
(s+1)_p\leqs2_p\max\{1,v_p/b_p\}\leqs f_5(z)
\end{equation}
for all sufficiently large $q \in \Lambda$, where $f_5(x)\in\mathbb{Q}[x]$ is defined for each fixed choice of rational scalars in our initial definition of the polynomials $f_i(x) \in \mathbb{Q}[x]$ for $i = 1,2,3,4$ (see \eqref{e:fi}).

Recall that $s+1$ divides $v$ and $\theta|X_{\ell}|$ by Lemmas \ref{HD}(iv) and \ref{large}(i), which means that $(s+1)_{p'}$ divides $g_1(z)$ and $\theta g_4(z)$ since $z$ is a power of $p$. Let
\[
g_5(x)=\gcd(g_1(x),g_4(x)) \in \mathbb{Q}[x].
\]
With the aid of the {\sc Magma} command \texttt{XGCD}, after clearing denominators, we can find integer polynomials $u_1(x),u_4(x)\in \mathbb{Z}[x]$ and a positive integer $c$ such that
\begin{equation}\label{polynomial-gcd}
u_1(x)g_1(x)+u_4(x)g_4(x)= cg_5(x) \in \mathbb{Z}[x].
\end{equation}
Therefore, $(s+1)_{p'}$ divides $c\theta g_5(z)$.

Set $d(x) = cf_5(x)g_5(x) \in \mathbb{Q}[x]$. Then $s+1\leqs\theta d(z)$ for all sufficiently large $q \in \Lambda$ and so the bounds
\[
\left(v^{2}/b\right)^{1/3}<s+1 \leqs \gcd(v,\theta|X_{\ell}|)
\]
in Corollary \ref{gcd_bound}(i) imply that
\begin{equation}\label{parameter-comparison}
f_1(z)^2<\theta^3d(z)^3f_2(z).
\end{equation}

Recall that $z=p^j$, so $q=p^{rj}$. By inspecting Table \ref{tab:G}, we observe that there is an absolute constant $k>0$ such that $\theta\leqs jk$ for all $q \in \Lambda$. Squaring \eqref{parameter-comparison} and using the elementary estimate $j^6 < 2048\cdot2^j\leqs2048z$, we obtain
\[
f_1(z)^4<2048k^6z\,f_2(z)^2d(z)^6.
\]

If $2\deg f_1>\deg f_2+3\deg d$, then Lemma \ref{polynomial-bound} implies that the above inequality is only valid for at most finitely many values of $z$ (moreover, the proof of the lemma gives an explicit bound on $z$). In this situation, for the finite set of remaining possibilities for $z$ (and hence $q = z^r$), we can seek to apply the techniques in Remark \ref{r:c1} to rule out the existence of feasible parameters for $\cQ$.

\vs

\noindent \emph{Method 3.} Bounds on $st+1$

\vs

Our final technique relies on the inequalities
\[
(vb/2)^{1/3}<st+1\leqs\gcd(v,b)
\]
from Corollary \ref{gcd_bound}(ii). Set
\begin{equation}\label{e:hx}
h(x)=\gcd(f_1(x),f_2(x)) \in \mathbb{Q}[x],
\end{equation}
which can be computed using the {\sc Magma} command \texttt{XGCD}. Then as in Method 2 above, there is a positive integer $c$ such that $\gcd(v,b)$ divides $ch(z)$ and thus
\begin{equation}\label{e:f1f2}
f_1(z)f_2(z)<2c^3h(z)^3.
\end{equation}
If
\begin{equation}\label{e:hcond}
\deg f_1 +\deg f_2 >3\deg h,
\end{equation}
then Lemma \ref{polynomial-bound} implies that the required bound in \eqref{e:f1f2} only holds for at most finitely many prime powers $q \in \Lambda$. So once again, this will allow us to eliminate the triple $(G,G_{\a},G_{\ell})$ away from a finite set of prime powers.
\end{rem}

We will now illustrate the general methods presented in Remark \ref{rem:computational-strategy} by considering the following concrete cases, each of which arises in the proof of Theorem \ref{t:main}:

\begin{itemize}\addtolength{\itemsep}{0.2\baselineskip}
\item[{\rm (a)}] $X = {\rm L}_n(q)$ with $G_{\a} \in \C_3$ of type ${\rm GL}_{n/k}(q^k)$ and $G_{\ell} \in \C_5$ of type ${\rm GL}_n(q_0)$, where $q = q_0^r$ and $r$ is a prime.
\item[{\rm (b)}] $X = {\rm P\O}_8^{+}(q)$ and $G_{\a}, G_{\ell} \in \mathcal{S}$ are both of type ${\rm P\O}_8^{-}(q_0)$, where $q = q_0^2$.
\item[{\rm (c)}] $X = G_2(q)$ with $G_{\a}$ of type $P_1$ and $G_{\ell}$ of type $P_2$, where $p \geqs 5$.
\end{itemize}

\begin{ex}\label{computational-example}
Let $X={\rm L}_n(q)$ and suppose that $G_\alpha$ and $G_\ell$ are of type ${\rm GL}_{n/k}(q^k)$ and ${\rm GL}_n(q_0)$, respectively, where $q=q_0^r$ for a prime $r$. So $G_{\a}$ is a field extension subgroup in Aschbacher's $\C_{3}$ collection and $G_{\ell}$ is a subfield subgroup in $\C_5$. By inspecting Table \ref{tab:B}, we see that we may assume
\[
(n,k)\in\{(3,3),(4,2),(5,5),(6,2),(6,3),(8,2),(9,3)\},\;\; r\in\{2,3,5\}.
\]
From \cite[Propositions 4.3.6, 4.5.3]{KL} (or by inspecting the relevant tables in \cite[Chapter 8]{BHR}) we read off
\[
v=\frac{|{\rm GL}_n(q)|}{k|{\rm GL}_{n/k}(q^k)|},\;\;
b=\frac{|{\rm SL}_n(q)|}{w|{\rm SL}_n(q_0)|},
\]
where $w=\gcd\bigl(n,(q-1)/(q_0-1)\bigr)$.

We claim that we can use the computational methods in Remark \ref{rem:computational-strategy} to eliminate all possible triples of this form, with the exception of
\[
(n,k,r) \in \{ (6,3,3), (8,2,2) \},
\]
which are retained in collection $\mathcal{E}$ for further analysis in Section \ref{ss:comp_class}.

For each triple $(n,k,r)$, we summarize the associated computations in Table \ref{example-outcome}, where we adopt the following notation:
\begin{itemize}\addtolength{\itemsep}{0.2\baselineskip}
\item[{\rm (a)}] An entry $v^4$ (respectively, $b^4$) indicates that $v^4 \geqs b^5$ (respectively, $b^4 \geqs v^5$) for all but a finite set $\Sigma$ of prime powers. For $q \not\in \Sigma$, the corresponding triple $(G,G_{\a},G_{\ell})$ is eliminated via Lemma \ref{HD}(v), as in Method 1 of Remark \ref{rem:computational-strategy}. And for each $q \in \Sigma$, we can use the techniques in Remark \ref{r:c1} to show that there are no feasible parameters for $\cQ$.

\item[{\rm (b)}] We write $s+1$ (respectively, $st+1$) in Table \ref{example-outcome} if the inequality in part (i) (respectively, part (ii)) of Corollary \ref{gcd_bound} fails to hold for all but a finite set $\Sigma$ of prime powers, working with the respective techniques in Methods 2 and 3 of Remark \ref{rem:computational-strategy}. So if $q \not\in \Sigma$, then $(G,G_{\a},G_{\ell})$ is immediately eliminated, otherwise we conclude by ruling out the existence of feasible parameters, as above.
\end{itemize}

{\small
\begin{table}
\[
\begin{array}{ccccc} \hline
n & k & r=2 & r=3 & r=5\\\hline
3 & 3 & v^4 & st+1 & s+1\\
4 & 2 & st+1 & b^4 & b^4\\
5 & 5 & v^4 & s+1 & s+1\\
6 & 2 & st+1 & b^4 & b^4\\
6 & 3 & v^4 &  & s+1\\
8 & 2 &  & b^4 & b^4\\
9 & 3 & v^4 & s+1 & s+1\\ \hline
\end{array}
\]
\caption{A summary of the methods in Example \ref{computational-example}}
\label{example-outcome}
\end{table}
}

For brevity, we will only give details for the following triples
\[
(n,k,r)\in\{(3,3,2),(3,3,5),(4,2,3),(6,2,2),(8,2,2)\},
\]
noting that very similar methods apply in the remaining cases.

\vs

\noindent \emph{Case 1.} $(n,k,r)=(3,3,2)$

\vs

First assume $(n,k,r)=(3,3,2)$ and set $e = \gcd(3,q_0+1)$, where $q = q_0^2$. Write $z = q_0$ and note that
\[
v=f_1(z) = \frac{1}{3}z^6(z^2-1)(z^4-1),\;\; b= f_2(z) = \frac{1}{e}z^3(z^2+1)(z^3+1),
\]
so we have $4\deg f_1 > 5\deg f_2$ and we are in a position to apply Method 1 in Remark \ref{rem:computational-strategy}.
Indeed, it is easy to check that $v^4>b^5$ for all $z \geqs 2$, so the case  $(n,k,r) = (3,3,2)$ is eliminated and we write $v^4$ in Table \ref{example-outcome}.

\vs

\noindent \emph{Case 2.} $(n,k,r)=(4,2,3)$

\vs

Similarly, for $(n,k,r)=(4,2,3)$ we have $q = z^3$ and
\begin{align*}
v& = f_1(z) =\frac{1}{2}z^{12}(z^3-1)(z^9-1)<z^{24}\\
b& = f_2(z)= z^{12}(z^4+z^2+1)(z^6+z^3+1)(z^8+z^4+1)>z^{30}.
\end{align*}
Therefore, $b^4>z^{120}>v^5$ for all $q$ and so this possibility is also ruled out by Lemma \ref{HD}(v). So for this case, we write $b^4$ in Table \ref{example-outcome}.

\vs

\noindent \emph{Case 3.} $(n,k,r)=(3,3,5)$

\vs

Next assume $(n,k,r)=(3,3,5)$ and write $z = q_0=p^j$, where $p$ is a prime, so we have
\begin{align*}
\theta & = 10j\gcd(3,q-1) \\
v & = f_1(z) = \frac{1}{3}z^{15}(z^5-1)(z^{10}-1) \\
b & = f_2(z) = z^{12}(z^8+z^6+z^4+z^2+1)(z^{12}+z^9+z^6+z^3+1)
\end{align*}
and $4\deg f_1 < 5 \deg f_2$ and $4 \deg f_2 < 5 \deg f_1$. This means that we cannot eliminate this case using Method 1 in Remark \ref{rem:computational-strategy}. Indeed, it is easy to check that $v^4 <b^5$ and $b^4 < v^5$ for all $z$. However, we will show that Method 2 is applicable.

To see this, write
\[
v_{p'} = g_1(z) = \frac{1}{3_{p'}}(z^5-1)(z^{10}-1),\;\; |X_{\ell}|_{p'} = g_4(z) = \frac{1}{e_{p'}}(z^2-1)(z^3-1),
\]
where $e = \gcd(3,q-1)$. Now $v<\frac{1}{3}z^{30}$
and $v_p/2_p \geqs \frac{1}{3}z^{15}$, so $(v_p/2_p)^5 \geqs 2v^2$ and thus Lemma \ref{inequalities}(iii) implies that
\begin{equation}\label{e:splus11}
(s+1)_p \leqs 2_p\max\{1,v_p/b_p\} \leqs 2_pz^3 = f_5(z).
\end{equation}
In addition, $\theta|X_\ell|=10jz^3(z^2-1)(z^3-1)$ and using the {\sc Magma} command \texttt{XGCD}, we obtain
\[
g_5(x)=\gcd(g_1(x),g_4(x))=(x-1)^2(x+1).
\]
Moreover, we can take $c=1$ in \eqref{polynomial-gcd} and thus
\[
d(x) = f_5(x)g_5(x) = 2_px^3(x-1)^2(x+1) \in \mathbb{Q}[x].
\]
We now check that
\begin{equation}\label{e:dcond}
2\deg f_1 > \deg f_2 + 3 \deg d,
\end{equation}
so the required inequality in \eqref{parameter-comparison} only holds for at most finitely many values of $z = q_0$.

More precisely, by combining the above GCD calculation with the expression for $\theta|X_\ell|$, we deduce that $(s+1)_{p'}$ divides $(10j)_{p'}g_5(z)$. Since $2_p\leqs(10j)_p$, combining this with the upper bound on $(s+1)_p$ in \eqref{e:splus11} yields
\[
s+1\leqs10jz^3(z-1)^2(z+1)<10jz^6.
\]
And since $v> \frac{1}{4}z^{30}$, $b<2z^{32}$ and Corollary \ref{gcd_bound}(i) gives $v^2/b < (s+1)^3$, we deduce that
\begin{equation}\label{e:32000}
z^{10}<32000j^3.
\end{equation}
Since $z \geqs 2^j$ and the function $2^{10j}/j^3$ is increasing for $j\geqs1$, we see that the inequality in \eqref{e:32000} fails for all $j \geqs2$. And since $3^{10} > 32000$, we reduce to the case $z=2$. But then $s+1\leqs240$ and $(s+1)^3 < v^2/b$, which is incompatible with Corollary \ref{gcd_bound}(i). So the triple $(n,k,r) = (3,3,5)$ is eliminated and we write $s+1$ in Table \ref{example-outcome} to indicate that we have used Method 2 in Remark \ref{rem:computational-strategy}.

\vs

\noindent \emph{Case 4.} $(n,k,r)=(6,2,2)$

\vs

Suppose $(n,k,r)=(6,2,2)$ and write $z=q_0=p^j$, where $p$ is a prime. Here
\[
v=f_1(z)=\frac{1}{2}z^{18}(z^2-1)(z^6-1)(z^{10}-1),\;\;
b=f_2(z)=\frac{1}{e} z^{15}\prod_{i=2}^6(z^i+1),
\]
where $e=\gcd(6,q_0+1)$. The inequalities $v^4<b^5$ and $b^4<v^5$ hold for all $z$, so Method 1 does not exclude this case. Moreover, $v_p=z^{18}/2_p$ and $b_p=z^{15}$, so we may take $f_5(x)=x^3$ in \eqref{e:splus1} and we compute
\[
g_5(x)=(x+1)(x^2-1)(x^5-1)(x^6-1).
\]
Therefore $\deg d=17$ and $2\deg f_1=72<86=\deg f_2+3\deg d$, so the degree condition in \eqref{e:dcond} is not satisfied and thus Method 2 is also ineffective.

Finally, we turn to Method 3. Using \texttt{XGCD} with $2f_1(x)$ and $ef_2(x)$, and clearing denominators, we obtain
\[
st+1\leqs\gcd(v,b) \leqs16z^{15}(z^3+1)(z^5+1)<19z^{23}.
\]
Since $v>\frac{1}{3}z^{36}$ and $b>\frac{1}{6}z^{35}$, Corollary \ref{gcd_bound}(ii) gives
\[
\frac1{36}z^{71}<\frac12vb<(st+1)^3<19^3z^{69},
\]
hence $z^2<246924$ and we then check that the required inequality
$vb<2\gcd(v,b)^3$ holds only if $z\in\{3,5,7,9,13,19\}$. In each of these cases, we use the method in Remark \ref{r:c1}(b) to show that there are no feasible parameters for $\cQ$. This eliminates the case $(n,k,r)=(6,2,2)$ and we write $st+1$ in Table \ref{example-outcome}.

\vs

\noindent \emph{Case 5.} $(n,k,r)=(8,2,2)$

\vs

Finally, let us assume $(n,k,r)=(8,2,2)$. Set $z=q_0$ and observe that
\[
v=f_1(z)=\frac{1}{2}z^{32}(z^2-1)(z^6-1)(z^{10}-1)(z^{14}-1),\;\;
b=f_2(z)=\frac{1}{e}z^{28}\prod_{i=2}^8(z^i+1),
\]
where $e=\gcd(8,q_0+1)$. The inequalities $v^4<b^5$ and $b^4<v^5$ hold for all $z$, so Method 1 does not exclude this case.

Turning to Method 2, we see that $v_p=z^{32}/2_p$ and $b_p=z^{28}$, so we may take $f_5(x)=x^4$. Furthermore, $g_5(x)$ has degree $22$, so $\deg d=26$ and thus
\[
2\deg f_1=128<141=\deg f_2+3\deg d.
\]
Therefore, the condition in \eqref{e:dcond} is not satisfied and so we cannot apply Method 2.

For Method 3, we compute
\[
h(x)=x^{28}(x^3+1)(x^5+1)(x^7+1)
\]
(see \eqref{e:hx}) and thus $\deg f_1+\deg f_2=127<129=3\deg h$. Since the inequality in \eqref{e:hcond} is not satisfied, we conclude that Method 3 is also ineffective.

This means that none of the computational techniques in Remark \ref{rem:computational-strategy} are applicable and so the case $(n,k,r) = (8,2,2)$ is retained in $\mathcal{E}$ for further analysis in Section \ref{ss:comp_class}.
\end{ex}

\begin{ex}\label{computational-example-orthogonal}
Let $X={\rm P}\Omega_8^+(q)$ and suppose $G_\alpha$ and $G_\ell$ are subgroups in $\mathcal S$ of type ${\rm P}\Omega_8^-(q_0)$, where $q=q_0^2$. Set $z=q_0=p^j$ and note that
\[
v=b=f_1(z)=f_2(z)=\frac1e z^{12}(z^2+1)(z^6+1)(z^8-1)
\]
and
\[
|X_\alpha|=|X_\ell|=f_3(z)=f_4(z)=\frac1e z^{12}(z^2-1)(z^6-1)(z^8-1),
\]
where $e=\gcd(2,q_0-1)$. In particular, we have $s=t$ and the inequalities $v^4<b^5$ and $b^4<v^5$ hold for all $z$, so Method 1 is not applicable. We will exclude this family of triples by applying Method 2.

To do this, first note that $v_p=b_p=z^{12}$ and $v<2z^{28}$, so $
(v_p/2_p)^3\geqs \frac{1}{8}z^{36}>2v$ for all $z$. Since $v=b$, Lemma \ref{inequalities}(iii) gives $(s+1)_p\leqs2_p$ and so we can take $f_5(x)=2_p$ in \eqref{e:splus1}.

Recall that $q = q_0^2 = p^f$ and $q_0 = p^j$, so $f=2j$. By inspecting \cite[Table 8.50]{BHR}, we see that the maximality of $G_\alpha$ and $G_\ell$ in $G$ implies that $[G:X]$ divides $2ef$. And since $s+1$ divides $|G_\ell|=[G:X]|X_\ell|$, we deduce that
\[
\mbox{$s+1$ divides $2e^2f|X_\ell|=4je^2f_4(z)$.}
\]
For the $p'$-parts of $f_i(z)$, we have
\begin{align*}
g_1(z)=g_2(z)&=\frac{1}{e}(z^2+1)(z^6+1)(z^8-1)\\
g_3(z)=g_4(z)&=\frac{1}{e}(z^2-1)(z^6-1)(z^8-1)
\end{align*}
and using the {\sc Magma} function \texttt{XGCD} we compute
\[
g_5(x)=\gcd(g_1(x),g_4(x))=x^8-1,
\]
which allows us to take $c=4/e \in \mathbb{N}$ in \eqref{polynomial-gcd}. Therefore, in the notation of Method 2, we have
\[
d(x)=cf_5(x)g_5(x)=2_pc(x^8-1).
\]
Since $(s+1)_p\leqs2_p$ and $(s+1)_{p'}$ divides both $g_1(z)$ and $4je^2g_4(z)$, we deduce that
\[
s+1\leqs 2_p\gcd\bigl(g_1(z),4je^2g_4(z)\bigr) \leqs 4je^2d(z)=32jg_5(z),
\]
where we have used the fact that $2_pe=2$. Moreover,
\[
2\deg f_1=56>52=\deg f_2+3\deg d,
\]
so the required degree condition in Method 2 is satisfied.

Since $v=b$, Corollary \ref{gcd_bound}(i) gives $v<(s+1)^3$. Squaring and using $j^6\leqs2048z$, we obtain
\[
f_1(z)^2<(s+1)^6\leqs j^6(32g_5(z))^6\leqs2048z(32g_5(z))^6.
\]
For each $e\in\{1,2\}$, we can work with the proof of Lemma \ref{polynomial-bound} to show that
\[
f_1(x)^2-2048x(32g_5(x))^6 \geqs 0
\]
for all $x\geqs142$. This allows us to assume that $2\leqs z<142$, and  for each prime power $z$ in this range, we check that
\[
(s+1)^3 \leqs \left(2_p\gcd\bigl(g_1(z),4je^2g_4(z)\bigr)\right)^3 < v,
\]
a contradiction. Therefore, this family of triples has been eliminated.
\end{ex}

\begin{ex}\label{computational-example-exceptional}
In this final example we illustrate Method 2 from Remark \ref{rem:computational-strategy} in a setting where $X$ is an exceptional group of Lie type.

Let $X=G_2(q)$, where $q=p^f$ and $p\geqs5$, and suppose that $G_\alpha$ and $G_\ell$ are maximal parabolic subgroups of type $P_1$ and $P_2$, respectively. Set $z=q$ and $j = f$, so $z = p^j$. By inspecting \cite[Table 8.41]{BHR} we see that
\[
v=b=f_1(z)=f_2(z)=z^5+z^4+z^3+z^2+z+1
\]
and
\[
|X_\alpha|=|X_\ell|=f_3(z)=f_4(z)=z^6(z^2-1)(z-1),
\]
so $s=t$ and the inequalities $v^4<b^5$ and $b^4<v^5$ hold for all $z$. Therefore, Method 1 is ineffective and so we turn to Method 2. Since  $v_p=b_p=1$ and $s+1$ divides $v$, it follows that $(s+1)_p=1$ and we may take $f_5(x)=1$ in Method 2. Note that $\theta = j$ (see Table \ref{tab:H}).

For the $p'$-parts of $f_i(z)$ we have
\[
g_1(z)=g_2(z)=f_1(z),\;\; g_3(z)=g_4(z)=(z^2-1)(z-1).
\]
Using \texttt{XGCD}, we obtain
\[
g_5(x)=\gcd(g_1(x),g_4(x))=x+1
\]
and we may take $c=3$ in \eqref{polynomial-gcd} since
\[
(3-2x)g_1(x)+(2x+x^2+2x^3)g_4(x)=3(x+1).
\]
Therefore, $d(x)=cf_5(x)g_5(x)=3(x+1)$ and Method 2 gives
\[
s+1\leqs\theta d(z)=3j(z+1).
\]
Since $2\deg f_1=10>8=\deg f_2+3\deg d$, the degree condition in Method 2 is satisfied.

By Corollary \ref{gcd_bound}(i) we have $v<(s+1)^3$. Squaring and using $j^6\leqs2048z$, we obtain
\[
f_1(z)^2<(s+1)^6\leqs j^6d(z)^6\leqs2048zd(z)^6,
\]
whereas the proof of Lemma \ref{polynomial-bound} shows that
$f_1(x)^2-2048xd(x)^6$
is non-negative for all $x\geqs229$. This reduces the problem to the prime powers $z = p^j < 229$ with $p \geqs 5$. Here the bounds $v<(s+1)^3$ and $s+1\leqs3j(z+1)$ require
\[
z^5+z^4+z^3+z^2+z+1<27j^3(z+1)^3,
\]
which is only valid for $(p,j)=(5,1)$. Here $z=5$, $v = 3906$ and it is easy to check that there are no feasible parameters for $\cQ$.
\end{ex}

We now apply the methods presented in Remark \ref{rem:computational-strategy} to each candidate triple $(G,G_{\a},G_{\ell})$ in the collection $\mathcal{D}$, using {\sc Magma} to organize and execute the process. In this way, we obtain the following result which reduces the proof of Theorem \ref{t:main} for classical groups to the restricted list of cases recorded in Tables \ref{tab:E1} and \ref{tab:E2}. We refer to the latter set of remaining triples as collection $\mathcal{E}$.

\begin{prop}\label{p:DtoE}
Let $\mathcal{Q} = (\mathcal{P},\mathcal{L})$ be a finite thick generalized quadrangle admitting an almost simple point-primitive, line-primitive and flag-transitive automorphism group $G$ with socle $X \in \mathcal{G}$. Then up to duality, either
\begin{itemize}\addtolength{\itemsep}{0.2\baselineskip}
\item[{\rm (i)}] $\cQ$ is a classical generalized quadrangle; or
\item[{\rm (ii)}] $(G,G_{\a},G_{\ell}) \in \mathcal{E}$ is one of the cases in Tables \ref{tab:E1} or \ref{tab:E2}.
\end{itemize}
\end{prop}

\begin{proof}
The sequence of reductions obtained in Sections \ref{s:subspace}-\ref{s:irred3} allow us to assume that the triple
$(G,G_{\a},G_{\ell})$ is in the collection $\mathcal{D}$. We now work systematically through the complete list of triples in $\mathcal{D}$, using a {\sc Magma} implementation of Methods 1, 2 and 3 in Remark \ref{rem:computational-strategy}. In order to test a given family of triples, we just need to input the following parameters:
\[
|X|, \; |{\rm Out}(X)|, \; |X_{\a}|, \; |X_{\ell}|,
\]
which we can do in every case (see Remark \ref{r:params}). No special difficulties arise in any of these cases and so we omit the details (see Examples \ref{computational-example} and \ref{computational-example-orthogonal}).

In several cases, we find that one of the methods in Remark \ref{rem:computational-strategy} is only effective if $q$ is sufficiently large (for instance, see Case 4 in Example \ref{computational-example}). In every such situation, we can eliminate the remaining finite list of possibilities for $q$ by appealing to the inequalities $v^4 < b^5$ and $b^4 < v^5$ in Lemma \ref{HD}(v), or the inequalities in Lemma  \ref{inequalities}, or by ruling out the existence of feasible parameters for $\cQ$, using the techniques presented in Remark \ref{r:c1}.

Finally, the families of triples in $\mathcal{D}$ that we cannot eliminate in this way, via the methods presented in Remark \ref{rem:computational-strategy}, comprise the collection $\mathcal{E}$.
\end{proof}

\begin{rem}
Some comments on the notation and entries in Tables \ref{tab:E1} and \ref{tab:E2}:
\begin{itemize}\addtolength{\itemsep}{0.2\baselineskip}
\item[{\rm (a)}] We set $q_0 = q^{1/2}$, unless stated otherwise. And in Table \ref{tab:E2} we set $\e=+$, unless stated otherwise. In both tables we record the subgroup collections containing $G_{\a}$ and $G_{\ell}$ in terms of the standard notation $\C_i,\mathcal{S},\mathcal{N}$ for $1 \leqs i \leqs 8$ from Section \ref{sss:sub}.

\item[{\rm (b)}] Various conditions on the defining parameters for $G$, $G_{\a}$ and $G_{\ell}$ are listed in the final column of both tables. In many cases, additional conditions are needed to ensure the existence and maximality of $G_{\a}$ and $G_{\ell}$. These latter conditions have been omitted in the tables, but they can be easily read off from \cite{BHR,KL}. For example, if $X = {\rm L}_n(q)$ and $G_{\a}$ or $G_{\ell}$ is a $\C_8$-subgroup of type ${\rm O}_n^{\delta}(q)$, then $q$ is odd (see \cite[Table 3.5.A]{KL}).

\item[{\rm (c)}] In Table \ref{tab:E2}, the $\mathcal{E}_i$ notation appearing in the final column is defined in Table \ref{tab:Ei}.
\end{itemize}
\end{rem}

{\small
\begin{table}
\[
\begin{array}{ll} \hline
\mathcal{E}_1 & \{ (8,2,\pm),(12,4,-),(14,6,-),(16,4,\pm),(16,6,-),(20,6,-)\} \\
\mathcal{E}_2 & \{(12,3),(14,5),(16,5),(20,5)\} \\
\mathcal{E}_3 & \{ (8,2,\pm),(12,4,-),(14,6,+),(16,4,\pm),(16,6,+),(20,6,+)\} \\
\mathcal{E}_4 & \{ (8,2,-),(12,4,\pm),(14,4,\pm),(16,4,+),(16,6,-)\} \\
\mathcal{E}_5 & \{(8,3,-),(12,3,\pm),(16,5,-)\} \\
\mathcal{E}_6 & \{ (8,+,-),(8,-,+),(8,-,-),(12,-,+)\} \\ \hline
\end{array}
\]
\caption{The sets $\mathcal{E}_i$ in Table \ref{tab:E2}, $1 \leqs i \leqs 6$}
\label{tab:Ei}
\end{table}
}

{\small
\begin{table}
\[
\begin{array}{lcclll} \hline
X & \mbox{Label} & \mbox{Class pair} & \mbox{Type of $G_{\a}$} & \mbox{Type of $G_{\ell}$} & \mbox{Conditions} \\ \hline
{\rm L}_n(q) & \texttt{L1} & (\C_1,\C_8) & {\rm GL}_2(q)\oplus {\rm GL}_4(q) & {\rm Sp}_6(q) & n = 6 \\
& \texttt{L2} & (\C_2,\C_3) & {\rm GL}_1(q)\wr S_3 & {\rm GL}_1(q^3) & n = 3 \\
& \texttt{L3} & (\C_3,\C_5) & {\rm GL}_2(q^3) & {\rm GL}_6(q_0) & n = 6, \, q=q_0^3 \\
& \texttt{L4} & (\C_3,\C_5) & {\rm GL}_4(q^2) & {\rm GL}_8(q_0) & n = 8 \\
& \texttt{L5} & (\C_3,\C_8) & {\rm GL}_{n/3}(q^3) & {\rm O}_n^\epsilon(q) & (n,\epsilon) \in \{(6,\pm),(12,+)\} \\
& \texttt{L6} & (\C_3,\C_8) & {\rm GL}_4(q^2) & {\rm U}_8(q_0) & n = 8 \\
& \texttt{L7} & (\C_5,\C_8) & {\rm GL}_8(q_0) & {\rm Sp}_8(q) & n = 8 \\
& \texttt{L8} & (\C_5,\C_8) & {\rm GL}_n(q_0) & {\rm O}_n(q) & n \in \{3,5,7\} \\
& \texttt{L9} & (\C_5,\C_8) & {\rm GL}_n(q_0) & {\rm O}_n^\epsilon(q) &
(n,\epsilon) \in \{(4,-),(6,+),(8,-)\} \\
& \texttt{L10} & (\C_5,\C_8) & {\rm GL}_n(q_0) & {\rm U}_n(q_0) & n \in \{4,5,7,8,9\} \\
& \texttt{L11} & (\C_8,\C_8) & {\rm O}_4^+(q) & {\rm O}_4^-(q) & n = 4 \\
& \texttt{L12} & (\C_8,\C_8) & {\rm O}_n(q) & {\rm U}_n(q_0) & n \in \{3,5,7\} \\
& \texttt{L13} & (\C_8,\C_8) & {\rm Sp}_8(q) & {\rm U}_8(q_0) & n = 8 \\
& \texttt{L14} & (\C_8,\C_8) & {\rm O}_n^\epsilon(q) & {\rm U}_n(q_0) &
(n,\epsilon)\in \{(4,-),(6,+),(8,-)\}\\
& & & & & \\
{\rm U}_n(q) & \texttt{U1} & (\C_1,\C_5) & {\rm GU}_2(q)\perp {\rm GU}_4(q) & {\rm Sp}_6(q) & n = 6 \\
& \texttt{U2} & (\C_1,\C_5) & {\rm GU}_3(q)\perp {\rm GU}_5(q) & {\rm O}_8^-(q) & n = 8 \\
& \texttt{U3} & (\C_3,\C_5) & {\rm GU}_{n/3}(q^3) & {\rm O}_n^\epsilon(q) & (n,\epsilon) \in \{(6,\pm),(12,+)\} \\
& \texttt{U4} & (\C_5,\C_5) & {\rm O}_4^+(q) & {\rm O}_4^-(q) & n = 4 \\
& & & & & \\
{\rm PSp}_n(q) & \texttt{S1} & (\C_2,\C_3) & {\rm Sp}_2(q)\wr S_3 & {\rm Sp}_2(q^3) & n = 6 \\
& \texttt{S2} & (\C_2,\C_3) & {\rm GL}_{n/2}(q).2 & {\rm Sp}_{n/2}(q^2) & n \in \{12,16\} \\
& \texttt{S3} & (\C_2,\C_3) & {\rm GL}_{n/2}(q).2 & {\rm GU}_{n/2}(q) & n \in \{6,10,16,18\} \\
& \texttt{S4} & (\C_3,\C_3) & {\rm GU}_{n/2}(q) & {\rm Sp}_{n/2}(q^2) & n \in \{12,16\}\\
& \texttt{S5} & (\C_3,\C_5) & {\rm Sp}_2(q^3) & {\rm Sp}_6(q_0) & n = 6, \, q=q_0^3\\
& \texttt{S6} & (\C_3,\C_5) & {\rm GU}_4(q) & {\rm Sp}_8(q_0) & n = 8 \\
& \texttt{S7} & (\C_3,\C_5) & {\rm Sp}_{n/k}(q^k) & {\rm Sp}_n(q_0) & (n,k) \in \{(6,3),(8,2)\} \\
& & & & & \\
\O_n(q) & \texttt{O1} & (\C_1,\C_5) & {\rm O}_3(q)\perp {\rm O}_4^-(q) & {\rm O}_7(q_0) & n = 7\\
& \texttt{O2} & (\C_1,\C_5) & {\rm O}_5(q)\perp {\rm O}_2^+(q) & {\rm O}_7(q_0) & n = 7\\
& \texttt{O3} & (\C_1,\C_5) & {\rm O}_3(q)\perp {\rm O}_6^-(q) & {\rm O}_9(q_0) & n = 9 \\ \hline
\end{array}
\]
\caption{The collection $\mathcal{E}$, Part I: $X \ne {\rm P\O}_n^{\pm}(q)$}
\label{tab:E1}
\end{table}
}

{\small
\begin{table}
\[
\begin{array}{cclll} \hline
\mbox{Label} & \mbox{Class pair} & \mbox{Type of $G_{\a}$} & \mbox{Type of $G_{\ell}$} & \mbox{Conditions} \\ \hline
\texttt{EO1} & (\C_1,\C_1) & {\rm O}_4^-(q)\perp {\rm O}_{16}^-(q) & {\rm O}_5(q)\perp {\rm O}_{15}(q) & n=20 \\
\texttt{EO2} & (\C_1,\C_1) & {\rm O}_2^-(q)\perp {\rm O}_{n-2}^-(q) & {\rm O}_2^+(q)\perp {\rm O}_{n-2}^+(q) & n \in \{8,12,16,20\} \\
\texttt{EO3} & (\C_1,\C_2) & {\rm O}_i^\delta(q)\perp {\rm O}_{n-i}^\delta(q) & {\rm GL}_{n/2}(q).2 & (n,i,\delta) \in \mathcal{E}_1 \\
\texttt{EO4} & (\C_1,\C_2) & {\rm O}_i(q)\perp {\rm O}_{n-i}(q) & {\rm GL}_{n/2}(q).2 & (n,i) \in \mathcal{E}_2 \\
\texttt{EO5} & (\C_1,\C_3) & {\rm O}_5(q)\perp {\rm O}_{11}(q) & {\rm O}_8^\epsilon(q^2) & \e=\pm, \, n = 16 \\
\texttt{EO6} & (\C_1,\C_3) & {\rm O}_i^\delta(q)\perp {\rm O}_{n-i}^{\epsilon\delta}(q) & {\rm GU}_{n/2}(q) & \epsilon=(-)^{n/2},\, (n,i,\delta) \in \mathcal{E}_3 \\
\texttt{EO7} & (\C_1,\C_3) & {\rm O}_i(q)\perp {\rm O}_{n-i}(q) & {\rm GU}_{n/2}(q) & \epsilon=(-)^{n/2}, \, (n,i) \in \mathcal{E}_2 \\
\texttt{EO8} & (\C_1,\C_5) & {\rm O}_i^-(q)\perp {\rm O}_{n-i}^-(q) & {\rm O}_n^\eta(q_0) & (n,i,\eta) \in \mathcal{E}_4 \\
\texttt{EO9} & (\C_1,\C_5) & {\rm O}_i(q)\perp {\rm O}_{n-i}(q) & {\rm O}_n^\eta(q_0) & (n,i,\eta) \in \mathcal{E}_5  \\
\texttt{EO10} & (\C_1,\mathcal{S}) & {\rm O}_2^-(q)\perp {\rm O}_6^-(q) &
{\rm P\Omega}_8^-(q_0) & n = 8 \\
\texttt{EO11} & (\C_1,\mathcal{S}) & {\rm O}_3(q)\perp {\rm O}_5(q) &
{\rm P\Omega}_8^-(q_0) & n = 8 \\
\texttt{EO12} & (\C_2,\C_2) & {\rm GL}_4(q).2 & {\rm GL}_4(q).2 & n = 8 \\
\texttt{EO13} & (\C_2,\C_3) & {\rm GL}_{n/2}(q).2 & {\rm O}_{n/2}^+(q^2) & n \in\{ 12,16 \} \\
\texttt{EO14} & (\C_2,\C_3) & {\rm GL}_9(q).2 & {\rm O}_9(q^2) & n = 18 \\
\texttt{EO15} & (\C_2,\C_3) & {\rm GL}_{n/2}(q).2 & {\rm GU}_{n/2}(q) & n \in \{8,12,16,20\} \\
\texttt{EO16} & (\C_2,\C_5) & {\rm O}_{n/2}^\delta(q)\wr S_2 & {\rm O}_n^\eta(q_0) & (n,\delta,\eta) \in \mathcal{E}_6 \\
\texttt{EO17} & (\C_2,\mathcal{S}) & {\rm O}_4^\delta(q)\wr S_2 & {\rm P\Omega}_8^-(q_0) & (n,\delta)=(8,\pm)\\
\texttt{EO18} & (\C_2,\mathcal{N}) & {\rm O}_4^+(q)\wr S_2 & G_2(q) & n = 8 \\
\texttt{EO19} & (\C_3,\C_3) & {\rm GU}_4(q) & {\rm GU}_4(q) & n = 8 \\
\texttt{EO20} & (\C_3,\C_3) & {\rm O}_{n/2}^+(q^2) & {\rm GU}_{n/2}(q) & n \in \{12,16\} \\
\texttt{EO21} & (\C_3,\C_3) & {\rm O}_9(q^2) & {\rm GU}_9(q) & (n,\e) = (18,-) \\
\texttt{EO22} & (\C_3,\C_4) & {\rm O}_6^+(q^2) & {\rm Sp}_2(q)\otimes {\rm Sp}_6(q) & n = 12 \\
\texttt{EO23} & (\C_3,\C_5) & {\rm O}_{n/2}^+(q^2) & {\rm O}_n^\eta(q_0) & n \in \{8,12,16\},\, \eta = \pm \\
\texttt{EO24} & (\C_3,\C_5) & {\rm GU}_{n/2}(q) & {\rm O}_n^\eta(q_0) & (n,\eta) \in \{ (8,-),(12,\pm),(16,\pm)\} \\
\texttt{EO25} & (\C_3,\C_5) & {\rm O}_{n/2}(q^2) & {\rm O}_n^\eta(q_0) & n \in \{10,14,18\},\, \eta = \pm  \\
\texttt{EO26} & (\C_3,\mathcal{S}) & {\rm O}_4^+(q^2) & {\rm P\Omega}_8^-(q_0) & n = 8 \\
\texttt{EO27} & (\C_3,\mathcal{S}) & {\rm GU}_4(q) & {\rm P\Omega}_8^-(q_0) & n = 8 \\
\texttt{EO28} & (\C_4,\C_5) & {\rm Sp}_2(q)\otimes {\rm Sp}_4(q) & {\rm O}_8^-(q_0) & n = 8 \\
\texttt{EO29} & (\C_4,\mathcal{S}) & {\rm Sp}_2(q)\otimes {\rm Sp}_4(q) & {\rm P\Omega}_8^-(q_0) & n= 8 \\ \hline
\end{array}
\]
\caption{The collection $\mathcal{E}$, Part II: $X = {\rm P\O}_n^{\e}(q)$, $\e=\pm$}
\label{tab:E2}
\end{table}
}

\subsection{Further techniques}\label{ss:further}

In view of Proposition \ref{p:DtoE}, in order to complete the proof of Theorem \ref{t:main} for classical groups we are required to eliminate all of the triples in collection $\mathcal{E}$ (see Tables \ref{tab:E1} and \ref{tab:E2}). The purpose of this section is to introduce several new techniques for ruling out a given triple $(G,G_{\a},G_{\ell})$, which we will use in the next section to rule out all of the remaining candidate triples in the collection $\mathcal{E}$.

Similar methods will also play a key role in Section \ref{ss:ex_complete}, where we complete the proof of Theorem \ref{t:main} for almost simple exceptional groups of Lie type. Therefore, we will adopt the following more general hypothesis throughout Section \ref{ss:further}.

\begin{hypothesis}
Let $\mathcal{Q}=(\mathcal{P},\mathcal{L})$ be a finite thick generalized quadrangle of order $(s,t)$ with an automorphism group $G$ that is flag-transitive, point-primitive and line-primitive. Assume $G$ is an almost simple group with socle $X$.
\end{hypothesis}

As before, let $(\a,\ell) \in \mathcal{P} \times \mathcal{L}$ and set
\[
v = |\cP| = [X:X_{\a}], \;\; b = |\cL| = [X:X_{\ell}].
\]

\subsubsection{Fixed point methods}

\begin{lem}\label{l:large}
Let $g \in X$ be nontrivial.
\begin{itemize}\addtolength{\itemsep}{0.2\baselineskip}
\item[{\rm (i)}] If $|\cP_g|^3>4vb$, then $|\mathcal L_g|^3 < 16v^2/b$.
\item[{\rm (ii)}] If $|\cL_g|^3>4vb$, then $|\mathcal P_g|^3 < 16b^2/v$.
\end{itemize}
\end{lem}

\begin{proof}
By inspecting the possibilities for the fixed substructure $\cQ_g = (\cP_g,\cL_g)$, as detailed in Lemma \ref{subquadrangle}, we observe that $|\cP_g| \leqs (s+1)(t+1)$ if $\cQ_g$ is not a grid. In particular, if $|\cP_g|>(s+1)(t+1)$ then $\cQ_g$ is a grid and $|\cL_g| \leqs 2(s+1)$ (see case  (3) in Lemma \ref{subquadrangle}).

By Lemma \ref{HD}(v), we have
\[
(s+1)^3(t+1)^3 < 4vb,\;\; (2(s+1))^3 < 16v^2/b.
\]
So if $|\cP_g|^3 >4vb$, then $|\cP_g| > (s+1)(t+1)$ and we deduce that
\[
|\cL_g|^3 \leqs (2(s+1))^3 < 16v^2/b.
\]
The statement in part (ii) follows by point-line duality.
\end{proof}

\begin{lem}\label{l:der}
Let $g \in X$ be nontrivial.
\begin{itemize}\addtolength{\itemsep}{0.2\baselineskip}
\item[{\rm (i)}] If $g^X \cap X_{\ell}$ is empty then $|\cP_g|^3 < vb$.
\item[{\rm (ii)}] If $g^X \cap X_{\alpha}$ is empty, then $|\cL_g|^3 < vb$.
\end{itemize}
\end{lem}

\begin{proof}
Suppose $g^X \cap X_{\ell}$ is empty. Then $g$ fixes no lines in $\cL$, so the points fixed by $g$ are pairwise non-collinear and thus
\[
|\cP_g| \leqs st+1 < (vb)^{1/3}
\]
(see Lemma \ref{HD}(v)). Similarly, we deduce that $|\cL_g| < (vb)^{1/3}$ if $g^X \cap X_{\a}$ is empty.
\end{proof}

\subsubsection{The Sylow method}

\begin{lem}\label{l:sylow0}
Let $r$ be a prime and let $P \leqs X_\a$ be an $r$-subgroup. Then there exists a line $\ell \in \cL$ through $\a$ and a subgroup $H \leqs P$ such that
\[
H \leqs X_{\alpha} \cap X_{\ell} \, \mbox{ and } \, [P:H]\leqs (t+1)_r.
\]
In particular, if $P$ is a Sylow $r$-subgroup of $X_{\a}$ and $(t+1)_r \leqs 2_r$, then $|X_\a|_r/|X_{\alpha} \cap X_{\ell}|_r \leqs 2_r$.
\end{lem}

\begin{proof}
Consider the action of $P$ on the set of $t+1$ lines through $\a$ and note that every $P$-orbit has length a power of $r$. Of course, if the length of every $P$-orbit is greater than $(t+1)_r$, then the $r$-part of the sum of all the orbit lengths would also be greater than $(t+1)_r$, which is absurd. Therefore, there exists a line $\ell$ through $\a$ such that
$H \leqs X_{\alpha} \cap X_{\ell}$ and $[P:H]\leqs (t+1)_r$ for $H = P_{\ell}$.
\end{proof}

\begin{lem}\label{l:sylow}
Let $r$ be a prime divisor of $|X_{\a}|$ such that $|X_{\a}|_r \leqs |X_{\ell}|_r$ and $(b_r/2_r)^5 \geqs 2b^2$.
\begin{itemize}\addtolength{\itemsep}{0.2\baselineskip}
\item[{\rm (i)}] If $r$ is odd, then $g^X \cap X_{\ell}$ is nonempty for all
$g \in X_{\alpha}$ of $r$-power order.
\item[{\rm (ii)}] If $r=2$ and $P$ is a Sylow $2$-subgroup of $X_{\a}$, then there exists an index-two subgroup $H$ of $P$ such that $g^X \cap X_{\ell}$ is nonempty for all $g \in H$.
\end{itemize}
\end{lem}

\begin{proof}
Let $P$ be a Sylow $r$-subgroup of $X_{\a}$ and note that $b_r \leqs v_r$. Since $(b_r/2_r)^5 \geqs 2b^2$, Lemma \ref{inequalities}(ii) implies that $(t+1)_r \leqs 2_r$.

If $r$ is odd, then by Lemma \ref{l:sylow0} there is a line $\ell$ through $\a$ such that $|X_{\a} \cap X_{\ell}|_r = |X_{\a}|_r$ and thus some $X_{\a}$-conjugate of $P$ is contained in $X_{\ell}$. Therefore, if there exists an $r$-element $g \in X_{\a}$ such that $g^X \cap X_{\ell}$ is empty, then $X_{\ell}$ does not contain a conjugate of $P$ and we have reached a contradiction.

Similarly, if $r=2$ then Lemma \ref{l:sylow0} implies that there is a line $\ell$ through $\a$ such that $\frac{1}{2}|X_{\a}|_2 \leqs |X_{\a} \cap X_{\ell}|_2$. This means that $X_{\ell}$ either contains a Sylow $2$-subgroup of $X_{\a}$, or it contains an index-two subgroup of such a Sylow. So if every index-two subgroup of a Sylow $2$-subgroup of $X_{\a}$ contains an element $g$ such that $g^X \cap X_{\ell}$ is empty, then we have reached a contradiction.
\end{proof}

\begin{rem}\label{r:sylow}
Suppose $r=2$ in Lemma \ref{l:sylow} and set $g = x^2$, where $x \in X_{\alpha}$ is a $2$-element. Then $g^X \cap X_{\ell}$ is nonempty.
\end{rem}

\subsubsection{Arithmetic methods}

\begin{lem}\label{lem_cubic}
Let $h=\gcd(v,b)$ and set $v'=v/h$, $b'=b/h$, $k = h/(st+1)$ and
\begin{equation}\label{eqn_f}
F(x)=v'b'x^3-(v'+b')x^2+2x-h \in \mathbb{Z}[x].
\end{equation}
Then the following hold:
\begin{itemize}\addtolength{\itemsep}{0.2\baselineskip}
\item[{\rm (i)}] $F(x)$ is an increasing function in the interval $[c,\infty)$, where $c=\frac{2}{3}(v'^{-1}+b'^{-1})$.
\item[{\rm (ii)}] $k$ is a positive integer such that $F(k)=0$ and
\[
\frac{h}{v'b'}<k^3<\frac{2h}{v'b'}.
\]
\end{itemize}
\end{lem}

\begin{proof}
First observe that
\[
F'(x)=3v'b'x^2-2(v'+b')x+2,\;\; c=\frac{2(v'+b')}{3v'b'},
\]
so $F'(x)\geqs 2$ for all $x\geqs c$ and thus (i) holds. Note that $c \leqs \frac{4}{3}$ since $v',b' \geqs 1$.

Since $st+1$ divides $\gcd(v,b)$ by Lemma \ref{HD}(iv), we see that  $k = h/(st+1)$ is a positive integer. And since $v = (s+1)(st+1)$ we deduce that $s=kv'-1$, and similarly $t=kb'-1$. Therefore
\[
(kv'-1)(kb'-1)+1=\frac{h}{k},
\]
which simplifies to $F(k)=0$. Finally, the upper and lower bounds on $k^3$ follow immediately from the fact that
\[
\frac{1}{2}(s+1)(t+1)<st+1<(s+1)(t+1). \qedhere
\]
\end{proof}

We will use Lemma \ref{lem_cubic} in several different ways to eliminate a given candidate triple $(G,G_{\a},G_{\ell})$. Let us assume $X$ is a simple group of Lie type defined over $\mathbb{F}_q$, where $q=p^f$ and $p$ is a prime, and assume we can compute $v$ and $b$ in terms of $q$. Let us also assume we can determine $h = \gcd(v,b)$, which in turn allows us to compute $v'$ and $b'$. This means that we can determine the polynomial $F(x) \in \mathbb{Z}[x]$, as well as the upper and lower bounds on $k^3$ given in part (ii) of the lemma. We will typically apply the lemma in one of three ways:
\begin{itemize}\addtolength{\itemsep}{0.2\baselineskip}
\item[{\rm (a)}] We may observe that there is no positive integer $k$ satisfying the upper and lower bounds on $k^3$ in (ii), or that $F(k) \ne 0$ for all integers in the required range.
\item[{\rm (b)}] We may be able to show that $F(a_1) <0$ and $F(a_2)>0$ for a pair of real numbers $a_1$ and $a_2$ with $\frac{4}{3} \leqs a_1 < a_2$. Then since $F(k)=0$ and $F$ is increasing on the interval $[a_1,\infty)$, it follows that $a_1 < k < a_2$ and this may be incompatible with the bounds on $k^3$ in part (ii).
\item[{\rm (c)}] We first apply Lemma \ref{inequalities} to obtain conditions on $k_p$, which we can then use to reduce the equation $F(k) = 0$ modulo certain powers $q^m$. This yields information on $k \imod{q^m}$, which we can sometimes use to obtain a congruence condition on $k$ that is incompatible with the bounds on $k^3$ in part (ii) of the lemma.
\end{itemize}

\begin{rem}\label{r:approx}
Here we present another application of Lemma \ref{lem_cubic}, which is based on approximating $k$ by a polynomial with rational coefficients. This method can be used to find real numbers $a_1,a_2$ as required by the approach in (b) above, providing an alternative to the congruence technique in (c). For example, several cases arising in Arumugam's treatment of the groups with socle $X = {}^2F_4(q)'$ in \cite{Arumugam_thesis} can be significantly simplified using this approach.

Fix notation as in Lemma \ref{lem_cubic} and write $q=z^r$, where $r$ is a fixed positive integer and $z$ ranges over an appropriate set of prime powers. Suppose that, after fixing any rational scalars taking finitely many
possible values, we have
\[
 v'=f_1(z),\;\; b'=f_2(z),\;\; h=f_3(z),
\]
where $f_i \in\mathbb{Q}[x]$ has degree $e_i \geqs 0$ and leading coefficient $c_i$.  Assume that $e_3-e_1-e_2=3m$ and $c_3 = \lambda^3c_1c_2$ for some integer $m>0$ and positive rational number $\lambda$.

Define
\[
\Phi(x,y)=f_1(x)f_2(x)y^3-(f_1(x)+f_2(x))y^2+2y-f_3(x) \in \mathbb{Q}[x,y].
\]
By Lemma \ref{lem_cubic}, the existence of the proposed generalized quadrangle $\cQ$ requires a positive integer $k$ satisfying $\Phi(z,k)=0$. And by considering the leading term of the expression $f_1(z)f_2(z)k^3 - f_3(z)$ we can view $\lambda z^m$ as a first approximation for $k$.

We can refine this approximation by iteratively constructing a polynomial
\[
K(x)=\lambda x^m+d_{1}x^{m-1}+\cdots+d_m\in\mathbb{Q}[x]
\]
such that $\deg\Phi(x,K(x))<e_1+e_2+2m$.  

Set $K_0(x)=\lambda x^m$. Our assumptions on the degrees
and leading coefficients ensure that $\deg\Phi(x,K_0(x)) < e_3$.
For each $1\leqs j\leqs m$, suppose $K_{j-1}(x)$ has been chosen so that
$\deg\Phi(x,K_{j-1}(x))\leqs e_3-j$. Let $b_j \in \mathbb{Q}$ be the coefficient of $x^{e_3-j}$ in $\Phi(x,K_{j-1}(x))$ and set
\[
d_j=-\frac{b_j}{3c_1c_2\lambda^2} \in \mathbb{Q}, \;\;
 K_j(x)=K_{j-1}(x)+d_jx^{m-j}.
\]
Then $\deg\Phi(x,K_j(x))\leqs e_3-j-1$ and thus $K=K_m$ has the desired property.

Set $R(x)=\Phi(x,K(x)) \in \mathbb{Q}[x]$ and assume $R$ is not the zero polynomial. We claim that this rules out the candidate family $(G,G_{\a},G_{\ell})$ for all sufficiently large $z$.

To see this, choose a positive integer $N$ such that $NK(x)\in\mathbb{Z}[x]$, and fix
\[
 0<\epsilon<\min\{\lambda,1/N\}.
\]
Since $\deg R<e_1+e_2+2m$, the leading term of $\Phi(z,K(z)\pm \epsilon)$ is $\pm3\epsilon z^{e_1+e_2+2m}$ and so we deduce that
$\Phi(z,K(z)-\epsilon)<0$ and $\Phi(z,K(z)+\epsilon)>0$
for all sufficiently large $z$. Therefore, Lemma \ref{lem_cubic}(i) implies that
\begin{equation}\label{e:Kz}
 K(z)-\epsilon<k<K(z)+\epsilon
\end{equation}
for all sufficiently large $z$. The assumption that $R$ is nonzero means that $R(z) \ne 0$ for all sufficiently large $z$, which in turn implies that $k \ne K(z)$ for $z \gg 0$. However, if $k \ne K(z)$ then
\[
|k-K(z)|\geqs \frac{1}{N}>\epsilon
\]
since $k$ and $K(z)N$ are integers. This is incompatible with \eqref{e:Kz} and we reach a contradiction. 

So this approach provides a way to reduce the elimination of a given family of candidate triples $(G,G_{\a},G_{\ell})$ to a finite set of possibilities for $z$ (and hence $q$). As before, we can then use the techniques in Remark \ref{r:c1} to check for the existence of feasible parameters for $\cQ$ for each of the remaining possibilities. In practice, with regard to approach (b) above, we have some considerable flexibility when choosing $a_1,a_2 \in \mathbb{R}$ such that $a_1 \leqs K(z)$ and $a_2>K(z)$. Indeed, we can allow $|a_i - K(z)|$ to be relatively large, if needed.
\end{rem}

\subsection{Completing the proof}\label{ss:comp_class}

We are now ready to complete the proof of Theorem \ref{t:main} for classical groups. Our main result is the following.

\begin{thm}\label{t:class}
Let $\mathcal{Q} = (\mathcal{P},\mathcal{L})$ be a finite thick generalized quadrangle admitting an almost simple point-primitive, line-primitive and flag-transitive automorphism group $G$ with socle $X \in \mathcal{G}$. Then $\mathcal{Q}$ is a classical generalized quadrangle.
\end{thm}

In order to do this, we will systematically work through the list of triples $(G,G_{\a},G_{\ell})$ in $\mathcal{E}$, applying a combination of the techniques in Section \ref{ss:further}. We divide the analysis into five separate propositions, according to the socle $X$ of $G$. In Table \ref{tab:methods}, we present a convenient summary of the techniques that we use to handle each family of triples in $\mathcal{E}$, according to the following labelling of the main methods:
\begin{itemize}\addtolength{\itemsep}{0.2\baselineskip}
\item[{\rm I}] Fixed point method, I: Lemma \ref{l:large}
\item[{\rm II}] Fixed point method, II: Lemma \ref{l:der}
\item[{\rm III}] Sylow method: Lemma \ref{l:sylow}
\item[{\rm IV}] Arithmetic methods: Lemma \ref{lem_cubic}
\item[{\rm V}] Miscellaneous methods
\end{itemize}
Note that an asterisk in Table \ref{tab:methods} indicates $p=2$, whereas a dagger signifies that $p \geqs 3$.

{\small
\begin{table}
\[
\begin{array}{ll} \hline
{\rm I} & \texttt{L1}, \, \texttt{L10} \, (n \geqs 7), \, \texttt{U1}, \, \texttt{O1}, \, \texttt{O2}, \, \texttt{O3} \\ [3pt]
& \texttt{EO2} \, (n \geqs 12), \, \texttt{EO3}^* \, (n \geqs 12), \, \texttt{EO6}^* \, (n \geqs 12), \, \texttt{EO8} \, (n \geqs 12), \, \texttt{EO15} \, (n \geqs 12), \, \texttt{EO16} \, (n=12) \\
& \\
{\rm II} & \texttt{L3}, \, \texttt{L4}, \, \texttt{L6}, \, \texttt{L12} \, (n=7), \, \texttt{L14} \, (n=8), \, \texttt{S1}, \, \texttt{S5}, \, \texttt{S6}, \, \texttt{S7} \\ [3pt]
& \texttt{EO13} \, (n = 16 \mbox{ or } (n,p) = (12,2)), \, \texttt{EO14}, \, \texttt{EO20} \, (n = 16 \mbox{ or } (n,p) = (12,2)), \, \texttt{EO21},  \, \texttt{EO22} \\
& \texttt{EO23} \, (n \geqs 12), \, \texttt{EO24}^* \, (n \geqs 12), \, \texttt{EO25} \, (n \geqs 14) \\
& \\
{\rm III} & \texttt{L5}, \, \texttt{L7}, \, \texttt{L11}, \, \texttt{L13}, \, \texttt{U2}, \, \texttt{U3}, \, \texttt{U4}, \, \texttt{S2}, \, \texttt{S3}, \, \texttt{S4} \\ [3pt]
& \texttt{EO1}, \, \texttt{EO3}^{\dagger}, \, \texttt{EO4},\, \texttt{EO5}, \, \texttt{EO6}^{\dagger}, \, \texttt{EO7},\, \texttt{EO8}^{\dagger}\, (n=8), \, \texttt{EO9} \, (n \geqs 12),\, \texttt{EO10}^{\dagger}, \, \texttt{EO12}^{\dagger},\, \texttt{EO18}^{\dagger},\, \texttt{EO19}^{\dagger} \\
& \texttt{EO24}^{\dagger} \\
& \\
{\rm IV} & \texttt{L2}, \, \texttt{L10}\, (n \leqs 5), \, \texttt{L12} \, (n \leqs 5),  \, \texttt{L14} \, (n \leqs 6) \\ [3pt]
& \texttt{EO2} \, (n=8), \, \texttt{EO3}^* \, (n=8), \, \texttt{EO6}^* \, (n=8), \, \texttt{EO8}^* \, (n=8),\,  \texttt{EO9} \, (n=8), \, \texttt{EO10}^*,\, \texttt{EO11}, \, \texttt{EO12}^* \\
& \texttt{EO13}^{\dagger} \, (n=12),\, \texttt{EO15} \, (n=8),\, \texttt{EO16} \, (n=8), \, \texttt{EO17}, \, \texttt{EO18}^*,\, \texttt{EO19}^*, \, \texttt{EO20}^{\dagger} \, (n=12) \\
& \texttt{EO23} \, (n=8),\, \texttt{EO24}^* \, (n=8), \, \texttt{EO25} \, (n=10), \,\texttt{EO26}, \, \texttt{EO27},\, \texttt{EO28},\,  \texttt{EO29} \\
& \\
{\rm V} & \texttt{L8}, \, \texttt{L9} \\ \hline
\end{array}
\]
\caption{Elimination techniques for the triples in $\mathcal{E}$}
\label{tab:methods}
\end{table}
}

\begin{prop}\label{p:linear_final}
The conclusion of Theorem \ref{t:main} holds when $X = {\rm L}_n(q)$ is a linear group.
\end{prop}

\begin{proof}
We need to eliminate the triples labelled \texttt{L1}-\texttt{L14} in Table \ref{tab:E1} and we will divide the proof into several cases according to the methods we use.

We will frequently work with a long root element $g \in X$, which we recall is a unipotent element with Jordan form $(J_2,J_1^{n-2})$ on the natural module for $X$. We have
\[
|g^X| = \frac{|{\rm SL}_n(q)|}{q^{2n-3}|{\rm GL}_{n-2}(q)|} = \frac{(q^{n-1}-1)(q^n-1)}{q-1}
\]
and there is a single conjugacy class of such elements in $X$.

\vs

\noindent \emph{Case 1. The triples {\rm \texttt{L1}} and {\rm \texttt{L10} ($n=7,8,9$)}}

\vs

To handle these cases we will apply Lemma \ref{l:large} with respect to a long root element $g \in X$. For example, in \texttt{L1} we have $X = {\rm L}_6(q)$, with $G_{\a} \in \C_1$ of type ${\rm GL}_2(q) \oplus {\rm GL}_4(q)$ and $G_{\ell} \in \C_8$ of type ${\rm Sp}_6(q)$, so
\[
v = \frac{|{\rm SL}_6(q)|}{|{\rm SL}_2(q)||{\rm GL}_4(q)|} = \frac{q^8(q^5-1)(q^6-1)}{(q-1)(q^2-1)},\;\; b = \frac{|{\rm SL}_6(q)|}{e|{\rm Sp}_6(q)|} = \frac{1}{e}q^6(q^3-1)(q^5-1)
\]
with $e = \gcd(3,q-1)$. By counting the number of unipotent elements in $X_{\a}$ and $X_{\ell}$ with Jordan form $(J_2,J_1^4)$, we obtain
\[
|g^X \cap X_{\a}| = q^2-1 + \frac{|{\rm SL}_4(q)|}{q^5|{\rm GL}_2(q)|} = (q^2-1)(1+ (q^2+1)(q^2+q+1))
\]
and $|g^X \cap X_{\ell}| = q^6-1$. And then by applying Lemma \ref{l:fpr} we deduce that
\[
|\cP_g| = q^8(1+(q^2+1)(q^2+q+1)), \;\; |\cL_g| = \frac{1}{e}q^6(q-1)(q^3-1).
\]
It is now routine to check that $|\cP_g|^3 > 4vb$ and $|\cL_g|^3 \geqs 16v^2/b$ and we conclude via Lemma \ref{l:large}.

The case \texttt{L10} with $n=7,8,9$ is very similar. However, this  approach is ineffective when $n=4,5$ and a different method is required (see Case 5 below).

\vs

\noindent \emph{Case 2. The triples {\rm \texttt{L3}}, {\rm \texttt{L4}}, {\rm \texttt{L6}}, {\rm \texttt{L12} ($n=7$)} and {\rm \texttt{L14} ($n=8$)}}

\vs

For these triples we apply Lemma \ref{l:der} with regard to a long root element. For example, take case \texttt{L3}. Here $X = {\rm L}_6(q)$, $G_{\a} \in \C_3$ is of type ${\rm GL}_2(q^3)$ and $G_{\ell} \in \C_5$ is of type ${\rm GL}_6(q_0)$, where $q = q_0^3$, so we have
\[
v = \frac{|{\rm GL}_6(q)|}{3|{\rm GL}_2(q^3)|} = \frac{1}{3}q^{12}(q-1)(q^2-1)(q^4-1)(q^5-1)
\]
and
\[
b = \frac{|{\rm SL}_6(q)|}{e|{\rm SL}_6(q_0)|} = \frac{1}{e}q^{10}\prod_{i=2}^6\frac{q^i-1}{q_0^i-1}
\]
with $e = \gcd(3,q_0-1)$. Let $g \in X$ be a long root element. Then $|g^X \cap X_{\ell}|$ coincides with the total number of unipotent elements in ${\rm L}_6(q_0)$ with Jordan form $(J_2,J_1^4)$, so by appealing to Lemma \ref{l:fpr} we deduce that
\[
|\cL_g| = b\left(\frac{(q_0^5-1)(q_0^6-1)(q-1)}{(q^5-1)(q^6-1)(q_0-1)}\right),
\]
which implies that $|\cL_g|^3 \geqs vb$. The result now follows from Lemma \ref{l:der} since $g^X \cap X_{\a}$ is empty (see \cite[Lemma 5.3.2(i)]{BG}, for example). The other cases are entirely similar and we omit the details.

\vs

\noindent \emph{Case 3. The triples {\rm \texttt{L5}}, {\rm \texttt{L7}}, {\rm \texttt{L11}} and {\rm \texttt{L13}}}

\vs

We can rule out each of these triples via Lemma \ref{l:sylow}. Consider \texttt{L7}, for example, where we have $X = {\rm L}_8(q)$, with $G_{\a} \in \C_5$ of type ${\rm GL}_8(q_0)$ and $G_{\ell} \in \C_8$ of type ${\rm Sp}_8(q)$ with $q = q_0^2$. Here
\[
|X_{\alpha}|_p  = q^{14} < |X_{\ell}|_p = q^{16},\;\; b = \frac{|{\rm SL}_8(q)|}{e|{\rm Sp}_8(q)|} = \frac{1}{e}q^{12}(q^3-1)(q^5-1)(q^7-1),
\]
where $e = \gcd(4,q-1)$, which means that $(b_p/2_p)^5 \geqs 2b^2$. Let $x \in X_{\a}$ be a unipotent element with Jordan form $(J_7,J_1)$ on $V$ and set $g = x^2$, so $g$ has Jordan form $(J_7,J_1)$ if $p$ is odd, and $(J_4,J_3,J_1)$ if $p=2$. In particular, the Jordan form of $g$ is incompatible with containment in $X_{\ell}$ (due to the occurrence of odd-sized Jordan blocks with odd multiplicity; see \cite[Corollary 3.6]{LS_book}) and thus $g^X \cap X_{\ell}$ is empty. We can now eliminate this triple by appealing to Lemma \ref{l:sylow} (also see Remark \ref{r:sylow} when $p=2$).

The case \texttt{L13} is analogous. In \texttt{L5}, we can take $g \in X_{\a}$ with Jordan form $(J_2^3,J_1^{n-6})$ on $V$ and then the result follows from Lemma \ref{l:sylow}(i), noting that $b_p^5 \geqs 2b^2$, $g^X \cap X_{\ell}$ is empty and $q$ is odd. Similarly, we can take $g = (J_2^2) \in X_{\a}$ in case \texttt{L11}, noting that $g^X \cap X_{\ell}$ is empty since there are no unipotent elements in ${\rm SO}_4^{-}(q)$ with Jordan form $(J_2^2)$ by \cite[Theorem 7.1]{LS_book}.

\vs

\noindent \emph{Case 4. The triples {\rm \texttt{L8}} and {\rm \texttt{L9}}}

\vs

First consider \texttt{L8}. Here $G_{\a} \in \C_5$ is a subfield subgroup of type ${\rm GL}_n(q_0)$ and $G_{\ell} \in \C_8$ is of type ${\rm O}_n(q)$, where $q = q_0^2$ and $n \in \{3,5,7\}$. Set $m = (n-1)/2$ and note that $q$ is odd. By inspecting \cite[Propositions 4.5.3, 4.8.4]{KL} we observe that
\[
v = \frac{|{\rm SL}_n(q)|}{e|{\rm SL}_n(q_0)|} = \frac{1}{e}q^{\frac{1}{4}n(n-1)}\prod_{i=2}^n(q_0^i+1)
\]
and
\[
b = \frac{|{\rm SL}_n(q)|}{d|{\rm SO}_n(q)|} = \frac{1}{d}q^{\frac{1}{4}(n^2-1)}\prod_{i=1}^{m}(q^{2i+1}-1),
\]
where $d = \gcd(n,q-1)$ and $e = \gcd(n,q_0+1)$. By the parity condition on the Jordan blocks of a unipotent element in ${\rm SO}_n(q)$, namely the property that every block of even size must have even multiplicity (see \cite[Corollary 3.6]{LS_book}), we see that $g^X \cap X_{\ell}$ is empty when $g \in X$ is a long root element.

Fix a subgroup $H \leqs {\rm L}_n(q_0) \leqs X_{\a}$ of order $q_0^{n-1}$ such that every nontrivial element in $H$ is a long root element. For example, we can take $H$ to be the image modulo scalars of the subgroup
\[
U = \left\{ I_n + \sum_{j=2}^n a_jE_{1,j} \,:\, a_j \in \mathbb{F}_{q_0} \right\} \leqs {\rm SL}_n(q_0),
\]
where $E_{1,j}$ is the $n \times n$ matrix with $1$ in the $(1,j)$ position and $0$ elsewhere. Since there are no long root elements in $X_{\ell}$,  it follows that $H$ acts semiregularly on the set of $t+1$ lines through $\a$, which in turn implies that $q_0^{n-1}$ divides $t+1$.

Suppose $p$ does not divide $s+1$. Then $q_0^{mn} = v_p = (st+1)_p$
and thus
\[
q_0^{\frac{1}{2}(n^2-1)} = b_p = (t+1)_pq_0^{mn},
\]
which means that $(t+1)_p = q_0^m$. But this is incompatible with the fact that $q_0^{n-1}$ divides $t+1$, so we have reached a contradiction and thus $p$ does divide $s+1$.

Since $p$ divides $s+1$ and $t+1$, it follows that $p$ does not divide $st+1$. Therefore,
\[
(s+1)_p = v_p = q_0^{mn},\;\; (t+1)_p = b_p = q_0^{\frac{1}{2}(n^2-1)}
\]
and thus
\[
v = (s+1)(st+1) \geqs q_0^{mn}((q_0^{mn}-1)(q_0^{(n^2-1)/2}-1)+1)  > q_0^{mn}\prod_{i=2}^n(q_0^i+1) \geqs v,
\]
which is absurd. This eliminates case \texttt{L8}, and an entirely similar argument rules out \texttt{L9}.

\vs

\noindent \emph{Case 5. The remaining triples}

\vs

To complete the proof of the proposition, it remains to eliminate the following cases:
\[
\mbox{\texttt{L2}, \texttt{L10} (with $n=4,5$), \texttt{L12} (with $n=3,5$), \texttt{L14} (with $n=4,6$).}
\]
Here the key tool is Lemma \ref{lem_cubic}, which we can apply in several different ways. As in the lemma, set
\[
h=\gcd(v,b),\;\; v'=v/h, \;\; b'=b/h, \;\; k = h/(st+1)
\]
and
\[
F(x)=v'b'x^3-(v'+b')x^2+2x-h \in \mathbb{Z}[x].
\]
Then $k$ is a positive integer such that $F(k)=0$ and
\[
\frac{h}{v'b'}<k^3<\frac{2h}{v'b'}.
\]

First consider \texttt{L2}. Setting $a = \frac{1}{2}q^3(q+1)$, we observe that
\[
v=\frac{|{\rm GL}_3(q)|}{6|{\rm GL}_1(q)|^3} = \frac{1}{3}a(q^2+q+1), \;\; b=\frac{|{\rm GL}_3(q)|}{3|{\rm GL}_1(q^3)|} = \frac{2}{3}a(q-1)^2.
\]
Suppose $q \not\equiv 1 \imod{3}$. Then $3$ divides $a$, but not $q^2+q+1$, so $h=\frac{1}{6}q^3(q+1)$ and we get $v'=q^2+q+1$ and $b'=2(q-1)^2$. But this implies that
\[
k^3<\frac{2h}{v'b'}=\frac{q^3(q+1)}{6(q^3-1)(q-1)}<1,
\]
which is absurd since $k$ is a positive integer. Similarly, if $q \equiv 1 \imod{3}$ then $(q^2+q+1)_3 = 3$ and $v/a$, $b/a$ are coprime integers. Therefore, $h=a$, $v'=\frac{1}{3}(q^2+q+1)$ and $b'=\frac{2}{3}(q-1)^2$, which implies that
\[
\frac{9}{4}<\frac{h}{v'b'} < k^3 < \frac{2h}{v'b'} < 8.
\]
Since no integer $k$ has this property, this completes the argument for the triple labelled \texttt{L2}.

Next we turn to \texttt{L10} and for brevity we will assume $n=4$. First assume $q$ is even, in which case
\[
h=q^3(q+1)(q^2+1), \;\; v'=q_0^3+1,\;\; b'=q_0^3-1.
\]
In the notation of Remark \ref{r:approx}, we have $z=q_0$, $r = 2$, $\lambda = 1$ and $K(z) = a_1 = q+\frac{1}{3}$, and we observe that $F(a_1) <0$ and $F(a_2)>0$, where $a_2 = q+\frac{2}{3}$. By Lemma \ref{lem_cubic}, the function $F(x)$ is increasing in the interval $[1,\infty)$ and we have $F(k)=0$, whence $q+\frac{1}{3}<k<q+\frac{2}{3}$ and we have reached a contradiction. Similarly, if $q$ is odd then $\lambda=2$ and $K(z)=2q+\frac{2}{3}$ in Remark \ref{r:approx}. Moreover, 
\[
k^3>\frac{h}{v'b'}=\frac{8q^3(q+1)(q^2+1)}{q^3-1}>8q^3
\]
and we check that $F(2q+1) >0$ for all $q = q_0^2 \geqs 9$. This implies that $2q<k<2q+1$, a contradiction.

Now consider \texttt{L12}, so $n \in \{3,5\}$ and $q = q_0^2$ is odd. We will give details for the case $n = 5$, noting that a similar argument applies when $n=3$.

By arguing as in Remark \ref{r:c1}, it is easy to check that there are no feasible parameters for $\cQ$ when $q_0\leqs 9$, so we may assume  $q_0\geqs 11$. Set $e = \gcd(5,q_0-1)$, $e'=\gcd(5,q_0+1)$ and observe that
\[
v=\frac{1}{4e'}aq(q_0^3+1)(q_0^5+1),\;\; b = \frac{1}{4}a(q+1)(q^2+1),
\]
where $a=\frac{4}{e}q^5(q_0^3-1)(q_0^5-1) \in \mathbb{N}$. Since $b/a$ is odd and $\gcd(v/a,b/a)=1$, we deduce that $h=a$ and thus
\[
v' = \frac{1}{4e'}q(q_0^3+1)(q_0^5+1),\;\; b'=\frac{1}{4}(q+1)(q^2+1).
\]
In addition, we note that $b_p<v_p$ and $b_p^5 \geqs 2b^2$, so $(t+1)_p=1$ by Lemma \ref{inequalities} and thus $k_p = 1$ since $t+1=kb'$.

Since $q_0 \geqs 11$, we get
\[
8^3 < \frac{32qe'}{e}<\frac{h}{v'b'}=\frac{64qe'}{e}\left(\frac{q^3(q_0^3-1)(q_0^5-1)}
  {(q+1)(q^2+1)(q_0^3+1)(q_0^5+1)}\right)<\frac{64qe'}{e} \leqs 320q
\]
and thus Lemma \ref{lem_cubic} implies that $8^3 < k^3< 640q$,
whence $k>8$. Finally, we recall that $F(k) = 0$. By reducing the equation $4e'F(k)=0$ modulo $q$, recalling that $k_p=1$, we obtain $e'(8-k)\equiv 0\imod{q}$, which means that $k\equiv 8\imod{q}$. Therefore, since $k>8$, it follows that $k\geqs q+8$, which is incompatible with the upper bound $k^3<640q$ for each $q=q_0^2\geqs11^2$. This final contradiction eliminates case \texttt{L12} with $n=5$.

Finally, let us turn to \texttt{L14}. Here $q = q_0^2$ is odd and we may assume $G_{\a} \in \C_8$ is of type ${\rm O}_n^{\e}(q)$, where $(n,\e) = (6,+)$ or $(4,-)$. We will give details for the case $(n,\e) = (6,+)$, noting that the argument is very similar to the previous case (\texttt{L12} with $n=5$).

First we observe that there are no feasible parameters for $\cQ$ when $q_0\leqs 23$, so we may assume that $q_0\geqs 25$. We compute
\[
h = \frac{4}{e}q_0^{15}(q_0^3-1)(q_0^5-1)(q^3+1), \;\;
  v'=\frac{1}{4e'}q_0^3(q_0^3+1)(q_0^5+1),\;\;
  b'=\frac{1}{4}(q+1)(q^2+1),
\]
where $e = \gcd(6,q_0-1)$ and $e'=\gcd(3,q_0+1)$. In particular, $b'$ is odd and by arguing as in the previous case we deduce that $k_p=1$.

By reducing the equation $16F(k)=0$ modulo $q_0^3$, we obtain the congruence condition
\[
-(1+q)k+8\equiv 0\imod{q_0^3},
\]
which means that $k=8(1-q)+mq_0^3$ for some integer $m$. Then by successively reducing the equation $16e'F(k)=0$ modulo $q_0^4$, $q_0^5$ and $q_0^6$, we deduce that
\[
  e'k\equiv 8(e'-e'q+q_0^3-2q_0^5)\imod{q_0^6}.
\]
Since $-q_0^6<8(e'-e'q+q_0^3-2q_0^5)<0$ for all $q_0\geqs 25$, it follows that
\[
e'k\geqs q_0^6+8e'(1-q)+8(q_0^3-2q_0^5).
\]
And since $1\leqs e'\leqs 3$ and $q_0^3-2q_0^5<0$, we get
\begin{equation}\label{kbd2}
k\geqs \frac{1}{3}q_0^6+8(1-q+q_0^3-q_0^5).
\end{equation}
Finally, we recall that Lemma \ref{lem_cubic} gives
\[
k^3<\frac{2h}{v'b'}=\frac{128e'}{e}q^{6}\left(
  \frac{(q_0^3-1)(q_0^5-1)(q^3+1)}{(q_0^3+1)(q_0^5+1)(q+1)(q^2+1)}\right)
  <384q^{6}
\]
and thus $k<\frac{73}{10}q^2$. Since $q = q_0^2 \geqs 25^2$, we can check that this is incompatible with the lower bound on $k$ in \eqref{kbd2}, and this final contradiction completes the argument.
\end{proof}

\begin{prop}\label{p:unitary_final}
The conclusion of Theorem \ref{t:main} holds when $X = {\rm U}_n(q)$ is a unitary group.
\end{prop}

\begin{proof}
We need to eliminate the triples \texttt{U1}-\texttt{U4} in Table \ref{tab:E1}. The cases \texttt{U1}, \texttt{U3} and \texttt{U4} can be handled by repeating the arguments for \texttt{L1}, \texttt{L5} and \texttt{L11}, respectively, in the proof of the previous proposition. We omit the details.

Finally, let us consider case \texttt{U2}. Here $X = {\rm U}_8(q)$ and $q$ is odd, with $G_{\a}\in \C_1$ of type ${\rm GU}_3(q) \perp {\rm GU}_5(q)$ and $G_{\ell} \in \C_5$ of type ${\rm O}_8^{-}(q)$. It is easy to check that
\[
|X_{\ell}|_p  = q^{12} < |X_{\a}|_p = q^{13}
\]
and $v_p^5 \geqs 2v^{2}$ for all $q \geqs 3$. Let $g \in X_{\ell}$ be a unipotent element with Jordan form $(J_7,J_1)$ on the natural module. Since $g^X \cap X_{\a}$ is empty, the result follows from Lemma \ref{l:sylow}.
\end{proof}

\begin{prop}\label{p:symp_final}
The conclusion of Theorem \ref{t:main} holds when $X = {\rm PSp}_n(q)$ is a symplectic group.
\end{prop}

\begin{proof}
We will eliminate the triples \texttt{S1}-\texttt{S7}. Set $d = \gcd(2,q-1)$ and note that $|g^X| = \frac{1}{d}(q^n-1)$ if $g \in X$ is a long root element.

We begin by handling \texttt{S1}. Here $n=6$ and we compute
\[
v = \frac{|{\rm Sp}_6(q)|}{6|{\rm Sp}_2(q)|^3} = \frac{1}{6}q^6(q^2+1)(q^4+q^2+1),\;\; b = \frac{|{\rm Sp}_6(q)|}{3|{\rm Sp}_2(q^3)|} = \frac{1}{3}q^6(q^2-1)(q^4-1).
\]
Let $g \in X$ be a long root element, so $g$ has Jordan form $(J_2,J_1^4)$ on the natural module $V$. By considering the embedding of $X_{\a}$ in $X$ as a $\C_2$-subgroup, we see that it contains long root elements of the form $(y_1,y_2,y_3) \in {\rm Sp}_2(q)^3$, where $y_i = (J_2)$ and $y_j = 1$ for $j \ne i$ and $i \in \{1,2,3\}$. Therefore,
$|g^X \cap X_{\a}| \geqs \frac{3}{d}(q^2-1)$ and by applying Lemma \ref{l:fpr} we deduce that
\[
|\cP_g| \geqs \frac{1}{2}q^6(q^2+1) \geqs (vb)^{1/3}.
\]
The result now follows from Lemma \ref{l:der} since $g^X \cap X_{\ell}$ is empty.

The cases \texttt{S5}, \texttt{S6} and \texttt{S7} are very similar, working with Lemma \ref{l:der} and a long root element $g \in X$. We omit the details.

Next let us turn to \texttt{S3}. Here $n \in \{ 6,10,16,18\}$ and $q$ is odd, with $G_{\a} \in \C_2$ of type ${\rm GL}_{n/2}(q).2$ and $G_{\ell} \in \C_3$ of type ${\rm GU}_{n/2}(q)$. We note that $|X_{\a}|_p = |X_{\ell}|_p$ and $b_p^5 \geqs 2b^{2}$.

By \cite[Theorem 7.1(ii)]{LS_book} (also see \cite[Proposition 3.4.10]{BG}), there are two conjugacy classes of unipotent elements in $X$ with Jordan form $(J_2^2,J_1^{n-4})$ on $V$. Fix class representatives $x_+$ and $x_{-}$, where
\[
|x_{\e}^X| = \frac{|{\rm Sp}_n(q)|}{q^{2n-5}|{\rm O}_{2}^{\e}(q)||{\rm Sp}_{n-4}(q)|}.
\]
Both $X_{\a}$ and $X_{\ell}$ contain unipotent elements with Jordan form $(J_2^2,J_1^{n-4})$. If $g \in X_{\a}$ is such an element, then by considering the centralizer of $g$ in $X_{\a}$ we deduce that $g$ is $X$-conjugate to $x_{+}$. On the other hand, each $g \in X_{\ell}$ with Jordan form $(J_2^2,J_1^{n-4})$ is conjugate to $x_{-}$. In other words, if $g \in X_{\a}$ is a unipotent element with Jordan form $(J_2^2,J_1^{n-4})$ then $g^X \cap X_{\ell}$ is empty and so we can rule out \texttt{S3} by applying Lemma \ref{l:sylow}.

Finally, let us consider cases \texttt{S2} and \texttt{S4}. In both cases, $q$ is odd, $|X_{\a}|_p < |X_{\ell}|_p$ and $b_p^5 \geqs 2b^{2}$. Let $g \in X_{\a}$ be a unipotent element with Jordan form $(J_3^2,J_1^{n-6})$ on $V$. Then $g^X \cap X_{\ell}$ is empty (see \cite[Lemma 5.3.2(i)]{BG}) and once again we conclude by applying Lemma \ref{l:sylow}.
\end{proof}

\begin{prop}\label{p:odd_orth_final}
The conclusion of Theorem \ref{t:main} holds when $X = \O_n(q)$ is an odd dimensional orthogonal group.
\end{prop}

\begin{proof}
It is sufficient to eliminate the triples labelled \texttt{O1}, \texttt{O2} and \texttt{O3} in Table \ref{tab:E1}. Here $n \in \{7,9\}$ and $q = q_0^2$ is odd. Set $m = (n-1)/2$ and let $g \in X$ be an involution of type $(-I_{n-1},I_1)$, where the $(-1)$-eigenspace of $g$ is a nondegenerate orthogonal space of plus-type. In the notation of \cite[Table B.8]{BG}, $g$ is an involution of type $t_m$ and we note that
\[
|g^X| = \frac{|{\rm SO}_n(q)|}{2|{\rm SO}_{n-1}^{+}(q)|} = \frac{1}{2}q^m(q^m+1).
\]
Moreover, every involution in $X$ of the form $(-I_{n-1},I_1)$ is conjugate to $g$ and thus
\[
|g^X \cap X_{\ell}| = \frac{|{\rm SO}_n(q_0)|}{2|{\rm SO}_{n-1}^{+}(q_0)|} + \frac{|{\rm SO}_n(q_0)|}{2|{\rm SO}_{n-1}^{-}(q_0)|} = q^m
\]
is the total number of such involutions in $X_{\ell} = {\rm SO}_n(q_0)$.

First consider case \texttt{O1}. Here
\[
v = \frac{1}{2}q^6(q^6-1),\;\;
b = \frac{1}{2}q^{9/2}(q+1)(q^2+1)(q^3+1)
\]
and we compute
\[
|g^X \cap X_{\a}| = \frac{|{\rm SO}_3(q)|}{2|{\rm SO}_{2}^{-}(q)|} + \frac{|{\rm SO}_4^{-}(q)|}{2|{\rm SO}_{3}(q)|} = \frac{1}{2}q^2(q+1).
\]
This implies that
\[
|\cP_g| = \frac{1}{2}q^5(q+1)(q^3-1),\;\; |\cL_g| = q^{9/2}(q+1)(q^2+1)
\]
and it is now routine to check that $|\cP_g|^{3} > 4vb$ and $|\cL_g|^3 \geqs 16v^2/b$, so we can conclude by applying Lemma \ref{l:large}.

The cases \texttt{O2} and \texttt{O3} are entirely similar and we omit the details.
\end{proof}

\begin{prop}\label{p:even_orth_final}
The conclusion of Theorem \ref{t:main} holds when $X = {\rm P\O}_n^{\e}(q)$ is an even-dimensional orthogonal group.
\end{prop}

\begin{proof}
Here we need to eliminate \texttt{EO1}-\texttt{EO29} in Table \ref{tab:E2} and we will divide the proof into several cases.

As above, we will often work with long root elements $g \in X$. Recall that long root elements have Jordan form $(J_2^2,J_1^{n-4})$ on the natural module $V$ and they form a single conjugacy class in $X$. Moreover, if $q$ is odd, then any two unipotent elements with Jordan form $(J_2^2,J_1^{n-4})$ are conjugate, whereas there are two such classes when $q$ is even, and the long root elements correspond to the involutions of type $a_2$ in terms of the notation of Aschbacher and Seitz \cite{AS}. For all $q$, we have
\[
|g^X| = \frac{|{\rm O}_{n}^{\e}(q)|}{q^{2n-7}|{\rm Sp}_2(q)||{\rm O}_{n-4}^{\e}(q)|}<2q^{2n-6}.
\]

In several cases, our main tool will be Lemma \ref{lem_cubic} and so throughout the proof we will freely adopt the notation $h,v',b',k,F(x)$ defined in the statement of this lemma.

\vs

\noindent \emph{Case 1. The triples {\rm \texttt{EO1}}, {\rm \texttt{EO4}}, {\rm \texttt{EO5}}, {\rm \texttt{EO7}}, {\rm \texttt{EO9}} and {\rm \texttt{EO10}}}

\vs

First consider case \texttt{EO1}. Here $q$ is odd, $|X_{\ell}|_p = q^{53} < |X_{\a}|_p = q^{58}$ and
\[
v = \frac{1}{2}q^{32}(q^4+1)(q^{10}-1)(q^{18}-1).
\]
In particular, $v_p^5 \geqs 2v^2$ for all $q \geqs 3$. If we take a unipotent element $g \in X_{\ell}$ with Jordan form $(J_{15},J_5)$ on $V$, then $g$ does not stabilize a nondegenerate $4$-space, so $g^X \cap X_{\a}$ is empty and we conclude via Lemma \ref{l:sylow}.

Next suppose we are in case \texttt{EO4}. Here $q$ is odd and it is easy to check that $v_p^5 \geqs 2v^2$ and $b_p^5 \geqs 2b^2$ for all $q$. If $(n,i) = (14,5)$, then $|X_{\a}|_p < |X_{\ell}|_p$ and $g^X \cap X_{\ell}$ is empty when $g \in X_{\a}$ is unipotent with Jordan form $(J_{3},J_1^{11})$ due to the occurrence of Jordan blocks with odd multiplicity, which is incompatible with containment in $X_{\ell}$. Otherwise, $|X_{\ell}|_p < |X_{\alpha}|_p$ and we note that $g^X \cap X_{\a}$ is empty when $g \in X_{\ell}$ has Jordan form $(J_{n/2}^2)$ on $V$. So all cases in \texttt{EO4} can be eliminated by applying Lemma \ref{l:sylow}. Case \texttt{EO7} is entirely similar and we omit the details.

The triple \texttt{EO5} is straightforward. Here $q$ is odd, $|X_{\ell}|_p < |X_{\a}|_p$ and $v_p^5 \geqs 2v^2$. If $g \in X_{\ell}$ has Jordan form $(J_7^2,J_1^2)$ on $V$, then $g^X \cap X_{\a}$ is empty and Lemma \ref{l:sylow} applies.

Now suppose we are in case {\rm \texttt{EO9}}, so $q =q_0^2$ is odd.
If $n \in \{12,16\}$ then $|X_{\ell}|_p < |X_{\a}|_p$ and $g^X \cap X_{\a}$ is empty when $g \in X_{\ell}$ is a unipotent element with Jordan form $(J_{n-1},J_1)$ on $V$. Since $v_p^5 \geqs 2v^2$, the desired result follows from Lemma \ref{l:sylow}. However, for $n=8$ we have $|X_{\a}|_p < |X_{\ell}|_p$ and so a different approach is required.

Assume $n=8$. Here $G_\alpha$ and $G_\ell$ are of type ${\rm O}_3(q)\perp{\rm O}_5(q)$ and ${\rm O}_8^-(q_0)$, respectively, and we note that
\[
v=\frac{1}{2}q^{7}(q^2+1)(q^{6}-1), \;\;
b=\frac{1}{2}q^{6}(q+1)(q^3+1)(q^4-1).
\]
With respect to the notation of Lemma \ref{lem_cubic}, we have
\[
h=\frac{1}{2}q^{6}(q-1)(q^2+1)(q^3+1), \;\; v'=q(q^2+q+1),\;\; b'=(q+1)^2
\]
and so the bound in part (ii) of the lemma implies that
 \[
  k^3<\frac{2h}{v'b'}=\left(\frac{q^2+1}{q^2+q+1}\right)\left(\frac{q-1}{q+1}\right)q^5(q^2-q+1)<q^{7}
\]
and thus $k<q^3$.

We have $b_p<v_p$ and $(b_p/2_p)^5 \geqs 2b^2$ for all $q$, so   Lemma \ref{inequalities}(ii) yields $(t+1)_p=1$. And since $t+1=kb'$, we deduce that $k_p=1$. By reducing the equation $2F(k)=0$ modulo $q$ we obtain $k\equiv 2\imod{q}$. And by successively reducing modulo $q^2$, $q^3$ and $q^4$, we deduce that
\[
k\equiv 4q^3+6q^2-2q+2\imod{q^4}.
\]
But this implies that $k\geqs 4q^3+6q^2-2q+2$, which is incompatible with the upper bound $k<q^3$ for all $q\geqs 9$. This means that we have reached a contradiction and the triple \texttt{EO9} is eliminated.

Finally, let us consider \texttt{EO10}. This is similar to the previous case with $n=8$. If $q$ is odd then $|X_{\ell}|_p = |X_{\a}|_p = q^6$, $v_p^5 \geqs 2v^2$ and $g^X \cap X_{\a}$ is empty when $g \in X_{\ell}$ is a unipotent element with Jordan form $(J_7,J_1)$ on $V$, so we can apply Lemma \ref{l:sylow}. Now assume $p=2$, so
\[
v=\frac12q^{6}(q-1)(q^2+1)(q^3-1), \;\; b=q^{6}(q+1)(q^3+1)(q^4-1)
\]
and
\[
h=\frac{1}{2}q^{6}(q-1)(q^2+1), \;\; v'=q^3-1, \;\; b'=2(q+1)^2(q^3+1).
\]
By Lemma \ref{lem_cubic} we have
\[
k^3<\frac{2h}{v'b'}=\frac{q^{6}(q-1)(q^2+1)}{2(q+1)^2(q^{6}-1)}<\frac{1}{2}q
\]
and by reducing the equation $F(k)=0$ modulo $2$, we deduce that $k=2m$ is even. Since $q$ divides $h/4$, whereas $v'\equiv-1 \imod{q}$ and $b'\equiv2\imod{q}$, it follows that
\[
0=\frac{1}{4}F(k)\equiv -m(4m^2+m-1)\imod{q}.
\]
However, this implies that
\[
q \leqs 4m^3+m^2-m<8m^3=k^3<\frac{1}{2}q,
\]
which is absurd. This eliminates \texttt{EO10}.

\vs

\noindent \emph{Case 2. The triples {\rm \texttt{EO13}}, {\rm \texttt{EO14}}, {\rm \texttt{EO20}}, {\rm \texttt{EO21}}, {\rm \texttt{EO22}}, {\rm \texttt{EO23}}  and {\rm \texttt{EO25}}}

\vs

First consider \texttt{EO14}. Here $n = 18$, $q$ is odd and
\[
v = q^{36}\prod_{i=1}^{8}(q^i+1),\;\;
b = \frac{1}{c}q^{40}(q^2-1)(q^6-1)(q^9-1)(q^{10}-1)(q^{14}-1),
\]
where $c = \gcd(4,q-1)$. Let $g \in X$ be a long root element. Then $|g^X|<2q^{30}$ and by considering the embedding of $X_{\a}$ in $X$ we see that $|g^X \cap X_{\a}|$ coincides with the total number of unipotent elements in ${\rm SL}_9(q)$ with Jordan form $(J_2,J_1^7)$ on the $9$-dimensional natural module. Therefore,
\[
|g^X \cap X_{\a}| = \frac{(q^8-1)(q^9-1)}{q-1} > \frac{1}{2}q^{16},
\]
and by applying Lemma \ref{l:fpr} we deduce that $|\cP_g|>\frac{1}{4}vq^{-14}$. Since $g^X \cap X_{\ell}$ is empty and $|\cP_g|^3 \geqs vb$, the desired result follows from Lemma \ref{l:der}. The case \texttt{EO21} is entirely similar.

Next consider \texttt{EO22}. Here we have
\[
v = \frac{1}{4}q^{18}(q^2-1)(q^6-1)(q^{10}-1),\;\; b = q^{20}(q^4+q^2+1)(q^8-1)(q^{10}-1).
\]
Let $g = g_1 \otimes g_2 \in X_{\ell}$, where $g_1 \in {\rm Sp}_2(q)$ is trivial and $g_2 = (J_2,J_1^4) \in {\rm Sp}_6(q)$. Then $g$ is a long root element in $X$ and thus $|g^X|<2q^{18}$. Moreover, by considering the embedding of $X_{\ell}$ in $X$, we deduce that
$|g^X \cap X_{\ell}|$ coincides with the total number of elements in ${\rm Sp}_6(q)$ with Jordan form $(J_2,J_1^4)$ on the $6$-dimensional natural module. Therefore, $|g^X \cap X_{\ell}| = q^6-1>\frac{1}{2}q^6$ and we get $|\cL_g| > \frac{1}{4}bq^{-12}$ and $|\cL_g|^3 \geqs vb$. To complete the argument via Lemma \ref{l:der}, it just remains to show that $g^X \cap X_{\a}$ is empty. This is clear when $q$ is odd since there are no elements in $X_{\a}$ with Jordan form $(J_2^2,J_1^8)$ on $V$. On the other hand, if $q$ is even then the only such elements are of type $c_2$, rather than $a_2$, since they correspond to $b_1$-type involutions in ${\rm O}_{6}^{+}(q^2)$. The result follows.

Now let us turn to case \texttt{EO13}, so $n \in \{12,16\}$ and we observe that
\[
v = \frac{1}{2}q^{\frac{1}{8}n(n-2)}\prod_{i=1}^{n/2-1}(q^i+1),\;\; b = \frac{1}{4}q^{\frac{1}{8}n^2}\prod_{i=1}^{n/4}(q^{4i-2}-1).
\]

First assume $q$ is even. Let $g \in X_{\ell}$ be the image of a $b_1$-type involution in ${\rm O}_{n/2}^{\delta}(q^2) \leqs X_{\ell}$. Then $|g^X \cap X_{\ell}|>\frac{1}{2}q^{n-2}$ and we note that
$|g^X| < 2q^{2n-4}$ since $g$ embeds in $X$ as an involution of type $c_2$. These estimates imply that $|\cL_g| >\frac{1}{4}bq^{2-n}$ and thus $|\cL_g|^3 \geqs vb$. Since $g^X \cap X_{\a}$ is empty, we can now conclude via Lemma \ref{l:der}.

Now suppose $q$ is odd. If $n = 16$ then we can repeat the above argument for \texttt{EO14}, working with Lemma \ref{l:der} and a long root element $g \in X$. However, this approach is ineffective when $n=12$ and so we switch to Lemma \ref{lem_cubic}. It is easy to check that there are no feasible parameters for $\cQ$ when $q=3$, so we may assume $q\geqs 5$.

Setting $a=\frac{1}{4}q^{15}(q+1)(q^3+1)(q^5+1)$, we have
\[
v=2a(q^2+1)(q^4+1),\;\; b=aq^3(q-1)(q^3-1)(q^5-1)
\]
and
\[
(2(q^2+1)(q^4+1))_2 = \gcd\bigl(2(q^2+1)(q^4+1),q^3(q-1)(q^3-1)(q^5-1)\bigr)=8,
\]
so we get
\[
h=8a, \;\; v'=\frac{1}{4}(q^2+1)(q^4+1),\;\; b'=\frac{1}{8}q^3(q-1)(q^3-1)(q^5-1).
\]
Then Lemma \ref{lem_cubic} gives
\[
k^3<\frac{2h}{v'b'}=128q^6\left(\frac{q^{6}(q+1)(q^3+1)(q^5+1)}{(q^2+1)(q^4+1)(q-1)(q^3-1)(q^5-1)}\right)<324q^6
\]
and thus $k<7q^2$.

Next observe that $v_p<b_p$ and $(v_p)^5 \geqs 2v^{2}$ for each odd $q$, so Lemma \ref{inequalities}(i) implies that $(s+1)_p=1$
and thus $k_p=1$ since $s+1=kv'$. By reducing the equation $8F(k)=0$ modulo $q^3$ we obtain $k(1+q^2)\equiv 8\imod{q^3}$, which means that $k=8-8q^2+mq^3$ for some positive integer $m$. Since $k<7q^2$ we see that $m\in\{1,2\}$. And since $32F(k)\equiv 0\imod{q^4}$ we get $-8(4+m)q^3\equiv 0\imod{q^4}$ and thus $q$ divides $4+m$. Since $q$ is an odd prime power, this holds only if $m=1$ and $q=5$, but this yields $k=-67$, which is absurd.

This completes the argument for \texttt{EO13} and an entirely similar argument handles \texttt{EO20}.

Next consider case \texttt{EO23}. Here $n \in \{8,12,16\}$ and we have
\[
v = \frac{1}{4}q^{\frac{1}{8}n^2}\prod_{i=1}^{n/4}(q^{4i-2}-1),\;\; b = \frac{1}{d^{(3+\eta)/2}}q^{\frac{1}{8}n(n-2)}(q^{n/4}+\eta)\prod_{i=1}^{(n-2)/2}(q^{i}+1),
\]
where $d = \gcd(2,q-1)$. Let $g \in X$ be a long root element. Then $|g^X|<2q^{2n-6}$ and $|g^X \cap X_{\ell}|$ coincides with the number of long root elements in ${\rm SO}_{n}^{\eta}(q_0)$. Therefore, $|g^X \cap X_{\ell}| > \frac{1}{2}q^{n-3}$ and we deduce that $|\cL_g| > \frac{1}{4}bq^{3-n}$. For $n \geqs 12$, it is straightforward to check that $|\cL_g|^3 \geqs vb$ and so the result follows from Lemma \ref{l:der} since $g^X \cap X_{\a}$ is empty.

To complete the proof for \texttt{EO23}, we may assume $n=8$ and $\eta = \pm$. It is easy to check that there are no feasible parameters for $\cQ$ when $q \in \{4,9,16\}$, so we may assume $q \geqs 25$.

Note that  $b_p=q^{6} \leqs v_p=q^8/4_p$ and $(b_p/2_p)^5>2b^2$, so  $(t+1)_p\leqs 2_p$ by Lemma \ref{inequalities}. Also note that if $p=2$, then $(t+1)_p=2_p$ since $t$ is odd by the fact that $b$ is even.

We will assume $\eta = +$ since a similar argument applies when $\eta = -$. Then
\begin{equation}\label{e:vdash}
h=q^{6}(q+1)^2(q^2-q+1), \;\; v'=\frac{1}{4}q^2(q-1)^2(q^2+q+1),\;\; b'=\frac{1}{d^2}(q^2+1)^2
\end{equation}
and we note that $k_p=2$ when $p=2$ since $b'$ is odd and $t+1=kb'$. By Lemma \ref{lem_cubic} we have
\[
4d^2<\frac{h}{v'b'} < k^3<\frac{2h}{v'b'}<8d^2\left(\frac{q+1}{q-1}\right)^2<16d^2,
\]
so $k=2$ if $p=2$, and $k=3$ otherwise. For $p=2$, if we reduce the equation $F(k)=0$ modulo $2q^2$ then $-q^2\equiv 0\imod{2q^2}$, which is absurd. Similarly, if $p$ is odd, then the equation $4F(k)=0$ yields $15\equiv 0\imod{q}$, which is equally absurd.

This completes the analysis of case \texttt{EO23}, and a very similar argument goes through for \texttt{EO25}. Indeed, in the latter case we appeal to Lemma \ref{l:der} when $n \in \{14,18\}$, with respect to a long root element $g \in X$, and for $n=10$ we can reach a contradiction via Lemma \ref{lem_cubic}. We omit the details.

\vs

\noindent \emph{Case 3. The triples {\rm \texttt{EO2}}, {\rm \texttt{EO8}}, {\rm \texttt{EO15}} and {\rm \texttt{EO16}}}

\vs

First consider \texttt{EO2}. Using \cite[Proposition 4.1.6]{KL} we  compute
\[
v = \frac{q^{n-2}(q^{(n-2)/2}-1)(q^{n/2}-1)}{2(q+1)},\;\; b = \frac{q^{n-2}(q^{(n-2)/2}+1)(q^{n/2}-1)}{2(q-1)}
\]
and we take $g \in X$ to be a long root element. Then $|g^X| <2q^{2n-6}$ and we observe that $|g^X \cap H| > \frac{1}{2}q^{2n-10}$ for $H = X_{\a}$ or $X_{\ell}$. For example, $|g^X \cap X_{\a}|$ coincides with the number of long root elements in $\O_{n-2}^{-}(q)$, which explains the given lower bound. From these estimates, it follows that
\[
|\cP_g|^3 > \frac{1}{64}v^3q^{-12} \geqs 4vb,\;\; |\cL_g|^3 > \frac{1}{64}b^3q^{-12} \geqs 16v^2/b
\]
for all $n \geqs 12$ and all $q \geqs 2$, with the exception of the case $(n,q) = (12,2)$. For $(n,q) = (12,2)$, it is easy to check that there are no feasible parameters for $\cQ$. So in view of Lemma \ref{l:large}, this reduces the problem to the case $n=8$.

For $n=8$ we first check that there are no feasible parameters for $\cQ$ when $q=2$, so we may assume $q \geqs 3$. We then compute
\[
h=\frac{1}{2}d^2q^6(q^2+1), \;\; v'=\frac{1}{d^2}(q-1)(q^3-1), \;\; b'=\frac{1}{d^2}(q+1)(q^3+1),
\]
where $d = \gcd(2,q-1)$, and we observe that
\[
\frac{1}{2}d^6<\frac{h}{v'b'}=\frac{q^6(q^2+1)}{(q^2-1)(q^6-1)} \cdot \frac{1}{2}d^6 <\frac{13}{10} \cdot \frac{1}{2}d^6
\]
for all $q\geqs 3$. Therefore, Lemma \ref{lem_cubic} implies that
$\frac{1}{2}d^6<k^3<\frac{13}{10}d^6$. For $q$ even, this forces $k=1$ but we check that $F(1)>0$ for all $q \geqs 4$, which is a contradiction. Similarly, for $q$ odd we deduce that $k=4$ is the only possibility and once again we reach a contradiction since $F(4)>0$ for all $q\geqs 3$.

Next consider \texttt{EO8} and let $g \in X$ be a long root element. Then $|g^X|<2q^{2n-6}$ and we compute
\[
|g^X \cap X_{\a}| > \frac{1}{2}q^{2(n-i)-6},\;\; |g^X \cap X_{\ell}| > \frac{1}{2}q^{n-3},
\]
so by applying Lemma \ref{l:fpr} we get $|\cP_g|> \frac{1}{4}vq^{-2i}$ and $|\cL_g| > \frac{1}{4}bq^{3-n}$. For $n \geqs 12$, these estimates imply that $|\cP_g|^3 > 4vb$ and $|\cL_g|^3 \geqs 16v^2/b$, in which case the result follows from Lemma \ref{l:large}. Finally, for $n = 8$ we have $i = 2$ and $\eta = -$, and we can repeat the previous argument for the triple \texttt{EO10} in Case 1.

A very similar argument handles the case \texttt{EO15}. For $n \geqs 12$ we reach a contradiction via Lemma \ref{l:large}, working with a long root element $g \in X$. And for $n=8$ we can repeat the above argument for \texttt{EO2}, simply interchanging $v$ and $b$.

Finally, let us consider \texttt{EO16}. Here $q = q_0^2$ and we can use Lemma \ref{l:large} to handle the case $n=12$, working once again with a long root element. Now assume $n=8$, so $(\delta,\eta) = (\pm, -)$ or $(-,+)$. We can use Lemma \ref{lem_cubic} to eliminate each of these cases. In fact, we observe that the two cases with $\delta = -$ can be handled by repeating the above arithmetic argument for \texttt{EO23} with $n=8$. For example, if $(\delta,\eta) = (-,+)$ then $h,v',b'$ are exactly as in \eqref{e:vdash} and so the same argument applies.

So to complete the argument for \texttt{EO16}, we may assume $n=8$ and $(\delta,\eta) = (+,-)$. For $q \in \{4,9,16\}$ it is easy to check that there are no feasible parameters for $\cQ$, so we may assume $q \geqs 25$. Since $b_p<v_p$ and $(b_p/2_p)^5>2b^2$, Lemma \ref{inequalities} implies that $(t+1)_p\leqs 2_p$. Let $e=\gcd(3,q-1)$ and note that $(q^2+q+1)_3=e$ and
\[
h=\frac{e}{d}q^{6}(q^2+1)(q^2-q+1), \;\; v'=\frac{d}{4e}q^2(q^2+1)(q^2+q+1),\;\; b'=\frac{1}{e}(q-1)(q+1)^3.
\]
In particular, if $p=2$ then $t$ and $b'$ are odd and thus $k_p=2$ since $t+1=kb'$. The bound in Lemma \ref{lem_cubic}(ii) gives
\[
k^3<\frac{2h}{v'b'}=\frac{8e^3}{d^2}\left(\frac{q^4(q^2-q+1)}{(q^3-1)(q+1)^3}\right)<\frac{8e^3}{d^2}
\]
and thus $k<2e$. By reducing the equation $eF(k)=0$ modulo $2_p^2q$ we get $k(k+2e)\equiv 0\imod{2_p^2q}$ and thus $k+2e\equiv0\imod{q}$ if $q$ is odd. Similarly, if $p=2$ and we write $k'=k/2$, which is an odd integer, then $k'+e\equiv 0\imod{q}$ and once again we conclude that $k+2e\equiv 0\imod{q}$. So we have $k\geqs q-2e$ for all $q$, which is incompatible with the bound $k<2e$ for all $q\geqs 25$.

\vs

\noindent \emph{Case 4. The triples {\rm \texttt{EO3}}, {\rm \texttt{EO6}} and {\rm \texttt{EO24}}}

\vs

First consider \texttt{EO3} and let us begin by assuming $q$ is odd. From \cite[Propositions 4.1.6, 4.2.7]{KL} we have
\[
v = \frac{|{\rm O}_n^{+}(q)|}{|{\rm O}_i^{\delta}(q)||{\rm O}_{n-i}^{\delta}(q)|},\;\; b = \frac{|{\rm SO}_n^{+}(q)|}{c|{\rm GL}_{n/2}(q)|}, \;\; |X_{\a}|_p = q^{\frac{1}{4}(n^2+2i^2-2ni-2n)},\;\; |X_{\ell}|_p = q^{\frac{1}{8}n(n-2)},
\]
where $c = 2$ if $n \equiv 0 \imod{4}$, otherwise $c = 1$. It is straightforward to check that $v_p^5 \geqs 2v^2$ and $b_p^5 \geqs 2b^2$ for all $q \geqs 3$.

Suppose $(n,i) = (8,2)$, $(12,4)$, $(14,6)$ or $(16,6)$, in which case $|X_{\a}|_p \leqs |X_{\ell}|_p$. If $g \in X_{\a}$ is a unipotent element with Jordan form $(J_{3},J_1^{n-3})$, then $g^X \cap X_{\ell}$ is empty and we can apply Lemma \ref{l:sylow}. Similarly, if $(n,i) = (16,4)$ or $(20,6)$ then $|X_{\ell}|_p < |X_{\a}|_p$ and we can take $g \in X_{\ell}$ with Jordan form $(J_{n/2}^2)$ on $V$. Then $g^X \cap X_{\a}$ is empty and once again the result follows via Lemma \ref{l:sylow}.

Now assume $q$ is even and let $g \in X$ be a long root element. Then $|g^X|< 2q^{2n-6}$ and we compute the following lower bounds
\[
|g^X \cap X_{\a}|>\frac{1}{2}q^{2(n-i-3)},\;\; |g^X \cap X_{\ell}| = \frac{(q^{n/2-1}-1)(q^{n/2}-1)}{q-1} > \frac{1}{2}q^{n-2}.
\]
These estimates imply that $|\cP_g|>\frac{1}{4}vq^{-2i}$ and $|\cL_g|>\frac{1}{4}bq^{4-n}$, which allow us to deduce that $|\cP_g|^3>4vb$ and $|\cL_g|^3 \geqs 16v^2/b$ for all $n \geqs 12$, with the single exception of the case $(n,q) = (12,2)$. For the latter we have $i=4$ and it is easy to check that there are no feasible parameters for $\cQ$.

So by appealing to Lemma \ref{l:large}, the argument for \texttt{EO3} is reduced to the case where $n=8$ and $p=2$. If $\delta = -$ then $v$ and $b$ are as in \texttt{EO2}, which we handled above in Case 3, so we may assume $\delta = +$ and thus
\[
v = b = \frac{1}{2}q^6(q+1)(q^2+1)(q^3+1).
\]
Since there are no feasible parameters for $\cQ$ when $q=2$, we may assume $q\geqs  4$. Now Lemma \ref{inequalities}(i) implies that $(s+1)_2 \leqs 2$ since $(v_2/2)^5 \geqs 2v^2$. Moreover, $v$ is even and thus $s$ is odd, so $(s+1)_2=2$ and $(s^2+1)_2=v_2/2 = \frac{1}{4}q^6$. But $s^2+1\equiv 2\imod{4}$ since $s$ is odd, so we have reached a contradiction.

The case \texttt{EO6} is very similar to \texttt{EO3} and we omit the details.

Finally, let us consider \texttt{EO24}. If $n=8$ then we can repeat our previous argument for \texttt{EO10} (see Case 1), so we may assume
$n \geqs 12$. If $q$ is odd, then $|X_{\a}|_p = |X_{\ell}|_p$, $b_p^5 \geqs 2b^2$ and we can take a unipotent element $g \in X_{\ell}$ with Jordan form $(J_{3},J_1^{n-3})$ on $V$. Since $g^X \cap X_{\a}$ is empty, the result follows from Lemma \ref{l:sylow}. Now assume $q$ is even and let $g \in X$ be a $c_2$-type involution, so $|g^X|< 2q^{2n-4}$, $|g^X \cap X_{\ell}|  > \frac{1}{2}q^{n-2}$ and we deduce that $|\cL_g|>\frac{1}{4}bq^{2-n}$. If $q \geqs 8$ then $|\cL_g|^3 \geqs vb$ and we can apply Lemma \ref{l:der} since $g^X \cap X_{\a}$ is empty, whereas for $q \in \{2,4\}$ we can check that there are no feasible parameters for $\cQ$.

\vs

\noindent \emph{Case 5. The remaining triples}

\vs

To complete the proof of the proposition, we need to eliminate the following triples:
\[
\mbox{\texttt{EO11}, \texttt{EO12},  \texttt{EO17},  \texttt{EO18}, \texttt{EO19}, \texttt{EO26}, \texttt{EO27},  \texttt{EO28}, \texttt{EO29}}
\]
In each case we have $X = {\rm P\O}_8^{+}(q)$ and we observe that our earlier work in Cases 1-4 means that we only need to consider the following triples:
\[
\texttt{EO12} \mbox{ ($q$ odd)}, \; \texttt{EO18}, \; \texttt{EO19} \mbox{ ($q$ odd)}.
\]

For example, we can immediately eliminate cases \texttt{EO11}, \texttt{EO28} and \texttt{EO29} by repeating the argument for \texttt{EO9} with $(n,i,\eta) = (8,3,-)$, which we handled in Case 1. Similarly, for the triples \texttt{EO17}, \texttt{EO26} and \texttt{EO27} we refer the reader to cases \texttt{EO16} (with $(n,\eta) = (8,-)$), \texttt{EO23} (with $(n,\eta) = (8,-)$) and \texttt{EO10}, respectively.

For \texttt{EO12} we observe that $v=b$ coincides with the values of $v$ and $b$ for \texttt{EO3} (with $(n,i,\delta) = (8,2,+)$), so the arithmetic argument for \texttt{EO3} in Case 4 for $p=2$ also eliminates \texttt{EO12} when $q$ is even. However, we used Lemma \ref{l:sylow} to handle \texttt{EO3} when $q$ is odd and this does not immediately imply the result for \texttt{EO12} without some additional reasoning because there are no unipotent elements in $X_{\a}$ or $X_{\ell}$ with Jordan form $(J_3,J_1^5)$ on $V$. Similarly, by appealing to \texttt{EO6}, we may assume $q$ is odd for the triple \texttt{EO19}.

First consider \texttt{EO18} and observe that we may view $X_{\ell}$ as the centralizer $C_X(\tau)$ of a suitable triality graph automorphism $\tau$ of $X$. Set $d = \gcd(2,q-1)$ and note that
$|X_{\a}|_p \leqs |X_{\ell}|_p$ and
\[
b = \frac{|\O_{8}^{+}(q)|}{d|G_2(q)|} = \frac{1}{d^2}q^6(q^4-1)^2.
\]

First assume $q$ is odd and let $g \in X_{\a}$ be a unipotent element with Jordan form $(J_3,J_1^5)$ on $V$. Since $g^{\tau}$ has Jordan form $(J_2^4)$ on $V$ (see \cite[Proposition 3.55]{FPR2}, for example), it follows that no $X$-conjugate of $g$ is centralized by $\tau$ and thus $g^X \cap X_{\ell}$ is empty. We can now appeal to Lemma \ref{l:sylow} since $b_p^5 \geqs 2b^2$.

Now assume $q$ is even. Here $\gcd(q^4+q^2+1,(q^2-1)^2)=3$ and
\[
h=3q^6(q^2+1)^2, \;\; v'=\frac{1}{12}q^2(q^4+q^2+1),\;\; b'=\frac{1}{3}(q^2-1)^2,
\]
whence Lemma \ref{lem_cubic} gives
\[
108<\frac{h}{v'b'}<k^3<\frac{2h}{v'b'}=216\left(\frac{q^4(q^2+1)^2}{(q^4+q^2+1)(q^2-1)^2}\right)< 480.
\]
Therefore, $k \in \{5,6,7\}$. Reducing the equation $F(k)=0$ modulo $2$, we deduce that $k$ is even and thus $k=6$. However, $q^2$ divides both $36v'b'$ and $h/6$, whereas $6v'\equiv q^2/2 \imod{q^2}$ and $6b'\equiv2\imod{q^2}$. Therefore,
\[
0=\frac{1}{6}F(6)=36v'b'-6(v'+b')+2-\frac{1}{6}h \equiv \frac{1}{2}q^2 \imod{q^2}
\]
and we have reached a contradiction. This eliminates the triple \texttt{EO18}.

Next let us turn to the case \texttt{EO12} with $q$ odd. Let $\tau \in {\rm Aut}(X)$ be a triality graph automorphism such that $X_{\a}^{\tau}$ is a $\C_1$-subgroup of type ${\rm O}_2^{+}(q) \perp {\rm O}_{6}^{+}(q)$ and $X_{\ell}^{\tau}$ is a $\C_2$-subgroup of type ${\rm GL}_{4}(q).2$. By our earlier work with \texttt{EO3} (see Case 4), we know that there exists an element $h = g^{\tau} \in X_{\a}^{\tau}$ of order $p$ such that $h^X \cap X_{\ell}^{\tau}$ is empty (indeed, we can take $h$ to be unipotent with Jordan form $(J_3,J_1^5)$). Therefore, by applying $\tau^{-1}$, we deduce that there exists an element $g \in X_{\a}$ of order $p$ such that $g^X \cap X_{\ell}$ is empty and we can now appeal to Lemma \ref{l:sylow} (more precisely, if $g \in X_{\a}$ has Jordan form $(J_2^4)$ on $V$, then $g^X \cap X_{\ell}$ is empty). An entirely similar argument handles the case \texttt{EO19}, with \texttt{EO6} in place of \texttt{EO3}.
\end{proof}

\vs

By combining Propositions \ref{p:linear_final}, \ref{p:unitary_final}, \ref{p:symp_final}, \ref{p:odd_orth_final} and \ref{p:even_orth_final} we conclude that the proof of Theorem \ref{t:main} for classical groups is complete.

\section{Proof of Theorem \ref{t:main}: Exceptional groups}\label{s:excep}

In this final section we complete the proof of Theorem \ref{t:main} by handling the almost simple exceptional groups. Let $\mathcal{Q}=(\mathcal{P},\mathcal{L})$ be a finite thick generalized quadrangle of order $(s,t)$ with an almost simple flag-transitive, point-primitive and line-primitive automorphism group $G$ with socle $X \in \mathcal{H}$ (see Table \ref{tab:H}). As usual, we assume $X$ is defined over $\mathbb{F}_q$, where $q = p^f$ and $p$ is a prime.

Fix a point $\a\in\cP$ and a line $\ell\in\cL$. As before, set
\[
v = |\cP| = [X:X_{\a}],\;\; b = |\cL| = [X:X_{\ell}],\;\; \theta = |{\rm Out}(X)|
\]
and recall that $|X|$ and $\theta$ are recorded in Table \ref{tab:H}. Our goal is to eliminate every triple $(G,G_{\a},G_{\ell})$ and our proof divides naturally into three parts, following the same strategy we used for classical groups in Sections \ref{s:subspace}-\ref{s:main1}.

In Section \ref{ss:ex_red}, we begin by working with the bound
\begin{equation}\label{e:xalpha}
|X_{\alpha}|> \theta^{-5/6}|X|^{1/6}
\end{equation}
from Lemma \ref{large}(ii) in order to restrict the set of candidate pairs $(G,G_{\a})$, as we did for classical groups in Section \ref{s:irred1}. This is a routine exercise for the groups with socle
\[
X \in \{ {}^3D_4(q), G_2(q), F_4(q), E_6(q), {}^2E_6(q) \},
\]
inspecting the list of maximal subgroups of $G$ available in \cite[Chapter 8]{BHR} (for $X = {}^3D_4(q)$, $G_2(q)$) and \cite{Craven1} (for $X = F_4(q)$, $E_6^{\e}(q)$). The case $X = E_7(q)$ is very similar, working with an almost complete list of maximal subgroups given in \cite{Craven2}, whereas some additional reasoning is required when $X = E_8(q)$ (see Lemma \ref{l:e8_red}) in order to reduce the problem to triples comprising known maximal subgroups. For a given family of pairs $(G,G_{\a})$ of the same type, if we find that \eqref{e:xalpha} fails for all but a finite set $\Sigma$ of prime powers $q$, then we can use the techniques in Remark \ref{r:c1} to show that there are no feasible parameters for $\cQ$ for each $q \in \Sigma$. Our main result for this initial reduction is Proposition \ref{p:ex_small}.

As a consequence, for the remainder of the proof we only need to consider triples $(G,G_{\a},G_{\ell})$ where $X_{\a}$ and $X_{\ell}$ are known in every case (in addition, by Lemma \ref{isomorphism}(i), we may assume $G_{\a}$ and $G_{\ell}$ are non-conjugate). In particular, this means that we can compute $v,b,|X_{\a}|,|X_{\ell}|$ and $\theta = |{\rm Out}(X)|$, so we are in a position to apply the computational methods presented in Remark \ref{rem:computational-strategy}, which can be implemented in {\sc Magma} \cite{magma}. In this way, we can reduce the problem to a much shorter list of candidate triples, which comprise the collection $\mathcal{F}$ in Table \ref{tab:F}; see Proposition \ref{p:ex_comp}.

Finally, in Section \ref{ss:ex_complete}, we use a combination of the techniques in Section \ref{ss:further} to eliminate all of the triples in $\mathcal{F}$, thereby completing the proof of Theorem \ref{t:main}.

\begin{rem}\label{r:ex_notation}
Let us record a few comments on our set-up in this section. For the purposes of this remark, we set $d = \gcd(2,q-1)$, $e=\gcd(3,q-1)$ and $e' = \gcd(3,q+1)$.
\begin{itemize}\addtolength{\itemsep}{0.2\baselineskip}
\item[{\rm (a)}] First recall that we may assume $q \geqs 3$ when $X=G_2(q)$, in view of the isomorphism $G_2(2)' \cong {\rm U}_3(3)$.

\item[{\rm (b)}] It will sometimes be convenient to use the notation $E_6^{\e}(q)$ with $\e=\pm$, where $E_6^{+}(q) = E_6(q)$ and $E_6^{-}(q)$ is the twisted group ${}^2E_6(q)$. In addition, we will write $E_6^{\rm sc}(q)$, ${}^2E_6^{\rm sc}(q)$ and $E_7^{\rm sc}(q)$ for the ``simply connected" versions of the simple groups $E_6(q)$, ${}^2E_6(q)$ and $E_7(q)$, i.e. the central extensions
\[
E_6^{\rm sc}(q) = e.E_6(q),\;\;  {}^2E_6^{\rm sc}(q) = e'.{}^2E_6(q),\;\; E_7^{\rm sc}(q) = d.E_7(q).
\]

\item[{\rm (c)}] As we did for classical groups, it will be convenient to refer to the \emph{type} of the maximal subgroups $G_{\a}$ and $G_{\ell}$, which provides a rough description of their structure. To do this, we will sometimes adopt the following Lie notation for classical groups:
\[
A_r(q) = A_r^{+}(q),\; A_r^{-}(q),\; B_r(q), \; C_r(q), \; D_r(q) = D_r^{+}(q), \;  D_r^{-}(q).
\]
For example, the simple group $G = E_7(q)$ has a maximal subgroup
\[
H = d^2\cdot({\rm P\O}_8^{+}(q) \times {\rm L}_2(q)^3)\cdot d^2\cdot S_3
\]
(see \cite[Table 4.1]{Craven2}) and we will refer to $H$ as a subgroup of type $D_4(q)A_1(q)^3$.

Similarly, if $G_{\a} = N_G(T)$ is the normalizer of a maximal torus $T$ of $X$, then the type of $G_{\a}$ provides a rough description of the structure of $T$. For instance, \cite[Table 10]{Craven1} indicates that if $q \geqs 3$ then $G = {}^2E_6(q)$ has a maximal subgroup
\[
H = (q^2-q+1)^3/e'.(3^{1+2}.{\rm SL}_2(3))
\]
and we will say that $H$ has type $(q^2-q+1)^3$.

\item[{\rm (d)}] We will adopt the standard $P_i$ notation for maximal parabolic subgroups, working with the standard labelling of the nodes of the associated Dynkin diagram (see Bourbaki \cite{Bou}, for example). In addition, for $X = E_6^{\e}(q)$ we will also refer to parabolic subgroups of type $P_{1,6}$ and $P_{3,5}$. Similarly, we use the notation $P_2$ and $P_{1,3,4}$ for representatives of the two conjugacy classes of maximal parabolic subgroups of the Steinberg triality group ${}^3D_4(q)$, where
\[
P_2 = [q^{9}]{:}((q-1) \circ {\rm SL}_2(q^3)).d,\;\; P_{1,3,4} = [q^{11}]{:}((q^3-1) \circ {\rm SL}_2(q)).d
\]
(see \cite[Table 8.51]{BHR}).
\end{itemize}
\end{rem}

\subsection{First reductions}\label{ss:ex_red}

In this section, we begin the proof of Theorem \ref{t:main} for exceptional groups by presenting two reduction results, which severely restrict the list of candidate triples we will need to consider in the next section. Indeed, our goal is to reduce the problem to the list of triples in the collection $\mathcal{F}$, which is presented in Table \ref{tab:F}.

\subsubsection{Small order reduction}\label{ss:ex_small}

As a first step, we appeal to the bounds in Lemma \ref{large}(ii), which allow us to exclude the triples $(G,G_{\a},G_{\ell})$ where $X_{\a}$ or $X_{\ell}$ is ``small" in the sense that
\[
|X_{\alpha}| \leqs \theta^{-5/6}|X|^{1/6} \; \mbox{ or } \; |X_{\ell}| \leqs \theta^{-5/6}|X|^{1/6}.
\]
Our main result is the following.

\begin{prop}\label{p:ex_small}
Let $\mathcal{Q} = (\mathcal{P},\mathcal{L})$ be a finite thick generalized quadrangle admitting an almost simple point-primitive and flag-transitive automorphism group $G$ with socle $X \in \mathcal{H}$.
\begin{itemize}\addtolength{\itemsep}{0.2\baselineskip}
\item[{\rm (i)}] If $G_{\a}$ is a subfield subgroup over $\mathbb{F}_{q_0}$, then $[\mathbb{F}_q : \mathbb{F}_{q_0}] \leqs 5$.
\item[{\rm (ii)}] Otherwise, either $X \ne E_8(q)$ and $(G,G_{\a})$ is one of the cases in Table \ref{tab:note8}, or $X = E_8(q)$ and $(G,G_{\a})$ is recorded in Table \ref{tab:e8}.
\end{itemize}
\end{prop}

\begin{rem}
Let us record a couple of comments on the statement of Proposition \ref{p:ex_small}.
\begin{itemize}\addtolength{\itemsep}{0.2\baselineskip}
\item[{\rm (a)}] With regard to the proof of Theorem \ref{t:main}, the  restrictions on the point stabilizer $G_{\a}$ obtained in the proposition also apply to the line stabilizer $G_{\ell}$, by point-line duality.
\item[{\rm (b)}] Suppose $X \ne E_8(q)$. In the second column of Table \ref{tab:note8} we refer to the relevant tables in \cite{BHR,Craven1,Craven2} where all the possibilities for $G_{\a}$ are recorded; in some cases, the given table depends on the characteristic
$p$ of $\mathbb{F}_q$, as indicated. And then in the final column we list all the candidates for $G_{\a}$ that we can eliminate, for all relevant values of $q$. So for example, if $X = F_4(q)$ and $p \geqs 3$, then the maximal subgroups of $G$ are listed in Tables 1 and 7 of \cite{Craven1} and we are stating that we only need to retain the subgroups in \cite[Table 7]{Craven1}, excluding those of type $A_1(q)$ and $3^3{:}{\rm SL}_3(3)$.
\item[{\rm (c)}] In both tables, we adopt the notational conventions described above in Remark \ref{r:ex_notation}.
\end{itemize}
\end{rem}

{\small
\begin{table}
\[
\begin{array}{llll} \hline
X & \mbox{Ref.} & & \mbox{Exclusions} \\ \hline
{}^3D_4(q) & \mbox{\cite[Table 8.51]{BHR}} & & (q^4-q^2+1),\, (q^2 \pm q +1)^2 \\
& & & \\
G_2(q) & \mbox{\cite[Table 8.41]{BHR}} \;\;\; p \geqs 5 & & 2^3.{\rm L}_3(2), \, {\rm L}_2(13), \, {\rm L}_2(8), \, {\rm U}_3(3), \, {\rm J}_1 \\
& \mbox{\cite[Table 8.42]{BHR}}\;\;\;  p =3 & & (q\pm 1)^2, \, q^2\pm q+1, \, 2^3.{\rm L}_3(2), \, {\rm L}_2(13) \\
& \mbox{\cite[Table 8.30]{BHR}}\;\;\;  p =2 & & {\rm L}_2(13), \, {\rm J}_2 \\
& & & \\
F_4(q) & \mbox{\cite[Table 1]{Craven1}}  & & \mbox{all} \\
& \mbox{\cite[Table 7]{Craven1}} \;\;\; p \geqs 3 & & A_1(q), \, 3^3{:}{\rm SL}_3(3) \\
& \mbox{\cite[Table 8]{Craven1}}\;\;\; p = 2 & & (q \pm 1)^4, \, (q^2 \pm q+1)^2, \, (q^2+1)^2, \, q^4-q^2+1 \\
& & & \\
E_6^{\e}(q) & \mbox{\cite[Tables 2, 3]{Craven1}} &&  \mbox{all} \\
& \mbox{\cite[Tables 9,10]{Craven1}} &  & (q-\e)^6, \, (q^2+\e q+1)^3, \, A_2^{\pm}(q),\, 3^{3+3}{:}{\rm SL}_3(3) \\
& &  & \\
E_7(q) & \mbox{\cite[Tables 1.1, 1.2]{Craven2}}  & & \mbox{all} \\
& \mbox{\cite[Table 4.1]{Craven2}}  & & (q\pm1)^7,\, A_1(q)^7, \, A_1(q^7),\, A_1(q)G_2(q), \, A_1(q)^2, \, A_2^{\pm}(q), \, A_1(q) \\ \hline
\end{array}
\]
\caption{The pairs $(G,G_{\a})$ arising in Proposition \ref{p:ex_small}(ii), $X \ne E_8(q)$}
\label{tab:note8}
\end{table}
}

{\small
\begin{table}
\[
\begin{array}{lll} \hline
\mbox{Type of $G_{\a}$} & |X_{\a}| & \mbox{Comments} \\ \hline
P_1 & q^{78}|{\rm Spin}_{14}^{+}(q)||{\rm GL}_1(q)| & \\
P_2 & q^{92}|{\rm GL}_{8}(q)| & \\
P_3 & q^{98}|{\rm GL}_{7}(q)||{\rm SL}_2(q)| & \\
P_4 & q^{106}|{\rm GL}_{5}(q)||{\rm SL}_3(q)||{\rm SL}_2(q)| & \\
P_5 & q^{104}|{\rm GL}_{5}(q)||{\rm SL}_4(q)| & \\
P_6 & q^{97}|{\rm Spin}_{10}^{+}(q)||{\rm GL}_3(q)| & \\
P_7 & q^{83}|E_6^{{\rm sc}}(q)||{\rm GL}_{2}(q)| & \\
P_8 & q^{57}|E_7^{{\rm sc}}(q)||{\rm GL}_1(q)| & \\
E_8(q_0) & |E_8(q_0)| & \mbox{$q = q_0^k$, $k \in \{2,3,5\}$} \\
D_8(q) & |{\rm Spin}_{16}^{+}(q)| & \\
A_1(q)E_7(q) & |{\rm SL}_2(q)||E_7^{\rm sc}(q)| & \\
A_8^{\e}(q) & 2|{\rm SL}_9^{\e}(q)| & \e = \pm \\
A_2(q)E_6(q) & 2|{\rm SL}_3(q)||E_6^{\rm sc}(q)| & \\
A_2^{-}(q)E_6^{-}(q) & 2|{\rm SU}_3(q)||{}^2E_6^{\rm sc}(q)| & \\
A_4^{\e}(q)^2 & 4|{\rm SL}_5^{\e}(q)|^2 & \e = \pm \\
A_4^{-}(q^2) & 4|{\rm SU}_5(q^2)| & \mbox{two classes} \\
D_4(q)^2 & 12|{\rm Spin}_8^{+}(q)|^2 & \\
D_4(q^2) & 12|{\rm Spin}_8^{+}(q^2)| & \\
{}^3D_4(q)^2 & 6|{}^3D_4(q)|^2 & \\
{}^3D_4(q^2) & 6|{}^3D_4(q^2)| & \\
G_2(q)F_4(q) & |G_2(q)||F_4(q)| & \\
F_4(q) & |F_4(q)| & p=3 \\ \hline
\end{array}
\]
\caption{The pairs $(G,G_{\a})$ arising in Proposition \ref{p:ex_small}(ii), $X = E_8(q)$}
\label{tab:e8}
\end{table}
}

\begin{lem}\label{l:e8_red}
The conclusion of Proposition \ref{p:ex_small} holds when $X = E_8(q)$.
\end{lem}

\begin{proof}
First recall that $q = p^f$, where $p$ is a prime, and we have $\theta = f$ and
\[
|X| = q^{120}(q^2-1)(q^8-1)(q^{12}-1)(q^{14}-1)(q^{18}-1)(q^{20}-1)(q^{24}-1)(q^{30}-1)
\]
as in Table \ref{tab:H}. If $G_{\a}$ is a subfield subgroup of type $E_8(q_0)$, where $q = q_0^k$ and $k \geqs 7$ is a prime, then
\[
|X_{\a}|^6 < q_0^{1488} < q^{213} <\theta^{-5}|X|
\]
and so the conclusion in part (i) of the proposition holds by Lemma \ref{large}(ii).

In order to complete the proof, we need to describe the maximal subgroups of $G$. To do this, it will be convenient to work in terms of the ambient simple algebraic group. So let $L$ be a simple algebraic group of type $E_8$ over an algebraically closed field of characteristic $p$ such that
\[
X = L_{\s} = \{ x \in L \,:\, x^{\s} = x\}
\]
for an appropriate Frobenius morphism $\s$ of $L$. Then according to \cite[Theorem 8]{LS03}, either $G_{\a}$ is a subfield subgroup or one of the following holds:
\begin{itemize}\addtolength{\itemsep}{0.2\baselineskip}
\item[{\rm (a)}] $G_{\a} = N_G(J_{\s})$, where $J$ is a maximal $\sigma$-invariant positive-dimensional closed subgroup of $L$.
\item[{\rm (b)}] $G_{\a}$ is the normalizer of an exotic $r$-local subgroup for some prime $r \neq p$.
\item[{\rm (c)}] $X_{\a} = (A_5 \times S_6).2$ with $p \geqs 7$.
\item[{\rm (d)}] $G_{\a}$ is almost simple and not of type (a).
\end{itemize}

In (b) and (c), it is easy to check that $|X_{\a}| \leqs \theta^{-5/6}|X|^{1/6}$. For example, in (b) we have $q \geqs 3$ and $X_{\a}$ is either $5^3.{\rm L}_3(5)$ or $2^{5+10}.{\rm L}_5(2)$ (see \cite[Theorem 1]{CLSS}) and the claim follows.

Next assume $G_{\a} = N_G(J_{\s})$ as in (a). The structure of the maximal parabolic subgroups of $G$ is well known and it is clear that the lower bound in \eqref{e:xalpha} is satisfied, so these possibilities are recorded in Table \ref{tab:e8}.

Now suppose $G_{\a}$ is a maximal rank subgroup, which means that $J$ contains a maximal torus of $L$. The maximal subgroups of this form were determined (up to conjugacy) by Liebeck, Saxl and Seitz in \cite{LSS} and by inspecting \cite[Tables 5.1, 5.2]{LSS} it is an easy exercise to read off all the cases satisfying the bound in \eqref{e:xalpha}. Note that if $q=2$ and $G_{\a}$ is the normalizer of a maximal torus of $G$ of order $3^8$, then
\[
|X_{\a}| = 3^8|W(E_8)| > \theta^{-5/6}|X|^{1/6}
\]
since the Weyl group of type $E_8$ has order $|W(E_8)| = 4|\O_8^{+}(2)| = 696729600$. However, in this special case we can check that there are no feasible parameters for $\cQ$, using the method in Remark \ref{r:c1}(a).

According to \cite[Theorem 1]{LS04}, the remaining possibilities for $J$ are as follows:
\[
\begin{array}{c}
G_2F_4, \; A_1G_2^2.2 \, (p \geqs 3), \; F_4 \, (p=3), \; A_1A_2.2 \, (p \geqs 5), \\ [2pt]
B_2 \, (p \geqs 5), \, A_1 \, (\mbox{$3$ classes; $p \geqs 23,29,31$}), \, A_1 \times S_5 \, (p \geqs 7).
\end{array}
\]
Here we include the special case $(J,p) = (F_4,3)$ from \cite{CST}, which was incorrectly omitted in \cite{LS04}. In each of these cases, it is easy to show that \eqref{e:xalpha} holds if and only if $G_{\a}$ is of type $G_2(q)F_4(q)$ or $F_4(q)$, and so both of these possibilities are included in Table \ref{tab:e8}.

To complete the proof of the lemma, we will show that $|X_{\a}| \leqs \theta^{-5/6}|X|^{1/6}$ for all of the almost simple subgroups arising in (d).

It remains an open problem to determine all the maximal subgroups of this form, even up to isomorphism, although there are strong restrictions on the possibilities that can arise. If $Y$ is the socle of $G_{\a}$, then by combining the work of several authors \cite{Craven0,Craven_alt,LS99,LS98,Litt} with Craven's ongoing work \cite{Craven3}, we know that one of the following holds (here ${\rm Lie}(p)$ denotes the set of finite simple groups of Lie type in characteristic $p$):
\begin{itemize}\addtolength{\itemsep}{0.2\baselineskip}
\item[(i)] $Y \in {\rm Lie}(p)$ and either $Y = {\rm L}_2(q_0)$ with $7 \leqs q_0 \leqs \gcd(2,p-1)\cdot 1312$, or
\[
Y \in \{ {\rm L}_3^{\e}(3), \, {\rm L}_3^{\e}(4), \, {\rm U}_3(8), \, {\rm U}_4(2), \, {}^2B_2(8) \}.
\]

\item[(ii)] $Y \not\in {\rm Lie}(p)$ and either

\vspace{1.5mm}

\begin{itemize}\addtolength{\itemsep}{0.2\baselineskip}
\item[(a)] $Y = A_6$ $(p \ne 5)$, $A_7$ $(p \ne 2)$, ${\rm M}_{11}$ $(p =3,11)$, ${\rm J}_3$ $(p=2)$ or ${\rm Th}$ $(p=3)$;

\item[(b)] $Y = {\rm L}_2(q')$ with $q' \in \{7,8,11,13,16,17,19,25,29,31,32,41,49,61\}$; or

\item[(c)] $Y = {\rm L}_3(3)$, ${\rm L}_3(5)$, ${\rm L}_4(5)$, ${\rm U}_3(3)$, ${\rm U}_4(2)$, ${\rm PSp}_4(5)$, ${}^2B_2(8)$, ${}^2B_2(32)$, ${}^3D_4(2)$ or ${}^2F_4(2)'$.
\end{itemize}
\end{itemize}

It is now an easy exercise to check that $|X_{\a}| \leqs \theta^{-5/6}|X|^{1/6}$ in every case. For example, suppose $Y = {\rm L}_2(q_0)$ with $q_0 = p^j$, so we have
\[
|X_{\a}| \leqs |{\rm Aut}(Y)| = jq_0(q_0^2-1)
\]
and $q_0 \leqs \gcd(2,p-1) \cdot 1312$. We then check that the given upper bound on $|X_{\a}|$ is at most $p^{40}$ for all the relevant pairs $(p,q_0)$ that can arise. And since $q = p^f$, we conclude that
\[
|X_{\a}|^6 \leqs p^{240} \leqs f^{-5}|X|,
\]
as required.
\end{proof}

\begin{lem}\label{l:note8_red}
The conclusion of Proposition \ref{p:ex_small} holds when $X \ne E_8(q)$.
\end{lem}

\begin{proof}
As in the proof of the previous lemma, it is an easy exercise to verify the claim in part (i) of the proposition concerning subfield subgroups. For example, if $X = G_2(q)$ with $q=q_0^k$, $p=3$ and $k \geqs 7$, then $\theta = 2f$ and we have
\[
|X_{\a}|^6 = |G_2(q_0)|^6 < q^{12} < \theta^{-5} \cdot \frac{1}{2}q^{14} < \theta^{-5}|X|.
\]
For the remainder, let us assume $G_{\a}$ is not a subfield subgroup.

First assume $X = E_7(q)$ and recall that $\theta = df$ and
\[
|X| = \frac{1}{d}q^{63}(q^2-1)(q^6-1)(q^8-1)(q^{10}-1)(q^{12}-1)(q^{14}-1)(q^{18}-1),
\]
where $q = p^f$ and $d = \gcd(2,q-1)$.

In \cite{Craven2}, the known core-free maximal subgroups of $G$ are listed in \cite[Tables 1.1, 4.1]{Craven2}. This is complemented by \cite[Table 1.2]{Craven2}, which records a small additional collection of candidate maximal subgroups; by the main theorem of \cite{Craven2}, we know that every possible maximal subgroup of $G$ is recorded in one of these three tables. It is now a routine exercise to work through each case arising in these tables, checking the status of the required inequality in \eqref{e:xalpha}. In this way, we deduce that all of the possibilities in Tables 1.1 and 1.2 of \cite{Craven2} can be discarded, as indicated in Table \ref{tab:note8}. And by inspecting \cite[Table 4.1]{Craven2}, reading off the order of $X_{\a}$ in each case, we can rule out the possibilities for $G_{\a}$ listed in the third column of Table \ref{tab:note8}.

For example, suppose $G_{\a}$ is a maximal rank subgroup of type $A_1(q^7)$. From \cite[Table 4.1]{Craven2} we see that
\[
|X_{\a}| = 7|{\rm L}_2(q^7)| = \frac{7}{d}q^7(q^{14}-1)
\]
and we now check that \eqref{e:xalpha} holds if and only if $q \in \{2,3,4,5,8\}$. And for each of these specific values of $q$, we can use the method in Remark \ref{r:c1}(b) to conclude that there are no feasible parameters for $\cQ$. Similarly, if $G_{\a}$ is of type $A_1(q)^7$ then we find that \eqref{e:xalpha} fails for all $q >256$ and once again we can rule out the remaining possibilities for $q$ via Remark \ref{r:c1}(b).

All of the remaining groups can be handled in an entirely similar fashion, working with the complete lists of maximal subgroups available in \cite{BHR,Craven1,Craven2}, as recorded in the second column of Table \ref{tab:note8}. No special difficulties arise in any of these cases and so we omit the details.
\end{proof}

This completes the proof of Proposition \ref{p:ex_small}.

\subsubsection{Arithmetic reduction}\label{ss:ex_arith}

For the remainder, we may assume that the pairs $(G,G_{\a})$ and $(G,G_{\ell})$ satisfy the conditions in Proposition \ref{p:ex_small}. In each of these cases we can compute
\[
|X|, \; |{\rm Out}(X)|, \; |X_{\a}|, \; |X_{\ell}|
\]
precisely and so we are in a position to apply the computational methods
presented in Remark \ref{rem:computational-strategy}. We refer the reader to Examples \ref{computational-example}-\ref{computational-example-exceptional} in Section \ref{ss:comp} to see how these techniques can be applied in a concrete setting. In particular, see Example \ref{computational-example-exceptional} for a case
involving the exceptional group $X = G_2(q)$.

In this way, we obtain Proposition \ref{p:ex_comp} below, which reduces the proof of Theorem \ref{t:main} for exceptional groups to the smaller collection of triples recorded in Table \ref{tab:F}. We will refer to the latter as collection $\mathcal{F}$.

{\small
\begin{table}
\[
\begin{array}{lcllll} \hline
X & \mbox{Label} & \mbox{Type of $G_{\a}$} & \mbox{Type of $G_{\ell}$} & \mbox{Conditions} \\ \hline
G_2(q) & \texttt{F1} & A_1(q)A_1(q)& {}^2G_2(q)&\mbox{$q = 3^f$, $f$ odd}\\
 &  \texttt{F2}  &A_1(q)A_1(q)& G_2(q_0)&\\
{}^3D_4(q)& \texttt{F3} & P_2&P_{1,3,4}&\\
       &\texttt{F4} & A_1(q^3)A_1(q) & G_2(q) &\\
       & \texttt{F5} & A_1(q^3)A_1(q) & {}^3D_4(q_0) &\\
F_4(q) & \texttt{F6} & P_1&P_4\\
       & \texttt{F7} & {}^3D_{4}(q) & D_4(q) &\\
       & \texttt{F8} & {}^3D_4(q) & {}^2F_4(q) & \mbox{$q = 2^f$, $f$ odd} \\
       & \texttt{F9} &C_2(q^2) & {}^2F_4(q) & \mbox{$q = 2^f$, $f$ odd}\\
       & \texttt{F10} & C_2(q^2) & F_4(q_0)   &\mbox{$q$ even}\\
E_6(q) & \texttt{F11} & P_2&P_{1,6}&\\
& \texttt{F12} &  D_5(q) \times (q-1) & F_4(q) &\\
& \texttt{F13} & D_5(q) \times (q-1) & E_6(q_0) &\\
& \texttt{F14} & D_5(q) \times (q-1) & {}^2E_6(q_0)& \\
       & \texttt{F15} &  {}^3D_4(q)\times (q^2+q+1)& C_4(q)& \mbox{$q$ odd}\\
       & \texttt{F16} & {}^3D_4(q) \times (q^2+q+1)&E_6(q_0)&\\
       & \texttt{F17} & {}^3D_4(q) \times (q^2+q+1)&{}^2E_6(q_0)&\\
       & \texttt{F18} & C_4(q) & E_6(q_0) &\mbox{$q$ odd}\\
       & \texttt{F19} & C_4(q) & {}^2E_6(q_0)&\mbox{$q$ odd}\\
       & \texttt{F20} & {}^2E_6(q_0) & E_6(q_0)&\\
{}^2E_6(q) & \texttt{F21}  & P_2 & P_{1,6} &\\
       & \texttt{F22} & P_4 & P_{1,6} &\\
       & \texttt{F23} & D_5^{-}(q) \times (q+1)&F_4(q)&\\
       & \texttt{F24} & {}^3D_4(q) \times (q^2-q+1)& C_4(q)& \mbox{$q$ odd} \\
E_7(q) & \texttt{F25} & P_1&P_7&\\
       & \texttt{F26} & P_2&P_6&\\
       & \texttt{F27} & A_1(q)F_4(q) & E_7(q_0) & \\
E_8(q) & \texttt{F28} & P_1&P_7&\\
       & \texttt{F29} & P_2&P_7&\\ \hline
\end{array}
\]
\caption{The collection $\mathcal{F}$}
\label{tab:F}
\end{table}}

\begin{rem}
In Table \ref{tab:F}, we adopt the notation and terminology introduced earlier in Remark \ref{r:ex_notation}. In addition, we set $q_0 = q^{1/2}$. For the cases with $X \ne E_8(q)$, the precise structure of $X_{\a}$ and $X_{\ell}$ can be read off from the relevant tables in \cite{BHR,Craven1,Craven2}. And for $X = E_8(q)$, we refer the reader to Table \ref{tab:e8} for the orders of $X_{\a}$ and $X_{\ell}$.
\end{rem}

\begin{prop}\label{p:ex_comp}
Let $\mathcal{Q} = (\mathcal{P},\mathcal{L})$ be a finite thick generalized quadrangle admitting an almost simple point-primitive, line-primitive and flag-transitive automorphism group $G$ with socle $X \in \mathcal{H}$. Then up to duality, $(G,G_{\a},G_{\ell}) \in \mathcal{F}$ is one of the cases in Table \ref{tab:F}.
\end{prop}

\begin{proof}
This is simply a matter of working through each remaining family of candidate triples $(G,G_{\a},G_{\ell})$ in turn, using a {\sc Magma} implementation of Methods 1, 2 and 3 in Remark \ref{rem:computational-strategy}. As before, if one of these methods is only effective when $q$ is sufficiently large, then each of the remaining possibilities for $q$ can be eliminated by appealing to the relevant inequalities in Lemma \ref{HD}(v) or Lemma \ref{inequalities}, or by using the techniques in Remark \ref{r:c1} to show that there are no feasible parameters for $\cQ$.

This approach is illustrated in Examples \ref{computational-example}-\ref{computational-example-exceptional}, and all of the other cases can be handled in a similar manner. Since no special difficulties arise, we omit the details.
\end{proof}

\subsection{Completing the proof}\label{ss:ex_complete}

In this final section of the paper, we complete the proof of Theorem \ref{t:main}. To do this, it suffices to eliminate all of the triples $(G,G_{\a},G_{\ell})$ in the collection $\mathcal{F}$ (see Table \ref{tab:F}), which will be achieved by combining some of the techniques we presented in Section \ref{ss:further}. Our main result is the following:

\begin{thm}\label{t:ex}
Let $\mathcal{Q} = (\mathcal{P},\mathcal{L})$ be a finite thick generalized quadrangle admitting an almost simple point-primitive, line-primitive and flag-transitive automorphism group $G$ with socle $X$. Then $X$ is not one of the exceptional groups of Lie type in $\mathcal{H}$.
\end{thm}

We now divide the analysis of the triples in $\mathcal{F}$ into several subcases, according to the specific techniques we will use to rule them out. We will refer freely to the labels \texttt{F1}-\texttt{F29} appearing in Table \ref{tab:F}.

\begin{lem}\label{l:ex_e6e7e8}
The triple $(G,G_{\a},G_{\ell})$ is not one of the cases labelled {\rm \texttt{F11}}-{\rm \texttt{F29}} in Table \ref{tab:F}.
\end{lem}

\begin{proof}
Here we are handling all of the triples where $X$ is either $E_6(q)$, ${}^2E_6(q)$, $E_7(q)$ or $E_8(q)$. Let $g \in X$ be a long root element. In other words, $g$ is a unipotent element in the class labelled $A_1$ in \cite[Tables 22.2.1-22.2.3]{LS_book} and we read off
\begin{equation}\label{e:long}
|g^X| = \left\{ \begin{array}{ll}
(q^3+\e)(q^4+1)(q^6+1)(q^9-\e) & \mbox{if $X = E_6^{\e}(q)$} \\
(q^8+q^4+1)(q^{12}+q^6+1)(q^{14}-1) & \mbox{if $X = E_7(q)$} \\
(q^6+1)(q^{10}+1)(q^{12}+1)(q^{30}-1) & \mbox{if $X = E_8(q)$}
\end{array}\right.
\end{equation}

Consider the triple labelled \texttt{F11}. Here $X = E_6(q)$ and $G_{\a}$ and $G_{\ell}$ are maximal parabolic subgroups of type $P_2$ and $P_{1,6}$, respectively (note that the maximality of $G_{\ell}$ implies that $G$ contains graph or graph-field automorphisms of $X$). By inspecting \cite[Table 9]{Craven1} we compute
\[
v = \frac{|E_6^{\rm sc}(q)|}{q^{21}|{\rm GL}_6(q)|} = (q^{6}+q^3+1)(q^6+1)(q^4+q^2+1)(q^4+1)(q+1)
\]
and
\[
b = \frac{|E_6^{\rm sc}(q)|}{q^{24}|{\rm Spin}_8^{+}(q)|(q-1)^2} = (q^8+q^4+1)(q^4+1)(q^5-1)(q^9-1)/(q-1)^2.
\]

The corresponding permutation characters for the action of $X$ on $\cP$ and $\cL$ are studied by Lawther, Liebeck and Seitz in \cite[Section 2]{LLS}. By adopting their character-theoretic methods, we can calculate the exact number of fixed points of $g$ on $\cP$ and $\cL$ (in particular, see \cite[Proposition 2.1]{LLS}) and we get
\begin{align*}
|\cP_g| = &  \; q^{15}+2q^{14}+3q^{13}+4q^{12}+5q^{11}+6q^{10}+5q^9+5q^8+5q^7+4q^6 \\
&  +3q^5+3q^4+2q^3+q^2+q+1, \\
|\cL_g|  = &  \; q^{18}+2q^{17}+5q^{16}+7q^{15}+11q^{14}+14q^{13}+18q^{12}+18q^{11}+19q^{10} \\
& +18q^{9}+17q^{8}+13q^{7}+11q^{6}+9q^{5}+7q^{4}+4q^{3}+3q^{2}+2q+1.
\end{align*}
It is now routine to check that $|\cL_g|^3 > 4vb$ and $|\cP_g|^3 \geqs 16b^2/v$ for all $q \geqs 2$. Therefore, the triple $(G,G_{\a},G_{\ell})$ is eliminated by Lemma \ref{l:large}(ii).

The triples corresponding to the cases \texttt{F21}, \texttt{F22}, \texttt{F25}, \texttt{F26}, \texttt{F28} and \texttt{F29} can all be eliminated in the same way, working with Lemma \ref{l:large} and a long root element $g \in X$. In each case, we can read off $v$ and $b$ from the relevant tables in \cite{Craven1,Craven2}, referring to Table \ref{tab:e8} when $X = E_8(q)$. If we write $\O = X/H$, where $H$ is any one of the relevant maximal parabolic subgroups of $X$, then a lower bound on $|\O_g|$ is presented in Table \ref{tab:longroot}. In every case, it is now easy to check that these bounds are sufficient to show that the pair $(|\cP_g|,|\cL_g|)$ is incompatible with Lemma \ref{l:large}, which eliminates every triple of this form.

\begin{table}
\[
\begin{array}{llccllccllc} \hline
X & H & & & & & & & & & \\ \hline
{}^2E_6(q) & P_2 & q^{15} & & E_7(q) & P_1 & q^{25} & & E_8(q) & P_1 & q^{60} \\
{}^2E_6(q) & P_4 & q^{19} & & E_7(q) & P_2 & q^{31} & & E_8(q) & P_2 & q^{70} \\
{}^2E_6(q) & P_{1,6} & q^{18} & & E_7(q) & P_6 & q^{32} & & E_8(q) & P_7 & q^{65} \\
& & & & E_7(q) & P_7 & q^{21} & & & & \\ \hline
\end{array}
\]
\caption{Lower bounds on $|\cP_g|$, $g \in X$ long root element}
\label{tab:longroot}
\end{table}

Next let us turn to the case labelled \texttt{F12}. Here $X = E_6(q)$ and $G_{\a}$ and $G_{\ell}$ are subgroups of type $D_5(q) \times (q-1)$ and $F_4(q)$, respectively. Once again, we can eliminate this triple via Lemma \ref{l:large}, but here we calculate $|\cP_g|$ and $|\cL_g|$ via Lemma \ref{l:fpr}. To do this, we simply observe that $|g^X \cap X_{\a}|$ and $|g^X \cap X_{\ell}|$ coincide with the number of long root elements in the groups $\O_{10}^{+}(q)$ and $F_4(q)$, respectively. Indeed, this is immediate from the construction of $X_{\a}$ and $X_{\ell}$ as subgroups of $X$ (also see \cite[Proposition 1.13]{LLS0}). The long root elements in $\O_{10}^{+}(q)$ correspond to the unipotent elements with Jordan form $(J_2^2,J_1^6)$ on the natural $10$-dimensional module, which are of type $a_2$ when $p=2$ (with respect to the standard notation of Aschbacher and Seitz \cite{AS} for unipotent involutions). Therefore,
\[
|g^X \cap X_{\a}| = \frac{|\O_{10}^{+}(q)|}{q^{13}|\O_6^{+}(q)||{\rm Sp}_2(q)|} = (q^2+1)(q^3+1)(q^4+1)(q^5-1)
\]
and by inspecting \cite[Table 22.2.4]{LS_book} we get
\[
|g^X \cap X_{\ell}| =\frac{|F_4(q)|}{q^{15}|{\rm Sp}_6(q)|} = (q^4+1)(q^{12}-1).
\]
Then using the expression for $|g^X|$ in \eqref{e:long}, we now compute
\begin{equation}\label{e:pl}
|\cP_g| = v\left(\frac{|g^X \cap X_{\a}|}{|g^X|}\right),\;\; |\cL_g| = b\left(\frac{|g^X \cap X_{\ell}|}{|g^X|}\right)
\end{equation}
and it is straightforward to check that $|\cP_g|^3 > 4vb$ and $|\cL_g|^3 \geqs 16v^2/b$, so this case is eliminated by Lemma \ref{l:large}(i).

The cases \texttt{F13}-\texttt{F20}, \texttt{F23} and \texttt{F24} can all be eliminated in exactly the same way and no complications arise. Indeed, the number of long root elements in $E_6^{\pm}(q_0)$ and $E_7(q_0)$ can be calculated from \eqref{e:long}, while the respective number of such elements in the groups ${}^3D_4(q)$, ${\rm Sp}_8(q)$, $\O_{10}^{-}(q)$ and $F_4(q)$ is as follows:
\[
(q^2-1)(q^8+q^4+1), \, q^8-1,\,  (q^2+1)(q^3-1)(q^4+1)(q^5+1),\, (q^4+1)(q^{12}-1).
\]
The case \texttt{F27} can also be handled in the same way, noting that
$|g^X \cap X_{\a}|$ coincides with the number of long root elements in $F_4(q)$ (in this case, there are no long root elements in the $A_1(q)$ factor of $X_{\a}$).
\end{proof}

\begin{lem}\label{l:ex_f4}
The triple $(G,G_{\a},G_{\ell})$ is not one of the cases labelled {\rm \texttt{F6}}-{\rm \texttt{F10}} in Table \ref{tab:F}.
\end{lem}

\begin{proof}
Here $X = F_4(q)$ and we begin by considering cases \texttt{F6}, \texttt{F7} and \texttt{F8}. As in the proof of the previous lemma, let $g \in X$ be a long root element, so
\[
|g^X| =  \frac{|F_4(q)|}{q^{15}|{\rm Sp}_6(q)|} = (q^4+1)(q^{12}-1).
\]

In \texttt{F6}, the stabilizers are maximal parabolic subgroups and we can proceed as in the proof of the previous lemma, noting that $|\cP_g| > q^{10}$ and $|\cL_g|>q^{11}$ via \cite[Proposition 2.1]{LLS}, which is incompatible with Lemma \ref{l:large}(ii).

Similarly, in \texttt{F7} we observe that $|g^X \cap X_{\a}|$ and $|g^X \cap X_{\ell}|$ coincide with the number of long root elements in ${}^3D_4(q)$ and $\O_8^{+}(q)$, respectively. Therefore,
\[
|g^X \cap X_{\a}| = \frac{|{}^3D_4(q)|}{q^9|{\rm SL}_2(q^3)|} =  (q^2-1)(q^8+q^4+1)
\]
and
\[
|g^X \cap X_{\ell}| = \frac{|\O_8^{+}(q)|}{q^{9}|\O_4^{+}(q)||{\rm Sp}_2(q)|} = (q^2+1)^2(q^6-1).
\]
We can now compute $|\cP_g|$ and $|\cL_g|$ via \eqref{e:pl}, and once again the triple $(G,G_{\a},G_{\ell})$ is eliminated by appealing to Lemma \ref{l:large}.

Now let us turn to case \texttt{F8}, so $q=2^f$ with $f$ odd and we have
\[
v = \frac{|F_4(q)|}{3|{}^3D_4(q)|} = \frac{1}{3}q^{12}(q^4-1)(q^8-1)
\]
and
\[
b = \frac{|F_4(q)|}{|{}^2F_4(q)|} = q^{12}(q+1)(q^3-1)(q^4+1)(q^6-1).
\]
Here $|g^X \cap X_{\a}| = (q^8+q^4+1)(q^2-1)$ and we observe that $g^X \cap X_{\ell}$ is empty (indeed, a graph automorphism of $X$ interchanges long and short roots, so it does not centralize a long root element). It is easy to check that $|\cP_g|^3 \geqs vb$ if $q \geqs 8$ and we can apply Lemma \ref{l:der}. And for $q=2$ we check that there are no feasible parameters for $\cQ$ using the method in Remark \ref{r:c1}(a).

Finally, let us consider cases \texttt{F9} and \texttt{F10}. Here $q$ is even and we take $g \in X$ to be an involution in the class labelled $(\tilde{A}_1)_2$ in \cite[Table 22.2.4]{LS_book}, so we have
\[
|g^X| = \frac{|F_4(q)|}{q^{20}|{\rm Sp}_4(q)|} = (q^4+1)(q^6-1)(q^{12}-1).
\]

First assume we are in case \texttt{F9}, so
\[
v = \frac{|F_4(q)|}{2|{\rm Sp}_4(q^2)|} = \frac{1}{2}q^{16}(q^2-1)(q^6-1)(q^8+q^4+1)
\]
and $b$ is given as above for \texttt{F8}. By inspecting the proof of \cite[Lemma 4.19]{BT} we see that
\[
|g^X \cap X_{\a}| = 2(q^8-1)+\frac{|{\rm Sp}_4(q^2)|}{|{\rm Sp}_4(q)|} = (q^2+1)(q^4+1)(q^4+2q^2-2).
\]
Now, there are two classes of involutions in $X_{\ell} = {}^2F_4(q)$ and they have the same labels as their corresponding classes in $X$. Therefore, by inspecting \cite[Table 22.2.5]{LS_book}, we deduce that
\[
|g^X \cap X_{\ell}| = \frac{|{}^2F_4(q)|}{q^{10}|{}^2B_2(q)|} = (q^2-1)(q^3+1)(q^6+1)
\]
and it is straightforward to check that $|\cP_g|^3 > 4vb$ and $|\cL_g|^3 \geqs 16v^2/b$ for all $q \geqs 8$, so we can conclude via Lemma \ref{l:large}. And as in the previous case, we check that there are no feasible parameters for $\cQ$ when $q=2$.

The case \texttt{F10} is very similar. Here $v$ and $|g^X \cap X_{\a}|$ are the same as the previous case, and we observe that
\[
|g^X \cap X_{\ell}| =(q^2+1)(q^3-1)(q^6-1)
\]
is the number of involutions in the $(\tilde{A}_1)_2$ class of $F_4(q_0)$. We now check that $|\cP_g|^3 > 4vb$ and $|\cL_g|^3 \geqs 16v^2/b$ for all $q \geqs 4$ and the result follows.
\end{proof}

\begin{lem}\label{l:ex_final}
The triple $(G,G_{\a},G_{\ell})$ is not one of the cases labelled {\rm \texttt{F1}}-{\rm \texttt{F5}} in Table \ref{tab:F}.
\end{lem}

\begin{proof}
Here $X = G_2(q)$ or ${}^3D_4(q)$. First consider case \texttt{F1}, so $X = G_2(q)$, $q=3^f$ with $f$ odd and we have
\[
|X_{\a}| = q^2(q^2-1)^2,\;\; |X_{\ell}| = q^3(q-1)(q^3+1), \;\; b = q^3(q+1)(q^3-1),
\]
so $|X_{\a}|_p < |X_{\ell}|_p$. Let $g \in X_{\a}$ be a long root element (note that we can take $g$ to be any nontrivial unipotent element in the ${\rm SL}_2(q)$ factor generated by long root subgroups). Then $g^X \cap X_{\ell}$ is empty (indeed, a graph automorphism of $X$ interchanges long and short roots, so it does not centralize $g$) and we calculate that
$b_p^5 \geqs 2b^2$ for all $q \geqs 27$. So for $q \geqs 27$, the result
follows from Lemma \ref{l:sylow}(i). And for $q=3$ we calculate $v = 7371$ and it is easy to check that there are no feasible parameters for $\cQ$.

We will eliminate the final four cases using arithmetic arguments based on Lemma \ref{lem_cubic}. As in the lemma, set $h = \gcd(v,b)$, $v' = v/h$, $b' = b/h$ and $k = h/(st+1)$, so $k$ is a positive integer and $F(k)=0$, where $F(x) \in \Z[x]$ is the polynomial defined in \eqref{eqn_f}.

Consider case \texttt{F2}. Here
\[
v = \frac{|G_2(q)|}{|{\rm SL}_2(q)|^2} = q^4(q^4+q^2+1),\;\; b = \frac{|G_2(q)|}{|G_2(q_0)|} = q^3(q+1)(q^3+1)
\]
and we compute
\[
h= q^3(q^2-q+1),\;\; v'=q(q^2+q+1),\;\; b'= (q+1)^2.
\]
Since
\[
\frac{h}{v'b'}=\frac{q^2(q^2-q+1)}{(q+1)^2(q^2+q+1)}<1,
\]
Lemma \ref{lem_cubic} implies that $k^3<2$ and thus $k=1$. However, we have $F(1)\equiv1\imod{q}$, which is incompatible with the fact that $F(k)=0$.

In case \texttt{F3}, we have
\[
v = \frac{|{}^3D_4(q)|}{q^{9}|{\rm SL}_2(q^3)|(q-1)} = (q+1)(q^8+q^4+1)
\]
and
\[
b = \frac{|{}^3D_4(q)|}{q^{11}|{\rm SL}_2(q)|(q^3-1)} = (q^3+1)(q^8+q^4+1),
\]
which means that
\[
h=(q+1)(q^8+q^4+1),\;\; v' = 1,\;\; b'= q^2-q+1.
\]
Since $F(q+1) < 0$ and $F(q^3)>0$ for all $q\geqs2$, Lemma \ref{lem_cubic} implies that $q+1<k<q^3$. By reducing the equation $F(k)=0$ modulo $q^8$, we obtain
\[
(k-q-1)a\equiv0\imod{q^8},
\]
where $a=(q^2-q+1)k^2+(q-1)(q^2+1)k+q^4+1$.

Now, if $p$ divides $k-q-1$, then $k\equiv1\imod{p}$ and thus $a\equiv1\imod{p}$. Therefore, the above congruence implies that $q^8$ divides $k-q-1$ or $a$. However, the bound $k<q^3$ implies that both $k-q-1$ and $a$ are less than $q^8$, which is a contradiction.

Next consider case \texttt{F4}. Here
\[
v = \frac{|{}^3D_4(q)|}{|{\rm SL}_2(q^3)||{\rm SL}_2(q)|} = q^8(q^8+q^4+1),\;\; b = \frac{|{}^3D_4(q)|}{|G_2(q)|} = q^6(q^8+q^4+1)
\]
and $h= q^6(q^8+q^4+1)$, $v'=q^2$ and $b' = 1$. In the notation of Remark \ref{r:approx}, we have $z=q$ and $r=1$. Moreover, $f_1(z)=z^2$, $f_2(z)=1$ and $f_3(z)=z^6(z^8+z^4+1)$, so $m=4$, $\lambda=1$ and we construct the polynomial  $K(z)=z^4+\frac{2}{3}$. Since $F(q^4)<0$ and $F(q^4+1)>0$ for all $q \geqs 2$, we deduce that $q^4<k<q^4+1$. But this contradicts the fact that $k$ is an integer.

Finally, we turn to case \texttt{F5}. Here $v = q^8(q^8+q^4+1)$ as in \texttt{F4} and
\[
b = \frac{|{}^3D_4(q)|}{|{}^3D_4(q_0)|} = q^6(q^4-q^2+1)(q^3+1)(q+1),
\]
where $q = q_0^2$. Therefore, $h=q^6(q^2-q+1)(q^4-q^2+1)$, $v' = q^2(q^2+q+1)$ and $b' = (q+1)^2$. If $q=4$ then it is easy to check that there are no feasible parameters for $\cQ$, so we may assume $q \geqs 9$.

Now $(b_p/2_p)^5 \geqs 2b^2$ and thus Lemma \ref{inequalities}(ii) implies that $(t+1)_p\leqs2_p$. Since $t+1=kb'$ and $b'$ is coprime to $p$, we deduce that $k_p\leqs2_p$. If $p=2$, then $b'$ is odd and both $v'$ and $h$ are even, so reducing $F(k)=0$ modulo $2$ gives $k^2\equiv0\imod{2}$. Thus $k$ is even and we conclude that $k_p=2_p$ in all cases. In addition, we calculate that
\[
F(q^2-4q+2)<0,\;\; F(2q^2-4q+2)>0
\]
for all $q\geqs9$, so Lemma \ref{lem_cubic} gives
\[
q^2-4q+2<k<2q^2-4q+2
\]
and we deduce that $k \not\equiv 2-4q \imod{q^2}$.

If $p$ is odd, then reducing $F(k)=0$ modulo $q^2$ yields
\[
k(2-(1+2q)k)\equiv0\imod{q^2}
\]
and thus $k\equiv2-4q\imod{q^2}$, a contradiction. Now assume $p=2$ and write $k=2k'$, where $k'$ is odd. Reducing the equation $\frac{1}{4}F(2k')=0$ modulo $q^2$ gives
\[
k'(1-(1+2q)k')\equiv0\imod{q^2},
\]
so $k'\equiv1-2q\imod{q^2}$ and we deduce that $k\equiv2-4q\imod{q^2}$. This final contradiction completes the proof of the lemma.
\end{proof}

\vs

This completes the proof of Theorem \ref{t:ex}. By combining this with Theorem \ref{t:class}, together with the earlier work in \cite{Arumugam_thesis, Arumugam,Feng,Lu,Ou} (see Remark \ref{r:earlier}), we conclude that the proof of Theorem \ref{t:main} is complete. Finally, Theorem \ref{c:main} follows by appealing to the reduction to almost simple groups of Lie type established in \cite{Bamberg2012}.

\end{document}